\documentclass[11pt,english]{smfart} 
 \usepackage[english,francais]{babel}
 \usepackage{amscd}
 \usepackage{amssymb,url,xspace,smfthm}
\usepackage{amsbsy,bm} 
\usepackage{paralist}
\usepackage{datetime} 
\usepackage{colortbl}
\usepackage{color}
\usepackage[dvipsnames]{xcolor}
\usepackage{mathrsfs}  
\usepackage[colorlinks,	
linkcolor=BlueViolet,
citecolor=Rhodamine,
urlcolor=YellowGreen,hypertexnames=true]{hyperref} 
\usepackage{mathtools}
\usepackage{euscript}
\usepackage[all]{xy}
\usepackage{hyperref}
\usepackage{graphicx}

\usepackage[T1]{fontenc}
\usepackage{newtxtext}
\usepackage{newtxmath}
\usepackage{textcomp}
 \usepackage[text={150mm,251mm},centering, marginparwidth=75pt]{geometry} 
\usepackage{comment}  
\usepackage{cite}  
\usepackage{thmtools}
\usepackage{enumitem}
\usepackage{letltxmacro} 
\usepackage{nameref}
\usepackage{cleveref}
\usepackage{calc}
\usepackage{interval}
\usepackage{manfnt}
\usepackage{tikz-cd}  
\usepackage{marginnote}
\usepackage{smfhyperref}
\usepackage[hyperpageref]{backref}
\renewcommand{\backref}[1]{$\uparrow$~#1}

\usepackage{faktor} 
\usepackage{scalerel}

\newlength{\perspective}
 \setlist[itemize]{wide = 0pt, labelwidth = 2em, labelsep*=0em, itemindent = 0pt, leftmargin = \dimexpr\labelwidth + \labelsep\relax, noitemsep,topsep = 1ex,}
 \setlist[enumerate]{wide = 0pt, labelwidth = 2em, labelsep*=0em, itemindent = 0pt, leftmargin = \dimexpr\labelwidth + \labelsep\relax, noitemsep,topsep = 1ex}

\theoremstyle{plain}
\newtheorem{thmx}{Theorem}
\renewcommand{\thethmx}{\Alph{thmx}} 
\newtheorem{thm}{Theorem}[section]  
\newtheorem{lem}[thm]{Lemma}
\newtheorem{claim}[thm]{Claim} 
\newtheorem{proposition}[thm]{Proposition}
\newtheorem{cor}[thm]{Corollary}

\newtheorem{corx}[thmx]{Corollary} 
\newtheorem{conjecture}[thm]{Conjecture}
\theoremstyle{definition}
\newtheorem{dfn}[thm]{Definition}

\theoremstyle{remark}
\newtheorem{rem}[thm]{Remark} 
\newtheorem{example}[thm]{Example}

\numberwithin{equation}{section}  

\theoremstyle{plain}
\newlist{thmlist}{enumerate}{1}
\setlist[thmlist]{wide = 0pt, labelwidth = 2em, labelsep*=0em, itemindent = 0pt, leftmargin = \dimexpr\labelwidth + \labelsep\relax, noitemsep,topsep = 1ex, font=\normalfont, label=(\roman*), ref=\thethm.(\roman{thmlisti})}

\addtotheorempostheadhook[thm]{\crefalias{thmlisti}{thm}}

\addtotheorempostheadhook[assumpsion]{\crefalias{thmlisti}{assumption}}

\addtotheorempostheadhook[cor]{\crefalias{thmlisti}{cor}}

\addtotheorempostheadhook[proposition]{\crefalias{thmlisti}{proposition}}

\addtotheorempostheadhook[dfn]{\crefalias{thmlisti}{dfn}}

\addtotheorempostheadhook[lem]{\crefalias{thmlisti}{lem}}
\addtotheorempostheadhook[main]{\crefalias{thmlisti}{main}}

\addtotheorempostheadhook[rem]{\crefalias{thmlisti}{rem}}

\newlist{thmenum}{enumerate}{1} % also creates a counter called 'propenumi'
\setlist[thmenum]{wide = 0pt, labelwidth = 2em, labelsep*=0em, itemindent = 0pt, leftmargin = \dimexpr\labelwidth + \labelsep\relax, noitemsep,topsep = 1ex, font=\normalfont, label=(\roman*), ref=\thethmx.(\roman{thmenumi})}%{label=\alph*), ref=\thethmx~(\alph*)}
\crefalias{thmenumi}{thmx} 

\newlist{corlist}{enumerate}{1} % also creates a counter called 'propenumi'
\setlist[corlist]{wide = 0pt, labelwidth = 2em, labelsep*=0em, itemindent = 0pt, leftmargin = \dimexpr\labelwidth + \labelsep\relax, noitemsep,topsep = 1ex, font=\normalfont, label=(\roman*), ref=\thecorx.(\roman{corlisti})}%{label=\alph*), ref=\thethmx~(\alph*)}
\crefalias{corlisti}{corx} 
 
\crefname{lem}{Lemma}{Lemmas} 
\crefname{conjecture}{Conjecture}{Conjectures}
\crefname{thm}{Theorem}{Theorems}
\crefname{proposition}{Proposition}{Propositions}
\crefname{dfn}{Definition}{Definitions}
\crefname{rem}{Remark}{Remarks}
\crefname{cor}{Corollary}{Corollaries}
\crefname{corx}{Corollary}{Corollaries}
\crefname{problem}{Problem}{Problems}
\crefname{thmx}{Theorem}{Theorems}
\crefname{claim}{Claim}{Claims}
\crefname{assumption}{Assumption}{Assumptions}
\crefname{main}{Main Theorem}{Main Theorems}

\def\ep{\varepsilon}
\def\O{\mathscr{O}}

\def\Res{{\rm Res}}

\def\rank{{\rm rank}\,}

\newcommand{\cS}{\mathcal{S}}

\newcommand{\sslash}{\mathbin{/\mkern-6mu/}}

\makeatletter
\newcommand*{\rom}[1]{\expandafter\@slowromancap\romannumeral #1@}
\makeatother

\makeatletter     %Replace the Section ...  by the symbol \S ...  in Cref
\newcommand{\crefnames}[3]{%
	\@for\next:=#1\do{%
		\expandafter\crefname\expandafter{\next}{#2}{#3}%
	}%
}
\makeatother

\crefnames{part,chapter,section}{\S}{\S\S}

\newcommand{\sD}{\mathscr{D}}

\newcommand{\sS}{\mathscr{S}}

\newcommand{\cA}{\mathcal A}

\newcommand{\cD}{\mathcal D}

\newcommand{\cO}{\mathcal O}

\newcommand{\bA}{\mathbb{A}}

\newcommand{\bC}{\mathbb{C}}
\newcommand{\bD}{\mathbb{D}}

\newcommand{\bF}{\mathbb{F}}
\newcommand{\bG}{\mathbb{G}}

\newcommand{\bQ}{\mathbb{Q}}
\newcommand{\bR}{\mathbb{R}}

\newcommand{\bV}{\mathbb{V}}
\newcommand{\bW}{\mathbb{W}}

\newcommand{\bZ}{\mathbb{Z}}

\newcommand{\mxp}{M_{\rm B}(X,N)_{\bF_p}}

 \newcommand{\xsp}{X^{\! \rm sp}}
  \newcommand{\zsp}{Z^{\! \rm sp}}

\newcommand{\kC}{\mathfrak{C}}

  \def\spec{\textrm{Spec}\,}
  
\def\hess{{\rm d}{\rm d}^{\rm c}}

\def\End{{\rm \small  End}}

\def\Im{{\rm Im}\,}

\def\brho{{\bm{\varrho}}}

\def\btau{{\bm{\tau}}}

\def\bsigma{{\bm{\sigma}}}

\newcommand{\Hom}{{\rm Hom}}

\newcommand{\diae}{{}^\diamond\! E}

 \newcommand{\ord}{{\rm ord}\,}

  \newcommand{\kR}{\mathfrak{R}}
  \newcommand{\vhs}{{\scaleto{V\!H\!S}{4pt}}}
    \newcommand{\GL}{{\rm GL}}

\begin{document} 
 	\title[Reductive Shafarevich conjecture]{Reductive Shafarevich conjecture} 
	
 	\date{\today} 
  
 \author{Ya Deng}
 
	\email{ya.deng@math.cnrs.fr}
	 \address{CNRS, Institut \'Elie Cartan de Lorraine, Universit\'e de Lorraine, F-54000 Nancy,
		France.}
	\urladdr{https://ydeng.perso.math.cnrs.fr}

 \author{Katsutoshi Yamanoi}
 
\email{yamanoi@math.sci.osaka-u.ac.jp}
\address{Department of Mathematics, Graduate School of Science,Osaka University, Toyonaka, Osaka 560-0043, Japan} 
	\urladdr{https://sites.google.com/site/yamanoimath/}

\dedicatory{ \large With an appendix joint  with \rm Ludmil Katzarkov} 

 \begin{abstract}  
In this paper, we prove the holomorphic convexity of the covering of a complex projective \emph{normal} variety $X$, which corresponds to the intersection of kernels of reductive representations $\varrho:\pi_1(X)\to {\rm GL}_{N}(\mathbb{C})$,  therefore   answering  a question by Eyssidieux, Katzarkov, Pantev, and Ramachandran in 2012.   It is worth noting that Eyssidieux had previously proven this result in 2004 when $X$ is smooth. While our approach follows the general strategy employed in Eyssidieux's proof, it introduces several improvements and simplifications.  Notably, it avoids the necessity of using the reduction mod $p$ method   in Eyssidieux's original proof.  

Additionally, we construct the Shafarevich morphism for complex reductive representations  of  fundamental groups of complex quasi-projective varieties unconditionally, and proving its algebraic nature at the function field level. 
 \end{abstract}   
 	\maketitle
\tableofcontents	
\section{Introduction}
\subsection{Shafarevich conjecture}
In his famous textbook ``Basic Algebraic Geometry\fg \cite[p 407]{Sha13}, Shafarevich raised  the following tantalizing conjecture.
\begin{conjecture}[Shafarevich]\label{conj:Sha}
		Let $X$ be a complex projective variety. Then its universal covering is holomorphically convex. 
\end{conjecture}
Recall that a complex normal space $X$  is \emph{holomorphically convex} if  it satisfies the following condition: for each compact $K \subset X$, its  \emph{holomorphic hull}
$$
\left\{x \in X \mid |f(x)| \leq \sup _K|f|, \forall f\in \cO(X)\right\},
$$ 
 is compact. $X$ is \emph{Stein} if it is holomorphically convex and holomorphically separable, i.e.  for distinct $x$ and $y$ in $X$, there exists $f\in \cO(X)$ such that $f(x)\neq f(y)$.   By the Cartan-Remmert theorem, a complex space $X$ is holomorphically convex if and only if it admits a proper surjective holomorphic mapping   onto some Stein space.  

The study of \cref{conj:Sha}  for smooth projective surfaces has been a subject of extensive research since the mid-1980s. Gurjar-Shastri \cite{GS85} and Napier \cite{Nap90} initiated this investigation, while  Koll\'ar \cite{Kol93} and   Campana \cite{Cam94}  independently explored the conjecture in the 1990s, employing the tools of Hilbert schemes and Barlet cycle spaces. In 1994, Katzarkov discovered  that non-abelian Hodge theories developed by  Simpson \cite{Sim92}  and Gromov-Schoen \cite{GS92}  can be utilized to prove \cref{conj:Sha}. His initial work \cite{Kat97} demonstrated \cref{conj:Sha} for projective varieties with nilpotent fundamental groups. Shortly thereafter, he and Ramachandran \cite{KR98} successfully established \cref{conj:Sha}  for smooth projective surfaces whose fundamental groups admit a faithful Zariski-dense representation in a reductive complex algebraic group. Building upon the ideas presented in \cite{KR98} and \cite{Mok92}, Eyssidieux further developed non-abelian Hodge theoretic arguments in higher dimensions. In \cite{Eys04} he proved that \cref{conj:Sha} holds for any \emph{smooth} projective variety whose fundamental group possesses a faithful representation that is Zariski dense in a reductive complex algebraic group. This result is commonly referred to as the ``\emph{Reductive Shafarevich conjecture}\fg. It is worth emphasizing  that the work of Eyssidieux \cite{Eys04} is not only ingenious  but also highly significant in subsequent research.   It serves as a foundational basis for advancements in the linear Shafarevich conjecture \cite{EKPR12} and the exploration of compact Kähler cases \cite{CCE15}.   More recently, there have been significant advancements in the quasi-projective setting by Green-Griffiths-Katzarkov \cite{GGK22} and Aguilar-Campana \cite{AC23}, particularly when considering the case of nilpotent fundamental groups.  
 
 \subsection{Main theorems}
 The aim of this paper is to present a relatively comprehensible  proof of Eyssidieux's results on the reductive Shafarevich conjecture   and its associated problems, as originally discussed in \cite{Eys04}. Additionally, we aim to extend these results to the cases of quasi-projective and singular varieties, thus answering a question raised in \cite[p. 1549]{EKPR12}.  Let us first give the definition of  the  {Shafarevich morphism} for linear representations of fundamental groups.
 \begin{dfn}[Shafarevich morphism]
 	Let $X$ be a   quasi-projective  normal variety, and let $\varrho:\pi_1(X)\to {\rm GL}_N(\bC)$ be a linear representation.  A dominant holomorphic map  ${\rm sh}_{\varrho}:X\to {\rm Sh}_{\varrho}(X)$ to a complex  normal space ${\rm Sh}_{\varrho}(X)$  whose general fibers are connected  is called the \emph{Shafarevich morphism} of $\varrho$ if  for any closed subvariety $Z\subset X$,  ${\rm sh}_\varrho(Z)$ is a point if and only if $\varrho({\rm Im}[\pi_1(Z^{\rm norm})\to \pi_1(X)])$ is finite. Here $Z^{\rm norm}$ denotes the normalization of $Z$. 
 \end{dfn}
Our first main result is the \emph{unconditional} construction of the \emph{Shafarevich morphism} for reductive representations. Additionally, we establish the algebraicity of the  {Shafarevich morphism} at the function field level. 
% \begin{thmx} [=\cref{thm:Sha3}]\label{maina}
 %	Let $X$ be a quasi-projective normal variety. 
% Let $\Sigma$ be a (non-empty) set of reductive representations $\varrho:\pi_1(X)\to {\rm GL}_{N_\varrho}(\bC)$. If $X$ is non-compact, we assume additionally that each $\varrho$ has infinite monodromy at infinity. 
% Then there is a proper surjective holomorphic fibration   ${\rm sh}_{\Sigma}:X\to {\rm Sh}_{\Sigma}(X)$ onto a complex normal space such that for closed subvariety $Z\subset X$, ${\rm sh}_\Sigma(Z)$   is a point if and only if $\varrho({\rm Im}[\pi_1(Z^{\rm norm})\to \pi_1(X)])$ is finite  for every $\varrho\in\Sigma$. Moreover, ${\rm Sh}_{\Sigma}(X)$ is a projective normal variety if $X$ is compact.   
% \end{thmx}  
%It is worth noting that  when $X$ is compact, the proof of \cref{maina} can be found in \cite{Eys04}. However, addressing the non-compact case requires the use of a crucial reduction theorem  in our recent work \cite[Theorem 0.10]{CDY22}, which builds upon the prior research on constructing harmonic mappings to Bruhat-Tits buildings in \cite{BDDM}.

%By investing efforts in \cref{prop:infinity}, we are able to eliminate the technical and impractical assumption of infinite monodromy at infinity. This refinement enhances the feasibility and practicality of the results.
\begin{thmx}[=\cref{thm:Sha5,lem:bimeromorphic2,lem:extended Shafarevich}] \label{main3}
Let $X$ be a   quasi-projective  normal variety, and let $\varrho:\pi_1(X)\to {\rm GL}_N(\bC)$ be a reductive representation.  Then  
\begin{thmenum}
\item \label{item:uncondition} there exists a dominant holomorphic map  ${\rm sh}_{\varrho}:X\to {\rm Sh}_{\varrho}(X)$ to a complex  normal space ${\rm Sh}_{\varrho}(X)$ whose general fibers are connected such that  for any connected Zariski closed subset   $Z\subset X$, the following properties are equivalent:
\begin{enumerate} [label=\rm (\alph*)]
	\item ${\rm sh}_\varrho(Z)$ is a point;
	\item  $ \varrho({\rm Im}[\pi_1(Z)\to \pi_1(X)])$ is finite;
	\item  for any irreducible component  $Z_o$ of $Z$,  $ \varrho({\rm Im}[\pi_1(Z_o^{\rm norm})\to \pi_1(X)])$ is finite. 
\end{enumerate}  
\item \label{item:sommese}  There exists 
\begin{enumerate}[label=\rm (\alph*)]
	\item  a proper bimeromorphic morphism  $\sigma:S\to {\rm Sh}_{\varrho}(X)$ from a smooth  quasi-projective   variety $S$;
	\item  a proper birational morphism  $\mu: Y\to X$ from a smooth  quasi-projective   variety $Y$; 
	\item an algebraic morphism $f:Y\to S$ with general fibers connected;  
\end{enumerate} 
such that we have the following commutative diagram: 
\[
\begin{tikzcd}
Y \arrow[r, "\mu"]\arrow[d, "f"]& X \arrow[d, "{\rm sh}_\varrho"]&\\
	S                   \arrow[r, "\sigma"]                             & {\rm Sh}_\varrho(X) 
\end{tikzcd}
\]   
\item If $X$ is  smooth, then there exists a smooth partial compactification $X'$ of $X$ and a proper surjective holomorphic fibration $\overline{{\rm sh}}_{\varrho}:X'\to 	{\rm Sh}_{\varrho}(X)$ such that its restriction on $X$ is the Shafarevich morphism of $\varrho$: 
 %and $\varrho(\pi_1(X))$ is torsion free,  then  there exists another smooth quasi-projective variety $X'$ containing $X$ as a Zariski dense open subset such that:
%\begin{thmenum}[resume]
	%\item	\label{item:extend}$\varrho$ extends to a reductive representation $\varrho_0:\pi_1(X')\to \GL_{N}(\bC)$;
	%\item \label{item:exist}  the Shafarevich morphism ${\rm sh}_{\varrho_0}:X'\to {\rm Sh}_{\varrho_0}(X')$ exists, which is a holomorphic proper fibration;
	%\item \label{item:restriction} ${\rm sh}_{\varrho}={\rm sh}_{\varrho_0}|_{X}$; namely, we have the following commutative diagram:
 	\begin{equation*}
		\begin{tikzcd}
			X\arrow[r, hook] \arrow[d,"{\rm sh}_\varrho"']& X'\arrow[d,"\overline{{\rm sh}}_{\varrho}"] \\
			{\rm Sh}_\varrho(X)\arrow[r,equal]&	{\rm Sh}_{\varrho}(X)  
		\end{tikzcd}
	\end{equation*}  
\end{thmenum}
\end{thmx}  
% The holomorphic map ${\rm sh}_\varrho:X\to {\rm Sh}_\varrho(X)$ that  satisfies the properties   in \cref{item:uncondition} will be called the \emph{Shafarevich morphism} of $\varrho$.

The proof of \cref{item:uncondition} relies on a more technical result but with richer information, cf. \cref{thm:Shafarevich1}.
\begin{rem}\label{rem:Griffiths}
	% It is noticeable that \cref{item:uncondition,item:extend,item:exist,item:restriction} extend the previous theorems of Griffiths  \cite{Gri70} when $\varrho$ underlies a $\bZ$-variation of Hodge structures. In this case, the representation $\varrho_0$   in \cref{item:extend} is constructed in \cite[Theorem 9.5]{Gri70}. Additionally, Griffiths proved in \cite[Theorem 9.6]{Gri70} that the period mapping $p:X'\to \sD/\Gamma$ associated with  $\varrho_0$ is proper, where $\sD$ represents the period domain of the $\mathbb{C}$-VHS, and $\Gamma$ denotes the monodromy group of $\varrho_0$. It can be easily verified that the Shafarevich morphism ${\rm sh}_{\varrho_0}:X'\to {\rm Sh}_{\varrho_0}(X')$ corresponds to the Stein factorization of the period mapping   $p:X'\to \sD/\Gamma$, and that \cref{item:restriction} holds.
 We  conjecture that ${\rm Sh}_{\varrho}(X)$ is quasi-projective and ${\rm sh}_{\varrho}$ is an algebraic morphism (cf. \cref{conj:algebraic}).  Our conjecture is motivated by Griffiths' conjecture, which predicted the same result when $\varrho$ underlies a $\mathbb{Z}$-VHS.  Consequently, we can interpret the results presented in \cref{item:sommese} as supporting evidence for our conjecture at the function field level.  It is worth noting that Sommese, in \cite{Som78}, proved \cref{item:sommese} when $\varrho$ underlies a $\mathbb{Z}$-VHS and $\varrho(\pi_1(X))$ is torsion free, using $L^2$-methods. In contrast, we employ a different approach to establish \cref{item:sommese}, which notably provides a simpler proof of Sommese's theorem.   Griffiths' conjecture was recently proved by Baker-Brunebarbe-Tsimerman \cite{BBT23} using o-minimal theory.
\end{rem}
Based on \cref{main3}, we  construct the Shafarevich morphism for families of representations. 
 \begin{corx}[=\cref{thm:Sha3}]\label{cor:Sha3}
	Let $X$ be a quasi-projective normal variety. 
	Let $\Sigma$ be a (non-empty) set of reductive representations $\varrho:\pi_1(X)\to {\rm GL}_{N_\varrho}(\bC)$. 	
	Then there is a dominant holomorphic map   ${\rm sh}_{\Sigma}:X\to {\rm Sh}_{\Sigma}(X)$ with general fibers connected onto a complex normal space such that for closed subvariety $Z\subset X$, ${\rm sh}_\Sigma(Z)$   is a point if and only if $\varrho({\rm Im}[\pi_1(Z^{\rm norm})\to \pi_1(X)])$ is finite  for every $\varrho\in\Sigma$. 
\end{corx}

Our second main result focuses on the holomorphic convexity of topological  Galois coverings associated with reductive representations of fundamental groups within \emph{absolutely constructible subsets} of character varieties  $M_{\rm B}(\pi_1(X),{\rm GL}_N)$, where $X$ represents a projective normal variety.
\begin{thmx}[=\cref{thm:HC,main4}]\label{main2}
	Let $X$ be a  projective normal variety, and let $\kC$ be an absolutely constructible subset  of $M_{\rm B}(\pi_1(X),{\rm GL}_N)(\bC)$ as defined in \cref{def:ac,def:ac2}.    We assume that  $\kC$  is  defined on $\bQ$.  Set $H:=\cap_{\varrho}\ker\varrho$, where  $\varrho:\pi_1(X)\to {\rm GL}_N(\bC)$ ranges over all reductive representations such that  $[\varrho]\in \kC(\bC)$. Let $\widetilde{X}$ be the universal covering of $X$, and denote  $\widetilde{X}_\kC:=\widetilde{X}/H$.    Then the complex space $\widetilde{X}_\kC$ is holomorphically convex.  In particular,  we have
	\begin{thmenum}
		\item the covering of $X$ corresponding to the intersections of the kernels of all reductive representations of
	$\pi_1(X)$ in $\GL_N(\bC)$ is holomorphically convex; 
	\item\label{item:singular}  if $\pi_1(X)$ is a subgroup of ${\rm GL}_N(\bC)$ whose Zariski closure is reductive, then the universal covering $\widetilde{X}$ of $X$ is holomorphically convex.
	\end{thmenum}
\end{thmx}
For large representations, we have the following result.
\begin{thmx}[=\cref{thm:Stein,thm:Stein singular}]\label{main}
Let $X$ and $\kC$ be as described in \cref{main2}. If $\kC$ is   \emph{large}, meaning that for any closed subvariety $Z$ of $X$, there exists a reductive representation $\varrho:\pi_1(X)\to {\rm GL}_N(\bC)$ such that $[\varrho]\in \kC(\bC)$ and $\varrho({\rm Im}[\pi_1(Z^{\rm norm})\to \pi_1(X)])$ is infinite, then all intermediate coverings of $X$ between $\widetilde{X}$ and $\widetilde{X}_\kC$  are Stein spaces.
\end{thmx}  
 In addition to employing new methods for the proof of \cref{main2,main},  it yields a stronger result compared to \cite{Eys04} in two aspects:
	 \begin{enumerate}[label=(\alph*)]
	 	\item The definition of absolutely constructible subsets (cf. \cref{def:ac,def:ac2}) in our proof is more general than the one provided in \cite{Eys04}. This generality is crucial for proving \cref{item:singular} for singular varieties since the original definition of absolutely constructible subsets, as introduced by Simpson \cite{Sim93b}, pertains specifically to the moduli spaces of semistable Higgs bundles with zero characteristic numbers over \emph{smooth} projective varieties \cite{Sim94,Sim94b}. However, such moduli spaces are not constructed for singular or quasi-projective varieties.  Our broader definition allows for a wider range of applications, including the potential extension of \cref{conj:Sha} to quasi-projective varieties. We are currently engaged in an ongoing project to explore this extension further. 
	 	\item Our result extends to the case where $X$ is a singular variety, whereas in \cite{Eys04}, the result is limited to smooth varieties. This expansion of our result answers a question raised by Eyssidieux, Katzarkov, Pantev and Ramachadran in their celebrated work on linear Shafarevich conjecture for smooth projective varieties  (cf. \cite[p. 1549]{EKPR12}).
	 \end{enumerate} 
 We remark that \cref{main} is not a direct consequence of \cref{main2}. It is important to note that \cref{main} holds significant practical value in the context of singular varieties. Indeed, finding a large representation over a \emph{smooth} projective variety can be quite difficult. In practice, the usual approach involves constructing large representations using the Shafarevich morphism in \cref{main3}, resulting in large representations of fundamental groups of \emph{singular normal} varieties. Therefore, the extension of \cref{main} to singular varieties allows for more practical applicability.

%It is worth mentioning that Eyssidieux has proven \cref{maina,main2} under conditions on absolutely constructible subsets that are   stronger than our own.  	 

\subsection{Comparison with \cite{Eys04} and Novelty}
It is worth noting that Eyssidieux \cite{Eys04} does not explicitly require absolutely constructible subsets $\kC$ to be defined over $\mathbb{Q}$, although it may seem to be an essential condition (cf. \cref{rem:Q}). Regarding \cref{main3}, it represents a new result that significantly builds upon our previous work \cite{CDY22}. While \cref{main} is not explicitly stated in \cite{Eys04}, it should be possible to derive it for smooth projective varieties $X$ based on the proof provided therein.  However, it is worth noting that the original proof in \cite{Eys04} is known for its notoriously difficult and involved nature. %, with certain aspects outlined without sufficient detail. 
One of the main goals of this paper is to provide a relatively accessible proof for \cref{main2} by incorporating more detailed explanations. We draw inspiration from some of the methods introduced in our recent work \cite{CDY22}, which aids in presenting a more comprehensible proof.   Our proofs of \cref{main2,main} require  us to apply Eyssidieux's Lefschetz theorem from \cite{Eys04}.  We also owe many ideas to Eyssidieux's work in \cite{Eys04} and frequently draw upon them without explicit citation.

Despite this debt, there are some novelties in our approach, including:
\begin{enumerate}[label=(\alph*)]
	\item An avoidance of the reduction mod $p$ method used in \cite{Eys04}.
	\item A new and more canonical construction of the Shafarevich morphism that incorporates both rigid and non-rigid cases, previously treated separately in \cite{Eys04}.
	\item We relax the definition of absolutely constructible subsets in \cite{Sim93b,Eys04}, enabling us to extend our results to projective normal varieties and establish the reductive Shafarevich conjecture for such varieties. We observe that adopting the original definition of absolutely constructible subsets by Simpson and Eyssidieux would pose significant challenges in extending the Shafarevich conjecture to the singular setting.
	\item The construction of the Shafarevich morphism for reductive representations over   quasi-projective varieties, along with a proof of its algebraic property at the function field level.
	\item A detailed exposition of the application of Simpson's absolutely constructible subsets to the proof of holomorphic convexity in \cref{main2,main} (cf.  \cref{subsec:tori,thm:crucial}).  
	This application was briefly outlined in \cite[Proof of Proposition 5.4.6]{Eys04}, but we present a more comprehensive  approach, providing complete  details.
\end{enumerate}

\medspace 
 
The main part of this paper was completed in February 2023 and was subsequently shared with several experts in the field in April for feedback. During the revision process, it came to our attention that Brunebarbe \cite{Bru23} recently announced a result similar to \cref{item:uncondition}.    In \cite[Theorem B]{Bru23} Brunebarbe claims the existence of the Shafarevich morphism under a stronger   assumption of infinite monodromy at infinity and torsion-freeness of the representation, and he does not address the algebraicity of Shafarevich morphisms. 
 It seems that some crucial aspects of the arguments  in \cite{Bru23} need to be carefully addressed, particularly those related to non-abelian Hodge theories  may have been overlooked (cf.   \cref{rem:gap of Brunebarbe}).  
 
 \subsection{Further developments} 
More recently, the techniques and results   in this paper have  applications in the following works. 
\begin{itemize}
	\item In \cite{DY24}, we constructed  the  Shafarevich morphism for any representation $\varrho:\pi_1(X)\to \GL_{N}(K)$, where $X$ is a complex normal quasi-projective variety and $K$ is a  field of positive characteristic.  We also proved the generalized Green-Griffiths-Lang conjecture for such $X$ provided that $\varrho$ is big. When $\varrho$ is faithful and $X$ is a projective normal surface, we prove that the universal cover of $X$ is holomorphically convex. 
\item  In \cite{DW24}, the first author and Botong Wang proved the linear Singer-Hopf conjecture:  let $X$ be a smooth complex projective  variety which is aspherical, i.e. its universal cover  is contractible. If there is an almost faithful representation $\pi_1(X)\to \GL_{N}(K)$ with $K$ any field, then $(-1)^{\dim X}\chi(X)\geq 0$, where $\chi(X)$ is the Euler characteristic of $X$.
	\end{itemize}
 \subsection*{Convention and notation} In this paper, we use the following conventions and notations:
 \begin{itemize}[noitemsep]
 	\item Quasi-projective varieties and their closed subvarieties are assumed to be positive-dimensional and irreducible unless specifically mentioned otherwise. Zariski closed subsets, however, may be reducible .
 	\item Fundamental groups are always referred to as topological fundamental groups.
 	\item If $X$ is a complex space, its normalization is denoted by $X^{\mathrm{norm}}$.
 	\item The bold letter Greek letter  $\brho$ (or $\btau$, $\bsigma$...) denotes a family of finite reductive representations  $\{\varrho_i:\pi_1(X)\to \GL_{N}(K)\}_{i=1,\ldots,k}$ where $K$ is some non-archimedean local field or complex number field. 
 \item A \emph{proper holomorphic fibration} between complex spaces $f:X\to Y$ is   \emph{surjective} and each fiber of $f$ is \emph{connected}.   
 \item Let $X$ be a compact normal  K\"ahler space and let $V\subset H^0(X,\Omega_X^1)$ be a $\bC$-linear subspace. The \emph{generic rank} of $V$ is the largest integer $r$ such that ${\rm Im}[\Lambda^rV\to  H^0(X,\Omega^r_X)]\neq 0$.  
 \item For a quasi-projective  variety $X$, we denote by $M_{\rm B}(X,N)$ the $\GL_{N}$-character variety of  $\pi_1(X)$ in characteristic zero and $\mxp$  the $\GL_{N}$-character variety of  $\pi_1(X)$ in characteristic $p>0$. %For any linear representation $\varrho:\pi_1(X)\to \GL_{N}(K)$ where $K$ is any field, we denote by $[\varrho]\in M_{\rm B}(X,N)(K)$ the equivalent class of $\varrho$.  
 \item For a finitely presented group $\Gamma$ and an algebraically closed field $K$, a representation $\varrho:\Gamma\to \GL_{N}(K)$ is semisimple if it is a direct sum of irreducible representations, and   is reductive if the Zariski closure of $\varrho(\Gamma)$ is a reductive group over $K$. 
 \item $\bD$ denotes the unit disk in $\bC$, and $\bD^*$ denotes the puncture unit disk.  
 \end{itemize} 
 
 \subsection*{Acknowledgments} We would like to thank Daniel Barlet, S\'ebastien Boucksom, Michel Brion,   Fr\'ed\'eric Campana, Philippe Eyssidieux, Ludmil Katzarkov, Bruno Klingler, J\'anos Koll\'ar, Mihai P\u{a}un, Carlos Simpson, Botong Wang and Mingchen Xia for answering our questions and helpful discussions. The impact of Eyssidieux's work \cite{Eys04} on this paper cannot be overstated.    This work was completed during YD’s visit at the University of Miami in February 2023. He would like to   extend special thanks to Ludmil Katzarkov for the warm invitation and fruitful discussions that ultimately led to the collaborative development of the joint appendix.
  
 \section{Technical preliminary}
 \subsection{Admissible coordinates}
 The following definition of \emph{admissible coordinates} introduced in \cite{Moc07}  will be used throughout the paper. 
 	\begin{dfn}(Admissible coordinates)\label{def:admissible}	Let $ {X}$ be a complex manifold and let $D$ be a simple normal crossing divisor. Let $x$ be a point of $D$, and assume that $\{D_{j}\}_{ j=1,\ldots,\ell}$ 
 	be components of $D$ containing $x$. An \emph{admissible coordinate} centered at $x$
 	is the tuple $(U ;z_1,\ldots,z_n;\varphi)$ (or simply  $(U ;z_1,\ldots,z_n)$ if no confusion arises) where
 	\begin{itemize}
 		\item $U $ is an open subset of $ {X}$ containing $x$.
 		\item there is a holomorphic isomorphism   $\varphi:U \to \mathbb D^n$ so that  $\varphi(D_j)=(z_j=0)$ for any
 		$j=1,\ldots,\ell$.
 	\end{itemize}  
 \end{dfn}
 
\subsection{Tame and pure imaginary harmonic bundles}\label{sec:tame}
Let $\overline{X}$ be a compact complex manifold, $D=\sum_{i=1}^{\ell}D_i$ be a   simple normal crossing divisor, $X=\overline{X}\backslash D$ be the complement of $D$ and $j:X\to \overline{X}$ be the inclusion.

\begin{dfn}[Higgs bundle]\label{Higgs}
	A \emph{Higgs bundle} on $X$  is a pair $(E,\theta)$ where $E$ is a holomorphic vector bundle, and $\theta:E\to E\otimes \Omega^1_X$ is a holomorphic one form with value in $\End(E)$, called the \emph{Higgs field},  satisfying $\theta\wedge\theta=0$.
	%	\begin{eqnarray}\label{higgs triple}
		%	(\db +\theta)^2=0.
		%	\end{eqnarray} 
\end{dfn}
Let  $(E,\theta)$ be a Higgs bundle  over a complex manifold $X$.  
Suppose that $h$ is a smooth hermitian metric of $E$.  Denote by $\nabla_h$  the Chern connection of $(E,h)$, and by $\theta^\dagger_h$  the adjoint of $\theta$ with respect to $h$.  We write $\theta^\dagger$ for $\theta^\dagger_h$ for short if no confusion arises.   The metric $h$ is  \emph{harmonic} if the connection  $\nabla_h+\theta+\theta^\dagger$
is flat.
\begin{dfn}[Harmonic bundle] A harmonic bundle on  $X$ is
	a Higgs bundle $(E,\theta)$ endowed with a  harmonic metric $h$.
\end{dfn}

Let $(E,\theta,h)$ be a  harmonic bundle on $X$.	 Let $p$ be any point of $D$, and $(U;z_1,\ldots,z_n)$  be an admissible coordinate centered at  $p$. On $U$, we have the description:
\begin{align}\label{eq:local}
	\theta=\sum_{j=1}^{\ell}f_jd\log z_j+\sum_{k=\ell+1}^{n}f_kdz_k.
\end{align}
%\begin{dfn}[Tame nilpotent harmonic bundle]
%The tame harmonic bundle $(E,\theta,h)$
%	is  \emph{ nilpotent } if   each $f_i(0)$ is nipotent.
%\end{dfn} 
\begin{dfn}[Tameness]\label{def:tameness}
	Let $t$ be a formal variable. For any $j=1,\dots, \ell$, the characteristic polynomial $\det (f_j-t)\in \mathcal{O}(U\backslash D)[t],$ is a polynomial in $t$   whose coefficients are holomorphic functions. If those functions can be extended
	to the holomorphic functions over $U$ for all $j$, then the harmonic bundle is called \emph{tame} at $p$.  A harmonic bundle is \emph{tame} if it is tame at each point.
\end{dfn} 

For a tame harmonic bundle  $(E,\theta,h)$ over $\overline{X}\backslash D$,  we prolong $E$ over $\overline{X}$ by a sheaf of $\cO_{\overline{X}}$-module $\diae_h$  as follows: 
\begin{align*} 
	\diae_h(U)=\{\sigma\in\Gamma(U\backslash D,E|_{U\backslash D})\mid |\sigma|_h\lesssim {\prod_{i=1}^{\ell}|z_i|^{-\ep}}\  \ \mbox{for all}\ \ep>0\}. 
\end{align*}  
In \cite{Moc07} Mochizuki proved that $\diae_h$ is locally free and that  $\theta$ extends to a morphism
$$
\diae_h\to \diae_h\otimes \Omega^1_{\overline{X}}(\log D),
$$
which we still denote by $\theta$.   

 \begin{dfn}[Pure imaginary]\label{def:nilpotency}
 	Let $(E, h,\theta)$ be a tame harmonic bundle on $\overline{X}\backslash D$. The residue $\Res_{D_i}\theta$ induces an endomorphism of $\diae_h|_{D_i}$. Its characteristic polynomial has constant coefficients, and thus the eigenvalues are all constant. We say that $(E,\theta,h)$ is \emph{pure imaginary}   if for each component $D_i$ of $D$, the   eigenvalues of $\Res_{D_i}\theta$ are all pure imaginary. 
 \end{dfn}   
 One can verify that \cref{def:nilpotency} does not depend on the compactification $\overline{X}$ of $\overline{X}\backslash D$. 
\begin{thm}[{Mochizuki \cite[Theorem 25.21]{Moc07b}}]\label{moc}
	Let $\overline{X}$ be a projective manifold and let $D$ be a   simple normal crossing divisor on $\overline{X}$.  Let $(E,\theta,h)$ be a tame pure imaginary harmonic bundle on $X:=\overline{X}\backslash D$. Then the flat bundle $(E, \nabla_h+\theta+\theta^\dagger)$ is semi-simple. Conversely, if $(V,\nabla)$ is a semisimple flat bundle on $X$, then there is a tame pure imaginary harmonic bundle $(E,\theta,h)$ on $X$ so that $(E, \nabla_h+\theta+\theta^\dagger)\simeq (V,\nabla)$. Moreover, when $\nabla$ is \emph{simple}, then any such harmonic metric $h$ is unique up to positive multiplication. 
\end{thm}  

The following important theorem by Mochizuki  will be used throughout the paper.
\begin{thm}\label{thm:reductive}
Let $f:X\to Y$ be a morphism of quasi-projective varieties, where $Y$ is smooth and $X$ is normal. For any reductive representation $\varrho:\pi_1(Y)\to \GL_{N}(K)$, where $K$ is a non-archimedean local field of characteristic zero or the  complex number field,  the pullback $f^*\varrho:\pi_1(X)\to \GL_{N}(K)$  is also reductive.
\end{thm}
\begin{proof}
If $K$ is a non-archimedean local field of characteristic zero, then there is an abstract embedding $K\hookrightarrow \mathbb{C}$. Therefore, it suffices to prove the theorem for $K=\mathbb{C}$.

  Let $\mu:X'\to X$ be a desingularization of $X$. By \cite[Theorem 25.30]{Moc07b}, $(f\circ\mu)^*\varrho:\pi_1(X')\to \GL_{N}(\bC)$ is reductive. Since $\mu_*:\pi_1(X')\to \pi_1(X)$ is surjective, it follows that $(f\circ\mu)^*\varrho(\pi_1(X'))=f^*\varrho(\pi_1(X))$. Hence, $f^*\varrho$ is also reductive.  
\end{proof}

 \subsection{Positive currents on  normal complex spaces}
 For this subsection, we refer to \cite{Dem85} for more details. 
 \begin{dfn}
 	Let $Z$ be an irreducible normal complex space.
 	A upper semi continuous function $\phi: Z \rightarrow \mathbb{R} \cup\{-\infty\}$ is plurisubharmonic if it is not identically $-\infty$ and every point $z \in Z$ has a neighborhood $U$ embeddable as a closed subvariety of the unit ball $B$ of some $\mathbb{C}^M$ in such a way that $\left.\phi\right|_U$ extends to a psh function on $B$.
 	
 	A closed positive current with continuous potentials $\omega$ on $Z$ is specified by a data $\left\{U_i, \phi_i\right\}_i$ of an open covering $\left\{U_i\right\}_i$ of $Z$, a continuous psh function $\phi_i$ defined on $U_i$ such that $\phi_i-\phi_j$ is pluriharmonic on $U_i \cap U_j$.
 	
 	A closed positive current with continuous potentials $Z$ is a Kähler form iff its local potentials can be chosen smooth and strongly plurisubharmonic.
 	
 	A psh function $\phi$ on $Z$ is said to satisfy $\hess \phi \geq \omega$ iff $\phi-\phi_i$ is psh on $U_i$ for every $i$. 
 \end{dfn} 
 
 In other words, a closed positive current with continuous potentials is a section of the sheaf $C^0 \cap P S H_Z / \operatorname{Re}\left(O_Z\right)$.
 
 \begin{dfn}
 	Assume $Z$ to be compact. The class of a closed positive current with continuous potentials is its image in $H^1\left(Z, \operatorname{Re}\left(O_Z\right)\right)$.
 	
 	A class in $H^1\left(Z, \operatorname{Re}\left(O_Z\right)\right)$ is said to be Kähler if it is the image of a Kähler form. 
 \end{dfn}
 
 To make contact with the usual terminology observe that if $Z$ is a compact Kähler manifold $H^1\left(Z, \operatorname{Re}\left(O_Z\right)\right)=H^{1,1}(Z, \mathbb{R})$. Hence we use abuse of notation to write $H^{1,1}(Z, \mathbb{R})$ instead of $H^1\left(Z, \operatorname{Re}\left(O_Z\right)\right)$ in this paper.

 \begin{lem}\label{lem:Demailly}
 	Let $f:X\to Y$ be a finite Galois cover with Galois group $G$, where $X$ and $Y$  are both irreducible normal complex spaces. Let $T$ be a positive $(1,1)$-current on $X$ with continuous potential. Assume that $T$ is invariant under $G$. Then there is a closed positive $(1,1)$-current $S$ on $Y$ with continuous potential such that $T=f^*S$.
 \end{lem} 
 \begin{proof}
 	Since the statement is local, we may assume that $T=\hess \varphi$ such that $\varphi\in C^0(X)$. Define a function on $Y$ by 
 	$$
 	f_*\varphi(y):=\sum_{x\in f^{-1}(y)}\varphi(x) 
 	$$ 
 	here the sums are counted with multiplicity.
 	By \cite[Proposition 1.13.(b)]{Dem85}, we know that
 	$
 	f_*\varphi 
 	$ is a psh function on $Y$ and 
 	$$
 	\hess f_*\varphi=f_*T.
 	$$
 	One can see that $f_*\varphi$ is also continuous.  
 	Define a current $S:=\frac{1}{\deg f}f_*T$.     Since $T$ is $G$-invariant, it follows that $f^*S=T$ outside the branch locus of $f$.  Since $f^*S=\frac{1}{\deg f}\hess (f_*\varphi)\circ f$, the potential of $f^*S$ is continuous. It follows that $f^*S=T$ over the whole $X$. 
 \end{proof}
 \subsection{Holomorphic forms on complex normal spaces}
There are many ways to define holomorphic forms on complex normal spaces. For our purpose of the current paper, we use the following definition in \cite{Fer70}.  
 \begin{dfn} \label{def:form}
  Let $X$ be a   normal complex space. 	 Let $\left(A_i\right)_{i \in I}$ be an open finite covering of $X$ such that each subset $A_i$ is an analytic subset of some open subset $\Omega_i \subset \mathbb{C}^{N_i}$.  The space of holomorphic $p$-forms,  denoted by $\Omega_X^p$, is defined by local restrictions of holomorphic $p$-forms on the sets $\Omega_i$ above to $A_i^{\rm reg}$, where $A_i^{\rm reg}$ is the smooth locus of $A_i$. 
 \end{dfn} 
The following fact will be used throughout the paper.
\begin{lem}\label{lem:pull}
	Let $f:X\to Y $ be a holomorphic map between normal complex spaces. Then for any holomorphic $p$-form $\omega$ on $Y$, $f^*\omega$ a holomorphic $p$-form on $Y$. 
\end{lem}
\begin{proof}
By \cref{def:form},    for any $x\in X$, there exist 
\begin{itemize}
	\item  a neighborhood $A$ (resp. $B$) of $x$ (resp. $f(x)$)    such that $A$ (resp. $B$) is  an analytic subset of some open $\Omega \subset \mathbb{C}^{m}$ (resp. $\Omega'\subset \mathbb{C}^{n}$).
\item   a holomorphic map $\tilde{f}:\Omega\to \Omega'$ such that $\tilde{f}|_{A}=f|_A$.  
\item A holomorphic $p$-form $\tilde{\omega}$ on $\Omega'$ such that $\omega=\tilde{\omega}|_{B}$. 
\end{itemize} 
Therefore, we can define $f^*\omega|_{A}:=\tilde{f}^*\tilde{\omega}|_{\Omega}$. One can check that this does not depend on the choice of local embeddings of $X$ and $Y$.
\end{proof}

\subsection{The criterion for K\"ahler classes}
We will need the following  extension of the celebrated Demailly-P\u{a}un's theorem \cite{DP04} on characterization of K\"ahler classes on complex compact normal Kähler  spaces by Das-Hacon-P\u{a}un in \cite{DHP22}.
\begin{thm}[\protecting{\cite[Corollary 2.39]{DHP22}}]\label{thm:DHP}
	 Let $X$ be a projective normal variety, $\omega$ a Kähler form on $X$, and $\alpha \in H_{\rm BC}^{1,1}(X)$. Then 
%\begin{thmlist}
	%\item  $\alpha$ is nef if and only if $\int_V \alpha^k \wedge \omega^{p-k} \geq 0$ for every analytic $p$-dimensional subvariety $V \subset X$ and for all $0<k \leq p$, and
	%\item 
	$\alpha$ is Kähler if and only if $\int_V \alpha^{\dim V} >0$ for  every
positive dimensional closed subvariety $V \subset X$. 
%	\end{thmlist}
\end{thm}
 
 \subsection{Some criterion for Stein space}
 We require the following criterion for the Stein property of a topological Galois covering of a compact complex normal space.
 \begin{proposition}[\protecting{\cite[Proposition 4.1.1]{Eys04}}]\label{prop:stein}
 	Let $X$ be a compact complex normal space and let $\pi:\widetilde{X}'\to X$ be some topological Galois covering.  Let $T$ be a positive current on $X$  with continuous potential such that $\{T\}$ is a K\"ahler class. Assume that there exists a continuous plurisubharmonic function $\phi: \widetilde{X}'\to \bR_{\geq 0}$ such that $\hess \phi\geq \pi^*T$.  Then $\widetilde{X}'$ is a Stein space. 
 \end{proposition}
 
  \subsection{Some facts on moduli spaces of rank 1 local systems} 
  For this subsection we refer the readers to \cite{Sim93b} for a systematic treatment. 
 Let $X$ be a smooth projective variety defined over  a field $K\subset \bC$. Let $M=M(X)$ denote the moduli space of complex local systems of rank one over $X$. We consider $M$ as a real analytic group under the operation of tensor product. There are three natural algebraic groups $M_{\mathrm{B}}$, ${\rm M}_{\rm DR}$ and ${\rm M}_{\rm Dol}$   whose underlying real analytic groups are canonically isomorphic to $M$.   The first is Betti moduli space $M_{\mathrm{B}}:= {\rm Hom}(\pi_1(X), \bC^*)$. The second is De Rham moduli space  $ M_{\mathrm{DR}} $ which consists of  pairs $(L,\nabla)$ where $L$ is a holomorphic line bundle on $X$ and $\nabla$ is an integrable algebraic connection on $L$. The last one $M_{\text {Dol }}$is moduli spaces of rank one Higgs bundles on $X$. Let ${\rm Pic}^\tau (X)$ be the group of line bundles on $X$ whose first Chern classes are torsion. We have  
 $$
 M_{\text {Dol }}={\rm Pic}^\tau (X)\times H^0(X, \Omega_X^1)
 $$
 For any subset $S\subset M$, let $S_{\rm B}$, $S_{\rm Dol}$ and $S_{\rm DR}$ denote the corresponding subsets of $M_{\mathrm{B}}$, ${\rm M}_{\rm DR}$ and ${\rm M}_{\rm Dol}$. 
 
 \begin{dfn}[Triple torus]
 	A triple torus is a closed, connected real analytic subgroup $N \subset M$ such that $N_{\mathrm{B}}, N_{\mathrm{DR}}$, and $N_{\text {Dol }}$ are algebraic subgroups defined over $\bC$. We say that a closed real analytic subspace $S \subset M$ is a translate of a triple torus if there exists a triple torus $N \subset M$ and a point $v \in M$ such that $S=\{v \otimes w, w \in N\}$. Note that, in this case, any choice of $v \in M$ will do.
 \end{dfn}
 
 We say that a point $v \in M$ is torsion if there exists an integer $a>0$ such that $v^{\otimes a}=1$. Let $M^{\text {tor }}$ denote the set of torsion points. Note that for a given integer $a$, there are only finitely many solutions of $v^{\otimes a}=1$. Hence, the points of $M_{\mathrm{B}}^{\text {tor }}$ are defined over $\overline{\mathbb{Q}}$, and the points of $M_{\mathrm{DR}}^{\text {tor }}$ and $M_{\mathrm{Dol}}^{\text {tor }}$ are defined over $\overline{K}$.
 
 We say that a closed subspace $S$ is a torsion translate of a triple torus if $S$ is a translate of a triple torus $N$ by an element $v \in M^{\text {tor }}$. This is equivalent to asking that $S$ be a translate of a triple torus, and contain a torsion point.
 
 Let $A$ be the Albanese variety of $X$ (which can be defined as $  {H}^0 ( {X}, \Omega_{ {X}}^1 )^* / {H}_1( {X},\mathbb{Z}) )$. Let $ {X} \rightarrow  {A}$ be the map from $ {X}$ into $ {A}$ given by integration (from a basepoint, which will be suppressed in the notation but assumed to be defined over $\overline{K})$. Pullback of local systems gives a natural map from $M(A)$ to $M(X)$, which is an isomorphism
 $$
 M(A) \cong M^0(X),
 $$
 where $M^0(X)$ is the connected component of $M(X)$ containing the trivial rank one local system.  
 The Albanese variety $A$ is defined over $\overline{ {K}}$. We recall the following result in \cite[Lemma 2.1]{Sim93b}.
 \begin{lem}[Simpson]\label{lem:triple}
 	Let $N \subset M$ be a closed connected subgroup such that $N_{\mathrm{B}} \subset M_{\mathrm{B}}$ is complex analytic and $N_{\mathrm{Dol}} \subset M_{\mathrm{Dol}}$ is an algebraic subgroup. Then there is a connected abelian subvariety $ {P} \subset {A}$, defined over $\overline{K}$, such that $N$ is the image in $M$ of $M(A/  {P})$. In particular, $N$ is a triple torus in $M$. \qed
 \end{lem}

 \subsection{Absolutely constructible subsets (I)}
 In this section we will recall some facts on \emph{absolutely constructible subsets} (resp. \emph{absolutely closed subsets}) introduced by Simpson in \cite[\S 6]{Sim93b} and later developed by Budur-Wang \cite{BW20}.  
 
 Let $X$ be a smooth projective variety defined over a subfield $\ell$ of $\bC$. Let $G$ be a reductive group defined over $\bar{\bQ}$. The representation scheme of $\pi_1(X)$ is an affine $\bar{\bQ}$-algebraic scheme described by its functor of points:
 $$R(X,G)(\spec A):=\Hom (\pi_1(X), G(A))$$
 for any $\bar{\bQ}$-algebra $A$. The character scheme of $\pi_1(X)$ with values in $G$ is the finite type 
 affine scheme 
  $
 M_{\rm B}(X,G):=R(X,G)\sslash G,$ where ``$\sslash$\fg denotes the GIT quotient.    
 If $G={\rm GL}_N$, we simply write $M_{\rm B}(X,N):=M_{\rm B}(X,{\rm GL}_N)$.
 Simpson constructed a quasi-projective scheme  $M_{\rm DR}(X,G)$, and $M_{\rm Dol}(X,G)$ over $\ell$. 
 The $\bC$-points of $M_{\rm DR}(X,G)$ are in bijection with the equivalence classes of flat $G$-connections with reductive monodromy. %and the $\bC$-points of $M_{\rm Dol}(X,G)$ are in bijection with the isomorphism classes of polystable $G$-Higgs  bundles with vanishing first and second Chern class.
There are natural isomorphisms
 $$
 \psi: M_{\rm B}(X, G)(\bC)\to M_{\rm DR}(X,G)(\bC)%; \ \ \  \varphi: M_{\rm B}(X, G)(\bC)\to M_{\rm Dol}(X, G)(\bC)
 $$
such that $\psi$ is an isomorphism  of complex analytic spaces. %and $\varphi$ is a  real analytic homeomorphism.
 For each automorphism $\sigma \in \operatorname{Aut}(\mathbb{C} / \mathbb{Q})$, let $X^\sigma:=X\times_\sigma\bC$ be the   conjugate variety of $X$, which is also smooth projective. There is a natural map
 $$
 \begin{gathered}
 	p_\sigma: M_{\mathrm{DR}}(X,G) \rightarrow M_{\mathrm{DR}}\left(X^\sigma,G^\sigma\right). %\\
 %	q_\sigma: M_{\mathrm{Dol}}(X,G) \rightarrow M_{\mathrm{Dol}}\left(X^\sigma,G\right)
 \end{gathered}
 $$
 %and similarly for the representation spaces. These are the transport of structure maps obtained from the fact that the spaces concerned are moduli spaces for algebraic geometric objects.

 Let us now introduce the following definition of absolutely  constructible subsets.
 \begin{dfn}[Absolutely constructible subset]\label{def:ac}
 	A subset $\kC \subset {M}_{\rm B}(X,G)(\bC)$ is an \emph{absolutely constructible subset} (resp. \emph{absolutely closed subset}) if the following conditions are satisfied. 
 	\begin{thmlist} 
 		\item $\kC$ is the   a $\bar{\bQ}$-constructible (resp. $\bar{\bQ}$-closed) subset of $M_{\rm B}(X,G)$.
 		%\item  $S_{\rm DR}=\psi(S)$ and $S_{\rm Dol}=\varphi(S)$ are constructible subsets of $M_{\rm DR }(X,G)$ and $M_{\rm Dol}(X)$ respectively.
 		\item For each $\sigma \in \operatorname{Aut}(\mathbb{C} / \mathbb{Q})$, there exists a $\bar{\bQ}$-constructible (resp. $\bar{\bQ}$-closed) set $\kC^\sigma \subset M_{\rm B}\left(X^\sigma, G^\sigma\right)(\bC)$  such that  $\psi^{-1} \circ p_\sigma \circ \psi(\kC)=\kC^\sigma$.
 		\item \label{item:R*}$\kC(\bC)$ is preserved by the action of $\bR^*$ defined in \cref{sec:C*action}. 
 	\end{thmlist}   
 \end{dfn}
 \begin{rem} 
 	\begin{thmlist}
\item Note that this definition is significantly weaker than the notion of absolutely constructible sets defined in \cite{Sim93b,Eys04}, as it does not consider moduli spaces of semistable Higgs bundles with trivial characteristic numbers, and it does not require that $\psi(\kC)$ is $\bar{\bQ}$-constructible in ${M}_{\rm DR}(X,G)(\bC)$. This revised definition allows for a broader range of applications, including quasi-projective varieties. In \cite{Sim93b,Eys04}, the preservation of $\kC(\mathbb{C})$ under the action of $\mathbb{C}^*$ is a necessary condition. It is important to emphasize that our definition only requires $\mathbb{R}^*$-invariance, which is weaker than $\mathbb{C}^*$-invariance. Our definition corresponds to the \emph{absolutely constructible subset} as defined in \cite[Definition 6.3.1]{BW20}, with the additional condition that $\kC(\mathbb{C})$ is preserved by the action of $\mathbb{R}^*$. 
\item It is crucial to emphasize the flexibility of our weaker definition of absolutely constructible subsets, as defined in \Cref{def:ac}, which allows for its extension to singular varieties, as will be shown in \cref{def:ac2}. This extension plays a crucial role in establishing  the reductive Shafarevich conjecture for projective normal varieties, thereby generalizing the previous findings by Eyssidieux \cite{Eys04}.
\end{thmlist}
 \end{rem}

 By \cite[Theorem 9.1.2.(2) \& Proposition 7.4.4.(2)]{BW20} we have the following result, which generalizes \cite{Sim93b}.
 \begin{thm}[Budur-Wang, Simpson]\label{thm:S1}
 	Let $X$ be a smooth projective variety over $\mathbb{C}$.  An \emph{absolute locally closed} subset in $M_{\rm B}( {X}, 1)(\bC)$ is the complement of a finite union of torsion translates of triple tori, within
 	another finite union of torsion translates of triple tori.    An absolute constructible subset in  $M_{\rm B}( {X}, 1)(\bC)$ is a finite union of  {absolute locally closed} subset.  \qed 
 \end{thm} 
 %where each $N^\circ_i$ is  a Zariski dense open subset of a  torsion-translated   subtori $N_i$ of $M_{\rm B}( {X}, 1)$. Moreover, let $A$ be the Albanese variety of $X$. Then  there are abelian subvarieties $P_i\subset A$   such that $N_i$ is the torsion translate of the image in $M^0_{\rm B}(X, 1)\simeq M_{\rm B}(A,1)$  of $M_{\rm B}(A/P_i,1)$. Here   $M^0_{\rm B}(X, 1)$  denotes the connected component of $M^0_{\rm B}(X, 1)$ containing the identity.

Absolute constructible subsets are preserved by the following operations:
 \begin{thm}[Simpson]\label{thm:S2}
 	Let $f:Z\to X$ be  a morphism   between   smooth projective varieties over $\bC$ and let $g: G\to G'$ be a  morphism of reductive groups over $\bar{\bQ}$.   Consider the natural map $i: M_{\rm B}(X,G)\to M_{\rm B}(X,G')$ and $j: M_{\rm B}(X, G)\to M_{\rm B}(Z, G)$. Then for any absolutely constructible subsets $\kC\subset M_{\rm B}(X,G)(\bC)$ and $\kC'\subset M_{\rm B}(X,G')(\bC)$,   we have $i(\kC)$,  $i^{-1}(\kC' )$ and $j(\kC)$  are all absolutely  constructible. \qed
 \end{thm}

 \begin{example}\label{example:ac}
 	$M_{\rm B}(X,G)(\bC)$, the isolated point in $M_{\rm B}(X,G)(\bC)$, and the   class of trivial representation  in $M_{\rm B}(X,G)(\bC)$ are all absolutely  constructible. 
 \end{example}

In this paper, absolutely constructible subsets are used to prove the holomorphic convexity of some topological Galois covering of $X$ in \cref{main2,main}. It will not be used in the proof of  \cref{main3}.

\subsection{Katzarkov-Eyssidieux reduction and  canonical currents}\label{sec:KE}
For this subsection, we  refer to the papers \cite[\S 3.3.2]{Eys04} or \cite{CDY22} for a comprehensive and systematic treatment. 

\begin{thm}[Katzarkov, Eyssidieux]\label{thm:KE}
	Let $X$ be a projective normal variety, and let $K$ be a non-archimedean local field. Let $\varrho: \pi_1(X)\to {\rm GL}_N(K)$ be a reductive representation. Then there exists a fibration $s_\varrho:X\to S_\varrho$ to a normal projective space, such that for any subvariety $Z$ of $X$, the following are equivalent:
	\begin{thmlist}
		\item\label{item:KZ1} $\varrho(\Im[\pi_1(Z^{\rm norm})\to \pi_1(X)])$ is bounded;
		\item \label{item:KZ2}  $s_\varrho(Z)$ is a point;
		\item $\varrho(\Im[\pi_1(Z)\to \pi_1(X)])$ is bounded.
	\end{thmlist} 
\end{thm}
We will call the above $s_\varrho$ the (Katzarkov-Eyssidieux) reduction map for $\varrho$.  When $X$ is smooth, the equivalence of \cref{item:KZ1} and \cref{item:KZ2} is proved by Katzarkov \cite{Kat97} and Eyssidieux \cite{Eys04}.  The more general result stated in the theorem can be found in \cite{CDY22}. 
We will outline the construction of certain  \emph{canonical}  positive closed $(1,1)$-currents  over $S_{\varrho}$. As demonstrated in the proof of \cite{Eys04}, we can establish the existence of a finite ramified Galois cover denoted by $\pi: \xsp \to X$ with the Galois group $H$, commonly known as the  \emph{spectral covering of $X$} (cf. \cite[Definition 5.14]{CDY22}).  This cover possesses holomorphic 1-forms ${\eta_1,\ldots,\eta_m} \subset H^0(\xsp, \pi^*\Omega_X^1)$, which can be considered as the $(1,0)$-part of the complexified differential of the $\pi^*\varrho$-equivariant harmonic mapping from $\xsp$ to the Bruhat-Tits building of $G$. These particular 1-forms, referred to as the \emph{spectral one-forms} (cf. \cite[Definition 5.16]{CDY22}),  play a significant role in the proof of \cref{main2,main}. Consequently, the Stein factorization of the \emph{partial Albanese morphism} $a:\xsp\to A$ (cf. \cite[Definition 5.19]{CDY22}) induced by ${\eta_1,\ldots,\eta_m}$  leads to the Katzarkov-Eyssidieux reduction map $s_{\pi^*\varrho}:\xsp\to S_{\pi^*\varrho}$ for $\pi^*\varrho$. Moreover, we have the  following commutative diagram:
\begin{equation*}
	\begin{tikzcd}
		\xsp\arrow[r, "\pi"] \arrow[d, "s_{\pi^*\varrho}"]\arrow[dd,bend right=37, "a"'] &  X\arrow[d, "s_{\varrho}"]\\
		S_{\pi^*\varrho} \arrow[d, "b"]\arrow[r, "\sigma_\pi"] & S_{\varrho}\\
		A &
	\end{tikzcd}
\end{equation*}
Here  $\sigma_\pi$ is also a finite ramified Galois cover with Galois group $H$. 
Note that there are one forms $\{\eta_1',\ldots,\eta_m'\}\subset H^0(A, \Omega_A^1)$ such that $a^*\eta_i'=\eta_i$. Consider the finite morphism $b: S_{\pi^*\varrho}\to A$. Then  we define a positive $(1,1)$-current $T_{\pi^*\varrho}:=b^*\sum_{i=1}^{m}i\eta_i'\wedge\overline{\eta_i}'$ on $S_{\pi^*\varrho}$. Note that $T_{\pi^*\varrho}$ is invariant under the Galois action $H$. Therefore, by \cref{lem:Demailly} there is a positive closed $(1,1)$-current $T_{\varrho}$ defined on $S_\varrho$ with continuous potential such that $\sigma_\pi^*T_{\varrho}=T_{\pi^*\varrho}$.
\begin{dfn}[Canonical current]\label{def:canonical}
	The closed positive $(1,1)$-current 	$T_\varrho$ on $S_\varrho$ is called the \emph{canonical current} of $\varrho$. 
\end{dfn}
More generally, let   $\{\varrho_i:\pi_1(X)\to {\rm GL}_N(K_i)\}_{i=1,\ldots,k}$ be   reductive representations where $K_i$ is a non-archimedean local field. We shall denote by the bolded letter $\bm{\varrho} :=\{\varrho_i\}_{i=1,\ldots,k}$ be such family of representations.   Let  $s_\brho:X\to S_{\brho}$  be the Stein factorization of  $(s_{\varrho_1},\ldots,s_{\varrho_k}):X\to S_{\varrho_1}\times\cdots\times S_{\varrho_k}$ where $s_{\varrho_i}:X\to S_{\varrho_i}$ denotes the reduction map associated with $\varrho_i$ and $p_i: S_{\brho}\to S_{\varrho_i}$ is the induced  morphism.    
$s_\brho:X\to S_{\brho}$  is called the \emph{reduction map} for  the family ${\brho}$ of representations.
\begin{dfn}[Canonical current II]\label{def:canonical2}
	The closed positive $(1,1)$-current 	$T_{\brho}:=\sum_{i=1}^{k}p_i^*T_{\varrho_i}$ on $S_\brho$ is called the canonical current of $\brho$. 
\end{dfn}

\begin{lem}[\protecting{\cite[Lemme 1.4.9 \& 3.3.10]{Eys04}}] \label{lem:functorial}
	Let $f:Z\to X$ be a morphism between  projective normal varieties and let $\brho:=\{\varrho_i:\pi_1(X)\to {\rm GL}_N(K_i)\}_{i=1,\ldots,k}$ be a family of  reductive representations where $K_i$ is a non-archimedean local field.   Then we have
	\begin{equation}\label{eq:functorial}
		\begin{tikzcd}
			Z \arrow[r, "f"] \arrow[d, "s_{f^*\brho}"] &  X\arrow[d, "s_{\brho}"]\\
			S_{f^*\brho} \arrow[r, "\sigma_f"] & S_{\brho}
		\end{tikzcd}
	\end{equation}
	where $\sigma_f$ is a \emph{finite} morphism. Here   $ f^*\brho= \{f^*\varrho_i\}_{i=1,\ldots,k}$ denotes the pull back of the family of   $\brho:=\{\varrho_i:\pi_1(X)\to {\rm GL}_N(K_i)\}_{i=1,\ldots,k}$. Moreover,    the following properties hold:
	 \begin{thmlist}
		\item The local potential of  $T_{\brho}$  is continuous. In particular, for any closed subvariety $W\subset X$,  we have
		$$
		\{T_\brho\}^{\dim W}\cdot W=\int_{W}T_\brho^{\dim W}\geq 0.
		$$
		\item \label{item:descend}$T_{f^*\brho}=\sigma_f^* T_{\bm{\varrho}}$;
		\item \label{item:same positivity} For every  closed subvariety $\Xi\subset S_{f^*\brho}$,  $\{T_\brho\}^{\dim \Xi}\cdot (\sigma_f(\Xi))>0$ if and only if    $\{T_{f^*\brho}\}^{\dim \Xi}\cdot \Xi>0$. \qed
	\end{thmlist}
\end{lem}
Note that \cref{item:same positivity} is a consequence of the first two assertions.

The current $T_{\varrho}$ will serve as a lower bound for the complex hessian of plurisubharmonic functions constructed by the method of harmonic mappings. 
\begin{proposition}[\protecting{\cite[Proposition 3.3.6, Lemme 3.3.12]{Eys04}}]\label{prop:Eys}
	Let   $X$ be a projective normal variety and let   $\varrho:\pi_1(X)\to G(K)$ be a Zariski dense representation where $K$ is a non archimedean local field and $G$ is a reductive group.	   Let $x_0 \in \Delta(G)$ be an arbitrary point.  Let $u: \widetilde{X}\to \Delta(G)$ be the associated the harmonic mapping, where $\widetilde{X}$ is  the universal covering of $X$. 
	The function $\phi: \widetilde{X} \rightarrow \mathbb{R}_{\geq 0}$ defined by
	$$
	\phi(x)=2d^2\left(u(x), u(x_0)\right)
	$$
	satisfies the following properties:
	\begin{enumerate}[label=(\alph*)]
		\item   $\phi$ descends to a function $\phi_{\varrho}$ on $\widetilde{X_{\varrho}}=\widetilde{X}  / \operatorname{ker}\left(\varrho\right)$.
		\item $\hess\phi_\varrho\geq (s_\varrho\circ\pi)^*T_\varrho$,  where  we denote by $\pi:\widetilde{X}_\varrho\to X$ the covering map. 
		%	\item Let $\pi: \widetilde{X}^{\rm un} \rightarrow \widetilde{X} \rightarrow \widetilde{X}_{\varrho}$ be a covering space of $X$ dominating $\widetilde{X}_{\varrho}$. By abuse of notation we denote by $\phi_{\varrho}$ the function $\phi_{\varrho}^* \circ \pi$.
		\item $\phi_{\varrho}$ is locally Lipschitz;
		\item  \label{item: descends} Let $T$ be a normal complex space and $r:\widetilde{X}_\varrho \to T$ a proper holomorphic fibration such that $s_{\varrho} \circ \pi: \widetilde{X}_\varrho \rightarrow S_{\varrho} $ factorizes via a morphism $\nu:T {\rightarrow} S_\varrho$. The function $\phi_{\varrho}  $ is of the form $\phi_{\varrho}=\phi_\varrho^T \circ r$  with $\phi_\varrho^T$ being a continuous  plurisubharmonic function on $T$;  
		\item $\hess\phi^T_\varrho \geq \nu^* T_{\varrho}$.\qed
	\end{enumerate}
\end{proposition}

\subsection{The generalization of Katzarkov-Eyssidieux reduction to quasi-projective varieties}
In our work \cite{CDY22} on hyperbolicity of quasi-projective varieties, we extended  \cref{thm:KE} to  quasi-projective varieties. The theorem we established is stated below.
 \begin{thm}[\protecting{\cite[Theorem H]{CDY22}}] \label{thm:KZ}
	Let $X$ be a complex smooth quasi-projective variety, and let $\varrho:\pi_1(X)\to {\rm GL}_N(K)$ be a reductive representation where $K$ is non-archimedean local field.  Then there exists a quasi-projective normal variety $S_\varrho$ and a dominant morphism $s_\varrho:X\to S_\varrho$ with connected general fibers, such that  for any connected Zariski closed  subset $T$ of $X$, the following properties are equivalent:
	\begin{enumerate}[label={\rm (\alph*)}]
		\item \label{item bounded} the image $\varrho({\rm Im}[\pi_1(T)\to \pi_1(X)])$ is a bounded subgroup of $G(K)$.
		\item \label{item normalization} For every irreducible component $T_o$ of $T$, the image $\varrho({\rm Im}[\pi_1(T_o^{\rm norm})\to \pi_1(X)])$ is a bounded subgroup of $G(K)$.
		\item \label{item contraction}The image $s_\varrho(T)$ is a point.\qed
	\end{enumerate} 
\end{thm}
This result plays a crucial role in the proof of \cref{main3}.    Its proof is built upon the   work by Brotbek, Daskalopoulos, Mese, and the first named author \cite{BDDM}, regarding the construction of $\varrho$-equivariant harmonic mappings from the universal covering of $X$ to the Bruhat-Tits building $\Delta(G)$ of $G$.

 \subsection{Simultaneous Sten factorization} 
 \begin{comment}
 
 \begin{lem}\label{lem:2023010511}
 	Let $V$ be a smooth quasi-projective variety.
 	For $i=1,2,\ldots$, let $W_i$ be normal quasi-projective varieties such that
 	\begin{itemize}
 		\item
 		there exist dominant morphisms $p_i:V\to W_i$, and
 		\item
 		there exist dominant morphisms $q_i:W_i\to W_{i-1}$ such that $q_i\circ p_i=p_{i-1}$.
 	\end{itemize}
 	Then there exists $i_0\in \mathbb Z_{\geq 2}$ such that for every $i\geq i_0$ and every subvariety $Z\subset V$, if $p_{i-1}(Z)$ is a point, then $p_i(Z)$ is a point.
	 \end{lem}
 
 \begin{proof}
 	Let $E_i\subset V\times V$ be defined by 
 	$$E_i=\{(x,x')\in V\times V; p_i(x)=p_i(x')\}.$$
 	Then $E_i\subset V\times V$ is a Zariski closed set. Indeed, $E_i=(p_i,p_i)^{-1}(\Delta_i)$, where $(p_i,p_i):V\times V\to W_i\times W_i$ is the morphism defined by $(p_i,p_i)(x,x')=(p_i(x),p_i(x'))$ and $\Delta_i\subset W_i\times W_i$ is the diagonal.
 	Now by $q_i\circ p_i=p_{i-1}$, we have $E_{i}\subset E_{i-1}$.
 	By the Noetherian property, there exists $i_0$ such that $E_{i+1}=E_{i}$ for all $i\geq i_0$.
 	Then the induced map $p_{i+1}(V)\to p_i(V)$ is injective.
	Hence if $p_{i-1}(Z)$ is a point, then $p_i(Z)$ is a point.
	 \end{proof}
\end{comment}
 
 \begin{lem}\label{lem:simultaneous}
 	Let $V$ be a quasi-projective normal variety and let $(f_{\lambda}:V\to S_{\lambda})_{\lambda\in\Lambda}$ be a family of morphisms into quasi-projective varieties $S_{\lambda}$.
 	Then there exist a quasi-projective normal
 	 variety $S_{\infty}$ and a morphism $f_{\infty}:V\to S_{\infty}$ such that 
 	\begin{itemize}
 	\item
 	$f_{\infty}$ is dominant and has connected general fibers,
 	\item for every subvariety $Z\subset V$, $f_{\infty}(Z)$ is a point if and only if $f_{\lambda}(Z)$ is a point for every $\lambda\in \Lambda$, 
	and
 		\item
 		there exist $\lambda_1,\ldots,\lambda_n\in\Lambda$ such that $f_{\infty}:V\to S_{\infty}$ is the quasi-Stein factorization of $(f_1,\ldots,f_n):V\to S_{\lambda_1}\times\cdots S_{\lambda_n}$.
 	\end{itemize}
 \end{lem}
 
\begin{proof}
 We define $E_{\lambda}$ by 
 $$E_{\lambda}=\{(x,x')\in V\times V; f_{\lambda}(x)=f_{\lambda}(x')\}.$$
 	Then $E_{\lambda}\subset V\times V$ is a Zariski closed set. 
 	Indeed, $E_{\lambda}=(f_{\lambda},f_{\lambda})^{-1}(\Delta_{\lambda})$, where $(f_{\lambda},f_{\lambda}):V\times V\to
 	S_{\lambda}\times S_{\lambda}$ is the morphism defined by $(f_{\lambda},f_{\lambda})(x,x')=(f_{\lambda}(x),f_{\lambda}(x'))$ and $\Delta_{\lambda}\subset S_{\lambda}\times S_{\lambda}$ is the diagonal.
By Noetherian property, we may take a finite subset $\Lambda'\subset \Lambda$ such that 
 \begin{equation}\label{eqn:20231025}
 \cap_{\lambda\in\Lambda'}E_{\lambda}=\cap_{\lambda\in\Lambda}E_{\lambda}.
 \end{equation}
 Let $f_{\infty}:V\to S_{\infty}$ be the quasi-Stein factorization of the map 
 $(f_{\lambda})_{\lambda\in\Lambda'}:V\to
 \Pi_{\lambda\in\Lambda'}S_{\lambda}$.
 Then $S_{\infty}$ is normal, and $f_{\infty}$ is dominant and has connected general fibers.
 For a closed subvariety $Z\subset V$, we have
 \begin{equation*}
 \text{
 $f_{\infty}(Z)$ is a point 
 $\Longleftrightarrow$
  $f_{\lambda}(Z)$ is a point for all $\lambda\in\Lambda'$.
 }
 \end{equation*}
 By the definition of $E_{\lambda}$, we note that 
  $f_{\lambda}(Z)$ is a point if and only if
  $Z\times Z\subset E_{\lambda}$.
  Hence $f_{\lambda}(Z)$ is a point for all $\lambda\in\Lambda'$ if and only if $Z\times Z\subset  \cap_{\lambda\in\Lambda'}E_{\lambda}$.
  Thus by \eqref{eqn:20231025}, we have
  \begin{equation*}
  \begin{split}
 \text{$f_{\infty}(Z)$ is a point}
 &\Longleftrightarrow
  \text{$Z\times Z\subset  \cap_{\lambda\in\Lambda}E_{\lambda}$}\\
  &\Longleftrightarrow
  \text{$f_{\lambda}(Z)$ is a point for all $\lambda\in\Lambda$
  }.
  \end{split}
 \end{equation*}
 The proof is completed.
\begin{comment}
 	We take $\lambda_1\in \Lambda$.
 	Let $p_1:V\to W_1$ be the quasi-Stein factorization of $f_{\lambda_1}:V\to S_{\lambda_1}$.
 	
 	Next we take (if it exists)  $\lambda_2\in \Lambda$ such that for the quasi-Stein factorization $p_2:V\to W_2$ of $(s_{\lambda_1},s_{\lambda_2}):X\to S_{\lambda_1}\times S_{\lambda_2}$, there exists a subvariety $Z\subset V$ such that $p_1(Z)$ is a point, but $p_2(Z)$ is not a point.

 	Similarly, we take (if it exists) $\lambda_3\in\Lambda$ such that for the Stein factorization $p_3:V\to W_3$ of $(f_{\lambda_1},f_{\lambda_2},f_{\lambda_3}):V\to S_{\lambda_1}\times S_{\lambda_2}\times S_{\lambda_3}$, there exists a subvariety $Z\subset V$ such that $p_2(Z)$ is a point, but $p_3(Z)$ is not a point. 	
	
 	We repeat this process forever we may continue.
 	However by \cref{lem:2023010511}, this process should terminate to get $\lambda_1,\ldots,\lambda_n\in \Lambda$.
 	We let $S_{\infty}=W_n$, namely $f_{\infty}:V\to S_{\infty}$ is the Stein factorization of $(f_{\lambda_1},\ldots,f_{\lambda_n}):V\to S_{\lambda_1}\times\cdots\times S_{\lambda_n}$.
 	
 	Now let $\lambda\in\Lambda$.
 	Then by the construction, if $f_{\infty}(Z)$ is a point, then $(f_{\lambda_1},\ldots,f_{\lambda_n},f_{\lambda})(Z)$ is a point.
	In particular, $f_{\lambda}(Z)$ is a point.	
	\end{comment}
 \end{proof}

 We also need the following generalized Stein factorization proved by Henri Cartan in \cite[Theorem 3]{Car60}. 
 \begin{thm}\label{lem:Stein}
 	Let $X, S$ be complex spaces and $f: X \rightarrow S$ be a morphism. Suppose a connected component $F$ of a fibre of $f$ is compact. Then, $F$ has an open neighborood $V$ such that $f(V)$ is a locally closed analytic subvariety $S$ and $V \rightarrow f(V)$ is proper.
 	
 	Suppose furthermore that $X$ is normal and that every connected component $F$ of a fibre of $f$ is compact. The set $Y$ of connected components of fibres of $f$ can be endowed with the structure of a normal complex space such that $f$ factors through  the natural map $e: X \rightarrow Y$ which is a proper holomorphic fibration.   \qed
 \end{thm}

 \section{Some non-abelian Hodge theories}  
In this section, we will build upon the previous work of Simpson \cite{Sim92}, Iyer-Simpson \cite{IS07, IS08}, and Mochizuki \cite{Moc07, Moc06}   to further develop non-abelian Hodge theories over quasi-projective varieties. We begin by establishing the functoriality of pullback for regular filtered Higgs bundles (cf.  \Cref{prop:functoriality}). Then we clarify the $\bC^*$ and $\bR^*$-action on the character varieties of smooth quasi-projective varieties, following   \cite{Moc06}.  Lastly, we prove \cref{prop:pullcommute}, which essentially states that the natural morphisms of character varieties induced by algebraic morphisms commute with the $\bC^*$-action.  This section's significance lies in its essential role in establishing   \Cref{lem:C*,prop:nonrigid}, which serves as a critical cornerstone of the whole paper. 
\subsection{Regular filtered Higgs bundles}  \label{sec:prolong}
In this subsection, we recall the notions of regular filtered Higgs bundles (or parabolic Higgs bundles). For more details refer to \cite{Moc06}. Let $\overline{X}$ be a complex manifold with a reduced simple normal crossing divisor $D=\sum_{i=1}^{\ell}D_i$, and let $X=\overline{X}\backslash D$ be the complement of $D$. We denote the inclusion map of $X$ into $\overline{X}$ by $j$.

\begin{dfn}\label{dfn:parab-higgs}
	A \emph{regular filtered Higgs bundle} $(\bm{E}_*,\theta)$ on $(\overline{X}, D)$ is holomorphic vector bundle $E$ on $X$, together with an $\mathbb{R}^\ell$-indexed
	filtration ${}_{ \bm{a}}E$ (so-called {\em parabolic structure}) by locally free subsheaves of $j_*E$ such that
	\begin{enumerate}[leftmargin=0.7cm]
		\item $\bm{a}\in \mathbb{R}^\ell$ and ${}_{\bm{a}}E|_X=E$. 
		\item  For $1\leq i\leq \ell$, ${}_{\bm{a}+\bm{1}_i}E = {}_{\bm{a}}E\otimes \cO_X(D_i)$, where $\bm{1}_i=(0,\ldots, 0, 1, 0, \ldots, 0)$ with $1$ in the $i$-th component.
		\item $_{\bm{a}+\bm{\epsilon}}E = {}_{\bm{a}}E$ for any vector $\bm{\epsilon}=(\epsilon, \ldots, \epsilon)$ with $0<\epsilon\ll 1$.
		\item  The set of {\em weights} \{$\bm{a}$\ |\  $_{\bm{a}}E/_{\bm{a}-\bm{\epsilon}}E\not= 0$  for any vector $\bm{\epsilon}=(\epsilon, \ldots, \epsilon)$ with $0<\epsilon\ll 1$\}  is  discrete in $\mathbb{R}^\ell$.
		\item There is a $\cO_{X}$-linear map, so-called Higgs field, 
		$$\theta:E\to \Omega_{X}^1 \otimes E$$
		such that
		$\theta\wedge \theta=0$, 
		and
		\begin{align*} 
			\theta(_{\bm{a}}E)\subseteq \Omega_{\overline{X}}^1(\log D)\otimes {}_{\bm{a}}E.
		\end{align*} 
	\end{enumerate}
\end{dfn}
Denote $_{\bm{0}}E$ by $\diae$, where $\bm{0}=(0, \ldots, 0)$. When disregarding the Higgs field, $\bm{E}_*$ is referred to as a \emph{parabolic bundle}.  %By the work of Borne-Vistoli  the parabolic structure of a parabolic bundle  is  \emph{locally abelian}, \emph{i.e.} it admits a local frame compatible with the filtration (see e.g. \cite{IS07}).  

A natural class of regular filtered   Higgs bundles comes from prolongations of tame harmonic bundles. We first   recall some notions in \cite[\S 2.2.1]{Moc07}.  Let $E$ be a holomorphic vector bundle with a smooth hermitian metric $h$ over $X$.

Let $U$ be an open subset of $\overline{X}$ with an admissible coordinate $(U; z_1, \ldots, z_n)$ with respect to $D$. For any section $\sigma\in \Gamma(U\backslash D,E|_{U\backslash D})$, let $|\sigma|_h$ denote the norm function of $\sigma$ with respect to the metric $h$. We denote $|\sigma|_h= \cO(\prod_{i=1}^{\ell}|z_i|^{-b_i})$ if there exists a positive number $C$ such that $|\sigma|_h\leq C\cdot\prod_{i=1}^{\ell}|z_i|^{-b_i}$. For any $\bm{b}\in \bR^\ell$, say $-\mbox{ord}(\sigma)\leq \bm{b}$ means the following:
$$
|\sigma|_h=\cO(\prod_{i=1}^{\ell}|z_i|^{-b_i-\varepsilon})
$$
for any real number  $\varepsilon>0$ and $0<|z_i|\ll1$. For any $\bm{b}$, the sheaf ${}_{\bm{b}} E$ is defined as follows: 
\begin{align}\label{eq:prolongation}
	\Gamma(U, {}_{\bm{b}} E):=\{\sigma\in\Gamma(U\backslash D,E|_{U\backslash D})\mid -\mbox{ord}(\sigma)\leq \bm{b} \}. 
\end{align}
The sheaf ${}_{\bm{b}} E$ is called the prolongment of $E$ by an increasing order $\bm{b}$. In particular, we use the notation ${}^\diamond E$ in the case $\bm{b}=(0,\ldots,0)$.

According to  Simpson \cite[Theorem 2]{Sim90} and Mochizuki \cite[Theorem 8.58]{Moc07}, the above prolongation gives a regular filtered Higgs bundle.
\begin{thm}[Simpson, Mochizuki] \label{thm:SM} Let $\overline{X}$ be a complex manifold and $D$ be a simple normal crossing divisor on $\overline{X}$. If $(E, \theta, h)$ is a tame harmonic bundle on $\overline{X}\backslash D$, then the corresponding filtration $_{\bm{b}}E$ defined above defines a regular filtered Higgs bundle $(\bm{E}_*, \theta)$ on $(\overline{X},D)$.  	\qed
\end{thm} 

\subsection{Pullback of parabolic bundles}\label{sec:adapt}
In this subsection, we   introduce the concept of pullback of parabolic bundles. We refer the readers to \cite{IS07,IS08} for a  more systematic treatment. We avoid the language of Deligne-Mumford stacks in \cite{IS07,IS08}.   This subsection is conceptional and we shall make precise computations in next subsection. 

A parabolic line bundle is a parabolic sheaf $F$ such that all the $_{\bm{a}}F$ are line bundles. An important class of examples is obtained as follows: let $L$ be a line bundle on $\overline{X}$, if ${\bm{a}}=(a_1,\ldots,a_\ell)$ is a $\mathbb{R}^\ell$-indexed, then we can define a parabolic line bundle denoted
$
L^{\bm{a}}_*
$ 
by setting
\begin{align}\label{eq:parabolic}
	_{\bm{b}}L^{\bm{a}}:=L\otimes \mathcal{O}_{\overline{X}}\left(\sum_{i=1}^\ell \lfloor a_i+b_i \rfloor D_i\right)
\end{align} 
for any ${\bm{b}}\in \mathbb{R}^\ell$.
%If $F$ is a parabolic sheaf, set $F_{\infty}$ equal to the extension $j_*\left(j^* F_\alpha\right)$ for any $\alpha$, where $j: X-D \hookrightarrow X$ is the inclusion. It is the sheaf of sections of $F$ which are meromorphic along $D$, and it doesn't depend on $\alpha$. Note that the $F_\alpha$ may all be considered as subsheaves of $F_{\infty}$
%We can define a tensor product of torsion-free parabolic sheaves: set
%$$
%(F \otimes G)_\alpha
%$$
%to be the subsheaf of $F_{\infty} \otimes \mathcal{O}_X G_{\infty}$ generated by the $F_{\beta^{\prime}} \otimes G_{\beta^{\prime \prime}}$ for $\beta^{\prime}+\beta^{\prime \prime} \leq \alpha$.
%On the other hand, if $E$ is a torsion-free sheaf on $X$ then it may be considered as a parabolic sheaf (we say "with trivial parabolic structure") by setting $E_\alpha$ to be $E\left(\sum a_i D_i\right)$ for $a_i$ the greatest integer $\leq \alpha_i$.

%With this notation we may define for any vector bundle $E$ on $X$ the parabolic bundle
%$$
%E\left(\sum_{i=1}^k \alpha_i D_i\right):=E \otimes \mathcal{O}_X\left(\sum_{i=1}^k \alpha_i %D\right) .
%$$
%Lemma 2.1. Any parabolic line bundle has the form $L\left(\sum_{i=1}^k \alpha_i D_i\right)$ for $L$ a line bundle on $X$. This may be viewed as $L(B)$ where $B$ is a rational divisor on $X$ (supported on $D$ ).
\begin{dfn}[Locally abelian parabolic bundle]\label{def:locally abelian}
	A parabolic sheaf $\bm{E}_*$ is a \emph{locally abelian parabolic bundle} if, in a  neighborhood of any point $x \in \overline{X}$ there is an isomorphism between $\bm{E}_*$ and a direct sum of parabolic line bundles. 
\end{dfn} 
Let $f:\overline{Y}\to \overline{X}$ be a    holomorphic map of complex manifolds. Let $D'=\sum_{j=1}^{k}D_j'$ and $D=\sum_{i=1}^{\ell}D_i$ be simple normal crossing divisors on $\overline{Y}$ and $\overline{X}$ respectively.  Assume tht $f^{-1}(D)\subset D'$.  Denote by $n_{ij}={\rm ord}_{D'_j}f^*D_i\in \bZ_{\geq  0}$.     Let   $L$ be a line bundle on $\overline{X}$ and let $ L^{\bm{a}}_*$  be the parabolic line bundle defined in \eqref{eq:parabolic}.  Set \begin{align}\label{eq:af}
	f^*\bm{a}:=(\sum_{i=1}^{\ell}n_{i1}a_i,\ldots,\sum_{i=1}^{\ell}n_{ik}a_i)\in \bR^k.
\end{align} 
Then $f^*(L^{\bm{a}}_*)$ is defined by setting
\begin{align}\label{eq:parabolic2}
	_{\bm{b}}(f^*L)^{f^*\bm{a}}:=f^*L\otimes \mathcal{O}_{\overline{Y}}\left(\sum_{j=1}^k \lfloor \sum_{i=1}^{\ell}n_{ij}a_i+b_j \rfloor D'_j\right)
\end{align} 
for any ${\bm{b}}\in \mathbb{R}^k$. 

Let $\overline{X}$ be a compact complex manifold. Consider a locally abelian parabolic bundle $\bm{E}_*$ defined on $\overline{X}$. We can cover $\overline{X}$ with open subsets $U_1,\ldots,U_m$, such that $\bm{E}_*|_{U_i}$ can be expressed as a direct sum of parabolic line bundles on each $U_i$.

Using this decomposition, we define the pullback $f^*(\bm{E}_*|_{U_i})$ as in \eqref{eq:parabolic2}. It can be verified that $f^*(\bm{E}_*|_{U_i})$ is compatible with $f^*(\bm{E}_*|_{U_j})$ whenever $U_i\cap U_j\neq\varnothing$. This allows us to extend the local pullback to a global level, resulting in the definition of the pullback of a locally abelian parabolic bundle denoted by $f^*\bm{E}_*$. In next section, we will see an explicit description of the pullback of regular filtered Higgs bundles induced by tame harmonic bundles.

\subsection{Functoriality of pullback of regular filtered Higgs bundle} \label{sec:pullback}
We recall some notions in \cite[\S 2.2.2]{Moc07}.  
Let $X$ be a complex manifold, $D$ be a simple normal crossing divisor on $X$, and $E$ be a holomorphic vector bundle on $X\backslash D$ such that  $E|_{X\backslash D}$ is equipped with  a hermitian metric $h$. Let $\bm{v}=(v_1,\ldots,v_r)$ be a smooth frame of $E|_{X\backslash D}$. We obtain the $H(r)$-valued function $H(h,\bm{v})$ defined over $X\backslash D$,whose $(i,j)$-component is given by $h(v_i,v_j)$.  %The frame $\mathbf{v}$ is called adapted, if $H(h,\mathbf{v})$ and $H(h,\mathbf{v})^{-1}$ are bounded. 

Let us consider the case $X=\bD^n$, and $D=\sum_{i=1}^{\ell}D_i$ with $D_i=(z_i=0)$. We have the coordinate $(z_1,\ldots,z_n)$. Let  $h$, $E$ and $\bm{v}$  be as above.

\begin{dfn}\label{def:adapted}
	A smooth frame $\bm{v}$ on $X\backslash D$ is called \emph{adapted up to log order}, if  the following inequalities hold over $X\backslash D$:
	$$
	C^{-1}(-\sum_{i=1}^{\ell}\log |z_i|)^{-M}\leq H(h,\bm{v})\leq  C(-\sum_{i=1}^{\ell}\log |z_i|)^{M}
	$$   
	for some positive numbers $M$ and $C$.
\end{dfn}

The goal of this subsection is to establish the following result concerning the functoriality of the pullback of a regular filtered Higgs bundle. This result will play a crucial role in proving \cref{lem:C*}.
\begin{proposition}\label{prop:functoriality}
	Consider a morphism $f:\overline{Y}\to \overline{X}$ of  smooth projective varieties    $\overline{X}$ and $\overline{Y}$. Let   $D$ and $D'$ be   simple normal crossing divisors on $\overline{X}$ and $\overline{Y}$ respectively. Assume that $f^{-1}(D)\subset D'$.   Let $(E,\theta,h)$ be a tame harmonic bundle on $X:=\overline{X}\backslash D$.   Let $(\bm{E}_*,\theta)$ be the regular filtered Higgs bundle defined in \cref{sec:prolong}.   Consider the pullback of $f^*\bm{E}_*$  defined in \cref{sec:adapt}, which is also a parabolic bundle over $(\overline{Y},D')$. Then 
	\begin{thmlist} 
		\item $f^*\bm{E}_*$ is  the prolongation $\tilde{E}_*$ of $f^*E$ using the norm growth with respect to the metric $f^*h$ as defined in \eqref{eq:prolongation}.  
		\item $(f^*\bm{E}_*,f^*\theta)$ is a filtered regular Higgs bundle.   
	\end{thmlist} 
\end{proposition}  
\begin{proof}
	Since this is a local result, we assume that  $\overline{X}:=\bD^n$ and $D:=\bigcup_{i=1}^{\ell}\left\{z_i=0\right\}$. Let $\overline{Y}:=\bD^m$ and $D':=\bigcup_{j=1}^k\left\{w_j=0\right\}$.  Then, $f^*\left(z_i\right)=\prod_{j=1}^k w_j^{n_{ij}} g_i$ for some invertible functions $\{g_i\}_{i=1, \ldots, \ell}\subset \cO(\overline{Y})$. 
	
	%For any $\bm{b}\in \bR^k$, we set
	%$$
	%\mathcal{S}(\bm{b}):=\left\{(\bm{a}, \bm{n}) \in \bR^{\ell} \times \mathbb{Z}_{\geq 0}^k \mid f^*\bm{a}+\boldsymbol{n} \leq \boldsymbol{b}\right\}
	%$$
	%where $f^*\bm{a}$ is defined in \eqref{eq:af}. 
	%Let $\boldsymbol{E}_*$ be a filtered sheaf on $(\overline{X}, D)$. According to \cite[\S 2.5.3.3]{Moc11}, there exists an explicit definition of pullback of parabolic bundle.  It is defined as follows:
	%\begin{align} \label{eq:pullback}
	%{}_{\bm{b}}f^*(\bm{E}_*):=\sum_{(a, n) \in \mathcal{S}(b)} \boldsymbol{w}^{-\bm{n}} f^*\left({ }_a E\right)
	%\end{align}
	%Here $\boldsymbol{w}^{-\bm{n}}:=\prod_{j=1}^{k}w_j^{-n_j}$.    It is independent of the choice of the coordinate systems $\boldsymbol{z}$ and $\boldsymbol{w}$. 
	%By \eqref{eq:theta}, we have
	%$$
	%f^*\theta: f^*\left({ }_a E\right)\to \Omega_{\overline{Y}}(\log D')\otimes f^*\left({ }_a E\right).
	%$$
	%Therefore,   it follows from \eqref{eq:pullback} that 
	%$$
	%f^*\theta:   {}_{\bm{b}}f^*(\bm{E}_*)\to  {}_{\bm{b}}f^*(\bm{E}_*)\otimes \Omega_{\overline{Y}}(\log D').
	%$$ 
	%Hence the pullback of $(\bm{E}_*,\theta)$  by $f$ is also a  regular filtered Higgs bundle, denoted by $f^*(\bm{E}_*,\theta)$.  The first assertion is proved. 
	
	%\medspace
	
	By \cite[Proposition 8.70]{Moc07}, there exists a holomorphic frame $\bm{v}=(v_1,\ldots,v_r)$   of $\diae|_{\overline{X}}$ and $\{a_{ij}\}_{i=1,\ldots,r; j=1,\ldots,\ell}\subset  \bR$  such that if we put 
	$
	\tilde{v}_i:=v_i\cdot \prod_{j=1}^{\ell}|z_j|^{-a_{ij}},
	$ 
	then for  the smooth frame  $\widetilde{\bm{v}}=(\tilde{v}_1,\ldots,\tilde{v}_r)$ over $X=\overline{X}\backslash D$,  $H(h,\widetilde{\bm{v}})$ is adapted to log order  in the sense of \cref{def:adapted}.   
	
	Define $L_i$ to be the sub-line bundle of $\diae$  generated by $v_i$. Write $\bm{a}_i:=(a_{i1},\ldots,a_{i\ell})\in \bR^\ell$.  Consider the parabolic line bundle $(L_i)^{\bm{a}_i}_*$ over $(\overline{X},D)$  defined in \eqref{eq:parabolic}, namely, \begin{align}\label{eq:parabolic3}
		_{\bm{b}}(L_i)^{\bm{a}_i}:=L_i\otimes \mathcal{O}_{\overline{X}}\left(\sum_{j=1}^\ell \lfloor a_{ij}+b_j \rfloor D_j\right)
	\end{align} 
	for any $\bm{b}\in \bR^\ell$. 
	\begin{claim}\label{claim:locally abelian}
		The parabolic bundles $\bm{E}_*$ and $\oplus_{i=1}^{r}(L_i)^{\bm{a}_i}_*$ are the same. In particular, $\bm{E}_*$ is locally abelian. 
	\end{claim}
	\begin{proof}
		By \eqref{eq:prolongation}, for any $\bm{b}\in \bR^\ell$, any holomorphic section $\sigma\in \Gamma(\overline{X}, {}_{\bm{b}}E)$ satisfies 
		$$
		|\sigma|_h=\cO(\prod_{j=1}^{\ell}|z_j|^{-b_j-\varepsilon})\quad \mbox{ for all real}\ep>0.
		$$
		As $\bm{v}$ is a frame for $\diae$, one can write $\sigma=\sum_{i=1}^{r}g_iv_i$ where $g_i$ is a holomorphic function defined on $X$. Write $\bm{g}:=(g_1,\ldots,g_r)$. Since $H(h,\widetilde{\bm{v}})$ is adapted to log order, it follows that
		$$
		C^{-1}(-\sum_{j=1}^{\ell}\log |z_j|)^{-M}\cdot \sum_{i=1}^{r}|g_i|^2\prod_{j=1}^{\ell}|z_j|^{2a_{ij}}\leq  \overline{\bm{g}}	H(h,\bm{v})\bm{g}^T=	|\sigma|_h^2=\cO(\prod_{j=1}^{\ell}|z_j|^{-2b_i-\varepsilon}) 
		$$
		for any $\ep>0$. 
		Hence for each $i$ and any $\ep>0$ we have $$|g_i|^2=\cO(\prod_{j=1}^{\ell}|z_j|^{-2(b_j+a_{ij})-\varepsilon}).$$ 
		Therefore, $\ord_{D_j}g_i\geq -\lfloor b_j+a_{ij}\rfloor$. This proves that 
		$$
		{}_{\bm{b}}E\subset \oplus_{i=1}^{r}{}_{\bm{b}}(L_i)^{\bm{a}_i}. 
		$$  
		
		\medspace
		
		On the other hand, we consider  any section $\sigma\in \Gamma(\overline{X}, {}_{\bm{b}}(L_i)^{\bm{a}_i})$.   Then $\sigma=gv_i$ for some meromorphic function $g$ defined over $\overline{X}$ such that $\ord_{D_j}g_i\geq -\lfloor b_j+a_{ij}\rfloor$ by \eqref{eq:parabolic3}.  Therefore, there exists some positive constant $C>0$ such that  
		$$
		|\sigma|_h^2=|g|^2|v_i|_h^2\leq C\prod_{j=1}^{\ell}|z_j|^{-2(b_j+a_{ij})}\cdot |\tilde{v}_i|_h^2\cdot \prod_{j=1}^{\ell}|z_j|^{2a_{ij}}=C \prod_{j=1}^{\ell}|z_i|^{-2b_j}\cdot |\tilde{v}_i|_h^2=\cO(\prod_{i=1}^{\ell}|z_i|^{-b_i-\varepsilon}).
		$$ 
	holds for any every $\ep>0$,	as we have $|\tilde{v}_i|_h^2\leq 	 C(-\sum_{j=1}^{\ell}\log |z_j|)^{M}$ for some $C,M>0$.  This implies that 
		$$ \oplus_{i=1}^{r}{}_{\bm{b}}(L_i)^{\bm{a}_i}\subset  {}_{\bm{b}}E.$$
		The claim is proved. 
	\end{proof}
	
	Consider the pullback $f^*\bm{v}:=(f^*v_1,\ldots,f^*v_m)$.  Then it is a holomorphic frame of $f^*E|_{Y}$ where $Y:=\overline{Y}\backslash D'$.  Note that  we have
	$$
	f^*\tilde{v}_i:=f^*v_i\cdot \prod_{j=1}^{\ell}|f^*z_j|^{-a_{ij}}=f^*v_i\cdot \prod_{j=1}^{\ell}\prod_{q=1}^{k}|w_q|^{-n_{jq}a_{ij}}\cdot g'_i
	$$
	for some invertible holomorphic function $g'_i\in \cO(\overline{Y})$. 
Similar to  \eqref{eq:af}, we set
	$$
	f^*\bm{a}_i:= (\sum_{j=1}^{\ell}n_{j1}a_{ij},\ldots,\sum_{j=1}^{\ell}n_{jk}a_{ij})\in \bR^k.
	$$ 
	Then we have $$
	f^*\tilde{v}_i:=f^*v_i\cdot |\bm{w}^{-f^*\bm{a}_i}|\cdot g'_i. 
	$$
	Since $H(h,\widetilde{\bm{v}})$ is adapted to log order, it is easy to check that $H(f^*h,f^*\widetilde{\bm{v}})$ also is adapted to log order. Set $e_i:=f^*v_i\cdot |\bm{w}^{-f^*\bm{a}_i}|$  for   $i=1,\ldots,r$ and $\bm{e}:=(e_1,\ldots,e_r)$. Then $\bm{e}$ is a smooth frame for $f^*E|_{Y}$.    Since $g_i'$ is invertible, it follows that  $H(f^*h,\bm{e})$  is also adapted to log order. Consider the prolongation $(\tilde{E}_*,\tilde{\theta})$ of the tame harmonic bundle $(f^*E,f^*\theta,f^*h)$ using the norm growth   as defined in \eqref{eq:prolongation}.     Applying the result from \cref{claim:locally abelian} to $(f^*E,f^*\theta,f^*h)$, we can conclude that the parabolic bundle $\tilde{E}_*$ is given by
	\begin{align}\label{eq:pullbackpara}
		\tilde{E}_*= \oplus_{i=1}^{r}(f^*L_i)^{f^*\bm{a}_i}_*,
	\end{align}
	where $(f^*L_i)^{f^*\bm{a}_i}_*$  are parabolic line bundles defined by \begin{align}\label{eq:parabolic4}
		_{\bm{b}}(f^*L_i)^{f^*\bm{a}_i}:=f^*L_i\otimes \mathcal{O}_{\overline{Y}}\left(\sum_{j=1}^k \lfloor \sum_{q=1}^{\ell}n_{qj}a_{iq}+b_j \rfloor D'_j\right).
	\end{align} 
	On the other hand, by our definition of pullback of parabolic bundles and \cref{claim:locally abelian}, we have
	$$
	f^*\bm{E}_*:=\oplus_{i=1}^{r}f^* (L_i)^{\bm{a}_i}_* 
	$$
	where $f^* (L_i)^{\bm{a}_i}_* $ is the pullback of parabolic line bundle $ (L_i)^{\bm{a}_i}_* $ defined in \eqref{eq:parabolic2}.  By performing a straightforward computation, we find that
	$$
	_{\bm{b}}	(f^* (L_i)^{\bm{a}_i}_*)  =f^*L_i\otimes \mathcal{O}_{\overline{Y}}\left(\sum_{j=1}^\ell \lfloor \sum_{q=1}^{\ell}n_{qj}a_{iq}+b_j \rfloor D'_j\right) 
	$$
	for every $\bm{b}\in \bR^\ell$. 
	This equality together with \eqref{eq:pullbackpara} and \eqref{eq:parabolic4} yields 
 $
	\tilde{E}_*= f^*\bm{E}_*.
	$ 
	We prove our first assertion.  The second assertion can be deduced from the first one, combined with \cref{thm:SM}.
\end{proof}

\subsection{$\bC^*$-action and $\bR^*$-action  on character varieties}\label{sec:C*action}
Consider a smooth projective variety $\overline{X}$ equipped with a simple normal crossing divisor $D$. We define $X$ as the complement of $D$ in $\overline{X}$. Additionally, we fix an ample line bundle $L$ on $\overline{X}$. Let $\varrho:\pi_1(X)\to \GL_{N}(\mathbb{C})$ be a reductive representation.

 According to \cref{moc}, there exists a tame pure imaginary harmonic bundle $(E,\theta,h)$ on $X$ such  that $(E, \nabla_h+\theta+\theta_h^\dagger)$ is flat, with the monodromy representation being precisely $\varrho$. Here $\nabla_h$ is the Chern connection of $(E,h)$ and $\theta_h^\dagger$ is the adjoint of $\theta$ with respect to $h$.   Let   $(\bm{E}_*,\theta)$ be the prolongation of $(E,\theta)$ on $\overline{X}$ defined in \cref{sec:prolong}. By \cite[Theorem 1.4]{Moc06}, $(\bm{E}_*,\theta)$ is a  $\mu_L$-polystable regular filtered Higgs bundle  on $(\overline{X}, D)$  with trivial characteristic numbers.  Therefore, for any $t\in \bC^*$, $(\bm{E}_*,t\theta)$ be also  a  $\mu_L$-polystable regular filtered Higgs bundle  on $(\overline{X}, D)$  with trivial characteristic numbers. By \cite[Theorem 9.4]{Moc06}, there is a pluriharmonic metric $h_t$ for $(E,t\theta)$ adapted to the parabolic structures of $(\bm{E}_*,t\theta)$. Then $(E,t\theta,h_t)$ is a harmonic bundle and thus the connection $\nabla_{h_t}+t\theta+\bar{t}\theta_{h_t}^\dagger$ is flat. Here $\nabla_{h_t}$ is the Chern connection for $(E,h_t)$ and $\theta_{h_t}^\dagger$ is the adjoint ot $\theta$ with respect to $h_t$. Let us denote by $\varrho_t:\pi_1(X)\to \GL_{N}(\bC)$ the monodromy representation of $\nabla_{h_t}+t\theta+\bar{t}\theta_{h_t}^\dagger$. It should be noted that the representation $\varrho_t$ is well-defined up to conjugation. As a result, the $\bC^*$-action is only well-defined over $M_{\rm B}(X,N)$ and we shall denote it by
$$
t.[\varrho]:=[\varrho_t] \quad \mbox{for any }t\in \bC^*.
$$   It is important to observe that unlike the compact case, $\varrho_t$ is not necessarily reductive in general, even if the original representation $\varrho$ is reductive. However, if $t\in \bR^*$, $(E,t\theta)$ is also pure imaginary and by \cref{moc}, $\varrho_t$ is reductive.  Nonetheless, we can obtain a family of  (might not be semisimple) representations $\{\varrho_t:\pi_1(X)\to \GL_N(\bC)\}_{t\in \bC^*}$.    By \cite[Proofs of Theorem 10.1 and Lemma 10.2]{Moc06} we have
\begin{lem}\label{lem:continuous}
	The map 
	\begin{align*}
		\Phi:\bR^*&\to M_{\rm B}(\pi_1(X), N)\\
		t&\mapsto [\varrho_t]
	\end{align*}
	is continuous. $\Phi(\{t\in \bR^*\mid |t|<1 \})$ is relatively compact in  $M_{\rm B}(\pi_1(X), N)$.\qed
\end{lem}
Note that \cref{lem:continuous} can not be seen directly from \cite[Lemma 10.2]{Moc06} as he did not treat the character variety in his paper.  Indeed, based on Uhlenbeck's compactness in Gauge theory, Mochizuki's proof can be read as follows: for any $t_n\in \bR^*$ converging to $0$, after subtracting to a subsequence, there exists some $\varrho_0:\pi_1(X)\to \GL_{N}(\bC)$ and $g_n\in \GL_{N}(\bC)$ such that $\lim\limits_{n\to\infty}g_n^*\varrho_{t_n}=\varrho_0$ in the representation variety $R(\pi_1(X),\GL_{N})(\bC)$. Moreover, one can check that $\varrho_0$ corresponds to some tame pure imaginary harmonic bundle, and thus by \cref{moc} it is reductive (cf. \cite{BDDM} for a more detailed study). For this reason, we can see that it will be more practical to work with $\bR^*$-action instead of $\bC^*$-action as the representations we encounter are all reductive.  %This fact is crucial in the proof of \cref{lem:C*}.

When $X$ is compact, Simpson proved that   $\lim_{t\to 0}\Phi(t)$ exists and underlies a $\bC$-VHS. However, this result is current unknown in the quasi-projective setting. Instead, Mochizuki proved that, we achieve a $\bC$-VHS after finite steps of deformations. Let us recall it briefly and the readers can refer to \cite[\S 10.1]{Moc06} for more details.
 
	Let $\varrho:\pi_1(X)\to \GL_N(\bC)$ be a  reductive representation.  Then there exists a tame and pure imaginary harmonic bundle $(E,\theta,h)$ corresponding to $\varrho$.    Then the induced regular filtered Higgs bundle   $(\bm{E}_*,\theta)$  on $(\overline{X}, D)$ is $\mu_L$-polystable   with trivial characteristic numbers. Hence we have a decomposition
	$$
	 (\bm{E}_*,\theta)=\oplus_{j\in \Lambda}(\bm{E}_{j*},\theta_j)\otimes \bC^{m_j} 
	$$ 
	where $(\bm{E}_{j*},\theta_j)$ is $\mu_L$-stable  regular filtered Higgs bundle with trivial characteristic numbers. Put $r(\varrho):=\sum_{j\in \Lambda}m_j$. Then $r(\varrho)\leq \rank E$.   For any $t\in \bR^*$, we know that $(E,t\theta)$ is still tame and pure imaginary and thus $\varrho_t$ is also reductive.   
	Since $\varrho(\{t\in \bR^*\mid |t|<1\})$ is relatively compact, then there exists some $t_n\in \bR^*$ which converges to zero such that     $\lim_{t_n\to 0}[\varrho_{t_n}]$ exists, denoting by $[\varrho_0]$. Moreover, $\varrho_0$ corresponds to some tame harmonic bundle. There are two possibilities:
	\begin{itemize}
		\item For each $j\in \Lambda$, $(\bm{E}_{j*},t_n\theta_j)$ converges to some $\mu_L$-stable regular filtered Higgs sheaf (cf.  \cite[p. 96]{Moc06} for the definition of convergence). Then by \cite[Proposition 10.3]{Moc06}, $\varrho_0$ underlies a $\bC$-VHS.
		\item For some $i\in \Lambda$, $(\bm{E}_{i*},t_n\theta_i)$ converges to some $\mu_L$-semistable regular filtered Higgs sheaf, but not  $\mu_L$-stable.  Then by \cite[Lemma 10.4]{Moc06}, we have $r(\varrho)<r(\varrho_0)$. In other words, letting $\varrho_i$ be the representation corresponding to $(\bm{E}_{j*},\theta_j)$ and $\varrho_{i,t}$ be the deformation under $\bC^*$-action. Then   $\lim_{n\to\infty}\varrho_{i,t_n}$ exists, denoted by $\varrho_{i,0}$. Then  $\varrho_{i,0}$ corresponds to some tame harmonic bundle, and thus also a  $\mu_L$-polystable   regular filtered Higgs bundle which is not stable.  In this case, we further deform $\varrho_0$ until we achieve Case 1.	\end{itemize}  
In summary, Mochizuki's result implies the following, which we shall refer to as \emph{Mochizuki's ubiquity}, analogous to the term \emph{Simpson's ubiquity} for the compact case (cf. \cite{Sim91}).
	\begin{thm}\label{thm:Moc ubi} 
 Let $X$ be a smooth quasi-projective variety. Consider $\kC$, a Zariski closed subset of $M_{\rm B}(X,G)(\bC)$, where $G$ denotes a complex reductive group. If $\kC$ is invariant under the action of $\bR^*$ defined above, then each geometrically connected component of $\kC(\bC)$ contains a $\bC$-point $[\varrho]$ such that $\varrho:\pi_1(X)\to \GL_{N}(\bC)$ is a reductive representation that underlies  a $\bC$-variation of Hodge structure. \qed
			\end{thm}
	
\subsection{Pullback of reductive representations commutes with $\bC^*$-action}
In this section, we prove that the $\bC^*$-action on character varieties commutes with the pullback. %The upshot is that, the pullback of a regular filtered Higgs bundle $(\bm{E}_*, \theta)$ only depends on the parabolic bundle $\bm{E}_*$, not the Higgs field $\theta$, and thus it commutes with the $\bC^*$-action. 
\begin{proposition}\label{prop:pullcommute}
Let $f:Y\to X$ be  a morphism of smooth quasi-projective varieties.  If $\varrho:\pi_1(X)\to \GL_{N}(\bC)$ is a reductive representation, then for any $t\in \bC^*$, we have
\begin{align}\label{eq:pullcommute}
	 f^*(t. [\varrho])=t. [f^*\varrho].
\end{align} 
\end{proposition}
\begin{proof}
Let    $\overline{X}$ and $\overline{Y}$ be smooth projective compactifications of $X$ and $Y$ such that $D:=\overline{X}\backslash X$ and $D':=\overline{Y}\backslash Y$ are  simple normal crossing divisors.   We may assume that $f$ extends to a morphism  $f:\overline{Y}\to \overline{X}$.  

 By  \cref{moc}, there is a tame pure imaginary harmonic bundle $(E,\theta,h)$ on $X$ such that $\varrho$ is the monodromy representation of  the flat connection $\nabla_h+\theta+\theta_h^\dagger$. Then $f^*\varrho$ is the monodromy representation of  $f^*(\nabla_h+\theta+\theta_h^\dagger)$, which is the flat connection corresponding to the harmonic bundle $(f^*E, f^*\theta, f^*h)$. 
 
 Let $(\bm{E}_*,\theta)$  be the induced regular filtered Higgs bundle   on $(\overline{X}, D)$  by $(E,\theta,h)$ defined in \cref{sec:prolong}.  According to \cref{sec:adapt,sec:pullback} we can define the pullback $(f^*\bm{E}_*,f^*\theta)$,  which also forms a regular filtered Higgs bundle on $(\overline{Y}, D')$ with trivial characteristic numbers. %By virtue of \cref{prop:functoriality}, we deduce that  $(f^*\bm{E}_*,f^*\theta)$ is induced by the harmonic bundle $(f^*E,f^*\theta,f^*h)$.     
 
 Fix some ample line bundle $L$ on $\overline{X}$.   It is worth noting that for any $t\in \bC^*$, $(\bm{E}_*,t\theta)$  is $\mu_L$-polystable   with trivial characteristic numbers. By \cite[Theorem 9.4]{Moc06}, there is a pluriharmonic metric $h_t$ for $(E,t\theta)$ adapted to the parabolic structures of $(\bm{E}_*,t\theta)$.    
   Recall that in \cref{sec:C*action}, $\varrho_t$ is defined to be the monodromy representation of the flat connection $\nabla_{h_t}+t\theta+\bar{t}\theta_{h_t}^\dagger$.   It follows that $f^*\varrho_t$ is the monodromy representation of the flat connection $f^*(\nabla_{h_t}+t\theta+\bar{t}\theta_{h_t}^\dagger)$.  
   
By virtue of \cref{prop:functoriality},  the  regular filtered Higgs bundle $(f^*\bm{E}_*,tf^*\theta)$ is the prolongation of the tame harmonic bundle $(f^*E,tf^*\theta,f^*h_t)$  using norm growth defined in \cref{sec:prolong}.   By the definition of $\bC^*$-action,  $(f^*\varrho)_t$ is the monodromy representation of the flat connection $\nabla_{f^*h_t}+tf^*\theta+\bar{t}(f^*\theta)_{f^*h_t}^\dagger$, which is equal to $f^*(\nabla_{h_t}+t\theta+\bar{t}\theta_{h_t}^\dagger)$.  It follows that $(f^*\varrho)_t=f^*\varrho_t$.  This concludes \eqref{eq:pullcommute}. 
\end{proof}
As a direct consequence of \cref{prop:pullcommute}, we have the following result.
\begin{cor}\label{cor:pull}
Let $f:Y\to X$ be  a morphism of smooth quasi-projective varieties.  Let $M\subset M_{\rm B}(X,N)(\bC)$ be a subset  which is invariant by $\bC^*$-action (or $\bR^*$-action). Then for  the morphism $f^*: M_{\rm B}(X,N)\to M_{\rm B}(Y,N)$ between character varieties, $f^*M$ is also   invariant by $\bC^*$-action (or $\bR^*$-action).  \qed
\end{cor}

\section{Construction of the Shafarevich morphism}
The aim of this section is to establish the proofs of \cref{main3}.  Additionally, the  techniques developed in this section will play a crucial role in \cref{sec:HC Stein} dedicated to the proof of the reductive Shafarevich conjecture.
\subsection{Factorizing through non-rigidity} \label{sec:factor}
In this subsection,  $X$ is assumed to be a smooth quasi-projective variety. %and let $G$ be a reductive algebraic group defined over $\bar{\bQ}$. 
Let $\kC\subset M_{\rm B}(X,N)(\bC)$ be a  $\bar{\bQ}$-constructible subset. Since $M_{\rm B}(X,N)$ is a  finite type affine scheme defined over $\bQ$,   $\kC$ is defined over some number field $k$.   

 Let us utilize \cref{lem:simultaneous,thm:KZ} to construct a reduction map  $s_{\kC}:X\to S_\kC$ associated with $\kC$, which allows us to factorize non-rigid representations into those underlying $\bC$-VHS with discrete monodromy.
\begin{dfn}\label{def:reduction ac}
The \emph{reduction map} $s_{\kC}:X\to S_\kC$ is obtained through the simultaneous Stein factorization of the reductions   $\{s_{\tau}:X\to S_\tau\}_{[\tau]\in \kC(K)}$, employing  \cref{lem:simultaneous}. Here   $\tau:\pi_1(X)\to \GL_N(K)$ ranges over all reductive representations with $K$  a non-archimedean local field containing $k$ such that  $[\tau] \in \kC(K)$ and $s_\tau:X\to S_\tau$ is the reduction map constructed  in \cref{thm:KZ}.  
\end{dfn} 
Note that $s_\kC:X\to S_\kC$ is a dominant morphism with connected general fibers. For every subvariety $Z\subset X$, $s_\kC(Z)$ is a point if and only if $s_\tau(Z)$ is a point for  any reductive representation $\tau:\pi_1(X)\to \GL_N(K)$  with $K$  a non-archimedean local field containing $k$ such that  $[\tau] \in \kC(K)$.   
%\begin{equation*}
%	\begin{tikzcd}
	%	X \arrow[r,"s_\kC"]\arrow[dr, "s_\tau"'] & S_{\kC}\arrow[d,"e_\tau"] \\ 
	%	 & S_{\tau}
%	\end{tikzcd}
%\end{equation*}

 The reduction map  $s_{\kC}:X\to S_\kC$  employs the following crucial property, thanks to \cref{thm:KZ}. 
\begin{lem}\label{lem:bounded}
	Let $F\subset X$ be a connected Zariski closed subset such that $s_{\kC}(F)$ is a single point in $S_{\kC}$.
	Then for any non-archimedean local field $L$ and any reductive representation $\tau:\pi_1(X)\to {\rm GL}_N(L)$,  the image $\tau(\mathrm{Im}[\pi_1(F)\to \pi_1(X)])$ is a bounded subgroup of   ${\rm GL}_N(L)$.\end{lem}

\begin{proof}
	By our construction of $s_\kC$,    $s_{\tau}(F)$ is a single point. 
	Hence by \cref{thm:KZ}, $\tau(\mathrm{Im}[\pi_1(F)\to \pi_1(X)])$  is bounded.
\end{proof}
Recall the following definition in \cite[Definition 2.2.1]{KP23}.
\begin{dfn}[Bounded set]
	Let $K$ be a non-archimedean local field. Let $X$ be an affine $K$-scheme of finite type. A subset $B\subset X(K)$ is \emph{bounded} if for every $f\in K[X]$,  the set $\{v(f(b)) \mid b\in B \}$ is bounded below, where $v:K\to \bR$ is the valuation of $K$. 
\end{dfn}

We have the following lemma in \cite[Fact 2.2.3]{KP23}.
\begin{lem}\label{lem:bounded criterion}
	If $B\subset X(K)$ is closed, then $B$ is bounded if and only if $B$ is compact with respect to the analytic topology of $X(K)$.  If $f:X\to Y$ is a morphism of affine $K$-schemes of finite type, then $f$ carries bounded subsets of $X(K)$ to bounded subsets in $Y(K)$. \qed
\end{lem}

\begin{dfn}
	Let $k$ be an algebraically closed field and let  $V$ be a finite dimensional $k$-vector space. Let  $G$ be a group and $\varrho:G\to \GL(V)$ be a representation.  A filtration $0= V_0\subsetneq V_1\subsetneq\cdots\subsetneq V_k=V$
	is called a Jordan-Holder series if all the $V_j$
	are subrepresentations of $\varrho$ and the induces representation $\varrho_j:G\to \GL(V_j/V_{j-1})$ by $\varrho$
	is irreducible. The \emph{semisimplification} of  $\varrho$
	is $\oplus_{i=j}^{k}\varrho_{j}$.  By the Jordan-Holder theorem,  the semisimplification of  $\varrho$ exists and it  is independent of the choice of Jordan-Holder filtation. 
\end{dfn}
We will establish a lemma that plays a crucial role in the proof of \cref{lem:conjugate3} and is also noteworthy in its own regard. 
\begin{lem}\label{lem:semisimplification}
	Let $\varrho:\pi_1(X)\to {\rm GL}_N(K)$ be a (un)bounded representation. Then its semisimplification $\varrho^{ss}:\pi_1(X)\to {\rm GL}_N(\bar{K})$ is also (un)bounded. 
\end{lem}
\begin{proof}
	Note that  
	there exists some $g\in {\rm GL}_N(\bar{K})$ such that  
	\begin{equation}\label{eq:upper}
		g\varrho g^{-1}=\ \left[\begin{array}{cccc}
			\varrho_1& a_{12} & \cdots & a_{1 n} \\
			0 & \varrho_2 & \cdots & a_{2 n} \\
			\vdots & \vdots & \ddots & \vdots \\
			0 & 0 & \cdots & \varrho_n
		\end{array}\right] 
	\end{equation}
	where $\varrho_i:\pi_1(X)\to \GL_{N_i}(\bar{K})$ is an irreducible representation such that $\sum_{i=1}^{n}N_i=N$ and $a_{ij}$ is a map from $\pi_1(X)$ to the set of $N_i\times N_j$ matrices $M_{N_i\times N_j}(\bar{K})$.  Note that $g\varrho g^{-1}$ is unbounded if and only if $\varrho$ is unbounded. Hence we may assume at the beginning that $\varrho$  has the form of \eqref{eq:upper}. The semisimplification of $\varrho$ is defined  
	by
	 \begin{equation*} 
		\varrho^{ss} =\left[\begin{array}{cccc}
			\varrho_1& 0 & \cdots & 0 \\
			0 & \varrho_2 & \cdots & 0 \\
			\vdots & \vdots & \ddots & \vdots \\
			0 & 0 & \cdots & \varrho_n
		\end{array}\right] 
	\end{equation*}
	It is obvious that if $\varrho$ is bounded, then $\varrho^{ss}$ is bounded.  
	
	Assume now $\varrho^{ss}$ is  bounded. Then each $\varrho_i$ is bounded. Let $L$ be a finite extension of $K$ such that $\varrho$ is defined over $L$. Then $\varrho_i(\pi_1(X))$ is contained in some maximal compact subgroup of $\GL_{N_i}(L)$. Since all maximal compact subgroups of $\GL_{N_i}(L)$ are conjugate to $\GL_{N_i}(\cO_L)$,   then there exists $g_i\in \GL_{N_i}(L)$ such that $g_i \varrho_i g_i^{-1}:\pi_1(X)\to \GL_{N_i}(\cO_L)$. 
	Define 
	\begin{equation} 
		\tau:= \left[\begin{array}{cccc}
			g_1& 0 & \cdots & 0 \\
			0 & g_2& \cdots & 0\\
			\vdots & \vdots & \ddots & \vdots \\
			0 & 0 & \cdots & g_n
		\end{array}\right]  \left[\begin{array}{cccc}
			\varrho_1& a_{12} & \cdots & a_{1 n} \\
			0 & \varrho_2 & \cdots & a_{2 n} \\
			\vdots & \vdots & \ddots & \vdots \\
			0 & 0 & \cdots & \varrho_n
		\end{array}\right]  \left[\begin{array}{cccc}
			g_1^{-1}& 0 & \cdots & 0 \\
			0 & g_2^{-1}& \cdots & 0\\
			\vdots & \vdots & \ddots & \vdots \\
			0 & 0 & \cdots & g_n^{-1}
		\end{array}\right] 
	\end{equation}
	which is conjugate to $\varrho$.  Then $\tau$ can be written as 
	\begin{equation*}
		\tau=\ \left[\begin{array}{cccc}
			g_1\varrho_1g_1^{-1}& h_{12} & \cdots & h_{1 n} \\
			0 & 	g_2\varrho_2g_2^{-1} & \cdots & h_{2 n} \\
			\vdots & \vdots & \ddots & \vdots \\
			0 & 0 & \cdots & g_n\varrho_ng_n^{-1}
		\end{array}\right] 
	\end{equation*}
	such that $g_i\varrho_ig_i^{-1}: \pi_1(X)\to {\rm GL}_N(\cO_L)$ is irreducible. 
	Write 
	$$
	\tau_1:= \left[\begin{array}{cccc}
		g_1\varrho_1g_1^{-1}& 0 & \cdots & 0 \\
		0 & 	g_2\varrho_2g_2^{-1} & \cdots & 0 \\
		\vdots & \vdots & \ddots & \vdots \\
		0 & 0 & \cdots & g_n\varrho_ng_n^{-1}
	\end{array}\right] 
	$$
	and 
	$$
	\tau_2:= \left[\begin{array}{cccc}
		0& h_{12} & \cdots & h_{1 n} \\
		0 & 0& \cdots & h_{2 n} \\
		\vdots & \vdots & \ddots & \vdots \\
		0 & 0 & \cdots &0
	\end{array}\right] 
	$$
	Note that $\tau_2$ is not a group homomorphism but only a map from $\pi_1(X)$ to $ {\rm GL}_N(L)$. 
	
	For any matrix $B$ with values  in $L$, we shall write $v(B)$ the matrix whose entries are the valuation of the  corresponding entries in $B$ by $v:L\to \bR$. Let us define $M(B)$ the lower bound of   the entries of $v(B)$. Then for another matrix $A$ with values  in $L$, one has $M(A+B)\geq \min \{M(A), M(B)\}$.  
	
	Let $x_1,\ldots,x_m$ be a generator of $\pi_1(X)$. Let $C$ be the lower bound of the entries of $ \{v(h_{ij}(x_k))\}_{i,j=1,\ldots,n; k=1,\ldots,m}$. We assume that $C<0$, or else it is easy to see that $\tau$ is bounded.  Note that $\min_{i=1,\ldots,m}M(g_i\varrho_ig_i^{-1}(x_i))\geq 0$. It follows that $M(\tau_1(x_i))\geq 0$ for each $x_i$. 
	Then  for any $x=x_{i_1}\cdots x_{i_\ell}$,
	\begin{align*}
		M(\tau(x))&=M(\sum_{j_1,\ldots,j_\ell=1,2}\tau_{j_1}(x_{i_1})\cdots\tau_{j_\ell}(x_{i_\ell}) )\\
		&\geq \min_{{j_1,\ldots,j_\ell=1,2}}\{M(\tau_{j_1}(x_{i_1})\cdots\tau_{j_\ell}(x_{i_\ell}) )\}.
	\end{align*}
	Note that $\tau_{j_1}(x_{i_1})\cdots\tau_{j_\ell}(x_{i_\ell})=0$ if $\#\{k\mid j_k=2\}\geq n$ since $\tau_2(x_i)$ is   nilpotent.   Hence
	$$
	M(\tau(x))\geq \min_{{j_1,\ldots,j_\ell=1,2};\#\{k\mid j_k=2\}< n }\{M(\tau_{j_1}(x_{i_1})\cdots\tau_{j_\ell}(x_{i_\ell}) )\}.
	$$
	Since $M(\tau_1(x_i))\geq 0$ for each $x_i$, it follows that $M(\tau_{j_1}(x_{i_1})\cdots\tau_{j_\ell}(x_{i_\ell})\geq (n-1)C$ if   $\#\{k\mid j_k=2\}< n$.    Therefore, $
	M(\tau(x))\geq (n-1)C$ for any $x\in \pi_1(X)$. $\tau$ is thus bounded. Since $\varrho$ is conjugate to $\tau$, $\varrho$ is also bounded. We finish the proof of the lemma. 
\end{proof}

We recall the following facts of character varieties (cf. \cite[Theorem 1.28]{LM85}). 
\begin{lem}\label{lem:character}
 Let $K$ be an \emph{algebraically closed} field.  Then the $K$-points	$M_{\rm B}(X,N)$ are  in one-to-one correspondence with the
	conjugate  classes of semisimple   representations $\pi_1(X)\to {\rm GL}_N(K)$. More precisely, if $\{\varrho_i:\pi_1(X)\to {\rm GL}_N(K)\}_{i=1,2}$ are two linear  representations such that $[\varrho_1]=[\varrho_2]\in M_{\rm B}(X, N)(K)$, then the semisimplification of $\varrho_1$ and $\varrho_2$ are conjugate.  \qed
\end{lem} 
The following result is thus a consequence of \cref{lem:semisimplification}.
\begin{lem}\label{lem:same bound}
	Let $K$ be a non-archimedean local field. Let $x\in M_{\rm B}(X,N)(K)$.  If $\{\varrho_i:\pi_1(X)\to {\rm GL}_N(\bar{K})\}_{i=1,2}$ are two linear  representations such that $[\varrho_1]=[\varrho_2]=x\in M_{\rm B}(X, N)(\bar{K})$, then $\varrho_1$ is bounded if and only if $\varrho_2$ is bounded. In other words, for the GIT quotient  $\pi: R(X,N) \to M_{\rm B}(X,N) $  where  $R(X,N)$ is  the representation variety of $\pi_1(X)$ into ${\rm GL}_N$, for any $x\in M_{\rm B}(X,N)(\bar{K})$, the representations in $\pi^{-1}(x)\subset R(X,N)(\bar{K})$ are either all bounded or all unbounded. 
\end{lem}
\begin{proof}
	By the assumption and \cref{lem:character}, we know that the semisimplificaitons $\varrho_1^{ss}:\pi_1(X)\to {\rm GL}_N(\bar{K})$ of $\varrho_2^{ss}:\pi_1(X)\to {\rm GL}_N(\bar{K})$ are conjugate by an element  $g\in {\rm GL}_N(\bar{K})$. Therefore, there exists a finite extension $L$ of $K$ such that $\varrho_i^{ss}$ and $\varrho_i$ are all defined in $L$ and $g\in {\rm GL}_N(L)$. Hence $\varrho_1^{ss}$ is bounded if and only if $\varrho_2^{ss}$ is bounded.  By \cref{lem:semisimplification}, we know that $\varrho_i^{ss}$ is bounded if and only if  $\varrho_i$ is bounded. Therefore, the lemma follows. 
\end{proof}
We thus can make the following definition. 
\begin{dfn}[Class of bounded representations]
	Let $K$ be a non-archimedean local field of characteristic zero. A point $x\in M_{\rm B}(X,N)(\bar{K})$ is called \emph{a class of bounded representations} if there exist some  $\varrho: \pi_1(X)\to {\rm GL}_N(\bar{K})$ (thus any $\varrho$ by \cref{lem:same bound}) such that $[\varrho]=x$ and $\varrho$ is bounded.
\end{dfn}

\begin{proposition}\label{lem:conjugate3} 
	Let $X$ be a smooth quasi-projective variety
	 and let $\kC$ be a $\bar{\bQ}$-constructible subset of $M_{\rm B}(X,N)$.  Let $f:F\to X$  be a morphism from  a quasi-projective \emph{normal} variety $F$  such that $s_{\kC}\circ f(F)$ is a  point. 
	Let $\{\tau_i:\pi_1(X)\to {\rm GL}_N(\bC)\}_{i=1,2}$ be  reductive representations  such that $[\tau_1]$ and $[\tau_2]$ are in the same geometric connected component of $\kC(\bC) $.
	Then $\tau_1\circ \iota$ is conjugate to $\tau_2\circ \iota$, where $\iota:\pi_1(F)\to \pi_1(X)$ is the  homomorphism of fundamental groups induced by $f$. In other words,  $j(\kC)$ is zero-dimensional, where $j:M_{\rm B}(X,N)\to M_{\rm B}(F,N)$ is the natural morphism of character varieties induced by $\iota:\pi_1(F)\to \pi_1(X)$. 
\end{proposition} 
\begin{proof}
%Though $F$ might not be normal, it is a finite CW-complex by \cite{Loj64} and thus $\pi_1(F)$ is also finitely presented. 	
Let $M_X$ (resp. $M$)  be the moduli space of representations of $\pi_1(X)$ (resp. $\pi_1(F^{\mathrm{norm}})$) in ${\rm GL}_N$.   Note that $M_X$ and $M$ are both affine schemes of finite type defined over ${\bQ}$.  Let $R_X$ (resp. $R$) be the  affine  scheme of finite type defined over $\bQ$ such that $R_X(L)=\Hom(\pi_1(X), {\rm GL}_N(L))$ (resp. $R(L)=\Hom(\pi_1(F), {\rm GL}_N(L))$)  for any field $L/\bQ$. Then we have 
	\begin{equation}\label{eq:cat}
		\begin{tikzcd}
			R_X \arrow[r, "\pi"] \arrow[d, "\iota^*"] & M_X\arrow[d, "j"]\\
			R\arrow[r, "p"] & M
		\end{tikzcd}
	\end{equation} where $\pi:R_X\to M_X$  and $p:R\to M$ are the GIT quotient that are both surjective.      For any field extension $K/\bQ$ and any $\varrho\in R_X(K)$, we write $[\varrho]:=\pi(\varrho)\in M_X(K)$.  Let $\mathfrak{R}:=\pi^{-1}(\kC)$ that is a constructible subset defined over some number field $k$.  Then $\tau_i\in \kR(\bC)$. 
	
%	We first assume that $\tau_1$ and $\tau_2$ are in the same geometric irreducible component of $\kR(\bC)$.     Since $\pi_1(X)$ is finitely generated, there is a  field extension $\mathfrak{K}$  of $k$, which is finitely generated over $k$, such  that $\tau_1$ and $\tau_2$ are both defined over $\mathfrak{K}$, i.e., $\tau_i:\pi_1(X)\to {\rm GL}_N(\mathfrak{K})$. Let $K$ be a finite extension of $\bQ_p$ such that $K$ contains  $\mathfrak{K}$  as a subfield.  
	\begin{claim}\label{claim:zero} 
	Let $\kR'$ be any geometric irreducible component of $\kR$. Then $j\circ\pi(\kR')$ is zero dimensional.  
	\end{claim}	 
	\begin{proof}   Assume, for the sake of contradiction, that $j\circ\pi(\kR')$ is positive-dimensional.  If we replace $k$ by a finite extension, we may assume that $\kR'$ is defined over $k$.    
Since $M$ is an affine $\bQ$-scheme of finite type, it follows that     there exist a  $k$-morphism $\psi: M \rightarrow \mathbb{A}^1$ such that the image $\psi\circ j\circ\pi(\kR')$ is Zariski dense in $\mathbb{A}^1$.   After replacing $k$ by a finite extension, we   can find   a locally closed irreducible curve $C \subset \kR'$ such that the restriction $ \psi \circ j\circ\pi|_C: C \rightarrow \mathbb{A}^1$ is a generically finite $k$-morphism.  We  take a Zariski open subset $U \subset \mathbb{A}^1$ such that  $\psi \circ j\circ \pi|_C$ is finite over $U$.   Let $\mathfrak{p}$ be a prime ideal of the ring of integer $\cO_k$ and let $K$ be its non-archimedean completion.  In the following, we shall work over  $K$.

		Let $x \in U(K)$ be a point, and let $y \in C(\bar{K})$ be a point over $x$. Then $y$ is defined over some extension of $K$ whose extension degree is bounded by the degree of $\psi\circ j\circ\pi|_C: C \rightarrow \mathbb{A}^1$. Note that there are only finitely many such field extensions. Hence there exists a finite extension $L/K$ such that the points over $U(K)$ are all contained in $C(L)$. Since $U(K)\subset \mathbb{A}^1(L)$ is unbounded, the image $\psi\circ j\circ\pi(C(L)) \subset \mathbb{A}^1(L)$ is unbounded.

		  Let $R_0$ be the set of bounded representations in $R(L)$.  Recall that by \cite{Yam10},    $M_0:=p(R_0)$  is compact in $M(L)$ with respect to analytic topology, hence $M_0$ is bounded by \cref{lem:bounded criterion}.  By \cref{lem:bounded criterion}  once again, $\psi(M_0)$ is a bounded subset in $\bA^1(L)$.   Recall that $\psi\circ j\circ\pi(C(L)) \subset \mathbb{A}^1(L)$ is unbounded.  Therefore, there exists $\varrho\in C(L)$ such that $\psi\circ j([\varrho])\not\in \psi(M_0)$. Note that $[\varrho\circ\iota]=j([\varrho])$ by \eqref{eq:cat}.  Hence $[\varrho\circ\iota]\not\in M_0$ which implies that $\varrho\circ\iota\not\in R_0$.  By the definition of $R_0$,  $\varrho\circ\iota$  is  unbounded.  
		
		Let $\varrho^{ss}:\pi_1(X)\to {\rm GL}_N(\bar{L})$  be the semisimplification of $\varrho$. 
		Then $[\varrho]=[\varrho^{ss}]\in \kC(\bar{L})$  by \cref{lem:character}.   Therefore, $[\varrho\circ\iota]=[\varrho^{ss}\circ\iota]\in M(\bar{L})$ by \eqref{eq:cat}. By \cref{lem:same bound}, $\varrho^{ss}\circ\iota:\pi_1(F)\to {\rm GL}_N(\bar{L})$ is also unbounded.  Note that  $\varrho^{ss}\circ\iota$ is reductive by \cref{thm:reductive}.  Since $\pi_1(F)$ is finitely generated, there exist a finite extension $L'$ of $L$ such that $\varrho^{ss}$ is defined over $L'$.   
				However, by \cref{lem:bounded}, $\varrho^{ss}\circ\iota$ is always bounded. We obtain a contradiction and thus  $j\circ\pi(\kR')$ is zero dimensional.  

We can also apply \cite[Corollary 4.3]{Kem78} instead of  \cref{lem:same bound}. As $\varrho\in C(L)$, its image $[\varrho]\in M_X(L)$.   Consider the fiber $\pi^{-1}([\varrho])$ which is an $L$-variety. Its closed orbit
		is defined over $L$ by Galois descent. As  $\pi^{-1}([\varrho])$ contains the $L$-point $\varrho$,  the closed orbit in $\pi^{-1}([\varrho])$ 
		has an $L$-point $\varrho':\pi_1(X)\to \GL_{N}(L)$ as well by \cite[Corollary 4.3]{Kem78}.  By \cref{lem:character}, $\varrho'$ is reductive and $[\varrho']=[\varrho]$. Hence $[\varrho'\circ\iota]=[\varrho\circ\iota]\not\in M_0$.   Therefore,  $\varrho'\circ\iota:\pi_1(F)\to \GL_{N}(L)$ is unbounded by our definition of $M_0$. However,  by the definition of $s_\kC:X\to S_\kC$  in \cref{def:reduction ac}, $\varrho'\circ\iota$ is always bounded. We obtain a contradiction and thus  $j\circ\pi(\kR')$ is zero dimensional.  
	\end{proof}
Let $\{\tau_i:\pi_1(X)\to {\rm GL}_N(\bC)\}_{i=1,2}$ be reductive representations such that $[\tau_1]$ and $[\tau_2]$ are contained in the same geometric   connected component $\kC'$ of $\kC(\bC)$. We aim to prove that $j(\kC')$ is a point in $M(\bC)$.

Consider a geometric  irreducible component $\kC''$ of $\kC'$. We can choose a geometric irreducible component $Z$ of $\pi^{-1}(\kC'')$ such that $\pi(Z)$ is dense in $\kC''$. It follows that $Z$ is an irreducible component of $\kR(\bC)$. By \cref{claim:zero}, we know that $j\circ \pi(Z)$ is a point in $M(\bC)$. Thus, $j(\kC'')$ is also a point in $M(\bC)$.

Consequently, $j(\kC')$ is a point in $M(\bC)$. As a result, we have $[\tau_1\circ\iota]=j([\tau_1])=j([\tau_2])=[\tau_2\circ\iota]$.  By \cref{thm:reductive}, $\tau_1\circ \iota$ and $\tau_2\circ\iota$ are reductive,  and according to \cref{lem:character}, they are conjugate to each other. 
We have   established the proposition. 
 \end{proof}
We will need the following lemma on the intersection of kernels of representations. 
\begin{lem}\label{lem:same kernel}
	Let $X$ be a quasi-projective normal variety and 	let $\kC$ be a constructible subset  of $M_{\rm B}(X, N)(\bC)$.  Then we have
	\begin{align}\label{eq:same}
		\cap_{[\varrho]\in \kC} \ker\varrho=\cap_{[\varrho]\in \overline{\kC}} \ker\varrho, 
	\end{align}
	where $\varrho$'s are   reductive representations of $\pi_1(X)$ into $\GL_N(\bC)$.  
\end{lem}
\begin{proof}
	Let   $M_X$ be the moduli space of representation of $\pi_1(X)$   in ${\rm GL}_N$.   Let   $R_X$  be the  affine  scheme of finite type  such that  $R(L)=\Hom(\pi_1(X), N)(L)$ for any field $\bQ\subset L$.  We write $M:=M_X(\bC)$ and $R:=R_X(\bC)$. Then the GIT quotient $\pi:R\to M$ is a surjective morphism.  It follows that  $\pi^{-1}(\kC)$ is a $\GL_N(\bC)$-invariant subset where $\GL_N(\bC)$ acts on $R$ by the conjugation. Define $H:=\cap_{[\varrho]\in \kC} \ker\varrho$, where $\varrho$'s are   reductive representations of $\pi_1(X)$ into $\GL_N(\bC)$. Pick any $\gamma\in H$.  Then the set $Z_\gamma:=\{\varrho\in R\mid \varrho(\gamma)=1\}$ is a Zariski closed subset of $R$.  Moreover, $Z_\gamma$ is $\GL_N(\bC)$-invariant. Define $Z:=\cap_{\gamma\in H} Z_\gamma$.  Then $Z$ is also $\GL_N(\bC)$-invariant.   Therefore, $\pi(Z)$ is also a Zariski closed subset of $M$ by \cite[Proposition 5.10]{Muk03}. Note that $\kC\subset \pi(Z)$. Therefore, $\overline{\kC}\subset \pi(Z)$.    Note that for any reductive $\varrho:\pi_1(X)\to \GL_N(\bC)$ such that $[\varrho]\in \pi(Z)$, we have $\varrho(\gamma)=1$ for any   $\gamma\in H$.  It follows that \eqref{eq:same} holds. 
\end{proof}

%By Simpson, there is a $\bC^*$-action on $M_{\rm B}(X,G)(\bC)$ defined as follows. Let $\tau:\pi_1(X)\to G(\bC)$ be  a reductive representation. We fix a faithful representation $G\to {\rm GL}_N$.  By Simpson again,  there is  	  a harmonic bundle $(E,\theta,h)$ such that the monodromy representation of the flat connection  $D_h+\theta+\theta^\dagger_h$  is $\tau$.  Let $t\in \bC^*$. Then $(E,t\theta)$ is poly-stable which vanishing Chern numbers. Hence it has a harmonic metric $h_t$ by Simpson's theorem. Then  $\nabla_t:=D_{h_t}+t\theta+\bar{t}\theta^\dagger_{h_t}$  is flat whose monodromy representation is denoted by $\tau_t$. By Simpson, the Zariski closure of $\tau_t(\pi_1(X))$ is the same as $\tau(\pi_1(X))$ hence is contained in $G$. Therefore, $[\tau_t]\in M_{\rm B}(X,G)$. We define $\bC^*$-action on $M_{\rm B}(X,G)(\bC)$ by  $t.[\tau]:=[\tau_t]$ for any $t\in \bC^*$.  
%\begin{thm}[Simpson]\label{thm:Caction}
	%Let  $\tau:\pi_1(X)\to G(\bC)$ be  a reductive representation such that $t.[\tau]=\tau$ for some $t\in U(1)$ which is not root of unity. Then $\tau$ underlies a $\bC$-VHS.  \qed
%\end{thm}   

Lastly, let us prove the main result of this subsection. This result will serve as a crucial cornerstone in the proofs of \cref{main2,main3,main}.
\begin{proposition}\label{prop:nonrigid}
Let $X$ be a smooth quasi-projective variety. Let $\kC$ be a constructible subset of $M_{\rm B}(X,N)(\bC)$, defined over $\bQ$, such that $\kC$ is invariant under $\bR^*$-action. When $X$ is non-compact, we further assume that $\kC$ is closed.  Then there exist reductive representations $\{\sigma^\vhs_i:\pi_1(X)\to {\rm GL}_N(\bC)\}_{i=1,\ldots,m}$ such that each  $\sigma^\vhs_i$ underlies a $\bC$-VHS, and for a  morphism  $\iota:Z\to X$ from any   quasi-projective \emph{normal} variety $Z$ with $s_\kC\circ\iota (Z)$  being a point, the following properties hold:
\begin{thmlist} 
	\item  \label{item:directsume}  For $\sigma:=\oplus_{i=1}^{m}\sigma^\vhs_i$,    $\iota^*\sigma(\pi_1(Z))$ is discrete in $\prod_{i=1}^{m}\GL_{N}(\bC)$. 
	\item \label{item:max}For each reductive representation $\tau:\pi_1(X)\to {\rm GL}_N(\bC)$ with $[\tau]\in \kC(\bC)$,   $\iota^*\tau$ is conjugate to some $\iota^*\sigma^\vhs_i$.   
	\item For each $\sigma_i^\vhs$, there exists a  reductive representation $\tau:\pi_1(X)\to {\rm GL}_N(\bC)$ with $[\tau]\in \kC(\bC)$ such that  $\iota^*\tau$ is conjugate to   $\iota^*\sigma^\vhs_i$. 
	\item For every $i=1,\ldots,m$, we have  \begin{align}\label{eq:inclusion ker}
		 \cap_{[\varrho]\in \kC }\ker\varrho\subset \ker \sigma_i^{\vhs}
	\end{align}  where $\varrho:\pi_1(X)\to \GL_{N}(\bC)$ varies among all reductive representations such that $[\varrho]\in \kC(\bC)$. 
\end{thmlist}
\end{proposition}
\begin{proof} 
Let $\kC_1,\ldots,\kC_\ell$ be  all   geometric connected components of $\kC$  which are  defined over $\bar{\bQ}$. We  can pick reductive representations $\{\varrho_i:\pi_1(X)\to {\rm GL}_N(\bar{\bQ})\}_{i=1,\ldots,\ell}$   such that $[\varrho_i]\in \kC_i(\bar{\bQ})$ for every $i$. Since $\pi_1(X)$ is finitely generated, there exists a number field $k$ which is a Galois extension of $\bQ$ such that       $\varrho_i:\pi_1(X)\to {\rm GL}_N(k)$ for every $\varrho_i$.
	
 Let ${\rm Ar}(k)$ be all archimedean places of $k$ with $w_1$ the identity map.  Then for any $w\in {\rm Ar}(k)$ there exists $a\in {\rm Gal}(k/\bQ)$ such that $w=w_1\circ a$.    Note that $\kC$ is defined over $\bQ$. Then $\kC$ is invariant under the conjugation $a$. Therefore, for any $w:k\to \bC$ in ${\rm Ar}(k)$,   letting $\varrho_{i,w}:\pi_1(X)\to {\rm GL}_N(\bC)$  be the composition  $w\circ\varrho_i$, we have $[\varrho_{i,w}] \in \kC(\bC)$.   
 
 	   For any $t\in \bR^*$, we consider the $\bR^*$-action  $\varrho_{i,w,t}:\pi_1(X)\to {\rm GL}_N(\bC)$ of $\varrho_{i,w}$ defined in \cref{sec:C*action}. Then  $\varrho_{i,w,t}$ is also reductive by the arguments in  \cref{sec:C*action}.  Since we assume that $\kC(\bC)$ is invariant under  $\bR^*$-action,  it follows that $[\varrho_{i,w,t}]\in \kC(\bC) $.    By \cref{lem:continuous}, $[\varrho_{i,w,t}]$ is a continuous deformation of $[\varrho_{i,w}]$. 
	Hence they are in the same geometric connected component of $\kC(\bC)$, and by \cref{lem:conjugate3} we conclude that $[\iota^*\varrho_{i,w,t}]=[\iota^*\varrho_{i,w}]$ for any $t\in \bR^*$.  
	
	We first assume that $X$ is compact. According to \cite{Sim92},  $\lim\limits_{t\to 0}[\varrho_{i,w,t}]$ exists, and there exists  a  reductive $\varrho_{i,w}^{\scaleto{V\!H\!S}{4pt}}:\pi_1(X)\to {\rm GL}_N(\bC)$ such that $[\varrho_{i,w}^{\scaleto{V\!H\!S}{4pt}}]=\lim\limits_{t\to 0}[\varrho_{i,w,t}]$.   Moreover, $\varrho_{i,w}^{\scaleto{V\!H\!S}{4pt}}$ underlies a $\bC$-VHS.    Therefore, $[\iota^*\varrho_{i,w}]=\lim\limits_{t\to 0}[\iota^*\varrho_{i,w,t}]=[\iota^*\varrho_{i,w}^{\scaleto{V\!H\!S}{4pt}}]$.     Since $[\varrho_{i,w,t}]\in \kC(\bC)$ for any $t\in \bR^*$, it follows that $[\varrho_{i,w}^{\vhs}]\in \overline{\kC}(\bC)$. By \cref{eq:same}, we conclude \begin{align} \label{eq:smallkernel}
		\cap_{[\varrho]\in \kC }\ker\varrho\subset \ker \varrho_{i,w}^{\vhs}.
	\end{align}  
	
	Assume now $X$ is non-compact. % Let ${\rm Ar}(k)$ be all archimedean places of $k$ with $w_1$ the identity map.  Since $k$ is a Galois extension of $\bQ$, for any $w\in {\rm Ar}(k)$ there exists $\sigma\in {\rm Gal}(k/\bQ)$ such that $w=w_1\circ\sigma$.    Note that $\kC$ is defined over $\bQ$. Then $\kC$ is invariant under the conjugation $\sigma$. Therefore, for any $w:k\to \bC\in {\rm Ar}(k)$,   letting $\varrho_{i,w}:\pi_1(X)\to {\rm GL}_N(\bC)$  be the composition  $w\circ\varrho_i$, $[\varrho_{i,w}] $ is also complex point in $\kC $. 
	 As we assume that $\kC $ is closed and invariant under $\bR^*$-action, by \cref{thm:Moc ubi}, we can choose a reductive representation $\varrho_{i,w}^\vhs:\pi_1(X)\to \GL_N(\bC)$ such that 
\begin{itemize}
	\item it underlies a $\bC$-VHS;
	\item    $[\varrho_{i,w}^\vhs]$ and  $[\varrho_{i,w}]$       are in  the same geometric connected component of $\kC(\bC)$.
\end{itemize} 
Note that \eqref{eq:smallkernel} is satisfied automatically. By \cref{lem:conjugate3}, we have $[\iota^*\varrho_{i,w}]=[\iota^*\varrho_{i,w}^\vhs]$. %Since $Z$ is quasi-projective normal, by \cref{thm:reductive},   $\iota^*\varrho_{i,w}$ and $\iota^*\varrho_{i,w}^\vhs$ are both reductive . It follows from \cref{lem:character} that $\iota^*\varrho_{i,w}$ is conjugate to $\iota^*\varrho_{i,w}^\vhs$. 

In summary, we construct reductive representations $\{\varrho_{i,w}^\vhs:\pi_1(X)\to \GL_{N}(\bC)\}_{i=1,\ldots,k;w\in {\rm Ar}(k)}$ in both compact and non-compact cases. Each of these representations underlies a $\bC$-VHS and satisfies $[\iota^*\varrho_{i,w}]=[\iota^*\varrho_{i,w}^\vhs]$ and \eqref{eq:smallkernel}.

	\medspace

	Let $v$ be any non-archimedean place of $k$ and   $k_v$  be the non-archimedean completion of $k$ with respect to $v$.  Write $\varrho_{i,v}:\pi_1(X)\to {\rm GL}_N(k_v)$ the induced representation by $\varrho_i$. By the construction of $s_{\kC}$, it follows that $\iota^*\varrho_{i,v}(\pi_1(Z))$   is bounded. Therefore, we have a factorization
	$$
\iota^*\varrho_{i}:\pi_1(Z)\to {\rm GL}_N(\cO_k).
	$$ 
	Note that ${\rm GL}_N(\cO_k)\to \prod_{w\in {\rm Ar}(k)}{\rm GL}_N(\bC)$ is a discrete subgroup  by \cite[Proposition 6.1.3]{Zim}. It follows that  for the product representation
	$$
	\prod_{w\in {\rm Ar}(k)}\iota^*\varrho_{i,w}:\pi_1(Z)\to \prod_{w\in {\rm Ar}(k)}{\rm GL}_N(\bC),
	$$
	its image is discrete.

Since $Z$ is  normal, by \cref{thm:reductive}, both $\iota^*\varrho_{i,w}$ and $\iota^*\varrho_{i,w}^{\scaleto{V\!H\!S}{4pt}}$ are reductive. Recall that $[\iota^*\varrho_{i,w}]=[\iota^*\varrho_{i,w}^{\scaleto{V\!H\!S}{4pt}}]$. It follows that      $\iota^*\varrho_{i,w}$ is conjugate to $\iota^*\varrho_{i,w}^{\scaleto{V\!H\!S}{4pt}}$ by \cref{lem:character}. Consequently, $\prod_{w\in {\rm Ar}(k)}\iota^*\varrho_{i,w}^{\scaleto{V\!H\!S}{4pt}}:\pi_1(Z)\to {\rm GL}_N(\bC)$ has discrete image.   Consider the product representation  of $\varrho_{i,w}^{\scaleto{V\!H\!S}{4pt}}$
	$$
	\sigma:=\prod_{i=1}^{\ell}\prod_{w\in {\rm Ar}(k)}\varrho_{i,w}^{\scaleto{V\!H\!S}{4pt}}: \pi_1(X)\to \prod_{i=1}^{\ell}\prod_{w\in {\rm Ar}(k)}{\rm GL}_N(\bC).
	$$
	Then $\sigma$ underlies a $\bC$-VHS and   $\iota^*\sigma:\pi_1(Z)\to \prod_{i=1}^{\ell}\prod_{w\in {\rm Ar}(k)}{\rm GL}_N(\bC)$ has discrete image. 
	
	Let $\tau:\pi_1(X)\to {\rm GL}_N(\bC)$ be any reductive representation such that $[\tau]\in \kC(\bC)$. Then $[\tau]\in \kC_i(\bC)$ for some $i$.   By \cref{lem:conjugate3}, it follows that 
	$
	[\iota^*\tau]=[\iota^*\varrho_{i,w_1}] =[\iota^*\varrho_{i,w_1}^{\scaleto{V\!H\!S}{4pt}}].
	$  
By \cref{thm:reductive,lem:character} once again, $\iota^*\tau$ is conjugate to $\iota^*\varrho_{i,w_1}^{\scaleto{V\!H\!S}{4pt}}$. % Hence  $\iota^*\tau$ also underlies a  $\bC$-VHS, which is conjugate to a  direct factor of $\iota^*\sigma$.  
	The proposition is proved if we let  $\{\sigma^\vhs_i:\pi_1(X)\to \GL_{N}(\bC)\}_{i=1,\ldots,m}$ be $\{\varrho_{i,w}^\vhs:\pi_1(X)\to \GL_{N}(\bC)\}_{i=1,\ldots,\ell;w\in {\rm Ar}(k)}$. 
\end{proof}
\begin{rem}\label{rem:Q}
	In the proof of \cref{prop:nonrigid}, we take the Galois conjugate of $\kC\subset M_{\rm B}(X,N)$ under  $a\in {\rm Gal}(k/\bQ)$. If $\kC$ is not defined over $\bQ$, it is not known that $a(\kC)\subset M_{\rm B}(X,N)$ is $\bR^*$-invariant.  This is why we include the assumption that $\kC$ is defined over $\bQ$ in our proof, whereas Eyssidieux disregarded such a condition in \cite{Eys04}.   It seems that this condition should also be necessary in \cite{Eys04}. 
\end{rem}

\subsection{Infinite monodromy at infinity}
When considering a non-compact quasi-projective variety $X$, it is important to note that the Shafarevich conjecture fails in simple examples. For instance, take $X:=A\backslash \{0\}$, where $A$ is an abelian surface. Its universal covering $\widetilde{X}$ is   $\mathbb{C}^2-\Gamma$, where $\Gamma$ is a lattice in $\mathbb{C}^2$. Then  $\widetilde{X}$ is not holomorphically convex. Therefore, additional conditions on the fundamental groups at infinity are necessary to address this issue.
\begin{dfn}[Infinity monodromy at infinity]\label{def:monodromy}
	Let $X$ be a quasi-projective normal variety and let $\overline{X}$ be a projective compactification of $X$. We say a subset  $M\subset M_{\rm B}(X,N)(\bC)$ \emph{has infinite monodromy at infinity} if  for any holomorphic map $\gamma:\bD\to \overline{X}$ with $\gamma^{-1}(\overline{X}\setminus X)=\{0\}$,    there exists a reductive $\varrho:\pi_1(X)\to \GL_N(\bC)$ such that $[\varrho]\in M$ and $\gamma^*\varrho:\pi_1(\bD^*)\to \GL_N(\bC)$ has infinite image. 
\end{dfn} 

\begin{dfn}[quasi-infinite monodromy at infinity]
	Let $X$ be  a smooth quasi-projective variety.	Let $\varrho:\pi_1(X)\to\mathrm{GL}_N(\bC)$ be a representation.
Say $\varrho$	 has  \emph{quasi-infinite monodromy at infinity} with respect to the partial compactification $X'$ of $X$   if for every boundary divisor $D\subset X'\backslash X$ the local monodromy is finite and for every $f:\mathbb D^*\to X$ such that $f(0)\not\in X'$, $f^*\varrho:\bZ\to \pi_1(X)$ is infinite.
\end{dfn}

Note that \cref{def:monodromy} does not depend on the projective compactification of $X$.
\begin{lem}\label{lem:infinite pull}
	Let $f:Y\to X$ be a \emph{proper} morphism between quasi-projective normal varieties. If $M\subset M_{\rm B}(X,N)(\bC)$ has infinite monodromy at infinity, then $f^*M\subset M_{\rm B}(Y,N)(\bC)$ also has   infinite monodromy at infinity.
\end{lem}
\begin{proof}
	We take   projective compactification $\overline{X}$ and $\overline{Y}$ of $X$ and $Y$ respectively such that $f$ extends to a morphism $\bar{f}:\overline{Y}\to \overline{X}$.   Let  $\gamma:\bD\to \overline{Y}$ be  any holomorphic map with $\gamma^{-1}(\overline{Y}\setminus Y)=\{0\}$. Then $\bar{f}\circ\gamma:\bD\to \overline{X}$ satisfies $(\bar{f}\circ \gamma)^{-1}(\overline{X}\setminus X)=\{0\}$ as $f$ is proper.   Then by \cref{def:monodromy} there exists a reductive $\varrho:\pi_1(X)\to \GL_N(\bC)$ such that $[\varrho]\in M$ and $\gamma^*(f^*\varrho)=(f\circ \gamma)^*\varrho:\pi_1(\bD^*)\to \GL_N(\bC)$ has infinite image. The lemma follows.  
\end{proof}
We have a precise local characterization of  a representation with infinite monodromy at infinity.
\begin{lem}\label{lem:equi}
 Consider a smooth quasi-projective variety $X$ along with a smooth projective compactification $\overline{X}$, where $D:=\overline{X}\backslash X$ is a simple normal crossing divisor.  A set  $M\subset M_{\rm B}(X,N)(\bC)$ has infinite monodromy at infinity is equivalent to the following: for any $x\in D$, there exists an admissible coordinate $(U ;z_1,\ldots,z_n)$ centered at $x$ with $U\cap D=(z_1\cdots z_k=0)$ such that for any $k$-tuple $(i_1,\ldots,i_k)\in \bZ_{> 0}^k$, there exists a reductive $\varrho:\pi_1(X)\to \GL_N(\bC)$ such that $[\varrho]\in M(\bC)$  and $\varrho(\gamma_1^{i_1}\cdots\gamma_k^{i_k})\neq 0$, where $\gamma_i$ is    the   anti-clockwise 
loop around the origin in the i-th factor of $U\setminus D\simeq (\bD^*)^k\times \bD^{n-k}$. For such condition we will say that $\varrho$ \emph{has infinite monodromy at $x$}.
	 \end{lem}
 \begin{proof}
For any holomorphic map $f:\bD\to \overline{X}$ with $f^{-1}(D)=\{0\}$,   let $x:=f(0)$ which lies on $D$. We take an admissible coordinate $(U ;z_1,\ldots,z_n)$ centered at $x$ in the lemma.   Then $f( {\bD_{2\ep}})\subset U$ for some small $\ep>0$.  We can write $f(t)=(f_1(t),\ldots,f_n(t))$ such that $f_1(0)=\cdots=f_k(0)=0$ and $f_{i}(0)\neq 0$ for $i=k+1,\ldots,n$. Denote by $m_i:={\rm ord}_0f_i$ the vanishing order of $f_i(t)$ at $0$.    Consider the  anti-clockwise loop $\gamma$ defined by $\theta\mapsto \ep e^{i\theta}$ which generates $\pi_1(\bD_{2\ep}^*)$. Then  $f\circ\gamma$ is homotopy equivalent  to $\gamma_1^{m_1}\cdots\gamma_k^{m_k}$ in $\pi_1(U\backslash D)$.  If  $M$ has infinite monodromy at infinity, by \cref{def:monodromy} there exists a reductive $\varrho:\pi_1(X)\to \GL_N(\bC)$ such that $[\varrho]\in M(\bC)$  and $f^*\varrho(\gamma)\neq 0$.  This is equivalent to    that $\varrho(\gamma_1^{m_1}\cdots\gamma_k^{m_k})\neq 0$.   The lemma is proved. 
 \end{proof} 
 \cref{def:monodromy} presents a stringent condition that is not be practically applicable in many situations. To address this issue, we establish the following result:
	  \begin{proposition}\label{prop:infinity}  
	  	Let $X$ be a smooth quasi-projective variety. Assume that    $\varrho:\pi_1(X)\to \GL_N(\bC)$ is a linear representation. 	Then there exists a  smooth partial compactification $X'$ of $X$ such that 
\begin{thmlist}
	\item    $\varrho$	 has  \emph{quasi-infinite monodromy at infinity} with respect to the partial compactification $X'$ of $X$; 
	\item  $X'\backslash X$ is simple normal crossing divisor.
\end{thmlist}	
In particular, if $\varrho(\pi_1(X))$ is torsion free, then $\varrho$ extends to a representation $\varrho':\pi_1(X')\to \GL_{N}(\bC)$ with infinite monodromy at infinity. 
 \end{proposition}  
\begin{proof}
We first introduce the definition of \emph{representation systems}. Let $\overline{X}$ be a smooth projective compactification such that the boundary $\overline{X}\setminus X$ is a simple normal crossing divisor $\sum_{i=1}^{\ell}D_i$.     If $D_{i_1}\cap \ldots\cap D_{i_n}\neq \varnothing$ for some $\{i_1,\ldots,i_n\}\subset \{1,\ldots,\ell\}$, then for a generic point $x\in D_{i_1}\cap \ldots\cap D_{i_n}$,  locally we have an coordinate system $(U;z_1,\ldots,z_d)$  centered at $x$ such that $\overline{X}\cap U$ is a polydisk with $D_{i_k}\cap U=(z_k=0)$. 
We then have $U\cap X=(\bD^*)^n\times \bD^{d-n}$ such that $\pi_1(U\cap X)\simeq \bZ^n$, and  the local monodromy of $\varrho$ around $x$ is a representation $\varrho:\bZ^n\to \GL_{N}(\bC)$. 
Note that $\varrho(\bZ^n)$ may have torsion $T\subset \varrho(\bZ^n)$.
Taking the quotient, we have $\bZ^m\simeq \varrho(\bZ^n)/T$ for some $m$. 
Then the local monodromy $\varrho$ induces a linear map $\bZ^n\to\bZ^m$, which is represented by a matrix in $M_{m,n}(\bZ)$.
%This matrix will be the monodromy matrix $A_{\varrho}$ below. Since $\varrho(\bZ^n)\subset \GL_{N}(\bC)$ is an abelian group without torsions, we have $\varrho(\bZ^n)\simeq \bZ^m$ for some $m$.
%Therefore, $\varrho$ is represented by a matrix in $M_{m,n}(\bZ)$.

 Let $A=(\boldsymbol{a}_1,\boldsymbol{a}_2,\ldots,\boldsymbol{a}_n)\subset M_{m,n}(\mathbb Z)$ be a matrix, where each $\boldsymbol{a}_i$ is a column vector.
Given $J\subset \{1,2,\ldots,n\}$, we denote by $A[J]$ a submatrix which consists of column vectors $\boldsymbol{a}_j$ for $j\in J$.
Let $\{D_1,\ldots,D_{n}\}$ be a set of irreducible components of the boundary divisors $\overline{X}-X$.
We say that $(A,\{D_1,\ldots,D_{n}\})$ is a {\it representation system} on $\overline{X}$ if for any $J\subset \{1,2,\ldots,n\}$ such that $\bigcap_{j\in J}D_j\not=\emptyset$, the local monodromy at $\bigcap_{j\in J}D_j$ \emph{modulo the torsion} is represented by the matrix $A[J]$.

Let $D,D'$ be distinct irreducible components of $\partial X$ such that $D\cap D'\not=\emptyset$.
We consider the blow-up $\mathrm{Bl}_{D\cap D'}\overline{X}\to\overline{X}$.
(We shall call these blow-ups {\it admissible blow-ups}.)
Given a representation system $S=(A,\{D_1,\ldots,D_{n}\})$, where $A=(\boldsymbol{a}_1,\boldsymbol{a}_2,\ldots,\boldsymbol{a}_n)\in  M_{m,n}(\mathbb Z)$, we set a system $S'$ on $\mathrm{Bl}_{D\cap D'}\overline{X}$ as follows.

\begin{itemize}
\item
When $D,D'\in \{D_1,\ldots,D_{n}\}$, we take $i,j\in\{1,\ldots,n\}$, where $i\not=j$, such that $D=D_i$ and $D'=D_j$.
We set $A'=(\boldsymbol{a}_1,\boldsymbol{a}_2,\ldots,\boldsymbol{a}_n,\boldsymbol{a}_i+\boldsymbol{a}_j)\in  M_{m,n+1}(\mathbb Z)$.
Let $D_1',\ldots,D_n'$ be the strict transforms of $D_1,\ldots,D_n$, respectively, and let $D_{n+1}'$ be the exceptional divisor.
We set $S'=(A',\{D_1',\ldots,D_n',D_{n+1}'\})$.
\item
Otherwise, we set $S'=(A,\{D_1',\ldots,D_n'\})$, where $D_1',\ldots,D_n'$ are the strict transforms of $D_1,\ldots,D_n$, respectively.
\end{itemize}

\begin{claim}\label{lem:20230725}
When $S$ is a representation system on $\overline{X}$, then $S'$ is a representation system on $\mathrm{Bl}_{D\cap D'}\overline{X}$.
\end{claim}
 \begin{proof}[Proof of \cref{lem:20230725}]
There are two cases as above.
The first case is that $D,D'\in \{D_1,\ldots,D_{n}\}$.
Assume $D=D_{j_1}$ and $D'=D_{j_2}$, where $j_1,j_2\in \{1,\ldots,n\}$.
Let $J\subset \{1,\ldots,n+1\}$ so that $\bigcap_{j\in J}D_j'\not=\emptyset$.
If $J\subset \{1,\ldots,n\}$, then $(\bigcap_{j\in J}D_j')\backslash D_{n+1}'\not=\emptyset$.
Hence the local monodromy at $\bigcap_{j\in J}D_j'$ coincides with $A[J]=A'[J]$.
If $n+1\in J$, we set $J'=J\backslash\{j_1,j_2,n+1\}$.
Then we have $\bigcap_{j\in J'\cup\{j_1,j_2\}}D_j\not=\emptyset$.
Then the local monodromy at $\bigcap_{j\in J'\cup\{j_1,n+1\}}D_j'$ is $A'[J'\cup\{j_1,n+1\}]$.
Similarly the local monodromy at $\bigcap_{j\in J'\cup\{j_2,n+1\}}D_j'$ is $A'[J'\cup\{j_2,n+1\}]$.
By $D_{j_1}'\cap D_{j_2}'=\emptyset$, we have $\{j_1,j_2\}\not\subset J$.
Hence either $J\subset J'\cup\{j_1,n+1\}$ or $J\subset J'\cup\{j_2,n+1\}$.
Hence the local monodromy at $\bigcap_{j\in J}D_j'$ coincides with $A'[J]$.

Next we consider the second case that $\{D,D'\}\not\subset \{D_1,\ldots,D_{n}\}$.
Then $(\bigcap_{j\in J}D_j')\backslash E\not=\emptyset$, where $E$ is the exceptional divisor.
Hence the local monodromy at $\bigcap_{j\in J}D_j'$ coincides with $A[J]=A'[J]$.
\end{proof}

Given a representation system $S=(A,\{D_1,\ldots,D_{n}\})$, we introduce notations $\mu(S)\in \mathbb Z_{\geq 0}$, $\alpha(S)\in \mathbb Z_{\geq 0}$, $\beta(S)\in \mathbb Z_{\geq 0}$ and $\gamma(S)\in\mathbb Z_{\geq 0}^3$ as follows.
\begin{itemize}
\item
$\mu(S)$ is the maximum $\mu$ such that there exists $(i,j)$ with $D_i\cap D_j\not=\emptyset$ and $a_{\mu,i}a_{\mu,j}<0$.
We set $\mu(S)=0$ if there is no such $i,j$ for all $\mu=1,\ldots,m$.
\item
$\alpha(S)$ is the maximum of $|a_{\mu(S),i}-a_{\mu(S),j}|$ among all $(i,j)$ with $D_i\cap D_j\not=\emptyset$ and $a_{\mu(S),i}a_{\mu(S),j}<0$.
We set $\alpha(S)=0$ if $\mu(S)=0$.
\item
$\beta(S)$ is the number of $(i,j)$ such that $D_i\cap D_j\not=\emptyset$, $a_{\mu(S),i}a_{\mu(S),j}<0$ and $|a_{\mu(S),i}-a_{\mu(S),j}|=\alpha(S)$.
We set $\beta(S)=0$ if $\mu(S)=0$.
\item
We set $\gamma(S)=(\mu(S),\alpha(S),\beta(S))\in\mathbb Z_{\geq 0}^3$.
\end{itemize}

\begin{claim}\label{lem:20230724}
There exists a sequence $\overline{X}^{(n)}\to \overline{X}^{(n-1)}\to\cdots\to \overline{X}$ of admissible blow-ups such that $\mu(S^{(n)})=0$.
\end{claim}
 \begin{proof}[Proof of \cref{lem:20230724}]
If $\mu(S)=0$, there is nothing to prove.
So suppose $\mu(S)>0$.
Then we take $i,j$ such that $D_i\cap D_j\not=\emptyset$, $a_{\mu(S),i}a_{\mu(S),j}<0$ and $|a_{\mu(S),i}-a_{\mu(S),j}|=\alpha(S)$.
We take an admissible blow-up $\mathrm{Bl}_{D_i\cap D_j}\overline{X}\to \overline{X}$.
Then $S'=(A',\{D_1',\ldots,D_n',D_{n+1}'\})$.

We claim $\gamma(S')<\gamma(S)$ under the lexicographic  order in $\mathbb Z_{\geq 0}^3$. 
Let $\mu>\mu(S)$.
Then for all $D_s\cap D_t\not=\emptyset$, we have $a_{\mu s}a_{\mu t}\geq 0$.
In particular, no sign change occurs in the three numbers $a_{\mu i}+a_{\mu j}$, $a_{\mu i}$ and $a_{\mu j}$.
This shows $\mu(S')<\mu$.
Hence $\mu(S')\leq \mu(S)$. 
We have $a_{\mu(S),i}a_{\mu(S),j}<0$.
Hence either $a_{\mu(S),i}<a_{\mu(S),i}+a_{\mu(S),j}<a_{\mu(S),j}$ or $a_{\mu(S),i}>a_{\mu(S),i}+a_{\mu(S),j}>a_{\mu(S),j}$.
In each case, we have $\alpha(S')\leq \alpha(S)$.
Moreover, when $\alpha(S')= \alpha(S)$, the condition $D_i'\cap D_j'=\emptyset$ yields $\beta(S')<\beta(S)$.
Thus we have proved $\gamma(S')<\gamma(S)$.

Now we have a sequence of admissible blow-ups $\overline{X}\leftarrow \overline{X}^{(1)}\leftarrow\cdots$ so that $\gamma(S)>\gamma(S^{(1)})>\cdots$.
This sequence should terminate to get the desired $\overline{X}^{(n)}$.
\end{proof}

\begin{claim}\label{lem:20230727}
Assume $\mu(S)=0$.
Then after an admissible blow-up $\overline{X}'\to\overline{X}$, we still have $\mu(S')=0$.
\end{claim}
 \begin{proof}[Proof of \cref{lem:20230727}]
Let $\overline{X}'=\mathrm{Bl}_{D\cap D'}\overline{X}$ and $S=(A,\{D_1,\ldots,D_n\})$.
There are two cases.
The first case is that $\{D,D'\}\subset \{D_1,\ldots,D_{n}\}$.
Then $S'=(A',\{D_1',\ldots,D_n',D_{n+1}'\})$.
Suppose $D_i'\cap D_j'\not=\emptyset$.
If $i,j\in\{1,\ldots,n\}$, then $D_i\cap D_j\not=\emptyset$.
Thus we have $a_{\mu,i}a_{\mu,j}\geq 0$ for all $\mu$.
If $j=n+1$, then $D_i\cap D_{j_1}\not=\emptyset$ and $D_i\cap D_{j_2}\not=\emptyset$, where $D=D_{j_1}$ and $D'=D_{j_2}$.
Hence $a_{\mu,i}a_{\mu,j_1}\geq 0$ and $a_{\mu,i}a_{\mu,j_2}\geq 0$ for all $\mu$.
Hence $a_{\mu,i}a_{\mu,n+1}=a_{\mu,i}(a_{\mu,j_1}+a_{\mu,j_2})\geq 0$ for all $\mu$.
Hence $\mu(S')=0$ in the case $D,D'\in \{D_1,\ldots,D_{n}\}$.

Next we assume $\{D,D'\}\not\subset \{D_1,\ldots,D_{n}\}$.
Then $S'=(A,\{D_1',\ldots,D_n'\})$.
If $D_i'\cap D_j'\not=\emptyset$, then $D_i\cap D_j\not=\emptyset$.
Thus we have $a_{\mu,i}a_{\mu,j}\geq 0$ for all $\mu$.
Hence $\mu(S')=0$
\end{proof}

We consider a finite set $\Lambda=\{S_1,\ldots,S_{\nu}\}$ of representation systems.
We say that $\Lambda$ is {\it maximum} if for every $D^1,\ldots,D^l$ with $\bigcap_{j=1}^lD^j\not=\emptyset$, there exists $S_{\lambda}=(A_{\lambda},\{D_1,\ldots,D_{n_{\lambda}}\})\in\Lambda$ such that $D^1,\ldots,D^l\in \{D_1,\ldots,D_{n_{\lambda}}\}$.
For an admissible blow-up $\overline{X}'\to\overline{X}$, we get the finite set $\Lambda'=\{S_1',\ldots,S_{\nu}'\}$ of representation systems on $\overline{X}'$.

\begin{claim}\label{lem:202307251}
If $\Lambda$ is maximum, then $\Lambda'$ is maximum.
\end{claim}
 \begin{proof}[Proof of \cref{lem:202307251}]
Let $\mathrm{Bl}_{D\cap D'}\overline{X}\to\overline{X}$ be the admissible blow-up.
Let $F_1,\ldots,F_l$ satisfy $\bigcap_{j=1}^lF_j\not=\emptyset$, where $F_j$ is a boundary divisor on $\mathrm{Bl}_{D\cap D'}\overline{X}$.

There are the two cases.
The first case is that $\{F_1,\ldots,F_l\}$ contains the exceptional divisor, say $F_l$.
Then $F_1,\ldots,F_{l-1}$ are the strict transforms of $G_1,\ldots,G_{l-1}$.
Then we have $(\bigcap_{j=1}^{l-1}G_j)\cap D\cap D'\not=\emptyset$.
Hence there exists $S_{\lambda}\in\Lambda$ such that $\{G_1,\ldots,G_{l-1},D,D'\}\subset \{D_1,\ldots,D_{n_{\lambda}}\}$.
Thus the set of divisors in $S_{\lambda}'$ contains the strict transforms $F_1,\ldots,F_{l-1}$ and the exceptional divisor $F_l$.

The second case is that $\{F_1,\ldots,F_l\}$ does not contain the exceptional divisor.
Then $F_1,\ldots,F_{l}$ are the strict transforms of $G_1,\ldots,G_{l}$.
We have $\bigcap_{j=1}^{l}G_j\not=\emptyset$.
Hence there exists $S_{\lambda}\in\Lambda$ such that $\{G_1,\ldots,G_{l}\}\subset \{D_1,\ldots,D_{n_{\lambda}}\}$.
Thus the set of divisors on $S_{\lambda}'$ contains the strict transforms $F_1,\ldots,F_{l}$.
Hence we have proved that $\Lambda'$ is maximum.
\end{proof}

\begin{claim}\label{lem:20230726}
Assume that there exists a finite set $\Lambda=\{S_1,\ldots,S_{\nu}\}$ of representation systems which is maximum and satisfies $\mu(S_{\lambda})=0$ for all $S_{\lambda}$.
Assume that $f:\bD^*\to X$ has finite monodromy.
Then for every irreducible boundary divisor $D$ with $f(0)\in D$, the local monodromy with respect to $D$ is finite.
\end{claim}
 \begin{proof}[Proof of \cref{lem:20230726}]
Let $D^1,\ldots,D^l$ be the boundary divisors which contains $f(0)$.
Then $\bigcap_{j=1}^lD^j\not=\emptyset$.
Hence there exists $S_{\lambda}=(A_{\lambda},\{D_1,\ldots,D_{n_{\lambda}}\})\in\Lambda$ such that $\{D^1,\ldots,D^l\}\subset \{D_1,\ldots,D_{n_{\lambda}}\})$.
We take $J\subset \{1,\ldots,n_{\lambda}\}$ such that $\{D^1,\ldots,D^l\}=\{D_j\}_{j\in J}$.
We set $b_j=\mathrm{ord}_0f^*D_j>0$ for $j\in J$.
Then since $A_{\lambda}[J]$ is a local monodromy representation modulo the torsion around $f(0)$, we have $\sum_{j\in J}b_j\boldsymbol{a}_j=0$, where $A_{\lambda}=(\boldsymbol{a}_1,\ldots,\boldsymbol{a}_{n_{\lambda}})$.

We claim that $A_{\lambda}[J]$ is a zero matrix. 
Indeed otherwise, we may take $a_{\mu j_1}\not=0$ from $A_{\lambda}[J]$.
Then by $\sum_{j\in J}b_j\boldsymbol{a}_j=0$, we may take $a_{\mu j_2}\not= 0$ so that $a_{\mu j_1}a_{\mu j_2}<0$, for $b_j>0$ for all $j\in J$.
By $D_{j_1}\cap D_{j_2}\not=\emptyset$, this contradicts to $\mu(S_{\lambda})=0$.
Hence $A_{\lambda}[J]$ is a zero matrix. 
Thus the local monodromy around $D_j$ is finite for all $j\in J$.
\end{proof} 
Let $\overline{X}$ be a compactification such that the boundary divisor $\overline{X}-X$ is simple normal crossing.
We shall show that after a sequence $\overline{X}^{(n)}\to \overline{X}^{(n-1)}\to\cdots\to \overline{X}$ of admissible blow-ups, $\overline{X}^{(n)}$ becomes a compactification with the good property.

We consider all sets $\{D_1,\ldots,D_l\}$ of boundary divisors such that $\bigcap D_j\not=\emptyset$ and the monodromy matrices $A$ for these $\{D_1,\ldots,D_l\}$ and construct $\Lambda=\{S_{\lambda}\}$.
This $\Lambda$ is maximum.

By \cref{lem:20230724} and \cref{lem:20230727}, there exists a sequence $\overline{X}^{(n)}\to \overline{X}^{(n-1)}\to\cdots\to \overline{X}$ of admissible blow-ups such that $\mu(S_{\lambda}^{(n)})=0$ for all $S_{\lambda}$.
Note that $\Lambda^{(n)}=\{S_{\lambda}^{(n)}\}$ is maximum (cf. \cref{lem:202307251}).
We replace $\overline{X}$ by $\overline{X}^{(n)}$ and $\Lambda$ by $\Lambda^{(n)}$.
Then $\overline{X}$ and $\Lambda$ satisfy the assumption of \cref{lem:20230726}.
Hence if $D$ is a boundary divisor with infinite local monodromy and $f:\mathbb D^*\to X$ satisfies $f(0)\in D$, then $f$ has infinite monodromy.
Now we remove all $D$ from $\overline{X}$ such that local monodromy at $D$ is infinite to get a  desired partial compactification $X'$ of $X$.

If $\varrho(\pi_1(X))$ is torsion free, then the local monodromy is trivial around any irreducible divisor of $X'\backslash X$. Hence $\varrho$ can be extended to a representation over $X'$ with infinite monodromy at infinity.  
\end{proof}

\subsection{Some lemmas on finitely generated linear groups}
\begin{lem}\label{lem:small}
	Let $M$ be an closed subvariety of $M_{\rm B}(X,N)$. Let $H:=\bigcap_{\tau [\tau]\in M(\bC)}\ker\tau$, where $\tau$ ranges over all reductive representations in $M(\bC)$. Then there exists a reductive $\varrho:\pi_1(X)\to \GL_{N}(\bC)$ such that $[\varrho]\in M(\bC)$ and $\ker\varrho =H$.  
\end{lem}
\begin{proof}
	Pick any $\gamma\in H$.  Then the set $Z_\gamma:=\{\varrho\in R_{\rm B}(X,N)\mid \varrho(\gamma)=1\}$ is a Zariski closed subset of the representation variety $R_{\rm B}(X,N)$.  Moreover, $Z_\gamma$ is $\GL_N(\bC)$-invariant.       Therefore, $W_\gamma:=\pi(Z_\gamma)$ is also a Zariski closed subset of $M$ by \cite[Proposition 5.10]{Muk03}. 
	\begin{claim}\label{claim:iff}
		For any reductive $\tau:\pi_1(X)\to \GL_{N}(\bC)$ and any $\gamma\in \pi_1(X)$, $\tau(\gamma)=1$ if and only if $[\tau]\in W_\gamma$. 
	\end{claim}
	\begin{proof}
		If $\tau(\gamma)=1$, then by the definition of $Z_\gamma$  we know that $\tau\in Z_\gamma$. Hence $[\tau]\in W_\gamma$. 
		
		If $[\tau]\in W_\gamma$, then there exists (possibly non-reductive) $\tau':\pi_1(X)\to \GL_{N}(\bC)$ such that $\tau'\in Z_\gamma$ and $[\tau']=[\tau]$ . Hence $\tau'(\gamma)=1$ and  the semisimplification of $\tau'$ is conjugate to $\tau$.  
		Hence $\tau(\gamma)=1$. The claim is proved. 
	\end{proof}
	\begin{claim}
		For any $\gamma\in \pi_1(X)$, $M \subset W_\gamma$ if and only if $\gamma\in H$. 
	\end{claim}
	\begin{proof}
		If $\gamma\in H$, it is easy to see that $M \subset W_\gamma$. 
		
		Assume now  $\gamma\notin H$. Then there exists a reductive $\tau:\pi_1(X)\to \GL_{N}(\bC)$ such that $[\tau]\in M(\bC)$ and $\tau(w)\neq 1$.  
		Note that for any $\tau':\pi_1(X)\to \GL_{N}(\bC)$, if $[\tau']=[\tau]$, then the semisimplification of $\tau'$ is conjugate to $\tau$. It follows that $\tau'(w)\neq  1$. Therefore, $\tau'\not\in Z_\gamma$. It follows that $[\tau]\not\in W_\gamma$. 
	\end{proof}
	Therefore, $M\setminus \bigcup_{\gamma\notin H}W_\gamma$ is non-empty since $M$ is irreducible. We pick a reductive $\varrho:\pi_1(X)\to \GL_{N}(\bC)$ such that $[\varrho]\in M\setminus \bigcup_{\gamma\notin H}W_\gamma$.  Then for any $\gamma\not\in H$, $\varrho(\gamma)\neq 1$ as  $[\varrho]\not\in W_\gamma$ thanks to \cref{claim:iff}. Hence $\ker \varrho\subset H$.  
	Note that $H\subset \ker\varrho$. The lemma is proved.  
\end{proof}

\begin{lem}\label{lem:closure}
	Consider any subset $\Sigma\subset M_{\rm B}(X,N)(\bC)$. Let $M$ be the Zariski closure of $\Sigma$. Let $H_1:=\bigcap_{  [\tau]\in M(\bC)}\ker\varrho$, where $\tau$ ranges over all reductive representations in $M(\bC)$, and $H_2:=\bigcap_{  [\tau]\in \Sigma}\ker\varrho$, where $\tau$ ranges over all reductive representations in $\Sigma$. Then $H_1=H_2$.   In particular, $\pi_1(X)/H_2$ is a finitely presented linear group. 
\end{lem}
\begin{proof}
	For any $\gamma\in \pi_1(X)$, we let $W_\gamma$ be the Zariski closed subset of $M_{\rm B}(X,N)$ defined  in the proof of \cref{lem:small}. Consider $M_0:=\bigcap_{\gamma\in H_2}W_\gamma$. Then $M_0$ is  a Zariski closed subset of $M_{\rm B}(X,N)$. For any reductive $\varrho:\pi_1(X)\to \GL_{N}(\bC)$ with $[\varrho]\in \Sigma$, by \cref{claim:iff} we have $[\varrho]\in M_0(\bC)$. Hence $\Sigma\subset M_0(\bC)$.
	
	For any reductive $\varrho:\pi_1(X)\to \GL_{N}(\bC)$ with $\varrho\in M_0(\bC)$, by \cref{claim:iff} again we have $\varrho(\gamma)=1$ for any $\gamma\in H_2$. Hence  $H_2\subset \ker\varrho$.  Recall that $\Sigma\subset M_0(\bC)$ and $M\subset M_0$. It follows that $H_1=H_2$. 
	
	Let $M_1,\ldots,M_k$ be the geometrically irreducible components of $M$. By \cref{lem:small}, there exists a reductive $\varrho_i:\pi_1(X)\to \GL_{N}(\bC)$ such that $[\varrho_i]\in M_i(\bC)$ and $\ker\varrho_i=\bigcap_{[\tau]\in M_i(\bC)}\ker\tau$, where $\tau$ ranges over all reductive representations in $M_i(\bC)$. Consider the product representation $\varrho:=\prod_{i=1}^{k}\varrho_i:\pi_1(X)\to \prod_{i=1}^{k}\GL_{N}(\bC)$. Then we have 
	$$
	\ker\varrho=\bigcap_{[\tau]\in M(\bC)}\ker\tau
	$$
	where $\tau$ ranges over all reductive representations in $M(\bC)$.  By \cref{lem:closure}, we have 
	 $
	\ker\varrho=H_2.
	$  Therefore, $\pi_1(X)/H_2=\varrho(\pi_1(X))$, which is a finitely generated linear group.  
\end{proof}

\subsection{Construction of Shafarevich morphism (I)}\label{sec:shmor}
Let $X$ be a quasi-projective smooth variety. We will construct the Shafarevich morphism ${\rm sh}_\kC:X\to {\rm sh}_\kC(X)$  associated to  a  Zariski closed subset $\kC$ of $M_{\rm B}(X,N)(\bC)$ defined over $\bQ$  that is invariant under $\bR^*$-action.  Then we  prove the  ${\rm sh}_\kC$ is algebraic when $X$ is projective. 
\begin{thm} \label{thm:Shafarevich1}
	Let $X$ be a smooth quasi-projective variety.  	Let $\kC $ be a Zariski closed subset of $M_{\rm B}(X,N)(\bC)$, defined over $\bQ$, such that  $\kC$ is invariant under $\bR^*$-action. 
	Suppose that $\kC $ has infinite monodromy at infinity in the sense of \cref{def:monodromy}.
%	Let $\kC $ be a   constructible subset of $M_{\rm B}(X,N)(\bC)$, defined over $\bQ$, such that  $\kC$ is invariant under $\bR^*$-action. When $X$ is non-compact, we make two additional assumptions:
%	\begin{itemize}
	%	\item $\kC$ is closed;
	%	\item $\kC $ has infinite monodromy at infinity in the sense of %\cref{def:monodromy}.
	%	\end{itemize}  
	Then there exists a proper surjective holomorphic fibration   ${\rm sh}_\kC:X\to {\rm Sh}_\kC(X)$ over a \emph{normal complex space} ${\rm Sh}_\kC(X)$   such that  for any connected Zariski  closed subset $Z$ of $X$, the following properties are equivalent: 
	\begin{thmlist} 
	\item ${\rm sh}_\kC(Z)$ is a point;
	\item $ \varrho({\rm Im}[\pi_1(Z)\to \pi_1(X)])$ is finite for any reductive representation $\varrho:\pi_1(X)\to {\rm GL}_N(\bC)$ such that $[\varrho]\in \kC(\bC)$;
	\item for any irreducible component $Z_1$ of $Z$, $ \varrho({\rm Im}[\pi_1(Z_1^{\rm norm})\to \pi_1(X)])$ is finite for any reductive representation $\varrho:\pi_1(X)\to {\rm GL}_N(\bC)$ such that $[\varrho]\in \kC(\bC)$. 
\end{thmlist}   \end{thm}  
\begin{proof} 
 By \cref{prop:nonrigid},   there exist    
     reductive representations $\{\sigma^{\vhs}_i:\pi_1(X)\to {\rm GL}_N(\bC)\}_{i=1,\ldots,m}$ that underlie   $\bC$-VHS such that, for a morphism  $\iota:Z\to X$ from any  quasi-projective \emph{normal} variety $Z$ with $s_\kC\circ\iota (Z)$  being a point, the following properties hold:
     \begin{enumerate}[label*=(\alph*)]
     	\item  \label{item:direct}    For $\sigma:=\oplus_{i=1}^{m}\sigma^\vhs_i$,    the image  $\iota^*\sigma(\pi_1(Z))$ is discrete in $\prod_{i=1}^{m}\GL_{N}(\bC)$. 
	\item \label{item:unique} For each reductive $\tau:\pi_1(X)\to {\rm GL}_N(\bC)$ with $[\tau]\in \kC(\bC)$,   $\iota^*\tau$ is conjugate to some $\iota^*\sigma^\vhs_i$. Moreover, for each $\sigma_i^\vhs$, there exists some reductive representation $\tau:\pi_1(X)\to {\rm GL}_N(\bC)$ with $[\tau]\in \kC(\bC)$ such that $\iota^*\tau$ is conjugate to   $\iota^*\sigma^\vhs_i$.  
	\item  \label{item:inclusionk} We have the following inclusion:
	\begin{align}\label{eq:inclusion ker2}
		\cap_{[\varrho]\in \kC(\bC)}\ker\varrho\subset \ker \sigma_i^{\vhs}
	\end{align}  where $\varrho$ varies in all reductive representations such that $[\varrho]\in \kC(\bC)$. 
\end{enumerate}  
  
Define $H:=\cap_{\varrho}\ker \varrho\cap \ker\sigma$, where $\varrho:\pi_1(X)\to \GL_N(\bC)$ ranges over all reductive representation   such that $[\varrho]\in \kC(\bC)$.  By \eqref{eq:inclusion ker2} we have  $
	H=\cap_{\varrho}\ker \varrho$,
where $\varrho:\pi_1(X)\to \GL_N(\bC)$ ranges over all reductive representation   such that $[\varrho]\in \kC(\bC)$. Denote by  $\widetilde{X}_H:=\widetilde{X}/H$.  Let $\sD$ be the period domain associated with the $\bC$-VHS   $\sigma$ and let $p: \widetilde{X}_H\to \sD$ be the period mapping.  
We define a  holomorphic map
\begin{align} \label{eq:Psi}
	\Psi: \widetilde{X}_H&\to S_\kC\times \sD,\\\nonumber
	z &\mapsto (s_\kC \circ\pi_H(z),  p(z))
\end{align}  
where $\pi_H:\widetilde{X}_H\to X$ denotes the covering map and $s_\kC:X\to S_\kC$ is the reduction map defined in \cref{def:reduction ac}.  
\begin{lem}\label{lem:compact}
	Each connected component of any fiber of $\Psi$ is compact.  
\end{lem}
\begin{proof}[Proof of \cref{lem:compact}]
	It is equivalent to prove that for any $(t,o)\in S_\kC\times \sD$, every connected component of $\Psi^{-1}(t,o)$ is compact.  We fix any $t\in S_\kC$. 
	
	\medspace
	
\noindent {\it Step 1: We first assume that each irreducible  component of   $(s_\kC)^{-1}(t)$ is normal.} 
  %Fix a reference point $o\in  \sD$. Note that there is a complex Lie group $G_0$ which is a subgroup of a complex general linear group which acts on $\sD$ transitively such that the fixator of $o$ is a compact subgroup $V$. Hence $\sD=G_0/V$.   
   Let $F$ be an irreducible  component of  $(s_\kC)^{-1}(t)$.    %   Let $\mu:X_1\to X$ be an embedded resolution of singularity of $F$ with $E$ the strict transform of $F$. Then we have
  %\begin{equation*}
  	%\begin{tikzcd}
  	%	E \arrow[r] \arrow[d] & F\arrow[d]\\
  	%	X_1\arrow[r] & X 
  %	\end{tikzcd}
  %\end{equation*}
%Since $\mu_*:\pi_1(X_1)\to \pi_1(X)$  is an isomorphism, it follows that ${\rm Im}$ 
 Then the natural morphism $\iota: F\to X$ is proper.    By \Cref{item:direct}, $\Gamma:=\sigma(\Im[{\pi_1(F)}\to \pi_1(X)]) $ is a discrete subgroup of $\prod_{i=1}^{m}\GL_{N}(\bC)$.

   %Then $F$   is a connected Zariski closed subset of $X$. 
  %By Malcev's theorem we can replace $X$ by a finite \'etale cover such that $\Gamma$ is torsion free. 
  
  %Given that $F$ is singular and may not be normal or irreducible, there seems to be a gap in the existing literature regarding the definition of the period mapping for the restricted $\bC$-VHS  over $F$. In this paper, we aim to address this issue and provide a precise definition.
 
 %To begin, we utilize the fact that $F$ is a Zariski closed subset of $X$ and refer to \cite[Theorem 4.5]{Hof09} for the existence of a tubular neighborhood $U$ of $F$ in $X$ such that  $F$ is a deformation retraction of $U$. Consequently, we have an isomorphism $\pi_1(F) \to \pi_1(U)$. Therefore, $\Gamma=\sigma(\Im[{\pi_1(U)}\to \pi_1(X)]) $.
 
 %Next, we restrict the $\bC$-VHS $\sigma$ over $U$ in order to define a period mapping denoted as $p_U: U \to \sD/\Gamma$. We define the period mapping for $F$ as the restriction of $p_U$ to $F$, resulting in $p_U|_F: F \to \sD/\Gamma$. 
	\begin{claim}\label{claim:proper}
	The period mapping $F\to \sD/\Gamma$ is proper. 
	\end{claim}
	\begin{proof} 
Although $F$ might be singular, we can still define its period mapping since it is normal. The definition is as follows: we begin by taking a resolution of singularities  $\mu:E\to F$. Since $F$ is normal,  each fiber of $\mu$ is connected, and we have $\Gamma = \sigma({\rm Im}[\pi_1(E)\to \pi_1(X)])$.  It is worth noting that $\sD/\Gamma$ exists as a complex normal space since $\Gamma$ is discrete.  Now, consider the period mapping $E\to \sD/\Gamma$ for the $\bC$-VHS induced $\mu^*\sigma$.  This mapping then induces a holomorphic mapping $F \to \sD/\Gamma$, which satisfies the following commutative diagram:
		\begin{equation*}
			\begin{tikzcd}
			E\arrow[r,"\mu"]\arrow[d] & F\arrow[dl]\\
			\sD/\Gamma & 
			\end{tikzcd}
 	\end{equation*}  
 The resulting holomorphic map $F\to \sD/\Gamma$ is the period mapping for the $\bC$-VHS on $F$ induced by $\sigma|_{\pi_1(F)}$. To establish the properness of $F\to \sD/\Gamma$, it suffices to prove that $E\to \sD/\Gamma$ is proper.  Let $\overline{X}$ be a smooth projective compactification such that $D:=\overline{X}\backslash X$ is a simple normal crossing divisor.  Given that $E\to X$ is a proper morphism, we can take a smooth projective compactification $\overline{E}$ of $E$ such that  
 \begin{itemize}
 	\item  the complement $D_E:=\overline{E}\setminus E$ is a simple normal crossing divisor;
 	\item   there exists  a morphism $j:\overline{E}\to \overline{X}$ such that $j^{-1}(D)=D_E$. 
 \end{itemize} We aim to prove that $j^*\sigma:\pi_1(E)\to \prod_{i=1}^{m}\GL_{N}(\bC)$  has infinite monodromy at infinity. 
 
Consider any holomorphic map  $\gamma:\bD\to \overline{E}$   such that $\gamma^{-1}(D_E)=\{0\}$.   Then $(j\circ \gamma)^{-1}(D)=\{0\}$. As we assume that $\kC(\bC)$ has infinite monodromy at infinity, there exists a reductive representation $\tau:\pi_1(X)\to \GL_{N}(\bC)$ such that $[\tau]\in \kC(\bC)$ and $(j\circ\gamma)^*\tau(\pi_1(\bD^*))$ is  infinite.      Using   \Cref{item:unique}, it follows that $j^*\tau$ is conjugate to some $j^*\sigma^\vhs_i$   as $E$ is smooth quasi-projective.  As $\sigma^\vhs_i$ is a direct factor of $\sigma$,  it follows that $(j\circ\gamma)^*\sigma(\pi_1(\bD^*))$ is also infinite. Hence, we conclude that $j^*\sigma$ has infinite monodromy at infinity.

% Let $D_E'$ be an irreducible component of $D_E$. Pick a   point $x\in D_E'$ which omits other irreducible components of $D_E$.   We then can pick admissible coordinates   $(U;x_1,\ldots,x_k)$ centered at $x$ and $(V;y_1,\ldots,y_n)$ centered at $j(x)$ such that $U\cap D_E=(x_1=0)$ and $V\cap D=(y_1\cdots y_\ell=0)$.  Consider the disk $\bD_\ep$ defined by $\{(x_1,0,\ldots,0\}\in U\mid |x_1|<\ep\}$. When $\ep$    is small enough, $j^{-1}(V)\cap\bD_\ep=\{0\}$. By our assumption, we have $j^*\sigma:\pi_1(\bD_\ep^*)\to \prod_{i=1}^{m}\GL_N(\bC)$ has infinite image. 
By a theorem of Griffiths (cf. \cite[Corollary 13.7.6]{CMP17}), we conclude that $E\to\sD/\Gamma$  is proper. Therefore,   $F\to \sD/\Gamma$ is  proper.     
	\end{proof}  
Take any point $o\in \sD$. Note that there is a real Lie group $G_0$ which   acts holomorphically and transitively on $\sD$. Let $V$ be the compact subgroup that fixes $o$. Thus, we have $\sD = G_0/V$.  	 Now, let $Z$ be any connected component of the fiber of $F\to \sD/\Gamma$ over $[o]$.  According to \cref{claim:proper}, $Z$ is guaranteed to be compact.  We have that $ \sigma(\Im[{\pi_1(Z)}\to \pi_1(X)])\subset V\cap \Gamma$. Notably, $V$ is compact, and $\Gamma$ is discrete. As a result, it follows that $\sigma(\Im[{\pi_1(Z)}\to \pi_1(X)])$ is finite.
	\begin{claim}\label{claim:finite index}
		$\Im[{\pi_1(Z)}\to \pi_1(X)]\cap H$ is a finite index subgroup of $\Im[{\pi_1(Z)}\to \pi_1(X)]$.  
	\end{claim}
	\begin{proof}
		By \Cref{item:unique} and \eqref{eq:inclusion ker2},  we have
	\begin{align}\label{eq:kerequal}
		 \ker \sigma\cap \Im[\pi_1(F)\to \pi_1(X)]= H\cap \Im[\pi_1(F)\to \pi_1(X)].
	\end{align} 
		Since $ \sigma(\Im[{\pi_1(Z)}\to \pi_1(X)])$ is finite, $\ker \sigma\cap \Im[\pi_1(Z)\to \pi_1(X)]$ is a finite index subgroup of $\Im[{\pi_1(Z)}\to \pi_1(X)]$.  The claim follows from \eqref{eq:kerequal}.
	\end{proof}
	Pick any connected component $Z_0$ of  $\pi_H^{-1}(Z)$.  Note that  ${\rm Aut}(Z_0/Z)=\frac{\Im[{\pi_1(Z)}\to \pi_1(X)]}{\Im[{\pi_1(Z)}\to \pi_1(X)]\cap H}$. According to \cref{claim:finite index}, ${\rm Aut}(Z_0/Z)$ is finite, implying that $Z_0$ is compact.   Hence,  $\pi_H^{-1}(Z)$    is a disjoint union of compact subvarieties of $\widetilde{X}_H$, each of which is a finite \'etale Galois cover of $Z$ under $\pi_H$, with the Galois group $\frac{\Im[{\pi_1(Z)}\to \pi_1(X)]}{\Im[{\pi_1(Z)}\to \pi_1(X)]\cap H}$.   %For each connected component $Z'$  of $\pi_H^{-1}(Z)$,  we have $p(Z')=\gamma. o$ for some $\gamma\in \Gamma$.
	  If we denote by $\widetilde{F}$ a connected component of $\pi_{H}^{-1}(F)$, then each connected component  of any fiber of $p|_{\widetilde{F}}:\widetilde{F}\to \sD$ is a connected component of $\pi_H^{-1}(Z)$, which is compact. This can be illustrated by the following commutative diagram: 
	\begin{equation*}
		\begin{tikzcd}
			\widetilde{F} \arrow[r] \arrow[d] & F\arrow[d]\\
			\sD \arrow[r]&  \sD/\Gamma
		\end{tikzcd}
	\end{equation*} Since we have assumed that each irreducible component of $(s_\kC)^{-1}(t)$ is normal, it follows that for any $o \in \sD$, each connected component of $\Psi^{-1}(t,o)$ is compact.
	
	\medspace
	
	\noindent {\it Step 2: we prove the general case.} In the general case, we consider an embedded resolution of singularities  $\mu:Y\to X$ for the fiber $(s_\kC)^{-1}(t)$ such that each irreducible component of $(s_\kC\circ\mu)^{-1}(t)$  is smooth.   It is worth noting that $s_{\kC}\circ\mu:Y\to S_{\kC}$ coincides with the reduction map $s_{\mu^*\kC}:Y\to S_{\mu^*\kC}$ for $\mu^*\kC\subset M_{\rm B}(Y,N)$.  Let $\widetilde{Y}_{H}:=\widetilde{X}_H\times_X Y$, which is connected.  
\begin{equation*}
	\begin{tikzcd}
&	\widetilde{Y}_{H} \arrow[r] \arrow[d, "\tilde{\mu}"] & Y\arrow[d, "\mu"]\\
S_{\kC}\times \sD&\widetilde{X}_H \arrow[l, "\Psi"]\arrow[r] & X
	\end{tikzcd}
\end{equation*}
We observe that  $\tilde{\mu}$ is a proper holomorphic fibration.  We define $H':=\cap_{\varrho}\ker \varrho\cap \ker\mu^*\sigma$, where $\varrho:\pi_1(Y)\to \GL_N(\bC)$ ranges over all reductive representation   such that $[\varrho]\in \mu^*\kC $. 
Since $\mu_*:\pi_1(Y)\to \pi_1(X)$ is an isomorphism,  we have $(\mu_*)^{-1}(H)=H'$. Consequently, $\widetilde{Y}_H$ is the covering of $Y$ corresponding to $H'$, and thus ${\rm Aut}(\widetilde{Y}_{H}/Y)=H'\simeq H$.  It is worth noting that $\mu^*\kC\subset M_{\rm B}(Y,N)\simeq M_{\rm B}(X,N)$ satisfying all the conditions required for $\kC$ as stated in  \cref{thm:Shafarevich1}, unless the $\bR^*$-invariance is not obvious.  However, we note that $\mu^*\kC$ is invariant by $\bR^*$-action by \cref{cor:pull}. This enables us to    work with $\mu^*\kC$ instead of $\kC$.

As a result, $\mu^*\sigma=\oplus_{i=1}^{m}\mu^*\sigma_i^{\vhs}$ satisfies  all the properties in \Cref{item:direct,item:unique,eq:inclusion ker2}.  Note that $\mu^*\sigma$  underlies a $\bC$-VHS with the period mapping  $p\circ\tilde{\mu}:\widetilde{Y}_H\to \sD$. It follows that $\Psi\circ\tilde{\mu}:\widetilde{Y}_H\to S_\kC\times \sD$ is defined in the same way as  \eqref{eq:Psi},  determined by $\mu^*\kC$ and $\mu^*\sigma$.

  Therefore, by Step 1, we can conclude that for any $o\in \sD$,  each connected component  of   $(\Psi\circ\tilde{\mu})^{-1}(t,o)$ is compact.   Let $Z$ be a connected component of $\Psi^{-1}(t,o)$. Then we claim that $Z$ is compact. Indeed, $\tilde{\mu}^{-1}(Z)$ is closed and connected as each fiber of $\tilde{\mu}$ is connected. Therefore,  $\tilde{\mu}^{-1}(Z)$ is contained in some connected component of $(\Psi\circ\tilde{\mu})^{-1}(t,o)$. So $\tilde{\mu}^{-1}(Z)$ is compact. As $\tilde{\mu}$ is proper and surjective, it follows that $Z=\tilde{\mu}(\tilde{\mu}^{-1}(Z))$  is compact.  \cref{lem:compact} is proved.  
\end{proof} 

\medspace
 
As a result of \cref{lem:compact,lem:Stein}, the set $\widetilde{S}_H$ of connected components of fibers of $\Psi$ can be endowed with the structure of a  complex normal space such that $\Psi=g\circ{\rm sh}_H$  where ${\rm sh}_H:\widetilde{X}_H\to \widetilde{S}_H$ is  a proper holomorphic fibration and $g:\widetilde{S}_H\to \widetilde{S}_\kC\times \sD$ is a holomorphic map. In \cref{claim:discrete} below, we will prove that each fiber of $g$ is discrete. 
\begin{claim}\label{claim:contract}
	${\rm sh}_H$ contracts every compact   subvariety of $\widetilde{X}_H$. 
\end{claim}
\begin{proof}
	Let $Z\subset \widetilde{X}_H$ be a compact irreducible subvariety. Then, $W:=\pi_H(Z)$ is also a compact irreducible subvariety in $X$ with  $\dim Z=\dim W$.   Hence $\Im[\pi_1(Z^{\rm norm})\to \pi_1(W^{\rm norm})]$ is a finite index subgroup of $\pi_1(W^{\rm norm})$. Note that $W$ can be endowed with an algebraic structure induced by $X$. As the natural map $Z\to W$ is finite, $Z$ can be equipped with an algebraic structure such that the natural map $Z\to X $ is algebraic. 
	
	For any reductive representation $\varrho:\pi_1(X)\to {\rm GL}_N(K)$ with $\varrho\in \kC(K)$ where $K$ is a non archimedean local field, we have $ \varrho(\Im[\pi_1(Z)\to \pi_1(X)])\subset \varrho(\Im[\pi_1(\widetilde{X}_H)\to \pi_1(X)])=\{1\}$.   Hence,   $\varrho(\Im[\pi_1(W^{\rm norm})\to \pi_1(X)])$ is   finite which is thus bounded.  By \cref{lem:bounded}, $W$ is contained in a fiber  of $s_\kC$. Consider a desingularization $Z'$ of $Z$ and let $i:Z'\to X$ be the natural  algebraic morphism. Note that $i^*\sigma(\pi_1(Z'))=\{1\}$.    It   follows that the variation of Hodge structure induced by  $i^*\sigma$ is trivial.  Therefore, $p(Z)$ is a point.  Hence  $Z$ is contracted by $\Psi$.    The claim follows. 
\end{proof} 
\begin{lem}\label{lem:properdis}
There is an action of ${\rm Aut}(\widetilde{X}_H/X)=\pi_1(X)/H$	 on $\widetilde{S}_H$ that  is  equivariant for the proper holomorphic fibration ${\rm sh}_{H}:\widetilde{X}_H\to \widetilde{S}_H$.  This action is  analytic and properly discontinuous. Namely,    for any point $y$
	of $\widetilde{S}_H$, there exists an open neighborhood $V_y$ of $y$ such that the set
	$$\{\gamma\in \pi_1(X)/H\mid \gamma.V_y\cap V_y\neq\varnothing\}$$ is finite.
\end{lem}
\begin{proof}
	 Take any $\gamma\in \pi_1(X)/H$.   We can consider  $\gamma$ as an analytic automorphism of $\widetilde{X}_H$.  According to \cref{claim:contract},  $ {\rm sh}_{H}\circ \gamma: \widetilde{X}_H\to \widetilde{S}_H$ contracts each fiber of  the proper holomorphic fibration ${\rm sh}_{H}:\widetilde{X}_H\to \widetilde{S}_H$. As a result, it induces a holomorphic map   $\tilde{\gamma}:\widetilde{S}_H\to \widetilde{S}_H$ such that we have the following commutative diagram: 
	\begin{equation*}
		\begin{tikzcd}
			\widetilde{X}_H\arrow[r,"\gamma"] \arrow[d, "{\rm sh}_H"] & \widetilde{X}_H\arrow[d, "{\rm sh}_H"]\\
			\widetilde{S}_H \arrow[r, "\tilde{\gamma}"] &\widetilde{S}_H
		\end{tikzcd}
	\end{equation*}
Let us define the action of $\gamma$ on $\widetilde{S}_H$ by $\tilde{\gamma}$. Then $\gamma$ is an analytic automorphism and ${\rm sh}_H$ is $\pi_1(X)/H$-equivariant. 	It is evident that  $\tilde{\gamma}:\widetilde{S}_H\to \widetilde{S}_H$ carries one fiber of ${\rm sh}_H$ to another fiber.  Thus, we have shown that $\pi_1(X)/H$ acts on $\widetilde{S}_H$ analytically and equivariantly with respect to ${\rm sh}_{H}:\widetilde{X}_H\to \widetilde{S}_H$. Now, we will prove that this action is properly discontinuous.
	
	\medspace

Take any $y\in \widetilde{S}_H$ and let $F:={\rm sh}_H^{-1}(y)$.  	Consider the subgroup $\cS$ of $\pi_1(X)/H$ that fixes $y$, i.e.  \begin{align}\label{eq:fixator}
	 \cS:=\{\gamma\in \pi_1(X)/H \mid \gamma\cdot F=F \}.
\end{align}
	Since $F$ is compact, $\cS$ is finite.
	\begin{claim}\label{claim:component}
	$F$ is a connected component of 	 $\pi_H^{-1}(\pi_H(F))$. 
	\end{claim}
\begin{proof}[Proof of \cref{claim:component}]
	Let $x\in \pi_H^{-1}(\pi_H(F))$. Then there exists $x_0\in F$  such that $\pi_H(x)=\pi_H(x_0)$. Therefore, there exists $\gamma\in \pi_1(X)/H $ such that $\gamma. x_0=x$. It follows that $\pi_H^{-1}(\pi_H(F))=\cup_{\gamma\in \pi_1(X)/H}\gamma. F$.  Since $\gamma$ carries one fiber of ${\rm Sh}_H$ to another fiber, and the group $\pi_1(X)/H$ is finitely presented, it follows that   $\cup_{\gamma\in \pi_1(X)/H}\gamma. F$ are countable union of fibers of ${\rm Sh}_H$.  It follows that $F$ is a connected component of 	 $\pi_H^{-1}(\pi_H(F))$.  
\end{proof}
\cref{claim:component} implies that  $\pi_H:F\to \pi_{H}(F)$ is a finite \'etale cover. Denote by $Z:=\pi_H(F)$ which is a connected Zariski closed subset of $X$. Then ${\rm Im}[\pi_1(F)\to \pi_1(Z)]$ is finite.    As a consequence of \cite[Theorem 4.5]{Hof09}, there is a connected open neighborhood $U$ of $Z$ such that $\pi_1(Z)\to \pi_1(U)$ is an isomorphism.  Therefore, ${\rm Im}[\pi_1(U)\to \pi_1(W)]={\rm Im}[\pi_1(F)\to \pi_1(W)]$ is also finite. As a result,   $\pi_H^{-1}(U)$ is a disjoint union of  connected open sets $\{U_\alpha\}_{\alpha\in I}$ such that
	 \begin{enumerate}[label=(\alph*)]
	\item For each $U_\alpha$, $\pi_H|_{U_\alpha}:U_\alpha\to U$ is a finite \'etale covering. 
	\item \label{item:exactone}Each $U_\alpha$ contains exactly one connected component of $\pi_H^{-1}(Z)$.   
\end{enumerate}      
We may assume that $F\subset U_{\alpha_1}$ for some $\alpha_1\in I$. By \Cref{item:exactone}, 	   for any $\gamma\in \pi_1(X,z)/H\simeq {\rm Aut}(\widetilde{X}_H/X)$, $\gamma\cdot U_{\alpha_1}\cap U_{\alpha_1}= \varnothing$ if and only if $\gamma\notin \cS$.    

 Since ${\rm sh}_H$ is a proper holomorphic fibration, we can take a neighborhood $V_y$ of $y$ such that ${\rm sh}_H^{-1}(V_y)\subset U_{\alpha_1}$.    Since $\gamma\cdot U_{\alpha_1}\cap U_{\alpha_1}=\varnothing $  if and only if $\gamma\notin \cS$, it follows that  $\gamma\cdot V_y\cap V_y= \varnothing$  if $\gamma\not\in \cS$. Since $\cS$ is finite and $y$ was chosen arbitrarily, we have shown that the action of $\pi_1(X)/H$ on $\widetilde{S}_H$ is properly discontinuous.  Thus  \cref{lem:properdis} is proven. 
 \end{proof}
 Let $\nu:\pi_1(X)/H\to {\rm Aut}(\widetilde{S}_H)$ be action of $\pi_1(X)/H$ on $ \widetilde{S}_H$ and let $\Gamma_0:=\nu(\pi_1(X)/H)$.    By \cref{lem:properdis} and \cite{Car60}, we know that the quotient ${\rm Sh}_{\kC}(X):=\widetilde{S}_H/\Gamma_0$ is a  complex normal space, and it is compact if $X$ is compact.  Moreover, since   ${\rm sh}_H:\widetilde{X}\to \widetilde{S}_H$  is $\nu$-equivariant, it  induces  a   proper  holomorphic fibration ${\rm sh}_\kC:X\to {\rm Sh}_{\kC}(X)$ from $X$ to a complex normal space ${\rm Sh}_{\kC}(X)$.
 \begin{equation} \label{eq:equivariant}
\begin{tikzcd}
&	\widetilde{X}_H\arrow[r,"\pi_H"] \arrow[dl, "\Psi"']\arrow[d, "{\rm sh}_H"] & X\arrow[d, "{\rm sh}_\kC"]\\
S_\kC\times \sD\arrow[d]	&\widetilde{S}_H\arrow[ld, "\phi"]\arrow[l, "g"'] \arrow[r, "\mu"] & {\rm Sh}_\kC(X)\\
\sD &&
 	\end{tikzcd}
 \end{equation}
\begin{claim}\label{lem:shafarevich}
	For any connected Zariski closed subset $Z\subset X$,  the following properties are equivalent: 
	\begin{enumerate}[label=\rm (\alph*)]
		\item ${\rm sh}_\kC(Z)$ is a point;
		\item $ \varrho({\rm Im}[\pi_1(Z)\to \pi_1(X)])$ is finite for any reductive representation $\varrho:\pi_1(X)\to {\rm GL}_N(\bC)$ such that $[\varrho]\in \kC(\bC)$;
		\item for any irreducible component $Z_1$ of $Z$, $ \varrho({\rm Im}[\pi_1(Z_1^{\rm norm})\to \pi_1(X)])$ is finite for any reductive representation $\varrho:\pi_1(X)\to {\rm GL}_N(\bC)$ such that $[\varrho]\in \kC(\bC)$. 
	\end{enumerate} 
\end{claim} 
\begin{proof}
\noindent {\em {\rm (c)}  $\Rightarrow$ {\rm (a)}}:   Let $f:Y\to Z_1$ be a desingularization.   For any non archimedean local field $K$, we note that there exists an abstract embedding $K\hookrightarrow \bC$.  Hence,  for  any reductive representation $\rho:\pi_1(X)\to {\rm GL}_N(K)$ with $[\rho]\in \kC(K)$, $f^*\rho(\pi_1(Y))\subset  \varrho({\rm Im}[\pi_1(Z_1^{\rm norm})\to \pi_1(X)])$ is finite, and therefore bounded. Hence,    $   f(Y)$ is contained in some fiber $F$ of $s_\kC$ by 
\cref{lem:bounded}.  Using \Cref{item:direct,item:unique}, we have that $\Gamma:=f^*\sigma(\pi_1(Y))$ is also finite.  Therefore,  $Y$ is mapped to one point by the period mapping $Y\to \sD/\Gamma$ of $f^*\sigma$.     As a result, ${\rm sh}_\kC(Z_1)$ is a point by \eqref{eq:equivariant}. Since $Z$ is connected, ${\rm sh}_\kC(Z)$ is also a point. 
	
	\medspace
	
	\noindent {\em {\rm (b)}  $\Rightarrow$ {\rm (c)}}:  Obvious. 
	
	\medspace
	
\noindent {\em {\rm (a)}  $\Rightarrow$ {\rm (b)}}:  We observe from \eqref{eq:equivariant} that  for any connected component $Z'$ of $\pi_H^{-1}(Z)$, it is contracted by $\Psi$.  By \cref{lem:compact}, $Z'$ is contained in some compact subvariety of $\widetilde{X}_H$. Since $Z'$ is closed, it is also compact.   Therefore, the map $Z'\to Z$ induced by $\pi_H$ is a finite étale cover. 

Let $\varrho:\pi_1(X)\to {\rm GL}_N(\bC)$ be any reductive representation such that $[\varrho]\in \kC(\bC)$. Note that $\varrho(\Im[\pi_1(Z')\to \pi_1(X)])$ is a finite index subgroup of $\varrho(\Im[\pi_1(Z)\to \pi_1(X)])$.  Since  $\varrho(\Im[\pi_1(Z')\to \pi_1(X)])=\{1\}$,   it follows that $\varrho(\Im[\pi_1(Z)\to \pi_1(X)])$ is finite.   The claim is proved. 
\end{proof}  
Therefore, we have  constructed the desired proper  holomorphic fibration ${\rm sh}_\kC:X\to {\rm Sh}_\kC(X)$.  The theorem is proved. 
\end{proof}

\begin{proposition}
	Let $X$ be a smooth projective variety and let $\kC$ be a constructible subset of $M_{\rm B}(X,N)(\bC)$ which is defined over $\bQ$ and is invariant under $\bR^*$-action. Then ${\rm Sh}_\kC(X)$ is a projective  variety and ${\rm sh}_\kC$ is algebraic.  
\end{proposition}
\begin{proof} 
	We will use the same notations as in the proof of  \cref{thm:Shafarevich1}. 
\begin{claim}\label{claim:free} 
	There exists a finite index normal subgroup $N$ of  $\pi_1(X)/H$ such that its action on $\widetilde{S}_H$ does not have fixed point.
\end{claim}
\begin{proof} [Proof of \cref{claim:free}]
	By \cref{lem:closure}, we know that $\pi_1(X)/H$ is a finitely presented linear group. Therefore, by Selberg's lemma it contains a finite index normal subgroup $N$ which is torsion free.   
By \eqref{eq:fixator}, we know that for any $y\in \widetilde{S}_H$, the subgroup of $\pi_1(X)/H$ fixing $y$ is finite.   It follows that the action of $N$ on  $y\in \widetilde{S}_H$ is torsion free. 
\end{proof}

  Let $Y:=\widetilde{X}_H/N$, where $N$ is the finite index torsion free normal subgroup of $\pi_1(X)/H$. Then $Y\to X$ is a finite Galois \'etale cover with ${\rm Aut}(\widetilde{X}_H/Y)=N$. Recall that we define $\nu:\pi_1(X)/H\to {\rm Aut}(\widetilde{S}_H)$ to be action of $\pi_1(X)/H$ on $ \widetilde{S}_H$.  Since   ${\rm sh}_H:\widetilde{X}_H\to \widetilde{S}_H$  is   $\nu$-equivariant, the group $N$ gives rise to a   proper holomorphic  fibration
$
Y\to \widetilde{S}_H/\nu(N) 
$  over a complex normal space $\widetilde{S}_H/\nu(N)$. By \cref{claim:proper,claim:free},  $\nu(N)$  acts on  $\widetilde{S}_H$ properly continuous and freely and thus the covering $\widetilde{S}_H\to \widetilde{S}_H/\nu(N)$   is \'etale. 
\begin{claim}\label{claim:discrete}
	Each fiber  of $g:\widetilde{S}_H\to S_\kC\times \sD$  is discrete.
\end{claim}
\begin{proof}
	Let $(t,o)\in S_\kC\times \sD$ be arbitrary point and take any point $y\in g^{-1}((t,o))$.  Then $Z:={\rm sh}_H^{-1}(y)$ is a 	   connected component of the fiber  $\Psi^{-1}((t,o))$, that is compact by \cref{lem:compact}.	By \cref{lem:Stein},  $Z$ has  an open neighborhood $U$ such that $\Psi(U)$ is a locally closed analytic subvariety of $S_\kC\times \sD$ and $\Psi|_{U}:U \rightarrow \Psi(U)$ is proper.  Therefore, for the Stein factorization $U\to V\stackrel{\pi_V }{\to}\Psi(U)$ of $\Psi|_U$,  $U\to V$ coincides with ${\rm sh}_H|_{U}:U\to {\rm sh}_H(U)$ and  $\pi_V:V\to \Psi(U)$ is finite.  Observe that $V$ is an open neighborhood of $y$ and    $\pi_V:V\to \Psi(U)$ coincides with $g|_{V}:V\to S_{\kC}\times \sD$. Therefore,  the set $V\cap g^{-1}((t,o))=V\cap  (\pi_V)^{-1}(t,o)$  is finite.  As a result,  $g^{-1}((t,o))$ is discrete.  The claim  is proven. 
\end{proof}	
%Since the holomorphic map $\phi:\widetilde{S}_H\to \sD$ is horizontal, we can pull back the canonical metric $\omega_\sD$  on $\sD$ to obtain a  semi-positive $(1,1)$-form $\omega$ on the     complex normal space $\widetilde{S}_H$. Note that $\omega$ is $\pi_1(X)/H$-invariant since $\phi$ is  $\pi_1(X)/H$-invariant.   Since $\nu(N)$ acts on  $\widetilde{S}_H$ properly continuous and freely,   $\omega$ induces a   semi-positive $(1,1)$-form  $T$   over $W:= \widetilde{S}_H/\nu(N) $ such that $q^*T=\omega$ where $q:\widetilde{S}_H\to \widetilde{S}_H/\nu(N)$  is the quotient map.  

In \cite{Gri70}, Griffiths discovered a  so-called \emph{canonical bundle $K_{\sD}$ on the period domain $\sD$}, which is invariant under $G_0$. Here $G_0$ is a real Lie group acting on $\sD$ holomorphically and transitively.  It is worth noting that $K_\sD$ is endowed with a $G_0$-invariant smooth metric $h_{\sD}$ whose curvature is positive-definite in the horizontal direction.  The period mapping $p:\widetilde{X}_H\to \sD$ induces a  holomorphic map $\phi:\widetilde{S}_H\to \sD$ which is horizontal. It is important to note that the monodromy representation $\sigma: \pi_1(X)/H \to G_0$ induces a representation $\nu(\pi_1(X)/H) \to G_0$, such that $\phi$ is equivariant with respect to this representation. As a result, $\phi^*K_{\sD}$ descends to a line bundle on the quotient $W:=\widetilde{S}_H/\nu(N)$, denoted by $L_G$.     
The smooth metric $h_{\sD}$ induces a smooth metric $h_G$ on $L_G$  whose curvature form is denoted by $T$. Let $x\in \widetilde{S}_H$ be a smooth point of $ \widetilde{S}_H$ and let $v\in T_{\widetilde{S}_H,x}$. Then $-iT(v,\bar{v})>0$   if $d\phi(v)\neq 0$.    %  Denote by $f:W\to S_\kC$ the natural morphism.   %Let $F$  be the normalization of any irreducible component of $f^{-1}(t)$. Recall that $\Gamma:=\sigma({\rm Im}(\pi_1(F)\to \pi_1(X)))$ is discrete. As each fiber of $g$ is discrete,  the period mapping of $\sigma$ induces a map $p_F:F\to \sD/\Gamma$.  By \cref{claim:discrete}, $p_F$ is a finite morphism. Note that  the pullback of $T$ to $F$, denoted by  $T|_F$, coincides with the pull back the canonical metric $\omega_\sD$  on $\sD$ via $p_F$. Therefore, by \cref{thm:DHP}, $\{T|_{F}\}$  is K\"ahler and $F$ is projective. 
\begin{claim}\label{claim:projective}
	${\rm Sh}_\kC(X)$ is a projective normal variety.
\end{claim}
\begin{proof}
	Note that $S_\kC$ is a projective normal variety. We take an ample line bundle $L$ over $S_{\kC}$. Recall that there is a line bundle   $L_{\rm G}$ on $W$ equipped with a smooth metric $h_{\rm G}$ such that its curvature form is $T$.    Denote by $f:W\to S_\kC$ the natural morphism induced by $g:\widetilde{S}_H\to S_\kC\times \sD$.  Let $\mu:W'\to W$  be a resolution of singularities of $W$. 
	 \begin{equation*}  
		\begin{tikzcd}
			&	\widetilde{X}_H\arrow[r,"\pi_H"] \arrow[dl, "\Psi"']\arrow[d, "{\rm sh}_H"] & Y\arrow[d ] & \\
			S_\kC\times \sD\arrow[d]	&\widetilde{S}_H\arrow[ld, "\phi"]\arrow[l, "g"'] \arrow[r] & W\arrow[d, "f"] & W'\arrow[l, "\mu"']\\
			\sD && S_\kC&
		\end{tikzcd}
	\end{equation*}

%Then $(f\circ\mu)^*L$    % Let $F$  be the normalization of any irreducible component of $f^{-1}(t)$, where $f:W\to S_\btau $.  %Note that  $T|_F$ is strictly positive at a general point. 
	We take a smooth metric $h$ on $L$ such that its curvature form $i\Theta_h(L)$ is K\"ahler. As shown in   \cref{claim:discrete}, the map $g:\widetilde{S}_H\to S_\kC\times \sD$ is discrete. Therefore, $g$ is an immersion at general points of $\widetilde{S}_H$.   Thus, for the line bundle $\mu^*( L_{\rm G}\otimes f^*L)$ equipped with the smooth metric $\mu^*( h\otimes f^*h_{\rm G})$, its curvature form is strictly positive at some points of $W'$.  By Demailly's holomorphic Morse inequality or Siu's solution for the Grauert-Riemenschneider conjecture, $\mu^*( L_{\rm G}\otimes f^*L)$  is a big line bundle and thus $W'$ is a Moishezon manifold. Hence $W$ is a Moishezon variety.

Moreover, we can verify that  for     irreducible  positive-dimensional  closed subvariety $Z$ of $W$, there exists a smooth point $x$ in $Z$ such that it has a neighborhood $\Omega$ that can be lifted to the \'etale covering $\widetilde{S}_H$ of $W$, and $g|_{\Omega}:\Omega\to S_\kC\times \sD$ is an immersion. 
It follows that 
 $
(if^*\Theta_h(L)+T)|_{\Omega}
	$  is strictly positive.  
Since $if^*\Theta_h(L)+T\geq 0$, it follows that
	$$
(L_G\otimes f^*L)^{\dim Z}\cdot [Z]=	\int_{Z^{\rm reg}}(if^*\Theta_h(L)+T)^{\dim Z}>0. 
	$$
By the Nakai-Moishezon criterion for Moishezon varieties  (cf.  \cite[Theorem 3.11]{Kol90}), $L_G\otimes f^*L$ is ample, implying that $W$ is projective.   
Recall that the compact complex normal space ${\rm Sh}_{\kC}(X):=\widetilde{S}_H/\Gamma_0$ is a quotient of $W=\widetilde{S}_H/\nu(N)$ by the finite group $\Gamma_0/\nu(N)$.   Therefore, ${\rm Sh}_\kC(X)$ is also projective.  The claim is proved. 
\end{proof} 
We  accomplish the proof of the proposition. 
\end{proof}
\begin{rem}
In \cref{thm:Shafarevich1}, when $X$ is compact,  \cref{prop:nonrigid} allows us to assume that $\kC$ is constructible rather than Zariski closed. Under this weaker assumption, we can still obtain the Shafarevich morphism for $\kC$. 
\end{rem}
\begin{rem}\label{rem:Brunebarbegap2}
We remark that \cref{lem:properdis} is claimed without a proof  in \cite[p. 524]{Eys04} and  \cite[Proof of Theorem 10]{Bru23}. It appears to us that the proof of \cref{lem:properdis} is not straightforward. 

It is worth noting that  \Cref{claim:free} is implicitly used in \cite[Proposition 5.3.10]{Eys04}. In that proof, the criterion for Stein spaces (cf.    \Cref{prop:stein}) is employed, assuming    \Cref{claim:free}.   
 Given its significance in   the proofs of \cref{main2,main}, we provide a complete proof.  
\end{rem}

\subsection{Construction of Shafarevich morphism (II)}
 In the previous subsection, we established the existence of the Shafarevich morphism  associated with a   constructible subset of $M_{\rm B}(X,N)(\bC)$ defined over $\mathbb{Q}$ that are invariant under $\bR^*$-action. In this section, we focus on proving an existence theorem for the Shafarevich morphism associated with a single reductive representation, based on \cref{thm:Shafarevich1}. 
 Initially, we assume that   the representation has infinite monodromy at infinity and that $X$ is smooth. However, we will subsequently  apply \cref{prop:infinity} to remove this assumption and establish the more general result.  
\begin{proposition}\label{thm:Sha2}
	Let $X$ be a  smooth quasi-projective  variety. 	Let $\varrho:\pi_1(X)\to {\rm GL}_N(\bC)$ be a reductive representation with infinite monodromy at infinity. 
	Then there exists a  proper holomorphic fibration   ${\rm sh}_{\varrho}:X\to {\rm Sh}_{\varrho}(X)$ onto a complex  normal space ${\rm Sh}_{\varrho}(X)$   such that for any Zariski  closed subvariety   $Z\subset X$,  $ \varrho({\rm Im}[\pi_1(Z)\to \pi_1(X)])$ is finite if and only if ${\rm sh}_\varrho(Z)$ is a point. If $X$ is smooth, then  for any Zariski  closed subset   $Z\subset X$, the following properties are equivalent: 
	\begin{thmlist} 
		\item ${\rm sh}_\varrho(Z)$ is a point;
		\item $ \varrho({\rm Im}[\pi_1(Z)\to \pi_1(X)])$ is finite;
		\item for each irreducible component $Z_1$ of $Z$, $ \varrho({\rm Im}[\pi_1(Z_1^{\rm norm})\to \pi_1(X)])$ is finite. 
	\end{thmlist}  
\end{proposition}   

We first prove the following crucial result. 
\begin{proposition}\label{lem:C*}
		Let $X$ be a smooth quasi-projective variety. Let $f: Z\to X$ be a \emph{proper} morphism from a \emph{smooth} quasi-projective variety $Z$. Let $\varrho:\pi_1(X)\to \GL_N(\bC)$ be a reductive representation. Define  $M:= j_{Z}^{-1}\{1\}$, where $1$  stands  for the trivial representation, and $j_Z: M_{\rm B}(X,N)\to M_{\rm B}(Z,N)$ is  the natural morphism of $\bQ$-scheme induced by $f$.
Then $M$ is a closed subscheme of $M_{\rm B}(X,N)$ defined over $\bQ$ such that $M(\bC)$ is invariant under $\bC^*$-action. 
\end{proposition}
\begin{proof}
We take a smooth projective compactification $\overline{X}$ (resp. $\overline{Z}$) of $X$ (resp.  $Z$) such that $D:=\overline{X}\backslash X$ (resp. $D_Z:=\overline{Z}\backslash Z$)  is a simple normal crossing divisor and $f$ extends to a morphism $\bar{f}:\overline{X}\to \overline{Z}$.  Note that the morphism $j_Z$ is a $\bQ$-morphism between  affine schemes of finite type  $M_{\rm B}(X,N)$ and $M_{\rm B}(Z,N)$ defined over $\bQ$. We remark that $M$ is   a closed subscheme of $M_{\rm B}(X,N)$ defined over $\bQ$ as we have $\{1\}\in M_{\rm B}(Z,N)(\bQ)$ . Let $\varrho:\pi_1(X)\to \GL_N(\bC)$ be a reductive representation such that $[\varrho]\in M(\bC)$. By  \cref{moc}, there is a tame pure imaginary harmonic bundle $(E,\theta,h)$ on $X$ such that $\varrho$ is the monodromy representation of  $\nabla_h+\theta+\theta_h^\dagger$.
  By definition of $M$, $f^*\varrho$ is a trivial representation.  Therefore, $f^*\varrho$ corresponds to a trivial harmonic bundle $(\oplus^N\cO_Z,0,h_0)$ where  $h_0$ is the canonical metric for the trivial vector bundle $\oplus^N\cO_Z$ with zero curvature. By  the unicity theorem in \cite[Theorem 1.4]{Moc06}, $(\oplus^N\cO_Z,0,h_0)$ coincides with $(f^*E,f^*\theta,f^*h)$ with some obvious ambiguity of $h_0$.  Therefore, $f^*E=\oplus^N\cO_Z$ and $f^*\theta=0$. In particular, the  regular filtered Higgs bundle  $(\tilde{E}_*,\tilde{\theta})$  on $(\overline{Z}, D_Z)$  induced by  the prolongation of $(f^*E,f^*\theta,f^*h)$ using norm growth defined in \cref{sec:prolong}  is trivial; namely  we have ${}_{\bm{a}}\tilde{E}=\cO_{\overline{Z}}^N\otimes \cO_{\overline{Z}}(\sum_{i=1}^{\ell}a_iD'_i)$ for any $\bm{a}=(a_1,\ldots,a_\ell)\in \bR^\ell$ and $\tilde{\theta}=0$.   Here we write $D_Z=\sum_{i=1}^{\ell} D'_i$. 
    
   Let $(\bm{E}_*,\theta)$  be the induced regular filtered Higgs bundle   on $(\overline{X}, D)$  by $(E,\theta,h)$ defined in \cref{sec:prolong}.  According to \cref{sec:adapt,sec:pullback} we can define the pullback $(f^*\bm{E}_*,f^*\theta)$,  which also forms a regular filtered Higgs bundle on $(\overline{Z}, D_Z)$ with trivial characteristic numbers. By virtue of \cref{prop:functoriality}, we deduce that  $(f^*\bm{E}_*,f^*\theta)=(\tilde{E}_*,\tilde{\theta})$.    Consequently, it follows that  $(f^*\bm{E}_*,f^*\theta)$ is  trivial.  Hence $(f^*\bm{E}_*,tf^*\theta)$ is trivial for any $t\in \bC^*$.

 Fix some ample line bundle $L$ on $\overline{Z}$.   It is worth noting that for any $t\in \bC^*$, $(\bm{E}_*,t\theta)$  is $\mu_L$-polystable   with trivial characteristic numbers. By \cite[Theorem 9.4]{Moc06}, there is a pluriharmonic metric $h_t$ for $(E,t\theta)$ adapted to the parabolic structures of $(\bm{E}_*,t\theta)$.  By \cref{prop:functoriality} once again, the  regular filtered Higgs bundle $(f^*\bm{E}_*,tf^*\theta)$ is the prolongation of the tame harmonic bundle $(f^*E,tf^*\theta,f^*h_t)$  using norm growth defined in \cref{sec:prolong}.  Since $(f^*\bm{E}_*,tf^*\theta)$ is trivial for any $t\in \bC^*$, by the unicity theorem in \cite[Theorem 1.4]{Moc06} once again, it follows that $(\oplus^N\cO_Z,0,h_0)$ coincides with $(f^*E,tf^*\theta,f^*h_t)$ with some obvious ambiguity of $h_0$.    Recall that in \cref{sec:C*action}, $\varrho_t$ is defined to be the monodromy representation of the flat connection $\nabla_{h_t}+t\theta+\bar{t}\theta_{h_t}^\dagger$.   It follows that $f^*\varrho_t$ is the monodromy representation of the flat connection $f^*(\nabla_{h_t}+t\theta+\bar{t}\theta_{h_t}^\dagger)$.    Therefore, $f^*\varrho_t$ is a trivial representation.
 
 However, it is worth noting that $\varrho_t$ might not be reductive as $(E,t\theta,h_t)$ might not be pure imaginary. Let $\varrho_t^{ss}$ be the semisimplification of $\varrho_t$. Then $[\varrho_t]=[\varrho_t^{ss}]$. Since $f^*\varrho_t$ is a trivial representation,   $f^*\varrho_t^{ss}$ is also trivial.  Note that $$f^*(t.[\varrho])=f^*[\varrho_t]=f^*[\varrho_t^{ss}]=[f^*\varrho_t^{ss}]=1.$$ 
 Therefore, $t.[\varrho]\in M(\bC)$ if $[\varrho]\in M(\bC)$.   The proposition is proved.   
\end{proof} 
\begin{rem}\label{rem:gap of Brunebarbe}
It is important to note that, unlike the projective case, the proof of \cref{lem:C*} becomes considerably non-trivial when $X$ is quasi-projective. This complexity arises from the utilization of the functoriality of pullback of regular filtered Higgs bundles, which is established in \cref{prop:functoriality}. Lemma \ref{lem:C*} plays a crucial role in the proof of \cref{thm:Sha2} as it allows us to remove the condition of $\bR^*$-invariance in \cref{thm:Shafarevich1}. However, we remark that \cref{lem:C*} is claimed without a  proof in the proof of \cite[Lemma 9.3]{Bru23}.
\end{rem}

 \begin{proof}[Proof of \cref{thm:Sha2}] 
 \noindent {\it Step 1: we assume that $X$ is smooth, and $\varrho$ has infinite monodromy at infinity.}	  Let $f: Z\to X$ be a \emph{proper} morphism from a \emph{smooth} quasi-projective variety $Z$.   Then $j_Z: M_{\rm B}(X,N)\to M_{\rm B}(Z,N)$ is a morphism of $\bQ$-scheme.  Define 
\begin{equation}\label{eq:kc}
\kC:=\bigcap_{\{f:Z\to X\mid f^*\varrho=1\} } j_{Z}^{-1}\{1\}, 
\end{equation}
where $1$  stands  for the trivial representation, and  $f: Z\to X$ ranges over all  proper morphisms from    smooth quasi-projective varieties $Z$ to $X$.   
  Then $\kC$ is a  Zariski closed   subset    defined over $\bQ$, and by \cref{lem:C*},  $\kC(\bC)$ is invariant under $\bC^*$-action.
 Note that $[\varrho]\in \kC(\bC)$. 
 As we assume that $\varrho$ has infinite monodromy at infinity,  conditions in \cref{thm:Shafarevich1} for $\kC$ are  fulfilled.   
 Therefore, we apply \cref{thm:Shafarevich1} to conclude that  the Shafarevich morphism ${\rm sh}_\kC:X\to {\rm Sh}_{\kC}(X)$ exists.   It is a proper holomorphic fibration over a complex normal space. 
  \begin{claim}\label{claim:contract3}
For any connected Zariski closed subset  $Z\subset X$, the following properties are equivalent:
	\begin{enumerate}[label=\rm (\alph*)]
	\item ${\rm sh}_\kC(Z)$ is a point;
	\item $ \varrho({\rm Im}[\pi_1(Z)\to \pi_1(X)])$ is finite;
	\item for each irreducible component $Z_1$ of $Z$, $ \varrho({\rm Im}[\pi_1(Z_1^{\rm norm})\to \pi_1(X)])$ is finite. 
\end{enumerate}  
  \end{claim} 
\begin{proof}

\noindent {\em {\rm (a)} $\Rightarrow$ {\rm (b)}}:   this follows from the fact that $[\varrho]\in \kC(\bC) $ and \cref{thm:Shafarevich1}.   

\medspace

\noindent {\em {\rm (b)} $\Rightarrow$ {\rm (c)}}:  obvious.

\medspace

\noindent {\em {\rm (c)} $\Rightarrow$ {\rm (a)}}:  Consider the desingularization $Y\to Z_1$, and let $f:Y\to X$ be the composite morphism. 
 Since $\pi_1(Y)\to \pi_1(Z_1^{\rm norm})$  is surjective, it follows that $f^*\varrho(\pi_1(Y))$ is finite.   We can take a finite \'etale cover $W\to Y$ such that  $f^*\varrho({\rm Im}[\pi_1(W)\to \pi_1(Y)])$ is trivial. Denote by $g:W\to X$  the composition of $f$ with  $W\to Y$. Then $g$ is proper and  $g^*\varrho=1$. 
Let $\tau:\pi_1(X)\to {\rm GL}_N(\bC)$ be any   reductive representation such that $[\tau]\in \kC(\bC)$. 
Then $g^*\tau=1$ by \eqref{eq:kc}. It follows that $f^*\tau(\pi_1(Y))$ is finite. 
According to \cref{thm:Shafarevich1}, ${\rm sh}_\kC\circ f(Y)={\rm sh}_\kC(Z_1)$ is a point.   Since $Z$ is connected,  ${\rm sh}_\kC(Z)$ is a point.  
\end{proof}  
Let ${\rm sh}_\varrho:X\to {\rm Sh}_\varrho(X)$ be  ${\rm sh}_\kC:X\to {\rm Sh}_{\kC}(X)$.  The proposition is proved.  
\end{proof}
The conditions in \cref{thm:Sha2} that   $\varrho$ has infinite monodromy at infinity and that $X$ is smooth pose  significant practical limitations for further applications. However, we will remove this assumption in next theorem. We first prove an important lemma. 
\begin{lem}\label{lem:surjective}
	Let $f:X\to Y$ be a proper surjective morphism between connected (possible reducible) quasi-projective varieties $X$ and $Y$.    Assume that each fiber of $f$ is connected. Then $f_*: \pi_1(X)\to \pi_1(Y)$ is surjective. 
\end{lem}
\begin{proof}
	Let $\widetilde{Y}\to  Y$ be the universal covering of $Y$. Consider the fiber product $X':=X\times_Y\widetilde{Y}$.  
	\begin{equation*}
		\begin{tikzcd}
			X' \arrow[r, "\pi_1"] \arrow[d, "f'"]& X\arrow[d, "f"]\\
			\widetilde{	Y} \arrow[r, "\pi_2"] & Y
		\end{tikzcd}
	\end{equation*}Since $f$ is proper and each fiber of $f$ is connected, it follows that $f':X'\to \widetilde{Y}$ is proper, surjective and each fiber of $f'$ is  connected.  
	\begin{claim}
		$X'$ is connected.
	\end{claim}
	\begin{proof}
		Assume by contradiction that  	$X'$ is not connected. Then $X'=\sqcup_{\alpha\in I}X_\alpha$, where $X_\alpha$  are  connected components   of $X'$.  Since each fiber of $f'$ is connected, it follows that any fiber of $f'$ is contained in  some $X_\alpha$. This implies that $f'(X_\alpha)\cap f'(X_\beta)=\varnothing$ if $\alpha\neq \beta$. Since $f'$ is surjective, it follows  $\sqcup_{\alpha\in I}f'(X_\alpha)=\widetilde{	Y}$.  Note that $f'(X_\alpha)$ and $f'(\sqcup_{\beta\in I, \beta\neq\alpha}X_\beta)$ are both closed since  $f'$ is proper.  This contradicts with the connectedness of  $\widetilde{	Y}$. Hence  	$X'$ is connected. 
	\end{proof}
	We choose a base point $x\in X$ and $y\in Y$  such that $y=f(x)$.  Let $x'\in X'$ and $y'\in \widetilde{	Y}$ be such that $x=\pi_1(x')$, $y=\pi_2(y')$ and $y'=f'(x')$. Then for any element $\gamma\in \pi_1(Y,y)$, the  lift of $\gamma$ in $\widetilde{Y}$ starting at $y'$ will end at some $y''$ such that $y=\pi_2(y'')$. 
	Since $X'=X\times_Y\widetilde{Y}$, it follows that there exists a unique $x''\in f'^{-1}(y'')$ such that $x=\pi_1(x'')$. 
	Since $X'$ is connected, there exists a continuous path $\sigma:[0,1]\to X'$ such that $x'=\sigma(0)$ and $x''=\sigma(1)$. Consider the continuous path $f'\circ\sigma:[0,1]\to \widetilde{	Y}$. Then $y'=f'\circ\sigma(0)$ and $y''=f'\circ\sigma(1)$. It follows that $\gamma=[\pi_2\circ f'\circ \sigma]=[f\circ \pi_1\circ \sigma]$.  Note that $x=\pi_1\circ \sigma(0)=\pi_1\circ \sigma(1)$.    This proves that $f_*([\pi_1\circ \sigma])=\gamma$. Therefore, $f_*:\pi_1(X,x)\to \pi_1(Y,y)$ is surjective. The lemma is proved. 
\end{proof}

\begin{thm}\label{thm:Sha5} 
Let $X$ be a   quasi-projective  normal varieties, and let $\varrho:\pi_1(X)\to {\rm GL}_N(\bC)$ be a reductive representation. 
	 Then there exists a dominant holomorphic map  ${\rm sh}_{\varrho}:X\to {\rm Sh}_{\varrho}(X)$ to a complex  normal space ${\rm Sh}_{\varrho}(X)$ whose general fibers are connected such that  for any Zariski closed subset  
	 $Z\subset X$, the following properties are equivalent:
	 \begin{enumerate} [label=\rm (\alph*)]
	 	\item \label{item:contract}${\rm sh}_\varrho(Z)$ is a point;
	 	\item  \label{item:nonnormal}$ \varrho({\rm Im}[\pi_1(Z)\to \pi_1(X)])$ is finite;
	 	\item \label{item:normal} for any irreducible component $Z_1$ of $Z$, $ \varrho({\rm Im}[\pi_1(Z_1^{\rm norm})\to \pi_1(X)])$ is finite. 
	 \end{enumerate}    
\end{thm}  
\begin{proof}  
\noindent {\it Step 1: We first assume that $X$ is smooth.} 
By \cref{prop:infinity}, there exists a smooth partial compactification $X'$ of $X$ such that $\varrho$ has quasi-infinite monodromy at infinity with respect to $X'$. 
\begin{claim}\label{claim:extendpartial}
	There exists a finite morphism $\nu':\widehat{X}'\to X'$ such that
	\begin{enumerate}[label={\rm (\arabic*)}]
		\item $\widehat{X}'$ is a quasi-projective smooth variety; 
		\item  We denote $\nu_0:\widehat{X}\to X$, where $\widehat{X}=(\nu')^{-1}(X)$.
		Then $\nu_0^*\varrho$ extends to a reductive representation $\varrho':\pi_1(\widehat{X}')\to \GL_{N}(\bC)$ that has infinite monodromy at infinity.  
	\end{enumerate}
\end{claim}
\begin{proof}
	Let $D_1+\cdots+D_l=X'-X$ be the irreducible decomposition of the boundary, which is simple normal crossing. 
	Let $m_1,\ldots,m_l$ be the orders of the local monodromy around $D_1,\ldots,D_l$.
	By Kawamata's  covering lemma (see \cite[Theorem 17]{Kaw81}), we may take a finite covering $\nu':\widehat{X}'\to X'$ such that 
	\begin{enumerate}[label={\rm (\arabic*)}]
		\item $\hat{X}'$ is smooth,
		\item $(\nu')^{-1}(D_1+\cdots+D_l)$ is simple normal crossing,
		\item let $(\nu')^*D_i=\sum n_{ij}F_j$ be the irreducible decomposition, then $m_i \vert n_{ij}$ for any $i$ and $j$.
	\end{enumerate}
	Then $\nu_0^*\varrho:\pi_1(\hat{X})\to \mathrm{GL}_N(\bC)$ has trivial monodromy at each boundary components $F_{ij}$. 
	Hence this extends to $\varrho':\pi_1(\widehat{X}')\to \GL_{N}(\bC)$, which has infinite monodromy at infinity.
\end{proof}

We proceed by finding a finite morphism $h:Y'\to \widehat{X}'$ from a normal quasi-projective variety $Y'$ such that the composition $f:Y'\to X'$ of $\widehat{X}'\to X'$  and $Y'\to \widehat{X}'$ is a Galois cover with Galois group $G$ (cf. \cite[\S 1.3]{CDY22}). By \cref{claim:extendpartial,lem:infinite pull}, $h^*\varrho':\pi_1(Y')\to \GL_{N}(\bC)$ also has infinite monodromy at infinity. Consequently, we can apply \cref{thm:Sha2} to deduce the existence of a proper holomorphic fibration ${\rm sh}_{h^*\varrho'}:Y'\to {\rm Sh}_{h^*\varrho'}(Y')$  such that for any closed subvariety $Z$ of $Y'$, ${\rm sh}_{h^*\varrho'}(Z)$ is a point if and only if $h^*\varrho'({\rm Im}[\pi_1(Z^{\rm norm})\to \pi_1(Y')])$ is finite. 
	\begin{claim}\label{claim:equivariant}
	The Galois group	$G$ acts analytically on ${\rm Sh}_{h^*\varrho'}(Y')$ such that ${\rm sh}_{h^*\varrho'}$ is $G$-equivariant. 
\end{claim} 
\begin{proof}
	Take any $y\in  {\rm Sh}_{h^*\varrho'}(Y')$ and any $g\in G$. Since ${\rm sh}_{h^*\varrho'}$ is surjective and proper,  the fiber ${\rm sh}_{h^*\varrho'}^{-1}(y)$ is thus non-empty and compact. 	Let $Z$ be an  irreducible component of the fiber ${\rm sh}_{h^*\varrho'}^{-1}(y)$. 
	Then $h^*\varrho'({\rm Im}[\pi_1(Z^{\rm norm})\to \pi_1(Y')])$ is finite, implying that $h^*\varrho'({\rm Im}[\pi_1((g. Z)^{\rm norm})\to \pi_1(Y')])$  is also finite.  Consequently, there exists a point $y'\in {\rm Sh}_{h^*\varrho'}(Y')$ such that ${\rm sh}_{h^*\varrho'}(g. Z)=y'$. Since each fiber of ${\rm sh}_{h^*\varrho'}$ is connected,  for any other   irreducible component $Z'$ of  ${\rm sh}_{h^*\varrho'}^{-1}(y)$,  we have ${\rm sh}_{h^*\varrho'}(g. Z')=y'$. Consequently, it follows that $g$ maps each fiber of ${\rm sh}_{h^*\varrho'}$ to another fiber.
	
	We consider $g$ as an analytic automorphism of $Y'$.   For the holomorphic map $ {\rm sh}_{h^*\varrho'} \circ g: Y'\to {\rm Sh}_{h^*\varrho'}(Y')$, since it contracts each fiber of  ${\rm sh}_{h^*\varrho'}:Y'\to {\rm Sh}_{h^*\varrho'}(Y')$ to a point, it induces a holomorphic map $\tilde{g}:{\rm Sh}_{h^*\varrho'}(Y')\to {\rm Sh}_{h^*\varrho'}(Y')$ such that we have   the following commutative diagram: 
	\begin{equation}\label{eq:contract}
		\begin{tikzcd}
			Y'\arrow[r,"g"] \arrow[d, "{\rm sh}_{h^*\varrho'}"] & Y'\arrow[d, "{\rm sh}_{h^*\varrho'}"]\\
			{\rm Sh}_{h^*\varrho'}(Y') \arrow[r, "\tilde{g}"] &{\rm Sh}_{h^*\varrho'}(Y')
		\end{tikzcd}
	\end{equation}   Let us define the holomorphic map $\tilde{g}:{\rm Sh}_{h^*\varrho'}(Y')\to {\rm Sh}_{h^*\varrho'}(Y')$ to be the action of $g\in G$ on ${\rm Sh}_{h^*\varrho'}(Y')$. Based on  \eqref{eq:contract}, it is  clear  that ${\rm sh}_{h^*\varrho'}$ is $G$-equivariant.   Therefore, the claim is proven.
\end{proof}
	Note that $X':=Y'/G$.  The quotient of ${\rm Sh}_{h^*\varrho'}(Y')$  by $G$, resulting in   a complex normal space  denoted by $Q$  (cf. \cite{Car60}).  
	Then  ${\rm sh}_{h^*\varrho'}$ induces a proper holomorphic fibration
	$
	c':X'\to Q.
	$ 
	Consider the restriction $c:=c'|_{X}$. 
	\begin{equation}\label{dia:noninfinite}
		\begin{tikzcd}
			Y\arrow[r, "f_0"]  \arrow[d, hook] & X \arrow[d, hook]\arrow[dd, bend left=37, "c"]\\
			Y'\arrow[r, "f"] \arrow[d, " {\rm sh}_{h^*\varrho'}"] & X' \arrow[d, "c'"]\\
			{\rm Sh}_{h^*\varrho'}(Y') \arrow[r]  & Q  
		\end{tikzcd}
	\end{equation}
	\begin{claim}\label{claim:same contract}
		For any closed subvariety $Z$ of $X$, $c(Z)$ is a point if and only if $\varrho({\rm Im}[\pi_1(Z^{\rm norm})\to \pi_1(Y')])$ is finite. 
	\end{claim}
	\begin{proof}
		Let $Y:=f^{-1}(X)$ and $f_0:=f|_{Y}$. Note that $f_0:Y\to X$ is a Galois cover with Galois group $G$.  We have  $h^*\varrho'|_{\pi_1(Y)}=f_0^*\varrho$.    Now, consider any closed subvariety $Z$ of $X$. There exists an irreducible closed subvariety  $W$ of $Y$ such that $f_0(W)=Z$. Let $\overline{W}$ be the closure of $W$ in $Y'$, which is   an  irreducible closed subvariety of $Y'$. 
		
		Observe that $c(Z)$ is a point if and only if ${\rm sh}_{h^*\varrho'}(\overline{W})$ is a point,   which is equivalent to 		$h^*\varrho'({\rm Im}[\pi_1(\overline{W}^{\rm norm})\to \pi_1(Y')])$ being finite by \Cref{thm:Sha2}. Furthermore, this is equivalent to $f_0^*\varrho({\rm Im}[\pi_1({W}^{\rm norm})\to \pi_1(Y)])$  being finite since $h^*\varrho'|_{\pi_1(Y)}=f_0^*\varrho$. Since ${\rm Im}[\pi_1(W^{\rm norm})\to \pi_1(Z^{\rm norm})]$ is a finite index subgroup of $\pi_1(Z^{\rm norm})$, the above condition is equivalent to $\varrho({\rm Im}[\pi_1({Z}^{\rm norm})\to \pi_1(X)])$   being finite. 
	\end{proof}
Let $ {\rm sh}_\varrho:=c$ and $ {\rm Sh}_\varrho(X):=Q$. This is  our construction of the Shafarevich morphism of $\varrho$.

\medspace 

\noindent {\it Step 2: We does not assume that $X$ is smooth.} We take a desingularization $\nu_1:X_1\to X$. Then $\nu_1^*\varrho:\pi_1(X_1)\to \GL_{N}(\bC)$ is also a reductive representation.   Based on Step 1,  the Shafarevich morphism ${\rm sh}_{\nu_1^*\varrho}:X_1\to {\rm Sh}_{\nu_1^*\varrho}(X_1)$ exists.   Let $Z$ be an irreducible component of a fiber of $\nu_1$. Then $\nu_1^*\varrho(\pi_1(Z))=\{1\}$. It follows that ${\rm sh}_{\nu_1^*\varrho}(Z)$ is a point. Note that each fiber of $\nu_1$ is connected as $X$ is normal. It follows that each fiber of $\nu_1$ is contracted to a point by    ${\rm sh}_{\nu_1	^*\varrho}$.  Therefore, by the universal property of the Stein factorization, there exists a dominant holomorphic map ${\rm sh}_{\varrho}:X\to {\rm Sh}_{\nu_1^*\varrho}(X_1) $ with connected general fibers such that we have the following commutative diagram:
\begin{equation}\label{eq:factor2}
	\begin{tikzcd}
		X_1\arrow[d,"\nu_1"'] \arrow[dr,"{\rm sh}_{\nu_1^*\varrho}"'']& \\
		X\arrow[r,"{\rm sh}_{\varrho}"]&	{\rm Sh}_{\nu_1^*\varrho}(X_1) 
	\end{tikzcd}
\end{equation}
\begin{claim} \label{claim:same contract2}
	For any closed subvariety $Z\subset X$, ${\rm sh}_{\varrho}(Z)$ is a point if and only if $\varrho({\rm Im}[\pi_1(Z^{\rm norm})\to \pi_1(X)])$ is finite. 
\end{claim}
\begin{proof}
	Let us choose an irreducible component $W$ of $\nu_1^{-1}(Z)$ which is surjective onto $Z$.   Since ${\rm Im}[\pi_1(W^{\rm norm})\to \pi_1(Z^{\rm norm})]$ is a finite index subgroup of $\pi_1(Z^{\rm norm})$, and $(\nu_1)_*:\pi_1(X_1)\to \pi_1(X)$ is surjective, it follows that $\varrho({\rm Im}[\pi_1(Z^{\rm norm})\to \pi_1(X)])$ is finite if and only  if $\nu_1^*\varrho({\rm Im}[\pi_1(W^{\rm norm})\to \pi_1(X_1)])$ is finite. 
	
	\medspace 
	
	\noindent {\it Proof of $\Rightarrow$:}  Note that ${\rm sh}_{\nu_1^*\varrho}(W)$ is a point and thus $\nu_1^*\varrho({\rm Im}[\pi_1(W^{\rm norm})\to \pi_1(Y)])$ is finite by \cref{claim:same contract}. Hence  $\varrho({\rm Im}[\pi_1(Z^{\rm norm})\to \pi_1(X)])$ is finite.
	
	\medspace
	
	\noindent {\it Proof of $\Leftarrow$:} Note that $\nu_1^*\varrho({\rm Im}[\pi_1(W^{\rm norm})\to \pi_1(X_1)])$ is finite.  
	Therefore, ${\rm sh}_{\nu_1^*\varrho}(W)$ is a point and thus $${\rm sh}_{\varrho}(Z)={\rm sh}_{\varrho}\circ\nu_1(W)={\rm sh}_{\nu_1^*\varrho}(W)$$ is a point by \cref{claim:same contract}. 
\end{proof}
Let us write ${\rm Sh}_{\varrho}(X):={\rm Sh}_{\nu_1^*\varrho}(X_1) $. Then ${\rm sh}_{\varrho}:X\to 	{\rm Sh}_{\varrho}(X)$ is the  construction of the Shafarevich morphism associated with $\varrho:\pi_1(X)\to \GL_{N}(\bC)$.   

\medspace 
 
{\it Step 3. We prove the last assertion on the three equivalent properties for ${\rm sh}_\varrho$.} 
Note that the equivalence between \Cref{item:normal} and \Cref{item:contract} is proved in \Cref{claim:same contract2}. The implication of \Cref{item:nonnormal}  to  \Cref{item:normal}  is obvious. 
We thus only need to  prove that \Cref{item:contract} to \Cref{item:nonnormal}.
Let $Z\subset X$ be a connected Zariski closed subset. 

We first prove the following claim.

\begin{claim}\label{lem:20230803}
	Let $g:Y\to X$ be a dominant morphism between normal quasi-projective varieties and $\varrho:\pi_1(X)\to \mathrm{GL}_n(\bC)$ be a reductive representation. Then   for every connected Zariski closed set $Z\subset Y$, ${\rm sh}_{\varrho}\circ g(Z)$ is a point if and only if $	{\rm sh}_{g^*\varrho}(Z)$ is a point.
\end{claim}
 \begin{proof}[Proof of \cref{lem:20230803}]
	It is enough to prove the case that $Z$ is irreducible.
	Hence we need to show that for every closed subvariety $Z\subset Y$, $g^*\varrho(\mathrm{Im}[\pi_1(Z^{\mathrm{norm}})\to \pi_1(Y)])$ is finite if and only if ${\rm sh}_{\varrho}\circ g(Z)$ is a point.
	So first suppose $g^*\varrho(\mathrm{Im}[\pi_1(Z^{\mathrm{norm}})\to \pi_1(Y)])$ is finite.
	Let $V\subset X$ be the Zariski closure of $g(Z)$.
	Then the induced map $Z\to V$ is dominant, hence induces $Z^{\mathrm{norm}}\to V^{\mathrm{norm}}$.
	Then ${\rm Im}[\pi_1(Z^{\mathrm{norm}})\to \pi_1(V^{\mathrm{norm}})]$ is a finite index subgroup of $\pi_1(V^{\mathrm{norm}})$. 
	Hence $\varrho(\mathrm{Im}[\pi_1(V^{\mathrm{norm}})\to \pi_1(X)])$ is finite.
	This shows ${\rm sh}_{\varrho}(V)$ is a point.
	Hence ${\rm sh}_{\varrho}\circ g(Z)$ is a point.
	Conversely, assume ${\rm sh}_{\varrho}\circ g(Z)$ is a point.
	Then ${\rm sh}_{\varrho}(V)$ is a point.
	Thus $\varrho(\mathrm{Im}[\pi_1(V^{\mathrm{norm}})\to \pi_1(X)])$ is finite.
	Hence $g^*\varrho(\mathrm{Im}[\pi_1(Z^{\mathrm{norm}})\to \pi_1(Y)])$ is finite.
\end{proof}

\medspace

\noindent {\it Step 3-1. \Cref{item:contract} $\Rightarrow$ \Cref{item:nonnormal} when $X$ is smooth.} 
Since $\varrho(\pi_1(X))$ is residually finite by Malcev's theorem, we can find a finite \'etale cover $\nu:\widehat{X}\to X$  such that $\nu^*\varrho(\pi_1(\widehat{X}))$  is torsion free. 
By \cref{prop:infinity}, there exists a partial smooth compactification $\widehat{X}'$ such that $\nu^*\varrho$ extends to a reductive representation $\sigma:\pi_1(\widehat{X}')\to \GL_{N}(\bC)$ with infinite monodromy at infinity. 
Let $\widehat{Z}\subset \widehat{X}$ be a connected component of $\nu^{-1}(Z)$.
Then $\widehat{Z}\to Z$ is finite \'etale.
Let $\widehat{Z}'\subset \widehat{X}'$ be the Zariski closure.
Then $\widehat{Z}'$ is connected.

Now by \cref{lem:20230803}   applied to $\nu:\widehat{X}\to X$, we conclude that $\mathrm{sh}_{\nu^*\varrho}(\widehat{Z})$ is a point.
Hence by \cref{lem:20230803}  applied to $\iota:\widehat{X}\hookrightarrow \widehat{X}'$, we conclude that $\mathrm{sh}_{\sigma}(\widehat{Z})$ is a point.
Hence $\mathrm{sh}_{\sigma}(\widehat{Z}')$ is a point.
Thus by \cref{thm:Sha2},  $\sigma({\rm }[\pi_1(\widehat{Z}')\to \pi_1(\widehat{X}')])$ is finite. 
Note that 
\begin{align*}
	\nu^*\varrho({\rm }[\pi_1(\widehat{Z})\to \pi_1(\widehat{X})])=\iota^*\sigma({\rm }[\pi_1(\widehat{Z})\to \pi_1( \widehat{X})])\\
	= \sigma({\rm }[\pi_1(\widehat{Z})\to \pi_1( \widehat{X}')])\subset \sigma({\rm }[\pi_1(\widehat{Z}')\to \pi_1(\widehat{X}')]).
\end{align*}  
It follows that $\nu^*\varrho({\rm }[\pi_1(\widehat{Z})\to \pi_1(\widehat{X})])$ is finite. 
Since the image $\pi_1(\widehat{Z})\to\pi_1(Z)$ has finite index, $\varrho({\rm }[\pi_1(Z)\to \pi_1(X)])$ is finite.

\medspace

\noindent {\it Step 3-2.  \Cref{item:contract} $\Rightarrow$ \Cref{item:nonnormal} in the general case.} 
Let $\nu:X_1\to X$ be a desingularization. 
Let $Z_1:=\nu^{-1}(Z)$, which might be reducible.  
Since $X$ is normal, each fiber of $\nu$ is connected.
Hence $Z_1$ is connected, and $\nu|_{Z_1}:Z_1\to Z$ is a proper surjective morphism between connected (possibly reducible) quasi-projective varieties such that each fiber is connected. 
By \cref{lem:surjective}, we have the surjectivity of $\pi_1(Z_1)\twoheadrightarrow \pi_1(Z)$.
Hence $ \varrho({\rm Im}[\pi_1(Z)\to \pi_1(X)])$ is finite if and only if  $ \nu^*\varrho({\rm Im}[\pi_1(Z_1)\to \pi_1(X_1)])$ is finite.
By \eqref{eq:factor2},    ${\rm sh}_{\nu^*\varrho}={\rm sh}_\varrho\circ\nu$. 
It  then follows that ${\rm sh}_{\nu^*\varrho}(Z_1)$ is a point.
Hence by step 3-1, $ \nu^*\varrho({\rm Im}[\pi_1(Z_1)\to \pi_1(X_1)])$ is finite.
Hence $ \varrho({\rm Im}[\pi_1(Z)\to \pi_1(X)])$ is finite.
The theorem is proved.   
\end{proof}

\subsection{On the algebraicity  of the Shafarevich morphism}
In \cref{thm:Sha5}, when $X$ is compact, we proved that the image ${\rm Sh}_\varrho(X)$ is projective. In general, as  mentioned in \cref{rem:Griffiths}, we propose the following conjecture. 
\begin{conjecture}[Algebraicity of Shafarevich morphism]\label{conj:algebraic}
	Let $X$, $\varrho$ and ${\rm sh}_\varrho:X\to {\rm Sh}_\varrho(X)$ be as in \cref{thm:Sha5}. Then  ${\rm Sh}_\varrho(X)$ is a quasi-projective normal variety and ${\rm sh}_\varrho:X\to {\rm Sh}_\varrho(X)$  is an algebraic morphism. 
\end{conjecture}
This conjecture seems to be a difficult problem, with the special case when $\varrho$ arises from a $\bZ$-VHS known as a long-standing Griffiths conjecture.  In this paper, we provide confirmation of such expectations at the function field level. 

We first recall the definition of  \emph{(bi)meromorphic maps} of complex spaces $X$ and $Y$ in the sense of Remmert.   
\begin{dfn}[Meromorphic map]
	A meromorphic map $f:X\dashrightarrow Y$ of complex spaces $X$ and $Y$ is 
	a holomorphic map $f^\circ:X^\circ\to Y $
defined on a dense open set $X^\circ\subset X$ such that 
	\begin{thmlist}
		\item $ X\backslash X^\circ$ 
		is a nowhere-dense analytic subset of $X$;
		\item  the closure $ {\Gamma_f}$ 
		of the graph   $\Gamma_{f^\circ}$ of $f^\circ$
		in  $X\times Y$
		is an analytic subset of  $X\times Y$;
		\item    the projection $\Gamma_f\to X$
		is a proper mapping. 
	\end{thmlist}
We say $\Gamma_f$ the \emph{graph} of the meromorphic map $f$.
\end{dfn}
\begin{dfn}[Bimeromorphic map]
	A meromorphic map $f:X\dashrightarrow Y$ of complex spaces  $X$ and $Y$ is \emph{bimeromorphic}  if the projection of $\Gamma_f$ to $Y$ is proper, and off some nowhere-dense analytic subset of $Y$, this projection  is biholomorphic. 
\end{dfn} 
\begin{lem}\label{lem:bime}
	Let $f:X\to Y$ be a proper surjective holomorphic fibration between irreducible complex normal spaces $X$ and $Y$ of the same dimension. Then $f$ is bimeromorphic. 
 \end{lem}
\begin{proof}
Since each fiber of $f$ is connected, and $\dim X=\dim Y$, there exists a nowhere-dense analytic subset  $A$ of $Y$   such that $f$ is a bijection outside $A$. Since $X$ and $Y$ are normal, and $f|_{X\backslash f^{-1}(A)}:X\backslash f^{-1}(A)\to Y\backslash A$ is a proper holomorphic fibration, it follows that $f|_{X\backslash f^{-1}(A)}$  is a biholomorphism. 	Consider the graph $\Gamma_f\subset X\times Y$ of $f$. The projection $p: \Gamma_f\to Y$ is proper as $f$ is proper, and $p$ a biholomophism over $Y\backslash A$.   Therefore, $f$ is bimeromorphic. 
\end{proof}

\begin{thm}\label{lem:bimeromorphic2}
	Let $X$ be a non-compact smooth quasi-projective variety and $\varrho:\pi_1(X)\to\GL_{N}(\bC)$ be a reductive representation.
	Then   there exists 
	\begin{enumerate}[label=(\alph*)]
		\item  a proper bimeromorphic morphism  $\sigma:S\to {\rm Sh}_{\varrho}(X)$ from a smooth  quasi-projective   variety $S$;
		\item  a proper birational morphism  $\mu:\tilde{X}\to X$ from a smooth  quasi-projective   variety $\tilde{X}$; 
		\item an algebraic morphism $f:\tilde{X}\to S$ with general fibers connected;  
	\end{enumerate} 
	such that we have the following commutative diagram: 
	\[
	\begin{tikzcd}
		\tilde{X} \arrow[r, "\mu"]\arrow[d, "f"]& X \arrow[d, "{\rm sh}_\varrho"]&\\
		S                   \arrow[r, "\sigma"]                             & {\rm Sh}_\varrho(X) 
	\end{tikzcd}
	\]
	Here ${\rm sh}_\varrho:X\to {\rm Sh}_\varrho(X)$ is the Shafarevich morphism of $\varrho$ constructed in \cref{thm:Sha5}.
\end{thm}

\begin{proof}
We apply \cite[Corollary 3.5 \& Remark 4.1.1]{Kol93} to conclude that after replacing $X$ by its birational model, there are a normal projective variety $Y$ and a dominant morphism $f:X\to Y$ so that there are at most countably many  closed subvarieties $Z_i\subsetneqq X$ so that for  every    closed subvariety $W\subset X$ such that   $W\not\subset \cup Z_i$, the image $\varrho\big({\rm Im}[\pi_1(W^{\rm norm})\to \pi_1(X))] \big)$ is finite if and only if $f(W)$ is a point. 

\begin{claim}\label{claim:20230620}
For the general fiber $F$ of $f:X\to Y$,
$\varrho({\rm Im}[\pi_1(F)\to \pi_1(X)])$ is finite.
\end{claim}

\begin{proof}
There exists a non-empty Zariski open $Y^o\subset Y$ such that the induced $X^o\to Y^o$ is locally trivial fibration of $C^{\infty}$-manifolds.
Hence ${\rm Im}[\pi_1(F)\to \pi_1(X)]$ are all the same for fibers $F$ over the points of $Y^o$.
Now $\varrho({\rm Im}[\pi_1(F)\to \pi_1(X)])$ is finite for very generic fiber over $Y$, hence the same is true for general fibers.
\end{proof}

We may take a finite {\'e}tale cover $X'\to X$ such that the induced representation $\varrho':\pi_1(X')\to\mathrm{GL}_N(\bC)$ has torsion free image.
Let $\overline{X'}$ be a partial compactification such that $\varrho'$ extends to $\bar{\varrho'}:\pi_1(\overline{X'})\to\mathrm{GL}_N(\bC)$ and $\bar{\varrho'}$ has infinite monodromy at infinity (cf. Proposition \ref{prop:infinity}).
By the proof of \cref{thm:Sha5}, we may assume that $X'\to X$ extends to a finite map $\overline{X'}\to\overline{X}$, where $\overline{X}$ is a partial compactification of $X$ so that the following commutative diagram exists:
\begin{equation}\label{eq:20230616}
		\begin{tikzcd}
			\overline{X'}\arrow[r,""] \arrow[d, "{\rm sh}_{\bar{\varrho'}}"] & \overline{X}\arrow[d, "{\overline{\rm sh}_{\varrho}}"]\\
			{\rm Sh}_{\varrho'}(X) \arrow[r, ""] &{\rm Sh}_{\varrho}(X)
		\end{tikzcd}
	\end{equation}  
Here the horizontal maps are finite and vertical maps are proper.
Since we are assuming that $Y$ is projective, we may assume that $f:X\to Y$ extends to $\bar{f}:\overline{X}\to Y$.

\begin{claim}\label{lem:20230616}
There is a commutative diagram
	\[
		\begin{tikzcd}
			\hat{X}\arrow[r, "\mu"] \arrow[d, "\hat{f}"] & \overline{X} \arrow[d, "\bar{f}"] \\
			\hat{Y} \arrow[r, "\nu"] & Y
		\end{tikzcd}
	\]
	where
	\begin{enumerate}[label=(\alph*)]
		\item \(\mu\) is a proper birational morphism from a quasi-projective variety $\widehat{X}$;
		\item $\nu$ is a birational, not necessarily proper morphism from a smooth quasi-projective variety $\widehat{Y}$,
		\item
		$\hat{f}$ is surjective;
	\end{enumerate} 
	 such that the following property holds:
	 Let $g:Z\to \hat{X}$ be a proper morphism from a quasi-projective manifold $Z$.
If $\hat{f}\circ g$ contracts to a point, then $\overline{{\rm sh}_{\varrho}}\circ \mu\circ g:Z\to {\rm Sh}_{\varrho}(X)$ contracts to a point.
\end{claim}

\begin{proof}
Let $\bar{f}':\overline{X'}\to Y'$ be the quasi-Stein factorization of the composite of $\overline{X'}\to \overline{X}\to Y$.
By \cref{claim:20230620}, for the general fiber $F$ of $f':X'\to Y'$, $\varrho'({\rm Im}[\pi_1(F)\to \pi_1(X')])$ is finite, hence trivial.
By \cite[Lemma 2.3]{CDY22}, there is a commutative diagram
	\[
		\begin{tikzcd}
			(\overline{X'})' \arrow[r, "\mu'"] \arrow[d, "f''"] & \overline{X'} \arrow[d, "\bar{f}'"] \\
			Y'' \arrow[r, "\nu'"] & Y'
		\end{tikzcd}
	\]
	where
	\begin{enumerate}[label=(\alph*)]
		\item \(\mu'\) is a proper birational morphism;
		\item  \(\nu'\) is a birational, not necessarily proper morphism;
		\item $Y''$ is smooth,
	\end{enumerate} 
	 and a representation \(\tau : \pi_{1}(Y'') \to \mathrm{GL}_N(\bC)\) such that $(f'')^*\tau=(\mu')^*\bar{\varrho'}$. 
 \begin{claim}\label{claim:proper3}
	$f''$ is a  proper morphism with each fiber connected. 
\end{claim}
\begin{proof}
	If $f''$ is not proper, 	we can find a partial projective compactification $	(\overline{X'})''\supsetneq 	(\overline{X'})'$ such that $f''$ extends to a proper morphism $f''':(\overline{X'})''\to Y$. Then $f'''^*\tau$ extends the representation $f''^*\tau=(\mu')^*\bar{\varrho'}$.  It is worth noting that  $(\mu')^*\bar{\varrho'}$ also has infinite monodromy at infinity by \cref{lem:infinite pull}. A contradiction is obtained, and thus $f''$ is a proper morphism. 
	
	Since general fibers of $f''$ is connected and $Y''$ is smooth, it follows that each fiber of $f$ is connected. The claim is proved. 
\end{proof}
We have the following commutative diagram:
$$
\begin{tikzcd}
			(\overline{X'})' \arrow[r, ""] \arrow[d, "f''"] & \overline{X} \arrow[d, "\bar{f}"] \\
			Y'' \arrow[r, ""] & Y
		\end{tikzcd}
$$
By applying Hironaka-Raynaud-Gruson's flattening theorem, we can take a birational modification $\tilde{Y}\to Y$ such that for the main component $\tilde{Y}''$ of $Y''\times_{Y}\tilde{Y}$, $\tilde{Y}''\to\tilde{Y}$ is flat.
We may assume that $\tilde{Y}$ is smooth. 
Then we get
\setlength{\perspective}{2pt}
\[\begin{tikzcd}[row sep={55,between origins}, column sep={55,between origins}]
	&[-\perspective] ((\overline{X'})'\times_{Y''}\tilde{Y}'')_{\mathrm{main}} \arrow[rr, "\varphi"]\arrow[dd, "\tilde{f}''"   {yshift=22pt}  ]\arrow[dl, "\tilde{\mu}"'] &[\perspective] &[-\perspective] (\overline{X}\times_Y\tilde{Y})_{\mathrm{main}} \arrow[dd, "\tilde{f}" {yshift=22pt} ]\arrow[dl, "\mu"] \\[-\perspective]
	(\overline{X'})'  \ar[crossing over]{rr}  \arrow[dd, "f''"] & & \overline{X}
   \\[\perspective]
	& \widetilde{Y}''  \ar{rr} \ar{dl} & &  \widetilde{Y} \ar{dl} \\[-\perspective]
Y''\ar{rr} && Y\ar[from=uu,crossing over, "\bar{f}"'  {yshift=14pt}] &{\rm Sh}_\varrho(X) \ar[from=uul,crossing over]
\end{tikzcd}\]

%$$
%\begin{tikzcd}
%((\overline{X'})'\times_{Y''}\tilde{Y}'')_{\mathrm{main}}
%			\arrow[r, "\varphi"] \arrow[d, "\tilde{f}''"] &
% (\overline{X}\times_Y\tilde{Y})_{\mathrm{main}}
%			\arrow[r, ""] \arrow[d, "\tilde{f}"] & \overline{X} \arrow[d, "\bar{f}"] \\
%		\tilde{Y}''	\arrow[r, ""] & \tilde{Y} \arrow[r, ""] & Y
%		\end{tikzcd}
%$$
Note that $\tilde{f}'':((\overline{X'})'\times_{Y''}\tilde{Y}'')_{\mathrm{main}}\to \widetilde{Y}''$ is proper.
Indeed $(\overline{X'})'\times_{Y''}\tilde{Y}''\to \widetilde{Y}''$ is proper, as it is a base change of the proper map $f''$.
Since $((\overline{X'})'\times_{Y''}\tilde{Y}'')_{\mathrm{main}}\hookrightarrow (\overline{X'})'\times_{Y''}\tilde{Y}''$ is a closed immersion, $\tilde{f}''$ is proper. 
Also $\varphi$ is proper.
Indeed since $(\overline{X'})'\to \overline{X}$ is proper, $(\overline{X'})'\times_{Y''}(Y''\times_Y\tilde{Y})\to \overline{X}\times_Y\tilde{Y}$ is proper.
Since $((\overline{X'})'\times_{Y''}\tilde{Y}'')_{\mathrm{main}}\hookrightarrow (\overline{X'})'\times_{Y''}(Y''\times_Y\tilde{Y})$ is a closed immersion, $((\overline{X'})'\times_{Y''}\tilde{Y}'')_{\mathrm{main}}\to \overline{X}\times_Y\tilde{Y}$ is proper.
Hence $\varphi$ is proper.
We let $\hat{Y}\subset \tilde{Y}$ to be the image, which is a smooth quasi-projective variety as $\tilde{Y}''\to\tilde{Y}$ is an open map.
Since $\varphi$ is surjective, we have $\tilde{f}( (\overline{X}\times_Y\tilde{Y})_{\mathrm{main}})\subset \hat{Y}$.
Since $\tilde{f}''$ and $\tilde{Y}''\to\hat{Y}$ are surjective, $\tilde{f}:\hat{X}= (\overline{X}\times_Y\tilde{Y})_{\mathrm{main}}\to \hat{Y}$ is surjective.
Now we set $\hat{X}=(\overline{X}\times_Y\tilde{Y})_{\mathrm{main}}$. 
Then we have the properties (a)-(c).

	 Now let $g:Z\to \hat{X}$ be a proper morphism from a quasi-projective manifold $Z$ such that $\hat{f}\circ g$ contracts to a point.
Let $g':Z'\to ((\overline{X'})'\times_{Y''}\tilde{Y}'')_{\mathrm{main}}$ be a proper map which induces a surjective map $Z'\to Z$.
Since $\tilde{Y}''\to\hat{Y}$ is flat, $\tilde{f}''\circ g'$ contracts to a point.
Hence $(\tilde{\mu}\circ g')^*(\mu')^*\varrho'$ is trivial, where $\tilde{\mu}:((\overline{X'})'\times_{Y''}\tilde{Y}'')_{\mathrm{main}}\to (\overline{X'})'$ is the natural map.
Hence $\mathrm{sh}_{\bar{\varrho'}}\circ \mu'\circ \tilde{\mu}\circ g':Z'\to \mathrm{Sh}_{\varrho'}(X')$ contracts to a point.
Hence $\overline{{\rm sh}_{\varrho}}\circ \mu\circ g:Z\to {\rm Sh}_{\varrho}(X)$ contracts to a point.
\end{proof}

\Cref{lem:20230616} yields $\sigma:\hat{Y}\to \mathrm{Sh}_{\varrho}(X)$.
Indeed, we consider $\phi=(\hat{f},\overline{\mathrm{sh}_{\varrho}}\circ \mu):\hat{X}\to \hat{Y}\times \mathrm{Sh}_{\varrho}(X)$.
Then $\phi$ is proper since $\overline{\mathrm{sh}_{\varrho}}\circ \mu$ is proper. 
Hence $\Gamma=\phi(\hat{X})\subset \hat{Y}\times \mathrm{Sh}_{\varrho}(X)$ is an analytic subset. 
\begin{claim}
The induced map $t:\Gamma\to\hat{Y}$  is an isomorphism of analytic spaces. 
\end{claim}
\begin{proof}
	By \cref{lem:20230616}, every fiber of $t:\Gamma\to\hat{Y}$ consists of finite points.
	Since $\hat{f}:\hat{X}\to \hat{Y}$ has general connected fibers, $\Gamma\to\hat{Y}$ is injective over a non-empty Zariski open set of $\hat{Y}$. 
	We prove this. 
	Let $E\subsetneqq \hat{Y}$ be a proper Zariski closed set such that $\Gamma\backslash t^{-1}(E)\to \hat{Y}\backslash E$ is injective. 
	To show that $t:\Gamma\to \hat{Y}$ is injective, we assume contrary that there exists $y\in \hat{Y}$ such that $t^{-1}(y)=\{x_1,x_2,\ldots,x_k\}$ consists $k$ points with $k\geq 2$.
	Let $U\subset\hat{Y}$ be a connected open neighbourhood of $y$.
	Then by \cite[p. 166]{GR84}, $t^{-1}(U)\backslash t^{-1}(E)\to U\backslash E$ is an isomorphism.
	Since $U$ is normal, $U\backslash E$ is connected.
	Hence $t^{-1}(U)\backslash t^{-1}(E)$ is connected.
	In particular, $t^{-1}(U)$ is connected.
	On the other hand, we may take disjoint open neighbourhoods $V_1,\ldots,V_k$ of $x_1,\ldots,x_k$.
	We may assume that $\partial V_1,\ldots,\partial V_k$ are compact and $y\not\in t(\partial V_1\cup\cdots\cup\partial V_k)$.
	Then we may take a connected open neighbourhood $U$ of $y$ such that $\overline{U}\cap t(\partial V_1\cup\cdots\cup\partial V_k)=\emptyset$.
	Then for each $i=1,\ldots,k$, $t^{-1}(U)\cap V_i(\ni x_i)$ is non-empty and $\overline{t^{-1}(U)\cap V_i}\subset V_i$.
	Hence $t^{-1}(U)$ is not connected.
	This is a contradiction.
	Hence $t:\Gamma\to\hat{Y}$ is injective. Note that   $t$ is   surjective, for $\hat{f}$ is surjective (cf. \cref{lem:20230616}). 
	Hence $t$ is an isomorphism by \cite[p. 166]{GR84}. 
\end{proof} 
 Hence we obtain the analytic map $\sigma:\hat{Y}\to \mathrm{Sh}_{\varrho}(X)$ whose graph is $\Gamma$.

Since $\mu$ and $\overline{{\rm sh}_\varrho}$ are proper, and $\hat{f}$ is surjective, $\sigma$ is proper. Moreover, $\sigma$ has connected fibers since $\overline{\rm sh_\varrho}\circ\mu$ is a proper holomorphic fibration. 
We claim that $\sigma:\hat{Y}\to {\rm Sh}_\varrho(X) $ is bimeromorphic.
Indeed, since $\sigma:\hat{Y}\to {\rm Sh}_\varrho(X) $ is a proper holomorphic fibration, it suffices to prove that $\dim \hat{Y}=\dim {\rm Sh}_\varrho(X)$ by \cref{lem:bime}.  
Assume, for the sake of contradiction, we have $\dim \hat{Y}>\dim {\rm Sh}_\varrho(X)$.  
Since $\hat{Y}$ is quasi-projective and $\sigma$ is proper, each fiber of $\sigma$ is thus algebraic subset of $\hat{Y}$.  
Let $F$ be a generic fiber of $\sigma$.
Then $\nu(F)$ is not a point.
Let $Z\subset X$ satisfies that the Zariski closure of $f(Z)$ is equal to that of $\nu(F)$.
Then $\mathrm{sh}_{\varrho}(Z)$ is a point.
Therefore, $\varrho({\rm Im}[\pi_1(Z)\to \pi_1(X)])$ is finite. 
On the other hand, since $Z$ is not contracted by $f$, $\varrho({\rm Im}[\pi_1(Z)\to \pi_1(X)])$ is infinite.
This is a contradiction.
Hence $\dim \hat{Y}=\dim {\rm Sh}_\varrho(X)$, therefore $\sigma:\hat{Y}\to {\rm Sh}_\varrho(X) $ is bimeromorphic.
\end{proof}

It is worth noting that while $\sigma:S\to {\rm Sh}_\varrho(X)$ is a proper bimeromorphic map,  showing that ${\rm Sh}\varrho(X)$ is quasi-projective requires understanding the structure of the Shafarevich morphism, as illustrated  in  the following example communicated with us by Kollár. 
\begin{rem}
Let $g:X\to \bA^1$ be   a family of  K3 surfaces over $\bA^1$, such that
	the generic fiber of $g$ has Picard number 1, but infinitely many fibers  $X_{t_i}:=g^{-1}(t_i)$ contain a (-2)-curve $C_i\subset X_{t_i}$  with $\{t_i\}_{i\in \bZ>0}$   a discrete set escaping  to infinity.
	One can contract all the $C_i$ to get a normal 3-fold $Y$ with infinitely many singular points. Then $f:X\to Y$ is a proper bimeromorphic map, but not a birational morphism. 
\end{rem}

\begin{rem}
	In the proof of \cref{lem:bimeromorphic2}, we apply the existence of \emph{Shafarevich maps} established by Koll\'ar \cite{Kol93}.
	
	Recall that when $\varrho:\pi_1(X)\to \GL_{N}(\bC)$ underlies a $\bZ$-VHS and has infinite monodromy at infinity, the Shafarevich morphism ${\rm sh}_\varrho:X\to {\rm Sh}_\varrho(X)$ arises as the Stein factorization of the period mapping  $X\to \sD/\Gamma$. Here, $\sD$ denotes the period domain of the corresponding $\bZ$-VHS, and $\Gamma$ represents the monodromy group of $\varrho$. In this context,  Sommese \cite{Som78} proved \cref{lem:bimeromorphic2} by utilizing   H\"ormander-Andreotti-Vesentini $L^2$-method. As a result, \cref{lem:bimeromorphic2} represents a significant extension of the results in \cite{Som78},  while also offering a simpler proof, even when $\varrho$ underlies a $\bZ$-VHS. Notably, proving \cref{conj:algebraic} using the $L^2$-method poses a highly challenging problem.
\end{rem}

\subsection{Construction of Shafarevich morphism (III)} 
In this subsection we will prove \cref{cor:Sha3}. 
%\begin{dfn}[quasi-infinite monodromy at infinity]
%Let $\varrho:\pi_1(X)\to\mathrm{GL}_N(\bC)$ be a representation.
%A partial compactification $X'$ of $X$ has quasi-infinite monodromy at infinity with respect to $\varrho$   if for every boundary divisor $D\subset X'\backslash X$ the local monodromy is finite and for every $f:\mathbb D^*\to X$ such that $f(0)\not\in X'$, $f^*\varrho:\bZ\to \pi_1(X)$ is infinite.
%\end{dfn} 
Recall that in Step 1 of the proof of \cref{thm:Sha5} we proved the following result. 
\begin{lem}\label{lem:extended Shafarevich} 
Let $X$ be a   smooth quasi-projective   variety, and let $\varrho:\pi_1(X)\to {\rm GL}_N(\bC)$ be a reductive representation. 
Let $X'$ be a smooth partial compactification of $X$ such that $X'$ is quasi-infinite monodromy at infinity and $X'-X$ is a simple normal crossing divisor.
Then we may construct ${\rm sh}_{\varrho}:X\to {\rm Sh}_{\varrho}(X)$ such that ${\rm sh}_{\varrho}$ extends to a proper map $X'\to {\rm Sh}_{\varrho}(X)$. \qed
\end{lem}

\begin{lem}\label{lem:20230909}
Assume $X$ is smooth.
Given a sequence of representations $\{\varrho_i:\pi_1(X)\to\mathrm{GL}_{N_i}(\bC)\}$ such that $\mathrm{ker}(\varrho_{i+1})\subset \mathrm{ker}(\varrho_i)$.
Then, there exist $i_0$ and a smooth partial compactification $X'$ such that $X'$ is quasi-infinite monodromy at infinity for every $\varrho_i$ with $i\geq i_0$.
Moreover $X'-X$ is a simple normal crossing divisor.
\end{lem}

\begin{comment}
\begin{rem}\label{rem:20230926}
In the proof of \cref{lem:20230909}, we need to modify the situation of the proof of \cref{prop:infinity}.
In the proof of \cref{prop:infinity}, we introduced representation systems. 
Namely let $\overline{X}$ be a smooth projective compactification such that the boundary $\overline{X}\setminus X$ is a simple normal crossing divisor $\sum_{i=1}^{\ell}D_i$.     If $D_{i_1}\cap \ldots\cap D_{i_n}\neq \varnothing$ for some $\{i_1,\ldots,i_n\}\subset \{1,\ldots,\ell\}$, then for a generic point $x\in D_{i_1}\cap \ldots\cap D_{i_n}$,  locally we have an coordinate system $(U;z_1,\ldots,z_d)$  centered at $x$ such that $\overline{X}\cap U$ is a polydisk with $D_{i_k}\cap U=(z_k=0)$. 
We then have $U\cap X=(\bD^*)^n\times \bD^{d-n}$ such that $\pi_1(U\cap X)\simeq \bZ^n$, and  the local monodromy of $\varrho$ around $x$ is a representation $\varrho:\bZ^n\to \GL_{N}(\bC)$. 

In the proof of \cref{prop:infinity}, $\varrho(\bZ^n)\subset \GL_{N}(\bC)$ is an abelian group without torsions.
However, in the proof of \cref{lem:20230909}, $\varrho(\bZ^n)$ may have torsion $T\subset \varrho(\bZ^n)$.
Taking the quotient, we have $\bZ^m\simeq \varrho(\bZ^n)/T$ for some $m$.
Then the local monodromy $\varrho$ induces a linear map $\bZ^n\to\bZ^m$, which is represented by a matrix in $M_{m,n}(\bZ)$.
This matrix will be the monodromy matrix $A_{\varrho}$ below. 
\end{rem} 
\end{comment}

\begin{proof}[Proof of \cref{lem:20230909}]
The proof is a modification of the proof of \cref{prop:infinity}.
Let $\overline{X}$ be a compactification such that the boundary divisor $\overline{X}-X$ is simple normal crossing.
We shall show that after a sequence $\overline{X}^{(n)}\to \overline{X}^{(n-1)}\to\cdots\to \overline{X}$ of admissible blow-ups, $\overline{X}^{(n)}$ becomes a compactification with the good property.

We consider all sets $\{D_1,\ldots,D_l\}$ of boundary divisors such that $\bigcap D_j\not=\emptyset$ and the monodromy matrices $A_{\varrho_i}$ for these $\{D_1,\ldots,D_l\}$ and construct $\Lambda_{\varrho_i}=\{S_{\lambda}\}$.
This $\Lambda_{\varrho_i}$ is maximum.

For each $\varrho_i$, we proceed the proof of \cref{prop:infinity}.
Namely by \cref{lem:20230724} and \cref{lem:20230727}, there exists a sequence $\overline{X}^{(n_i)}\to\cdots\to \overline{X}$ of admissible blow-ups such that $\mu(S_{\lambda}^{(n_i)})=0$ for all $S_{\lambda}\in \Lambda_{\varrho_i}$.
Note that $\Lambda^{(n_i)}=\{S_{\lambda}^{(n_i)}\}$ is maximum (cf. \cref{lem:202307251}).

Let $f:\mathbb D^*\to X$ such that $f(0)\in \overline{X}^{(n_i)}\backslash X$.
If $f^*\varrho_i:\pi_1(\mathbb D^*)\to \mathrm{GL}_{N_i}(\bC)$ has infinite image, then $f^*\varrho_{i'}:\pi_1(\mathbb D^*)\to \mathrm{GL}_{N_{i'}}(\bC)$ has infinite image for all $i'\geq i$.
This follows from the assumption $\mathrm{ker}(\varrho_{i'})\subset \mathrm{ker}(\varrho_i)$.
So we consider the case that $f^*\varrho_i:\pi_1(\mathbb D^*)\to \mathrm{GL}_{N_i}(\bC)$ has finite image.

\begin{claim}\label{claim:20230922}
There exists $i_0$ such that if $f(0)\in \overline{X}^{(n_{i_0})}\backslash X$ is contained in only one irreducible boundary component and $f^*\varrho_{i_0}:\pi_1(\mathbb D^*)\to \mathrm{GL}_{N_{i_0}}(\bC)$ has finite image, then $f^*\varrho_{i}:\pi_1(\mathbb D^*)\to \mathrm{GL}_{N_{i}}(\bC)$ has finite image for all $i\geq i_0$.
\end{claim}

\begin{proof}
Let $\{D_1,\ldots,D_l\}$ be boundary divisors on $\partial{X}$ such that $\bigcap D_j\not=\emptyset$ as above.
We associate a $\bQ$-vector space $V=\bQ^l$ generated by the formal basis $(D_1),\ldots,(D_l)$.
Let $f:\mathbb D^*\to X$ such that $f(0)\in \bigcap D_j$, we associate $(f)\in V$ by $(f)=\sum b_j(D_j)$, where $b_j=\mathrm{ord}_0f^*D_j$. 
We set $\Sigma=\{(f)\in V\}$ and
$$\Sigma_i=\{ (f)\in V; \text{$f^*\varrho_{i}:\pi_1(\mathbb D^*)\to \mathrm{GL}_{N_{i}}(\bC)$ has finite image}\}.
$$
Let $W_i\subset V$ be the $\bQ$-vector space spanned by $\Sigma_i$.
Then we have
\begin{equation}\label{eqn:20230921}
\Sigma_i=\Sigma\cap W_i.
\end{equation}
Indeed $\Sigma_i\subset \Sigma\cap W_i$ is obvious.
To show the converse, we take $(f)\in \Sigma\cap W_i$.
By $(f)\in W_i$, there exist $(f_1),\ldots,(f_k)\in \Sigma_i$ such that $(f)=c_1(f_1)+\cdots+c_k(f_k)$.
Then using the monodromy matrix $A_{\varrho_i}$, we have $A_{\varrho_i}(f)=c_1A_{\varrho_i}(f_1)+\cdots+c_kA_{\varrho_i}(f_k)=0$. 
Hence $(f)\in \Sigma_i$.
This shows \eqref{eqn:20230921}.

By the definition of $\Sigma_i$, we have $\Sigma\supset \Sigma_1\supset \Sigma_2\supset\cdots$.
Hence $W_1\supset W_2\supset \cdots$.
Thus there exists $i$ such that $W_i=W_{i+1}=\cdots$.
Hence by \eqref{eqn:20230921}, we have $\Sigma_i=\Sigma_{i+1}=\cdots$.
In particular, if $f^*\varrho_{i}:\pi_1(\mathbb D^*)\to \mathrm{GL}_{N_{i}}(\bC)$ has finite image, then $f^*\varrho_{i'}:\pi_1(\mathbb D^*)\to \mathrm{GL}_{N_{i'}}(\bC)$ has finite image for all $i'\geq i$.

So far, we have fixed $\{D_1,\ldots,D_l\}$ and chose $i=i_{\{D_1,\ldots,D_l\}}$.
Next we consider all such $\{D_1,\ldots,D_l\}$ and chose $i_0$ big enough such that $i_0$ is greater than every $i_{\{D_1,\ldots,D_l\}}$ above.

Now suppose $f(0)\in \overline{X}^{(n_{i_0})}\backslash X$ is contained in only one irreducible boundary component and $f^*\varrho_{i_0}:\pi_1(\mathbb D^*)\to \mathrm{GL}_{N_{i_0}}(\bC)$ has finite image.
Let $D_1,\ldots,D_l$ be boundary divisors in $\overline{X}$ so that $f(0)\in D\subset\overline{X}$.
Then by the choice of $i_0$, $f^*\varrho_{i}:\pi_1(\mathbb D^*)\to \mathrm{GL}_{N_{i}}(\bC)$ has finite image for all $i\geq i_0$.
The proof of \cref{claim:20230922} is completed. 
\end{proof}

Now we return to the proof of \cref{lem:20230909}.
We remove all boundary divisors $D$ from $\overline{X}^{(n_{i_0})}$ such that local monodromy with respect to $\varrho_{i_0}$ at $D$ is infinite to get a desired partial compactification $X'$ of $X$.
\end{proof}

\begin{lem}\label{lem:202309261}
Let $A_N$ be the set of all $\mathrm{ker}(\varrho)\subset \pi_1(X)$ for the reductive representations $\varrho:\pi_1(X)\to\mathrm{GL}_N(\bC)$.
Then $A_N$ is countable.
In particular, $\cup_NA_N$ is countable. 
\end{lem}

\begin{proof}
We use the notation in the proof of \cref{lem:small}.
Given a reductive representation $\varrho$, we define $Z_{\varrho}\subset M$ by $Z_{\varrho}=\cap_{\gamma\in \mathrm{ker}(\varrho)}W_{\gamma}$.
Then $Z_{\varrho}\subset M$ is a Zariski closed set. 
Using $Z_{\varrho}$, we may construct $\mathrm{ker}(\varrho)\subset \pi_1(X)$ by $\mathrm{ker}(\varrho)=\{\gamma\in\pi_1(X); Z_{\varrho}\subset W_{\gamma}\}$.
By the Noetherian property, there exists $\gamma_1,\ldots,\gamma_l$ such that $Z_{\varrho}=W_{\gamma_1}\cap\cdots\cap W_{\gamma_l}$.
In this way, we may label $\mathrm{ker}(\varrho)$ a finite subset $\{\gamma_1,\ldots,\gamma_l\}\subset \pi_1(X)$.
Since $\pi_1(X)$ is countable, it has only countable finite subsets.
Hence $A_N$ is countable.
\end{proof}

\begin{lem}\label{lem:simultaneous1} 
 	Let $V$ be a quasi-projective normal variety and let $(f_{\lambda}:V\to S_{\lambda})_{\lambda\in\Lambda}$ be a family of proper morphisms into normal complex spaces $S_{\lambda}$.
 	Then there exist a normal complex space $S_{\infty}$ and a proper morphism $f_{\infty}:V\to S_{\infty}$ such that  
 	\begin{enumerate}[label=(\alph*)]
	\item \label{item:factor} for each $\lambda\in \Lambda$, there exists a morphism $e_{\lambda}:S_{\infty}\to S_{\lambda}$ such that $f_{\lambda}=e_{\lambda}\circ f_{\infty}$;
	\item \label{item:simul}for every algebraic subvariety $Z\subset V$, if $f_{\lambda}(Z)$ is a point for every $\lambda\in\Lambda$, then $f_{\infty}(Z)$ is a point.
\end{enumerate}
 \end{lem}

\begin{proof}
The proof is a modification of that of \cref{lem:simultaneous}.
Let $E_{\lambda}\subset V\times V$ be defined by 
 	$$E_{\lambda}=\{(x,x')\in V\times V; f_{\lambda}(x)=f_{\lambda}(x')\}.$$
 	Then $E_{\lambda}\subset V\times V$ is an analytic subset. 
 	Indeed, $E_{\lambda}=(f_{\lambda},f_{\lambda})^{-1}(\Delta_{\lambda})$, where $(f_{\lambda},f_{\lambda}):V\times V\to S_{\lambda}\times S_{\lambda}$ is the morphism defined by $(f_{\lambda},f_{\lambda})(x,x')=(f_{\lambda}(x),f_{\lambda}(x'))$ and $\Delta_{\lambda}\subset S_{\lambda}\times S_{\lambda}$ is the diagonal.

Let $K_1\subset K_2\subset \cdots\subset V$ be a sequence of domains such that
\begin{itemize}
\item $\overline{K_n}$ is compact and $\overline{K_{n}}\subset K_{n+1}$ for all $n$, and
\item
$V=\cup_{n}K_n$.
 \end{itemize}
 Given $n$, we take a finite subset $\Lambda_n\subset \Lambda$ as follows.
 Note that $\cap_{\lambda\in\Lambda}E_{\lambda}\cap (K_n\times K_n)$ is an analytic subset of $K_n\times K_n$.
 Since $\overline{K_n\times K_n}$ is compact, we may take a finite subset $\Lambda_n\subset \Lambda$ such that 
 \begin{equation}\label{eqn:20231023}
 \cap_{\lambda\in\Lambda_n}E_{\lambda}\cap (K_n\times K_n)=\cap_{\lambda\in\Lambda}E_{\lambda}\cap (K_n\times K_n).
 \end{equation}
 We may assume that $\Lambda_1\subset \Lambda_2\subset \cdots\subset\Lambda$.
 Let $g_n:V\to \Sigma_n$ be the Stein factorization of the proper map $(f_{\lambda})_{\lambda\in\Lambda_n}:V\to\Pi_{\lambda\in\Lambda_n}S_{\lambda}$.
 Then $\Sigma_n$ is normal and $g_n:V\to \Sigma_n$ is proper, hence surjective.
 By \eqref{eqn:20231023}, for a closed subvariety $Z\subset V$ such that $Z\subset K_n$, we have
 \begin{equation}\label{eqn:20231021}
 \text{
 $g_n(Z)$ is a point 
 $\Longleftrightarrow$
  $f_{\lambda}(Z)$ is a point for all $\lambda\in\Lambda$.
 }
 \end{equation}
By $\Lambda_n\subset\Lambda_{n+1}$, we have a natural morphism $\Sigma_{n+1}\to\Sigma_n$.
 We set $\varphi_n:\Sigma_n\to \Sigma_1$ for the induced map.

Let $H_1\subset H_2\subset \cdots\subset \Sigma_1$ be a sequence of domains such that
\begin{itemize}
\item $\overline{H_k}$ is compact and $\overline{H_{k}}\subset H_{k+1}$ for all $k$, and
\item
$\Sigma_1=\cup_{k}H_k$.
 \end{itemize}
 We take this sequence inductively as follows.
  %\footnote{Since $\Sigma_1$ is only a complex space, how can we find such exhausting sequence $H_i$?
 %
 %ADDED:
 %Inductively we take $H_i$ as follows.
 %Since $g_1(\overline{K_1})$ is compact, we may take $H_1$ such that $g_1(\overline{K_1})\subset H_1$ and that $\overline{H_1}$ is compact.
% Assume that we have taken $H_1\subset\cdots \subset H_k$ such that for all $i=1,\ldots,k$, we have
% \begin{itemize}
% \item
% $g_1(\overline{K_i})\subset H_i$
% \item
% $\overline{H_i}$ is compact and $\overline{H_{i-1}}\subset H_{i}$, where we set $H_0=\emptyset$.
% \end{itemize}
% Since $g_1(\overline{K_{k+1}})\cup \overline{H_k}$ is compact, we may take $H_{k+1}$ so that $ g_1(\overline{K_{k+1}})\cup \overline{H_k}\subset H_{k+1}$  and $\overline{H_{k+1}}$ is compact.
%Thus inductively we have taken $H_k$ which have two properties above.
%Now since $g_1:V\to \Sigma_1$ is surjective, we have $\Sigma_1=\cup_k g_1(K_k)\subset \cup_k H_k$. 
%
%ADDED: this is a nice argument. Do we need to put it into the proof?
%
%ADDED:
%According to Wiki, a topologigal space is exhaustible by compact sets iff the space is $\sigma$-compact and weakly locally compact.
%However I can not find a reference so I added an argument in the text.
%}
Set $H_0=\emptyset$.
 Assume we are given $H_0\subset \cdots\subset H_k$, where $k\geq 0$.
 Since $g_1(\overline{K_{k+1}})\cup \overline{H_k}$ is compact, we may take a domain $H_{k+1}$ so that $ g_1(\overline{K_{k+1}})\cup \overline{H_k}\subset H_{k+1}$  and $\overline{H_{k+1}}$ is compact.
 Now since $g_1:V\to \Sigma_1$ is surjective, we have $\Sigma_1=\cup_k g_1(K_k)\subset \cup_k H_k$. 
 
 For all $k$, we claim that there exists $n_k$ such that the induced maps $\varphi_{n_k}^{-1}(H_k)\gets \varphi_{n_k+1}^{-1}(H_{k})\gets\cdots$ are all isomorphisms.
 Indeed since $g_1:V\to\Sigma_1$ is proper, we may take $n_k$ so that $g_1^{-1}(H_k)\subset K_{n_k}$.
 Let $n\geq n_k$.
 Since every fiber of $g_{n}:V\to \Sigma_{n}$ is connected, \eqref{eqn:20231021} yields that every fiber of $g_{n}:V\to \Sigma_{n}$ over $\varphi_{n}^{-1}(H_{k})\subset \Sigma_{n}$ is contracted to a point by $g_{n+1}:V\to \Sigma_{n+1}$.
 Hence $\varphi_{n+1}^{-1}(H_k)\to \varphi_n^{-1}(H_k)$ is injective.
 Since this is also surjective, the induced maps $\varphi_{n_k}^{-1}(H_k)\gets \varphi_{n_k+1}^{-1}(H_{k})\gets\cdots$ are all isomorphism, for $\Sigma_n$ are normal (cf. \cite[p. 166]{GR84}).
Hence $\varphi_{n_k}^{-1}(H_{k}) =\varphi_{n_{k+1}}^{-1}(H_{k}) \subset \varphi_{n_{k+1}}^{-1}(H_{k+1})$.
We set 
$$S_{\infty}=\cup \varphi_{n_k}^{-1}(H_{k}).$$
The map $f_{\infty}:V\to S_{\infty}$ is defined by $f_{\infty}|_{g_1^{-1}(H_{k})}=g_{n_k}|_{g_1^{-1}(H_{k})}$.
We note that $g_{n_k}|_{g_1^{-1}(H_{k})}:g_1^{-1}(H_{k})\to \varphi_{n_k}^{-1}(H_k)$ is proper.
Hence $f_{\infty}:V\to S_{\infty}$ is proper.

Now we prove the property {\rm (a)}.
We take $\lambda\in\Lambda$.
Let $Z\subset V$ be an algebraic subvariety such that $f_{\infty}(Z)$ is a point in $S_{\infty}$.
We shall show that $f_{\lambda}(Z)$ is a point.
Indeed we take $k$ such that $f_{\infty}(Z)\in \varphi_{n_k}^{-1}(H_{k})$.
Then $Z\subset g_{1}^{-1}(H_k)\subset K_{n_k}$ and $g_{n_k}(Z)$ is a point.
Hence by \eqref{eqn:20231021}, $f_{\lambda}(Z)$ is a point.
This shows that $f_{\lambda}:V\to S_{\lambda}$ factors through $f_{\infty}:V\to S_{\infty}$.
Thus we get $e_{\lambda}:S_{\infty}\to S_{\lambda}$. 

To show the property {\rm (b)}, we take an algebraic subvariety $Z\subset V$ such that $f_{\lambda}(Z)$ is a point for every $\lambda\in\Lambda$.
Then $g_n(Z)$ is a point of $\Sigma_n$ for every $n$.
In particular, $g_1(Z)$ is a point in $\Sigma_1$.
We take $k$ such that $g_1(Z)\in H_k$.
Then $Z\subset g_1^{-1}(H_k)$.
Since $g_{n_k}$ and $f_{\infty}$ coincide as morphisms $g_1^{-1}(H_k)\to \varphi_{n_k}^{-1}(H_k)$, where $\varphi_{n_k}(H_k)$ is contained in both $\Sigma_{n_k}$ and $S_{\infty}$.
Hence $f_{\infty}(Z)$ is a point of $S_{\infty}$.
\end{proof}

\begin{lem}\label{lem:Sha}
	Assume $X$ is smooth.
	Let $\Sigma=\{\varrho_i:\pi_1(X)\to \mathrm{GL}_{N_i}(\bC)\}_{i\in \bZ_{> 0}}$ be a sequence of reductive representations such that $\mathrm{ker}(\varrho_{i})\subset \mathrm{ker}(\varrho_{i-1})$.
	Then there exists a   dominant holomorphic map with general fibers connected   ${\rm sh}_\Sigma:X\to {\rm Sh}_\Sigma(X)$  onto a complex normal space ${\rm Sh}_\Sigma(X)$ such that for closed subvariety $Z\subset X$, ${\rm sh}_\Sigma(Z)$   is a point if and only if $\varrho_i({\rm Im}[\pi_1(Z^{\rm norm})\to \pi_1(X)])$ is finite  for every $i$. 
\end{lem}

\begin{proof} 
	By \cref{lem:20230909},  there exist $i_0$ and a smooth partial compactification $X'$ such that $X'$ is quasi-infinite monodromy at infinity for every $\varrho_i$ with $i\geq i_0$. By \cref{lem:extended Shafarevich}, there exists   a proper surjective holomorphic map $f_i:X'\to S_i$  for each $i\geq i_0$ such that the Shafarevich morphism ${\rm sh}_{\varrho_i}:X\to {\rm Sh}_{\varrho_i}(X)$ of $\varrho_i$ is the restriction $f_i|_{X}$. 
	By   \cref{lem:simultaneous1},  there exists a surjective proper  holomorphic fibration $f_{\infty}:X'\to S_\infty$ onto a complex normal space $S_\infty$ and holomorphic maps $e_i:S_\infty\to S_i$  such that
	\begin{enumerate}[label=(\alph*)]
		\item   $f_i=e_i\circ f_\infty$;
		\item for every algebraic subvariety $Z\subset X'$, if $f_{i}(Z)$ is a point for every $i\geq i_0$, then $f_{\infty}(Z)$ is a point. 
	\end{enumerate}
	 	For an algebraic subvariety $Z$ of $X$ with its closure denoted as $\overline{Z}$ in $X'$,   $f_\infty(\overline{Z})$   is a point if and only if $f_i(\overline{Z})$ is a point for any $i\geq i_0$. Additionally, by the property of the Shafarevich morphism the same holds true if and only if  
	  $\varrho_i({\rm Im}[\pi_1(Z^{\rm norm})\to \pi_1(X)])$ is finite  for every for any $i\geq i_0$, and, given that $\mathrm{ker}(\varrho_{i})\subset \mathrm{ker}(\varrho_{i-1})$, it is equivalent to $\varrho_i({\rm Im}[\pi_1(Z^{\rm norm})\to \pi_1(X)])$ for any $i\geq 1$. To conclude the proof, let ${\rm sh}_{\Sigma}:X\to {\rm Sh}_\Sigma(X)$ be the restriction of $f_\infty$ to $X$.
\end{proof}

 \begin{cor}\label{thm:Sha3}
	Let $X$ be a quasi-projective normal variety. 
	Let $\Sigma$ be a (non-empty) set of reductive representations $\varrho:\pi_1(X)\to {\rm GL}_{N_\varrho}(\bC)$. 	
	Then there is a dominant holomorphic map   ${\rm sh}_{\Sigma}:X\to {\rm Sh}_{\Sigma}(X)$ with general fibers connected onto a complex normal space such that for closed subvariety $Z\subset X$, ${\rm sh}_\Sigma(Z)$   is a point if and only if $\varrho({\rm Im}[\pi_1(Z^{\rm norm})\to \pi_1(X)])$ is finite  for every $\varrho\in\Sigma$. 
\end{cor}  
\begin{proof} 
By Step 2 of the proof of \cref{thm:Sha5}, we may assume that $X$ is smooth. 	 By \cref{lem:202309261}, we can take a set  of  reductive representations $\Sigma_1:=\{\varrho_i:\pi_1(X)\to \GL_{N_i}(\bC)\}_{i\in \bZ_{>0}}$ such that $\varrho_i\in \Sigma$, and for any $\varrho\in \Sigma$, there exists $\varrho_i$ such that $\ker\varrho_i=\ker\varrho$.   Therefore, if $\Gamma\subset \pi_1(X)$ is a subgroup such that $\varrho_i(\Gamma)$ is finite for each $\varrho_i\in \Sigma_1$, then $\varrho(\Gamma)$ is finite for each $\varrho\in \Sigma$. 
	 
	 Define $\tau_i$ to be the product  representation  $$(\varrho_1,\ldots,\varrho_i):\pi_1(X)\to \prod_{j=1}^{i}\GL_{N_j}(\bC).$$
	 Then it is reductive and we have $\ker\tau_i\subset \ker\tau_{i+1}$.  Let $\Sigma_2:=\{\tau_i\}_{i\in \bZ_{> 0}}$.  By \cref{lem:Sha}, there exists a dominant holomorphic map   ${\rm sh}_{\Sigma_2}:X\to {\rm Sh}_{\Sigma_2}(X)$ with general fibers connected onto a complex normal space such that for closed subvariety $Z\subset X$, ${\rm sh}_{\Sigma_2}(Z)$   is a point if and only if $\tau_i({\rm Im}[\pi_1(Z^{\rm norm})\to \pi_1(X)])$ is finite  for every $\tau_i\in\Sigma_2$.   
	 
We note that  $\tau_i({\rm Im}[\pi_1(Z^{\rm norm})\to \pi_1(X)])$ is finite if and only if $\varrho_j({\rm Im}[\pi_1(Z^{\rm norm})\to \pi_1(X)])$ are all finite for each $j=1,\ldots,i$. Therefore,${\rm sh}_{\Sigma_2}(Z)$   is a point if and only if   $\varrho_i({\rm Im}[\pi_1(Z^{\rm norm})\to \pi_1(X)])$ is finite  for every $\varrho_i\in\Sigma_1$, and if only only if $\varrho({\rm Im}[\pi_1(Z^{\rm norm})\to \pi_1(X)])$ is finite  for every $\varrho\in\Sigma$.  Let ${\rm sh}_{\Sigma}$ be ${\rm sh}_{\Sigma_2}$, and thus, the theorem is proven. 
\end{proof}

 \section{Proof of the reductive Shafarevich conjecture} \label{sec:HC Stein}
The  goal of this section is to provide proofs for \cref{main,main2}   when $X$ is a \emph{smooth} projective variety. It is important to note that our methods  differs   from the approach presented in \cite{Eys04}, although we do follow the general strategy   in that work.

 In this section, we will use the notation $\cD G$ to denote the derived group of any given group $G$. Throughout the section, our focus is on non-archimedean local fields with characteristic zero. More precisely, we consider finite extensions of $\bQ_p$ for some prime $p$.
\subsection{Reduction map of representation into algebraic tori}\label{subsec:tori} 
Let $X$ be a smooth projective variety. Let $a:X\to A$ be the Albanese morphism of $X$.   %and let $\pi_1(X)\to T(K)$ be a representation where $T$ is a torus over $K$. The Bruhat-Tits building  is $V(T):=\bR\times_{\bZ} X_*(T)$ where $X_*(T)=\Hom_{K}(\GL_1, T)$  is group of $K$-rational characters on $T$.  

\begin{lem}  
%\begin{thmlist}
%	\item  The reduction map $s_\tau:X\to S_\tau$ is the Stein factorization of $X\to A/A_\tau$. 
%	\item 
\label{lem:factor}
	Let $P\subset A$ be an abelian subvariety of the Albanese variety $A$ of $X$ and $K$ be a non-archimedean local field.  If $\tau:\pi_1(X)\to \GL_1(K)$ factors through $\sigma:\pi_1(A/P)\to \GL_1(K)$, then   the Katzarkov-Eyssidieux reduction map $s_\tau:X\to S_\tau$  factors through  the Stein factorization of  the map $q:X\to A/P$. 
%\end{thmlist}	   
\end{lem}
\begin{proof}
 As $\tau=q^*\sigma$, it follows that for each connected component $F$ of the fiber of $q:X\to A/P$, $\tau(\pi_1(F))=\{1\}$. Therefore, $F$ is contracted by $s_\tau$.  The lemma follows. 
\end{proof}

\begin{lem}\label{lem:enough}
	Let $P\subset A$ be an abelian subvariety of $A$. 	Let $N$ be a Zariski dense open set of the image $j: M_{\rm B}(A/P, 1)\to M_{\rm B}(A,1)$ where we consider $M_{\rm B}(A/P, 1)$ and $M_{\rm B}(A,1)$ as algebraic tori defined over $\bar{\bQ}$. Then there  are   non-archimedean local fields $K_i$ and  a family of representations $\btau:=\{\tau_i:\pi_1(X)\to \GL_1(K_i)\}_{i=1,\ldots,m}$ such that
	\begin{itemize}
		\item $\tau_i\in N(K_i)$, where we use the natural identification $M^0_{\rm B}(X,1)\simeq M_{\rm B}(A,1)$. Here $M^0_{\rm B}(X,1)$  denotes the connected component of $M^0_{\rm B}(X,1)$ containing the trivial representation. 
		\item The reduction map $s_{\btau}:X\to S_\btau$ is the Stein factorization of $X\to A/P$.  
		\item  For the canonical current $T_{\btau}$ defined over $S_{\btau}$, $\{T_{\btau}\}$ is a K\"ahler class.  
	\end{itemize}  
\end{lem}
\begin{proof}
 Let  $e_1,\ldots,e_{m}$ be a basis of $\pi_1(A/P)\simeq H_1(A/P,\bZ)$. Note that  $\bar{\bQ}$-scheme $M_{\rm B}(A/P, 1)\simeq (\bar{\bQ}^\times)^{m}$.  Denote by $S\subset U(1)\cap \bar{\bQ}$   the set of roots of unity. Then $S$ is   Zariski dense in $\bar{\bQ}^\times$. Since $j^{-1}(N)$ is  a Zariski dense open set of $M_{\rm B}(A/P, 1)$,   it follows that   there are $\{a_{ij}\}_{i,j=1,\ldots,m}\in \bar{\bQ}^\times$ and   representations  $\{\varrho_i:\pi_1(A/P)\to \bar{\bQ}^\times\}_{i=1,\ldots,m}$
	defined by $
	\varrho_i(e_j)=a_{ij} $
	such that 
	\begin{itemize}
		\item  $[\varrho_i]\in j^{-1}(N)(\bar{\bQ})$;
		\item 	If $i= j$,  $a_{ij}\in 	 \bar{\bQ}^\times\setminus U(1) $;
		\item  	If $i\neq j$,  $a_{ij}\in   	S $.  
	\end{itemize} 
	Consider a number field $k_i$ containing $a_{i1},\ldots,a_{im}$ endowed with  a discrete non-archimedean valuation $v_i:k_i\to \bR$ such that $v_i(a_{ii})\neq 0$. Then $v_i(a_{ij})=0$ for every $j\neq i$. Indeed, for every $j\neq i$, since $a_{ij}$ is a root of unity, there exists $\ell\in \bZ_{> 0}$ such that $a_{ij}^\ell=1$. It follows that $0=v(a_{ij}^\ell)=\ell v(a_{ij})$.     Let $K_i$ be   the non-archimedean local field which is the completion of $k_i$ with respect to $v_i$.    It follows that each $\varrho_i:\pi_1(A/P)\to K_i^\times$ is unbounded.    Consider $\nu_i:\pi_1(A/P)\to \bR$ by composing $\varrho_i$ with $v_i: K_i^\times \to \bR$.   Then $\{\nu_1,\ldots,\nu_{m}\}\subset H^1(A/P,\bR)$ is a basis for the $\bR$-linear space $H^1(A/P,\bR)$.  It follows that $\nu_i(e_j)=\delta_{ij}$ for any $i,j$.  Let $\eta_i\in H^0(A/P, \Omega_{A/P}^1)$ be the $(1,0)$-part of the Hodge decomposition of $\nu_i$.  Therefore, $\{\eta_1,\ldots,\eta_{m}\}$ spans the $\bC$-linear space $H^0(A/P, \Omega_{A/P}^1)$. Hence $\sum_{i=1}^{m}i\eta_i\wedge\overline{\eta_i}$  is a K\"ahler form on $A/P$.   Let $\tau_i:\pi_1(X)\to K_i^\times$ be the composition of $\varrho_i$ with $\pi_1(X)\to \pi_1(A/P)$. 
	
Let  $q:A\to A/P$ be the quotient map. Let $P'$ the largest abelian subvariety of $A$ such that $q^*\eta_i|_{P'}\equiv 0$ for each $i$. Since $\{\eta_1,\ldots,\eta_{m}\}$ spans $H^0(B, \Omega_B^1)$, it follows that $P'=P$. Therefore, the reduction map $s_{\btau}:X\to S_\btau$ is the Stein factorization of $X\to A/P$ with $g:S_\btau\to A/P$ be the finite morphism. According to  \cref{def:canonical2}, $T_{\btau}=g^*\sum_{i=1}^{m}i\eta_i\wedge\overline{\eta_i}$.  Since  $\sum_{i=1}^{m}i\eta_i\wedge\overline{\eta_i}$  is a K\"ahler form on $A/P$, it follows that  $\{T_{\btau}\}$ is a K\"ahler class by \cref{thm:DHP}.   
	The lemma is proved. 
\end{proof}

\begin{cor}\label{cor:positivity}
	Let $X$ be a smooth projective variety.  If $\kC \subset M_{\rm B}( {X}, 1)$ is an  absolutely constructible subset. Consider the reduction map $s_{\kC}:X\to S_{\kC}$ defined in \cref{def:reduction ac}. Then there is a family of   representations  $\brho:=\{\varrho_i:\pi_1(X)\to \GL_1(K_i)\}_{i=1,\ldots,\ell}$ where $K_i$ are non-archimedean local fields  such that
	\begin{itemize}
		\item For each $i=1,\ldots,\ell$, $\varrho_i\in \kC(K_i)$; 
		\item  The reduction map $s_{\brho}:X\to S_{\brho}$ of $\brho$ coincides with $s_{\kC}$.
		\item For the canonical current $T_{\brho}$ defined over $S_{\kC}$, $\{T_{\brho}\}$ is a K\"ahler class. 
	\end{itemize} 
\end{cor}
\begin{proof}
Let $A$ be the Albanese variety of $X$.  Since $\kC \subset M_{\rm B}( {X}, 1)$ is an absolute constructible subset, by \cref{thm:S1,lem:triple},  there are abelian subvarieties $\{P_i\subset A\}_{i=1,\ldots,m}$  and torsion points $v_i\in M_{\rm B}(X,1)(\bar{\bQ})$ such that $\kC=\cup_{i=1}^{m}v_i. N_i^\circ$. Here $N_i$ is the image in $M^0_{\rm B}(X, 1)\simeq M_{\rm B}(A,1)$  of the natural morphism $M_{\rm B}(A/P_i,1)\to M_{\rm B}(A,1)$ and $N_i^\circ$ is a Zariski dense open subset of $N_i$ such that $N_i\backslash N_i^\circ$ is defined over $\bar{\bQ}$. %Here   $M^0_{\rm B}(X, 1)$  denotes the connected component of $M^0_{\rm B}(X, 1)$ containing the identity.  
%Moreover, both $N_i$ and $N_i\setminus N_i^\circ$ are  defined over $\bar{\bQ}$.   
Let $k$ be a number field such that $v_i\in M_{\rm B}(X,1)(k)$ for each $i$.  
 \begin{claim}
 	Denote by $P:=\cap_{i=1}^{m}P_i$. Then $s_{\kC}:X\to S_{\kC}$ is the Stein factorization of $X\to A/P$.  
 \end{claim}
\begin{proof}
	Let   $\tau:\pi_1(X)\to \GL_1(K)$ be a reductive representation with $K$ a non-archimedean local field such that $\tau\in \kC(K)$.   Note that the reduction map $s_\tau$ is the same if we replace $K$ by a finite extension. We thus can assume that $k\subset K$. Note that there exists some $i\in \{1,\ldots,\ell\}$  such that   $[v_i^{-1}.\tau] \in N_i(K)$. Write $\varrho:=v_i^{-1}.\tau$.  Since $v_i$ is a torsion element, it follows that $v_i(\pi_1(X))$ is finite, and thus the reduction  map $s_\varrho$ coincides with $s_\tau$.  Since $\varrho$ factors through $\pi_1(A/P_i)\to \GL_1(K)$,  by \cref{lem:factor} $s_\varrho$ factors  through the Stein factorization of  $X\to A/P_i$.  Hence $s_\varrho$ factors  through the Stein factorization of  $X\to A/P$. By \cref{def:reduction ac}, it follows that $s_\kC:X\to S_{\kC}$ factors through  the Stein factorization of $X\to A/P$.   
	
Fix any $i$. By \cref{lem:enough}  there  are   non-archimedean local fields $K_j$ and  a family of reductive representations $\btau:=\{\tau_j:\pi_1(X)\to \GL_1(K_j)\}_{j=1,\ldots,n}$ such that
\begin{itemize}
	\item $\tau_j\in N_i^\circ(K_j)$. 
	\item The reduction map $s_{\btau}:X\to S_\btau$   is the Stein factorization of $X\to A/P_i$.  
	\item For the canonical current $T_{\btau}$   over $S_{\btau}$,  $\{T_{\btau}\}$ is a K\"ahler class.  
\end{itemize}  
We can replace $K_i$ by a finite extension such that  $k\subset K_i$ for each $K_i$. Then $v_i. \tau_i\in \kC(K_i)$ for every $i$. Note that the Katzarkov-Eyssidieux reduction map   $s_{v_i.\tau_j}:X\to S_{v_i.\tau_j}$ coincides with $s_{\tau_j}:X\to S_{\tau_j}$. Therefore, the Stein factorization of $X\to A/P_i$ factors through $s_\kC$.   Since this holds for each $i$, it follows that  the Stein factorization $X\to A/P_1\times \cdots\times A/P_m$   factors through $s_\kC$. Note that  the Stein factorization $X\to A/P_1\times \cdots\times A/P_m$   coincides with the Stein factorization of $X\to A/P$. Therefore,  the Stein factorization of $X\to A/P$ factors through $s_\kC$.  The claim is proved. 
\end{proof}
 By the above arguments, for each $i$, there  exists  a family of reductive representations into non-archimedean local fields $\brho_i:=\{\varrho_{ij}:\pi_1(X)\to \GL_1(K_{ij})\}_{j=1,\ldots,k_i}$ such that \begin{itemize}
 	\item  $\varrho_{ij}\in \kC(K_{ij})$ 
 	\item   $s_{\brho_i}:X\to S_{\brho_i}$ is the Stein factorization of $X\to A/P_i$ 
 	\item  For the canonical current $T_{\brho_i}$ defined over $S_{\brho_i}$, $\{T_{\brho_i}\}$ is a K\"ahler class.  
 \end{itemize}
By the above claim, we know that $s_\kC:X\to S_\kC$ is the Stein factorization of  $X\to S_{\brho_1}\times \cdots\times S_{\brho_m}$.  Then for the representation $\brho:=\{\varrho_{ij}:\pi_1(X)\to \GL_1(K_{ij})\}_{i=1,\ldots,m;j=1,\ldots,k_i}$, $s_{\brho}:X\to S_{\brho}$ is the Stein factorization of $X\to A/P$ hence $s_{\brho}$ coincides with $s_\kC$. Moreover, the canonical current $T_{\brho}=\sum_{i=1}^{m}g_i^*T_{\brho_i}$ where   $g_i:S_{\kC}\to S_{\brho_i}$ is the natural map. As $S_\kC\to S_{\brho_1}\times\cdots\times S_{\brho_m}$ is finite,  by \cref{thm:DHP}  $\{T_{\brho}\}$ is K\"ahler. 
\end{proof}

Let us prove the main result in this subsection. 
\begin{thm}\label{thm:positivity}
	Let $X$ be a smooth projective variety and let $T$ be an algebraic tori defined over some number field $k$.  Let $\kC \subset M_{\rm B}( {X}, T)(\bC)$ be an absolutely constructible subset. Consider the reduction map $s_{\kC}:X\to S_{\kC}$. Then there is a family of   reductive representations  $\btau:=\{\tau_i:\pi_1(X)\to T(K_i)\}_{i=1,\ldots,N}$ where $K_i$ are non-archimedean local fields  containing $k$ such that
	\begin{itemize}
		\item For each $i=1,\ldots,N$, $[\tau_i]\in \kC(K_i)$; 
		\item  The reduction map $s_{\btau}:X\to S_{\btau}$ of $\btau$ coincides with $s_{\kC}$.
		\item For the canonical current $T_{\btau}$  over $S_{\kC}$ defined in \cref{def:canonical2}, $\{T_{\btau}\}$ is a K\"ahler class. 
	\end{itemize} 
\end{thm}
\begin{proof}
We replace $k$ by a finite extension such that $T$ is split over $k$. Then  we have $T\simeq \bG_{m,k}^\ell$.   Note that this does not change the reduction map $s_{\kC}:X\to S_{\kC}$.  We take $p_i:T\to \bG_{m,k}$ to be the $i$-th projection which is a $k$-morphism. It induces a morphism of $k$-schemes $\psi_i:M_{\rm B}(X,T)\to M_{\rm B}(X,\GL_1)$. By \cref{thm:S2}, $\kC_i:=\psi_i(\kC)$ is also an absolutely constructible subset.  Consider the reduction maps $\{s_{\kC_i}:X\to S_{\kC_i}\}_{i=1,\ldots,\ell}$ defined by \cref{def:reduction ac}.  
\begin{claim}
$s_{\kC}:X\to S_{\kC}$ is the Stein factorization of $s_{\kC_1}\times\cdots\times s_{\kC_\ell}:X\to S_{\kC_1}\times\cdots\times S_{\kC_\ell}$. 
\end{claim}
\begin{proof}
	Let $\varrho:\pi_1(X)\to T(K)$ be any reductive representation where $K$ is a non-archimedean local field containing $k$ such that $[\varrho]\in \kC(K)$.  Write $\varrho_i=p_i\circ \varrho:\pi_1(X)\to \GL_1(K)$. Then $[\varrho_i]=\psi_i([\varrho])\in \kC_i(K)$.   Note that for any subgroup $\Gamma\subset \pi_1(X)$, $\varrho(\Gamma)$ is bounded if and only if $\varrho_i(\Gamma)$ is bounded for any $i$. Therefore, $s_{\varrho}:X\to S_\varrho$ is the Stein factorization of $X\to S_{\varrho_1}\times\cdots\times  S_{\varrho_\ell}$. Hence $s_{\kC}:X\to S_{\kC}$ factors through the Stein factorization of $X\to S_{\kC_1}\times\cdots\times S_{\kC_\ell}$. 
	
	On the other hand, consider any $\varrho_i\in \kC_i(K)$ where $K$ is a non-archimedean local field containing $k$. Then there is a finite extension $L$ of $K$ such that 
	\begin{itemize}
		\item  there is  a reductive representation   $\varrho:\pi_1(X)\to T(L)$ with $[\varrho]\in \kC(L)$;
		\item $p_i\circ\varrho=\varrho_i$. 
	\end{itemize}
By the above argument,    $s_{\varrho_i}:X\to S_{\varrho_i}$ factors through $s_{\varrho}:X\to S_{\varrho}$. Note that $s_{\varrho}$ factors through $s_{\kC}$. It follows that     the Stein factorization of $X\to S_{\kC_1}\times\cdots\times S_{\kC_\ell}$ factors through $s_{\kC}$. The claim is proved.  
\end{proof}
We now apply \cref{cor:positivity} to conclude that  for each $i$, there  exists  a family of reductive representations into non-archimedean local fields $\brho_i:=\{\varrho_{ij}:\pi_1(X)\to \GL_1(K_{ij})\}_{j=1,\ldots,k_i}$ such that \begin{itemize}
	\item  $\varrho_{ij}\in \kC_i(K_{ij})$;  
	\item  The reduction map $s_{\brho_i}:X\to S_{\brho_i}$ of $\brho_i$ coincides with $s_{\kC_i}:X\to S_{\kC_i}$;
	\item  for the canonical current $T_{\brho_i}$ defined over $S_{\brho_i}$, $\{T_{\brho_i}\}$ is a K\"ahler class.   
\end{itemize}
Denote by  $\brho:=\{\varrho_{ij}\}_{i=1,\ldots,\ell;j=1,\ldots,k_i}$. Then $s_{\brho}:X\to S_{\brho}$ coincides with $s_\kC:X\to S_{\kC}$ by the above claim. Then $T_{\brho}$ is a K\"ahler class.

By the definition of $\kC_{i}$, we can find a finite extension $L_{ij}$ of $K_{ij}$ such that    
\begin{itemize}
	\item  there is  a reductive representation   $\tau_{ij}:\pi_1(X)\to T(L_{ij})$ with $[\tau_{ij}]\in \kC(L_{ij})$;
	\item $p_i\circ\tau_{ij}=\varrho_{ij}$. 
\end{itemize}
Therefore, for the family $\btau:=\{\tau_{ij}\}_{i=1,\ldots,\ell;j=1,\ldots,k_i}$,   $s_{\btau}:X\to S_{\btau}$  coincides with $s_{\kC}$ by the above claim.  
Note that for any $i,j$, there exists an morphism  $e_{ij}:S_{\tau_{ij}}\to S_{\varrho_{ij}}$ such that $s_{\varrho_{ij}}:X\to S_{\varrho_{ij}}$ factors through $e_{ij}$. We also note that $e_{ij}^*T_{\varrho_{ij}}\leq T_{\tau_{ij}}$ for the canonical currents.  It follows that $T_{\brho}\leq T_{\btau}$ (note that $S_\btau=S_\brho=S_\kC$). Therefore, $\{T_{\btau}\}$ is a K\"ahler class. We prove the theorem.  
\end{proof}

\subsection{Some criterion for representation into tori} 
We recall a lemma in \cite[Lemma 5.3]{CDY22}.
\begin{lem}\label{lem:BT}
	Let $G$ be an almost simple algebraic group over the non-archimedean local field $K$.  Let $\Gamma\subset G(K)$ be a finitely generated subgroup so that
	\begin{itemize}
		\item  it is a Zariski dense subgroup in $G$,
		\item it is not contained in any bounded subgroup of $G(K)$. 
	\end{itemize} 
	Let $\Upsilon$ be a normal subgroup of $\Gamma$ which is \emph{bounded}. Then $\Upsilon$ must be finite. 
\end{lem}
This lemma enables us to prove the following result. 
\begin{lem}\label{lem:tori}
	Let $G$ be a reductive algebraic group over the non-archimedean local field $K$ of characteristic zero.  Let $X$ be a projective manifold and let $\varrho:\pi_1(X)\to G(K)$ be a Zariski dense representation. If $\varrho(\cD\pi_1(X))$ is bounded, then after replacing $K$ by some finite extension, for the  reductive representation  $\tau: \pi_1(X)\to G/\cD G(K)$ which is the composition of $\varrho$ with $G\to G/\cD G$,  the reduction map $s_\tau: X\to S_\tau$ coincides with  $s_\varrho:X\to S_\varrho$.   
\end{lem} 
\begin{proof}
	Since $G$ is reductive,  then after replacing $K$ by a finite extension,  there is an isogeny $G\to H_1\times\cdots\times H_k\times T$, where $H_i$ are almost simple algebraic groups over $K$ and $T=G/\cD G$ is an algebraic tori over $K$.%\footnote{\cite[Theorem 14.3]{Mil17} state this when $G$ is simply connected semisimple.}.
Write $G':=H_1\times\cdots\times H_k\times T$.  We denote by $\varrho': \pi_1(X)\to G'(K)$ the induced representation by the above isogeny. 
\begin{claim}\label{claim:samereduction}
	The Katzarkov-Eyssidieux reduction map $s_{\varrho}:X\to S_{\varrho}$ coincides with $s_{\varrho'}:X\to S_{\varrho'}$.
\end{claim}
\begin{proof}
It suffices to prove that, for any subgroup $\Gamma$ of $\pi_1(X)$, $\varrho(\Gamma)$ is bounded if and only if $\varrho'(\Gamma)$ is bounded. 	Note that we have the following short exact sequence of algebraic groups
	$$
	0\to \mu\to G\to G'\to 0
	$$
	where $\mu$ is finite.  Then
	 we have
	 	$$
	 0\to \mu(K)\to G(K)\stackrel{f}{\to} G'(K)\to H^1(K,\mu),
	 $$
	where $H^1(K,\mu)$  is the Galois cohomology.  Note that $\mu(K)$ is finite. Since $K$ is a finite extension of some $\bQ_p$, it follows that $H^1(K,\mu)$ is also finite. Therefore, $f:G(K)\to G'(K)$ has finite kernel and cokernel. Therefore, $\varrho(\Gamma)$ is bounded if and only if $\varrho'(\Gamma)$ is bounded.  
 \end{proof} 

Set $\Gamma:=\varrho'(\pi_1(X))$ and $\Upsilon:=\varrho'(\cD\pi_1(X))$.   Let  $\Upsilon_i\subset H_i(K)$ and $\Gamma_i$ be the image of $\Upsilon$ and $\Gamma$ under the   projection $G(K)\to H_i(K)$. Then $\Gamma_i$ is Zariski dense in $H_i$ and $\Upsilon_i\triangleleft \Gamma_i$ is also bounded. Furthermore, $\cD \Gamma_i=\Upsilon_i$. 
	\begin{claim}
		$\Gamma_i$ is bounded for every $i$. 
	\end{claim}
\begin{proof}Assuming a contradiction, let's suppose that some $\Gamma_i$ is unbounded. Since $\Upsilon_i\triangleleft \Gamma_i$ and $\Upsilon_i$ is bounded, we can refer to \cref{lem:BT} which states that $\Upsilon_i$ must be finite. We may replace $X$ with a finite étale cover, allowing us to assume that $\Upsilon_i$ is trivial. Consequently, $\Gamma_i$ becomes abelian, which contradicts the fact that $\Gamma_i$ is Zariski dense in the almost simple algebraic group $H_i$. 
\end{proof} 
Based on the previous claim, it follows that the induced representations $\tau_i:\pi_1(X)\to H_i(K)$ are all bounded for every $i$. Consequently, they do not contribute to the reduction map of $s_{\varrho'}:X\to S_{\varrho'}$. Therefore, the only contribution to $s_{\varrho'}$ comes from $\tau:\pi_1(X)\to T(K)$, where $\tau$ is the composition of $\varrho:\pi_1(X)\to G(K)$ and $G(K)\to T(K)$.

According to \cref{claim:samereduction}, we can conclude that $s_\varrho$ coincides with the reduction map $s_\tau:X\to S_\tau$ of $\tau:\pi_1(X)\to T(K)$. This establishes the lemma.
\end{proof}

 \subsection{Eyssidieux-Simpson Lefschetz theorem and its application}
Let $X$ be a compact K\"ahler manifold and let $V\subset H^0(X,\Omega_X^1)$ be a $\bC$-subspace.  Let $a:X\to \cA_X$ be the Albanese morphism of $X$.  Note that $a^*:H^0(\cA_X, \Omega_{\cA_X}^1)\to H^0(X,\Omega_X^1)$ is an isomorphism.  Write $V':=(a^*)^{-1}(V)$.    Define $B(V)\subset \cA_X$ to be the largest abelian subvariety of $\cA_X$ such that $\eta|_{{B(V)}}=0$ for every $\eta\in V'$.  Set $\cA_{X,V}:= \cA_X/B(V)$.  The \emph{partial Albanese morphism associated with $V$} is the composition of $a$ with the quotient map $\cA_X\to \cA_{X,V}$, denoted  by $g_V: X\to \cA_{X,V}$.   Note that there exists $V_0\subset H^0(\cA_{X,V}, \Omega_{\cA_{X,V}}^1)$ with $\dim_\bC V_0=\dim_\bC V$ such that $g_V^*V_0=V$.   Let $\widetilde{\cA_{X,V}}\to \cA_{X,V}$ be the universal covering and let $X_V$ be  $X\times_{\cA_{X,V}} \widetilde{\cA_{X,V}}$. Note that $V_0$ induces a natural linear map $\widetilde{\cA_{X,V}}\to V_0^*$. Its composition with $X_V\to \widetilde{\cA_{X,V}}$ and $g_V^*:V_0\to V$ gives rise to   a holomorphic map 
\begin{align}\label{eq:integration}
	\widetilde{g}_V:X_V\to V^*.
\end{align} 
Let $f: X\to S$ be the Stein factorization of $g_V:X\to \cA_{X,V}$ with $q:S\to \cA_{X,V}$ the finite morphism.  Set $\bV:=q^*V_0$. 
\begin{dfn}\label{def:perfect}
	$V$ is called \emph{perfect} if for any closed subvariety $Z\subset S$ of dimension $d\geq 1$, one has ${\rm Im}[\Lambda^d \bV\to H^0(Z, \Omega_Z^d)]\neq 0$.   
\end{dfn} 
The terminology of ``perfect $V$\fg in \cref{def:perfect} is called ``SSKB factorisable\fg in \cite[Lemme 5.1.6]{Eys04}. 

Let us recall the following Lefschetz theorem  by Eyssidieux, which is a generalization of previous work by Simpson \cite{Sim92}. This theorem plays a crucial role in the proofs of \cref{main2,main}.
\begin{thm}[\protecting{\cite[Proposition 5.1.7]{Eys04}}]\label{thm:ES} 
	Let $X$ be  a  compact K\"ahler normal space and let $V\subset H^0(X,\Omega_X^1)$ be a subspace.  Assume that 
	$$
	\Im[\Lambda^{\dim V}V\to H^0(X, \Omega_X^{\dim V})]=\eta\neq 0. 
	$$  
	Set $(\eta=0)=\cup_{i=1}^{k}Z_k$ where $Z_i$ are proper closed subvarieties of $X$.  For each $Z_i$, denote by $V_i:={\rm Im}[V\to H^0(Z_i, \Omega_{Z_i})]$.  Assume that $V_i$ is perfect for each $i$.  
	Then there are two possibilities which exclude  each other:
	\begin{itemize}
		\item either $V$ is perfect;
		\item or  for  the holomorphic map 	$\widetilde{g}_V:X_V\to V^*$  defined as \eqref{eq:integration}, $(X_{V}, \widetilde{g}_{V}^{-1}(t))$ is $1$-connected for any $t\in V^*$; i.e. $\widetilde{g}_{V}^{-1}(t)$  is connected and $\pi_1(\widetilde{g}_{V}^{-1}(t))\to \pi_1(X_{V})$ is surjective. 
	\end{itemize}
\end{thm}
We need the following version of the Castelnuovo-De Franchis theorem proved by Catanese.
\begin{thm}[Castelnuovo-De Franchis-Catanese]\label{thm:CD}
	Let $X$ be a compact K\"ahler normal space and let $W\subset H^0(X,\Omega_X)$ be the subspace of dimension $d\geq 2$ such that
	\begin{itemize}
		\item ${\rm Im}\big(\Lambda^d W\to H^0(X,\Omega^d_X)\big)=0$;
		\item for every hyperplane $W'\subset W$,  ${\rm Im}\big(\Lambda^{d-1} W'\to H^0(X,\Omega^{d-1}_X)\big)\neq 0$.
	\end{itemize} 
	Then there is a projective normal variety $S$ of dimension $d-1$ and a fibration $f:X\to S$ such that $W\subset f^*H^0(S,\Omega_S)$.  
\end{thm}
To apply \cref{thm:CD}, we need to show   the existence of a linear subspace $W\subset H^0(X,\Omega_X)$   as in the theorem.  
\begin{lem}\label{lem:non-empty}
	Let $X$ be a projective normal variety and let $V\subset H^0(X, \Omega_X)$. 	Let $r$ be the largest integer such that  $\Im[\Lambda^{r}V\to H^0(X, \Omega^{r}_X)]\neq 0$. Assume that $r<\dim_\bC V$. 
	 There exists $W\subset H^0(X, \Omega_X)$  such that 
	 \begin{thmlist} 
	 	\item     $2\leq \dim W\leq r+1$.   
	 	\item $\Im[\Lambda^{\dim W}W\to H^0(X, \Omega^{\dim  W}_X)]=0$;
	 	\item for every hyperplane $W'\subsetneq W$,  we always have $\Im[\Lambda^{\dim W-1}W'\to H^0(X, \Omega^{\dim  W-1}_X)]\neq0$. 
	 \end{thmlist} 
\end{lem}
\begin{proof}
By our assumption there exist $\{\omega_1,\ldots,\omega_r\}\subset V$ such that $\omega_1\wedge\cdots\wedge \omega_r\neq 0$. Let $W_0\subset V$ be the subspace generated by $\{\omega_1,\ldots,\omega_r\}$.  Since $r<\dim_\bC V$, there exists    $\omega\in V\setminus W_0$. 
	
	Pick a point $x\in X$ such that $\omega_1\wedge\cdots\wedge \omega_r(x)\neq 0$.  Then there exists  a coordinate system $(U;z_1,\ldots,z_n)$  centered at $x$ such that $dz_i=\omega_i$ for $i=1,\ldots,r$. Write $\omega=\sum_{i=1}^{n}a_i(z)dz_i$.  By our choice of $r$, we have
 $\omega_1\wedge\cdots\wedge \omega_r\wedge \omega= 0.$ 
It follows that 
	\begin{itemize}
		\item  $a_{j}(z)=0$ for $j=r+1,\ldots,n$;
		\item at least one of  $a_1(z),\ldots,a_r(z)$ is not constant.
	\end{itemize}Let $k+1$ be the transcendental degree of $\{1, a_1(z),\ldots,a_r(z)\}\subset \bC(U)$. Then $k\geq 1$.   We assume that $1, a_1(z),\ldots,a_k(z)$ is linearly independent for  the transcendental extension $\bC(U)/\bC$.     One can check by an easy linear algebra that the subspace  $W$ generated  $\{\omega_{1},\ldots,\omega_{k},\omega\}$ is an element of $E$.  The lemma is proved. 
\end{proof}

\begin{lem}\label{lem:rank}
	Let $X$ be a projective normal variety and let $V\subset H^0(X, \Omega_X)$.   	Let $r$ be the largest integer such that  $\Im[\Lambda^{r}V\to H^0(X, \Omega^{r}_X)]\neq 0$, which will be called \emph{generic rank of $V$}.  
	 Consider the partial Albanese morphism $g_{V}: X\to \cA_{X,V}$ induced by $V$.  Let $V_0\subset H^0(\cA_{X,V}, \Omega^1_{\cA_{X,V}})$ be the linear subspace such that $g_V^*V_0=V$.  Let $f: X\to S$ be the Stein factorization of $g_V$ with $q:S\to \cA_{X,V}$ the finite morphism.  Consider $\bV:=q^*V_0$. Assume that $${\rm Im}[\Lambda^{\dim Z} \bV\to H^0(Z, \Omega_Z^{\dim Z})]\neq 0$$ 
	for every proper closed subvariety $Z\subsetneq S$. Then there are two possibilities.
	\begin{itemize}
		\item either $${\rm Im}[\Lambda^{\dim S} \bV\to H^0(S, \Omega_S^{\dim S})]\neq 0;$$
		\item or $r=\dim_\bC V$.     % $$\Im[\Lambda^{\dim V}V\to H^0(X, \Omega_X^{\dim V})]\neq 0.$$  
	\end{itemize}
\end{lem}
\begin{proof}
	Assume that  both $${\rm Im}[\Lambda^{\dim S} \bV\to H^0(S, \Omega_S^{\dim S})]= 0,$$ and $r<\dim_\bC V$.  Therefore, $r<\dim S\leq \dim X$. 
	By \cref{lem:non-empty}	there is a subspace $W\subset V$ with $\dim_\bC W=k+1\leq r+1$  such that $\Im[\Lambda^{\dim W}W\to H^0(X, \Omega^{\dim  W}_Y)]=0$,  and  for any subspace $W'\subsetneq W$,  we always have $\Im[\Lambda^{\dim W'}W'\to H^0(X, \Omega^{\dim  W'}_X)]\neq0$.    By our assumption, we have $\dim_\bC W\leq \dim X$.  By \cref{thm:CD}, there is a fibration $p:X\to B$ with $B$ a projective normal variety with $\dim B=\dim W-1\leq \dim X-1$   such that $W\subset p^*H^0(B,\Omega_B^1)$.   In particular, the generic rank of the forms in $W$ is $\dim W-1$. 
	 	Consider the partial Albanese morphism $g_{W}:X\to \cA_{X,W}$ associated with $W$.  We shall prove that $p$ can be made as the Stein factorisation of $g_W$.
	 	
Note that each fiber of $p$ is contracted by $g_W$. Therefore, we have a factorisation $X\stackrel{p}{\to} B\stackrel{h}{\to} \cA_{X,W}$. Note that there exists a linear space $W_0\subset H^0(\cA_{X,W}, \Omega^1_{\cA_{X,W}})$   such that $W=g_W^*W_0$.  %Then $W=p^*(h^*W_0)$. 
 If $\dim h(B)<\dim B$, then the generic rank of $W$ is less or equal to $\dim h(B)$.    This contradicts with \cref{thm:CD}. Therefore, $\dim h(B)=\dim B$.  Let $X\stackrel{p'}{\to} B'\to \cA_{X,W}$ be  the Stein   factorisation of $g_W$. Then there exists a birational morphism $\nu:B\to B'$ such that $p'= \nu\circ p$.  We can thus replace $B$ by $B'$, and $p$ by $p'$.      
 
Recall that  $f:X\to S$ is the Stein factorisation of the partial Albanese morphism $g_V:X\to \cA_{X,V}$  associated with $V$.  As $g_W$ factors through the natural quotient map $\cA_{X,V}\to \cA_{X,W}$, it follows that   $p:X\to B$   factors through   $X\stackrel{f}{\to }S\stackrel{\nu}{\to} B$.   
	
	Assume that $\dim S=\dim B$. Then $\nu$ is birational. Since   $\dim B=\dim W-1$ and the generic rank of $W$ is $\dim W-1$,   it follows that $${\rm Im}[\Lambda^{\dim S} \bV\to H^0(S, \Omega_S^{\dim S})]\neq 0.$$ This  contradicts with our assumption at the beginning.  Hence  $\dim S>\dim B$. 
	
	Let $Z$ be a general fiber of $\nu$ which is positive-dimensional. %Then $Z=f(F)$   is a positive dimensional proper subvariety of $S$.  
	 Since $W\subset p^*H^0(B,\Omega_B^1)$, and we have assumed that the generic rank of $\bV$ is less than $ \dim S$, it follows that the generic rank of $\Im [\bV\to H^0(Z, \Omega_Z^1)]$ is less than $\dim Z$. This implies that $${\rm Im}[\Lambda^{\dim Z} \bV\to H^0(Z, \Omega_Z^{\dim Z})]=0,$$  which contradicts with our assumption.  Therefore we obtain a contradiction. The lemma is proved. 
\end{proof}

\begin{rem}\label{rem:positive}
Let $Y$ be a normal projective variety. 	Let $\brho=\{\varrho_i:\pi_1(Y)\to {\rm GL}_N(K_i)\}_{i=1,\ldots,k}$ be a family of reductive representations where $K_i$ are non-archimedean local field. Let $\pi: X\to Y$ be a Galois cover dominating all spectral covers induced by $\varrho_i$. Let $V\subset H^0(X, \Omega_X)$ be the set of all spectral forms (cf. \cref{sec:KE} for definitions). We use the same notations as in \cref{lem:rank}. Considering Katzarkov-Eyssidieux reduction maps $s_{\brho}:Y\to S_\brho$ and $s_{\pi^*\brho}:X\to S_{\pi^*\brho}$. One can check that, for every closed subvariety $Z\subset S_{\brho}$,   $\{T_{\brho}^{\dim Z}\}\cdot Z>0$ if and only if for any closed subvariety $W\subset S_{\pi^*\brho}$ dominating $Z$ under $\sigma_\pi: S_{\pi^*\brho}\to S_{\brho}$ defined in \eqref{eq:functorial}, one has  
	$${\rm Im}[\Lambda^{\dim W} \bV\to H^0(W, \Omega_W^{\dim W})]\neq 0.$$ 
	In particular, $V$ is perfect if and only if $\{T_{\brho}\}$  is a K\"ahler class by \cref{thm:DHP}.
\end{rem}

\begin{thm}\label{thm:dichotomy}
	Let $X$ be a smooth projective variety and let $\brho:=\{\varrho_i:\pi_1(X)\to {\rm GL}_N(K_i)\}_{i=1,\ldots,k}$ be a family of   reductive representations where $K_i$ is a non-archimedean local field.   
Let  $s_\brho:X\to S_\brho$  be the Katzarkov-Eyssidieux reduction map. %Stein factorization of  $(s_{\varrho_1},\ldots,s_{\varrho_k}):X\to S_{\varrho_1}\times\cdots\times S_{\varrho_k}$ where $s_{\varrho_i}:X\to S_{\varrho_i}$ denotes the reduction map associated with $\varrho_i$ with $p_i: S\to S_{\varrho_i}$.  
	 Let   $T_\brho$  be the canonical $(1,1)$-current  on   $S_\brho$   associated with $\brho$ defined in \cref{def:canonical2}.  Denote by $H_i$ the Zariski closure of $\varrho_i(\pi_1(X))$.   	   
Assume that for any proper closed subvariety  $\Sigma\subsetneq S_\brho$, one has $\{T_\brho\}^{\dim \Sigma}\cdot \Sigma>0$. Then 
	\begin{itemize}
		\item either $\{T_\brho\}^{\dim S_\brho}\cdot S_\brho>0$;
		\item or    the reduction map $s_{\sigma_i}: X\to S_{\sigma_i}$ coincides with  $s_{\varrho_i}:X\to S_{\varrho_i}$ for each $i$, where $\sigma_i:\pi_1(X)\to (H_i/\cD H_i)(K_i)$  is the composition of $\varrho_i$ with the group  homomorphism $H_i\to H_i/\cD H_i$. 
	\end{itemize}  
\end{thm}
\begin{proof}
	Assume that $\{T_\brho\}^{\dim S_\brho}\cdot S_{\brho}=0$. Let $Y\to X$ be a Galois cover which dominates all spectral covers of $\varrho_i$.  We pull back all the spectral one forms on $Y$ to obtain a subspace $V\subset H^0(Y,\Omega_Y^1)$.   Consider the partial Albanese morphism   $g_{V}: Y\to \cA_{Y,V}$ associated to $V$, then  $s_{\pi^*\brho}: Y\to S_{\pi^*\brho}$ is  its Stein factorization with $q: S_{\pi^*\brho}\to \cA_{Y,V}$ the finite morphism. Note that there is a $\bC$-linear subspace $\bV\subset H^0( {S_{\pi^*\brho}},\Omega^1_{S_{\pi^*\brho}})$ such that $s_{\pi^*\brho}^*\bV=V$. 
	\begin{equation*}
		\begin{tikzcd}
			Y\arrow[r, "\pi"] \arrow[d, "s_{\pi^*\brho}"]\arrow[dd, "s_{\pi^*\varrho_i}"', bend right=40] &  X\arrow[dd, "s_{\varrho_i}", bend left=40]\arrow[d, "s_\brho"]\\	S_{\pi^*\brho}  \arrow[r, "\sigma_\pi"] \arrow[d, "q_i"]& S_\brho\arrow[d, "p_i"']\\
			S_{\pi^*\varrho_i} \arrow[r, ] & S_{\varrho_i}
		\end{tikzcd}
	\end{equation*}
	Note that  $\sigma_{\pi}$ is   finite surjective morphism.  By \cref{lem:functorial}  we have $T_{\pi^*\brho}=\sigma_\pi^*T_{\brho}$. By our assumption, for an  proper closed subvariety  $\Xi\subsetneq S_\brho$, one has $\{T_\brho\}^{\dim \Xi}\cdot \Xi>0$.  Hence for an  proper closed subvariety  $\Xi\subsetneq S_{\pi^*\brho}$, one has $\{T_{\pi^*\brho}\} ^{\dim \Xi}\cdot \Xi>0$. According to \cref{rem:positive}, this implies that  $${\rm Im}[\Lambda^{\dim \Xi} \bV\to H^0(\Xi, \Omega_\Xi^{\dim \Xi})]\neq 0.$$  
	Since $\{T_\brho\}^{\dim S}\cdot S=0$, it follows that  $\{T_{\pi^*\brho}\}^{\dim S_{\pi^*\brho}}\cdot S_{\pi^*\brho}=0$.    This implies that  $$\Im[\Lambda^{\dim S_{\pi^*\brho}}\bV\to H^0(S_{\pi^*\brho}, \Omega_{S_{\pi^*\brho}}^{\dim S_{\pi^*\brho}})]=0.$$   Let $r$ be the generic rank $V$. According to \cref{rem:positive},   we have $r= \dim S_{\pi^*\brho}-1$.  By \cref{lem:rank}, we have $r=\dim_\bC V$.  Therefore, $\Im[\Lambda^{r}V\to H^0(Y, \Omega_Y^{r})]\simeq \bC$. 
	\begin{claim} 
For any non-zero   $\eta\in \Im[\Lambda^{r}V\to H^0(Y, \Omega_Y^{r})]$,  
 each irreducible component $Z'$ of $(\eta=0)$ satisfies that $s_{\pi^*\brho}(Z')$ is a proper subvariety of $S_{\pi^*\brho}$. 
	\end{claim}
	\begin{proof}
		Assume that this is not the case.  Let $Z\to Z'$ be a desingularization.    Set  $V':= \Im[V\to H^0(Z,\Omega_Z^1)]$. Denote by $r'$ the generic rank of $V'$. Then $r'<r$ as  $Z'$ is  an irreducible component  of $(\eta=0)$. 
		Write $\iota:Z\to Y$ and  $g: Z\to X$ for the natural map. Then the Katzarkov-Eyssidieux reduction $s_{g^*\brho}: Z\to S_{g^*\brho} $ associated with $g^*\brho$ is  the Stein factorization of the partial Albanese morphism $g_{V'}:Z\to \cA_{Z,V'}$. We have the diagram
		\begin{equation*}
			\begin{tikzcd}
				Z \arrow[r, "\iota"]\arrow[rr, "g", bend left=20]  \arrow[d, "s_{g^*\brho}"] &  Y\arrow[d, "s_{\pi^*\brho}"]\arrow[r]&   X\arrow[d, "s_\brho"]\\
				S_{g^*\brho} \arrow[r, "\sigma_{\iota}"]\arrow[rr,"\sigma_g"', bend right=20]  & S_{\pi^*\brho}\arrow[r]&S_{\brho}  \\ 
			\end{tikzcd}
		\end{equation*}  
		such that $\sigma_{\iota}$ is a finite  \emph{surjective} morphism as we assume that $s_{\pi^*\brho}(Z')=S_{\pi^*\brho}$.    Let $\Sigma\subsetneq S_{g^*\brho}$ be  a proper closed subvariety.  Let $\Sigma':=\sigma_g(\Sigma)$. Since  $\{T_\brho\}^{\dim \Sigma'}\cdot \Sigma'>0$ by our assumption, by \cref{lem:functorial} $\{T_{g^*\brho}\}^{\dim \Sigma}\cdot \Sigma>0$.
By \cref{rem:positive}, it follows that the generic rank $r'$ of $V'$ is equal to $\dim S_{g^*\brho}-1=\dim S_{\pi^*\brho}-1$. This contradicts with the fact that $r'<r=\dim S_{\pi^*\brho}-1$. The claim  is proved.     % for $V'\subset H^0(Z, \Omega_Z^1)$ conditions in   \cref{lem:rank} are fulfilled. Note that the second possibility does not happen.  By the first possibility of \cref{lem:rank} occurs, and thus $\{T_{g^*\brho}\}$ is also big, which implies that $T_{\brho}$ is big by \cref{lem:functorial}. This contradicts with our assumption at the beginning.
	\end{proof}
By the above claim, $s_{\pi^*\brho}(Z')$ is a proper subvariety of $S_{\pi^*\brho}$. 
Therefore, we have  $\{T_\brho\}^{\dim S_{g^*\brho}}\cdot S_{g^*\brho}>0$ . 	
Hence for each irreducible component $Z'$  
of $(\eta=0)$,  $  \Im[V\to H^0(Z,\Omega_Z^1)]$ is perfect by \cref{rem:positive} once again. We can apply \cref{thm:ES} to conclude that  for  the holomorphic map $\widetilde{g}_{V}: Y_V\to V^*$  defined as \eqref{eq:integration}, $(Y_V, \widetilde{g}_{V}^{-1}(t))$ is $1$-connected for any $t\in V^*$.      For the covering $ Y_V\to Y$, we know that ${\rm Im}[\pi_1(Y_V)\to \pi_1(Y)]$ contains the derived subgroup $\cD\pi_1(Y)$ of $\pi_1(Y)$. Then $\pi^*\varrho_i(\Im[\pi_1(Y_V)\to \pi_1(Y)])$ contains $\pi^*\varrho_i(\cD\pi_1(Y))$.       On the other hand, since $(Y_V, \widetilde{g}_{V}^{-1}(t))$ is $1$-connected for any $t\in V^*$, it follows that $\pi^*\varrho_i(\Im[\pi_1(\widetilde{g}_{V}^{-1}(t))\to \pi_1(Y)])$ contains $\pi^*\varrho_i(\cD\pi_1(Y))$. Note that $V$ is consists of all the spectral forms of $\pi^*\varrho_i$ for all $i$, hence each $\pi^*\varrho_i$-equivariant harmonic mapping $u_i$ vanishes over each connected component $p^{-1}(\widetilde{g}_{V}^{-1}(t))$ where $p:\widetilde{Y}\to Y$ is the universal covering. Then $\pi^*\varrho_i(\Im[\pi_1(\widetilde{g}_{V}^{-1}(t))\to \pi_1(Y)])$ fixes a point $P$ in the Bruhat-Tits building, which implies that it is bounded. Therefore, $\pi^*\varrho_i(\cD\pi_1(Y))$ is also bounded.  Note that the image of $\pi_1(Y)\to \pi_1(X)$  is a finite index subgroup of $\pi_1(X)$. 
Hence $ \varrho_i(\cD\pi_1(X))$    is also bounded for each $\varrho_i$. The theorem then follows from \cref{lem:tori}.
\end{proof}
 
 \subsection{A factorization theorem}
 As an application  of \cref{thm:dichotomy}, we will  prove the following factorization theorem which partially generalizes previous theorem by Corlette-Simpson \cite{CS08}.  This result is also a warm-up for the proof of \cref{thm:crucial}.
 \begin{thm}\label{thm:factor}
 	Let $X$ be a  projective normal variety and let $G$ be an almost simple algebraic group defined over  non-archimedean local field $K$ of characteristic zero.  Assume that $\varrho:\pi_1(X)\to G(K)$ is an unbounded,  Zariski dense representation such that for   any closed subvariety $Z$ of $X$,   the Zariski closure of  $\varrho({\rm Im}[\pi_1(Z_{\rm norm})\to \pi_1(X)])$ is a semisimple algebraic group (including the finite group). Then after replacing $X$ by a finite \'etale cover  and a birational modification, there exists a fibration $g:X\to Y$, a representation $\tau:\pi_1(Y)\to G(K)$ such that   $g^*\tau=\varrho$, and  $\dim Y\leq {\rm rank}_KG$. 
 \end{thm}
 \begin{proof}
 By \cite[Proposition 2.5]{CDY22},  after replacing $X$ by a finite \'etale cover and a birational modification, there exists a fibration $g:X\to Y$ over a smooth projective variety $Y$, and  a big  representation $\tau:\pi_1(Y)\to G(K)$ such that we have   $g^*\tau=\varrho$. Since $\varrho(\pi_1(X))$ is Zariski dense, it follows that the Zariski closure $G'$ of $\tau(\pi_1(Y))$  contains the identity component of $G$, which is also almost simple. We   replace $G'$ by $G$ and thus $\tau$ is Zariski dense.  Let $s_\tau:Y\to S_\tau$ be the Katzarkov-Eyssidieux reduction of $\tau$.    
 	\begin{claim}\label{claim:kahler}
 		The $(1,1)$-class $\{T_{\tau}\}$ on $S_{\tau}$ is K\"ahler, where $T_{\tau}$  is the canonical current on $S_\tau$ associated to $\tau$. 
 	\end{claim}
 	\begin{proof}
 		By \cref{thm:DHP}, it is equivalent to prove that for any closed subvariety $\Sigma\subset S_\tau$, $\int_\Sigma \{T_\tau\}^{\dim \Sigma}>0$.   We will prove it by induction on $\dim \Sigma$.    
 		
 		\medspace
 		
 		\noindent  \textbf{Induction}. Assume that for every closed subvariety $\Sigma\subset S_\tau$ of dimension $\leq r-1$,   $\{T_{\tau}\}^{\dim \Sigma}\cdot \Sigma>0$. 
 		
 		Let $\Sigma$ be any closed subvariety of $S_\tau$ with $\dim \Sigma=r$. Let $Z$ be a desingularization of any irreducible component  in $s_\tau^{-1}(\Sigma)$ which is surjective over $\Sigma$.  Denote by $f:Z\to Y$. 
 		\begin{equation*}
 			\begin{tikzcd}
 				Z\arrow[r, "f"] \arrow[d, "s_{f^*\tau}"] & Y \arrow[d, "s_\tau"]\\
 				S_{f^*\tau} \arrow[r, "\sigma_f"] &  S_\tau 
 			\end{tikzcd}
 		\end{equation*}
 	By \cref{lem:functorial},  $\sigma_f$ is a finite morphism whose image is $\Sigma$ and $T_{f^*\tau}=\sigma_f^*T_\tau$.

 		We first prove the induction for  $\dim \Sigma=1$. In this case $\dim S_{f^*\tau}=1$.   Since the spectral forms associated to $f^*\tau$ are not constant,  it  follows that  $T_{f^*\tau}$ is big.  By \cref{lem:functorial}, $\{T_{\tau}|_\Sigma\}$ is big. Therefore, we prove the induction when $\dim \Sigma=1$. 
 		
 	Assume now the induction holds for closed subvariety $\Sigma\subset S_\tau$ with $\dim \Sigma\leq r-1$. Let us treat  the  case $\dim \Sigma=r$.  	By \cref{lem:functorial} and the induction, we know that for any closed proper positive dimensional subvariety $\Xi\subset S_{f^*\tau}$, we have $\{T_{f^*\tau}\}^{\dim \Xi}\cdot \Xi>0$.   Note that the conditions in \cref{thm:dichotomy}  for $f^*\tau$ is  fulfilled.  Therefore, there are two possibilities: 	
 		\begin{itemize}
 			\item either $\{T_{f^*\tau}\}^{r}\cdot S_{f^*\tau}>0$; 
 			\item or    the reduction map $s_{f^*\tau}: Z\to S_{f^*\tau}$ coincides with  $s_{\nu}:Z\to S_\nu$, where $\nu:\pi_1(Z)\to (H/\cD H)(K)$  is the composition of $\tau$ with the group  homomorphism $H\to H/\cD H$. Here $H$ is the Zariski closure of $f^*\tau$.  
 		\end{itemize}  
 		If the first case happens, by \cref{lem:functorial} again we have $\int_\Sigma \{T_\tau\}^{\dim \Sigma}>0$.  we finish the proof of the induction for $\Sigma\subset S_\tau$ with $\dim \Sigma=r$. Assume that the second situation occurs. Since  $H$ is assumed to be semisimple, it follows that $H/\cD H$ finite. Therefore, $\nu$ is bounded and thus $S_{f^*\tau}$ is a point. This contradicts with the fact that $\dim S_{f^*\tau}=\dim \Sigma=r>0$.  Therefore, the second situation cannot occur. We finish the proof of the induction. The claim is proved.  
 	\end{proof}
 	This claim in particular implies that 
 	\begin{enumerate}
 \item [($\diamondsuit$)] 
 	the \emph{generic rank}  $r$ of the multivalued holomorphic 1-forms on $Y$ induced by the differential of  harmonic mappings of $\tau$ is equal to $\dim S_\tau$.  
\end{enumerate}
Since $G$ is almost simple, by \cite[Theorem 6.1]{CDY22} we know that the Katzarkov-Eyssidieux reduction map $s_\tau:Y\to S_{\tau}$ is birational.       
Therefore $r=  \dim Y$.  On the other hand, we note that $r$  is less than or equal to the dimension of the Bruhat-Tits building $\Delta(G)_K$, which is equal to ${\rm rank}_KG$. It follows that $\dim Y\leq \rank_KG$. The theorem is proved. 
 \end{proof}
 
 \begin{rem}
 The authors do not know the proof of $(\diamondsuit)$ which does not go through the stronger conclusion \cref{claim:kahler}.
In \cite[p. 148]{Zuo96}, Zuo claimed $(\diamondsuit)$ for much weaker assumption for $\varrho$ without providing a proof.  Using our terminology, it can be stated as follows: 
 \begin{enumerate}
 \item [($\spadesuit$)]  let $X$, $G$ and $K$ be as in \Cref{thm:factor}. If  $\varrho:\pi_1(X)\to G(K)$  is a Zariski dense, unbounded, and big representation, then the generic rank of the multivalued form equals to   $\dim X$. \label{Zuo}
 \end{enumerate}  Zuo's claim is essential in his proof of the result that such $X$ is of general type.  While we cannot provide a counter-example to Claim $(\spadesuit)$, we believe his statement is too strong. Indeed, as we will see in next subsection, we have to apply \cref{thm:positivity} to  add sufficiently many new unbounded representations of $\pi_1(X)$ into non-archimedean local fields such that the collection of all the induced multivalued forms have the rank equal to $\dim X$.  
 
 A new proof, along with a much more general statement on the hyperbolicity of $X$, can be found in  \cite[Theorem I]{CDY22}. We  emphasize that  in \cite{CDY22} we avoid the use of Claim $(\spadesuit)$, and use a completely different method. 
 \end{rem}

\subsection{Constructing K\"ahler classes via representations into non-archimedean fields} 
Let $X$ be a smooth projective variety.  In this subsection we will prove a more general theorem than \cref{thm:positivity}. 
\begin{thm} \label{thm:crucial}
	Let $\kC$ be absolutely constructible subset of $M_{\rm B}(X, N)(\bC)$. Then   there is a family of   representations  $\btau:=\{\tau_i:\pi_1(X)\to {\rm GL}_N(K_i)\}_{i=1,\ldots,M}$ where $K_i$ are non-archimedean local fields   such that
	\begin{itemize}
		\item For each $i=1,\ldots,M$, $[\tau_i]\in \kC(K_i)$; 
		\item  The reduction map $s_{\btau}:X\to S_{\btau}$ of $\btau$ coincides with $s_{\kC}:X\to S_\kC$ defined in \cref{def:reduction ac}.
		\item For the canonical current $T_{\btau}$ defined over $S_{\kC}$, $\{T_{\btau}\}$ is a K\"ahler class. 
	\end{itemize} 
\end{thm}

To prove this theorem, we need several preparations.
Let $\tau:\pi_1(X)\to \mathrm{GL}_N(K)$ be a reductive representation, where $K$ is a non-archimedian local field of characteristic zero.
 Then we get the spectral covering $\pi:\xsp\to X$, the spectral one forms $\eta_1,\ldots,\eta_l\in H^0(\xsp,\pi^*\Omega^1_X)$, the reduction map $s_{\tau}:X\to S_{\tau}$ and the canonical current $T_{\tau}$ on $S_{\tau}$ (cf. \cref{def:canonical}).

We define a Zariski closed set $W_{\tau}\subset TX$ as follows.
Each spectral one form $\eta_i\in  H^0(\xsp,\pi^*\Omega^1_X)$ determines a section of $H^0(\pi^*TX,\mathcal{O}_{\pi^*TX})$, where $\pi^*TX=TX\times_X\xsp $.
Hence we get the zero locus $\{\eta_i=0\}\subset \pi^*TX$.
We set $W_{\tau}'=\cap_{i=1}^l\{\eta_i=0\}$.
Then $W_{\tau}'\subset \pi^*TX$ is Zariski closed.
We set $W_{\tau}=p(W_{\tau}')$.
Since the natural map $p:\pi^*TX\to TX$ is finite, $W_{\tau}$ is Zariski closed.
%By the construction, $W_{\tau}$ is $\mathbb C^*$ invariant.
%We define the quotient $V_{\tau}\subset PTX$ of $W_{\tau}$.
Note that the set $\{\eta_1,\ldots,\eta_l\}$ is invariant under the action of the Galois group $\mathrm{Gal}(\xsp/X)$.
Hence we have
\begin{equation}\label{eqn:20240417}
p^{-1}(W_{\tau})=W_{\tau}'.
\end{equation}

Let $f:Z\to X$ be a morphism from another smooth projective variety $Z$.
Then we get $W_{f^*\tau}\subset TZ$ from $f^*\tau:\pi_1(Z)\to \mathrm{GL}_N(K)$.

\begin{lem}\label{lem:20240411}
Under the map $f_*:TZ\to TX$, we have $W_{f^*\tau}=(f_*)^{-1}W_{\tau}$.
\end{lem}

\begin{proof}
We have a natural map $\zsp\to \xsp$ induced from $f:Z\to X$.
This induces the following commutative diagram:
 \begin{equation*}
   	\begin{tikzcd}
   		\pi_Z^*TZ \arrow[r, " g"]  \arrow[d, "	p_Z"] &  \pi^*TX\arrow[d, "p"]\\
   		TZ \arrow[r, "f_*"]  & TX\\ 
   	\end{tikzcd}
   \end{equation*}  
   Let $\{\sigma_1,\ldots,\sigma_k\}\subset H^0(\pi_Z^*TZ,\mathcal{O}_{\pi_Z^*TZ})$ be the set of spectral one forms defined by $f^*\tau:\pi_1(Z)\to\mathrm{GL}_N(K)$.
   Then we have $\{\eta_1\circ g,\ldots,\eta_l\circ g\}=\{\sigma_1,\ldots,\sigma_k\}$ as sets.
   Hence $g^{-1}W'_{\tau}=W'_{f^*\tau}$.
   By \eqref{eqn:20240417}, we have $p_Z^{-1}(f_*)^{-1}W_{\tau}=W'_{f^*\tau}$.
   Since $\pi_Z:\zsp\to Z$ is surjective, $p_Z$ is also surjective.
   Hence $W_{f^*\tau}=p_Z(W'_{f^*\tau})=(f_*)^{-1}W_{\tau}$.
\end{proof}

 Let $\kC$ be an absolutely constructible subset of $M_{\rm B}(X, N)(\bC)$.
 Let $\btau=\{\tau_1,\ldots,\tau_k\}$ be a set of reductive representations over non-archimedian local fields of characteristic zero such that $[\tau_i]\in\kC$.
 Then we get $s_{\btau}:X\to S_{\btau}$ and $T_{\btau}$ (cf. \cref{def:canonical2}).
We define $W_{\btau}\subset TX$ by $W_{\btau}=\cap_{i=1}^kW_{\tau_i}$.

\begin{cor}\label{cor:20240412}
Let $f:Z\to X$ be a morphism of smooth projective varieties.
 Then we have $(f_*)^{-1}(W_{\btau})=W_{f^*\btau}$ under the map $f_*:TZ\to TX$.
\end{cor}
 
 \begin{proof}
 This follows directly from \cref{lem:20240411}.
 \end{proof}

 \begin{lem}\label{lem:20240412}
 Assume that $\dim S_{\btau}=\dim X$.
 Then $\{T_{\btau}\}^{\dim S_{\btau}}\cdot S_{\btau}>0$ if and only if  $W_{\btau}\cap T_xX=\{0\}$ for generic points $x\in X$. 
 \end{lem}
 
 \begin{proof}
 Let $\pi:Y\to X$ be a Galois cover dominating all spectral covers induced by $\tau_1,\ldots,\tau_l$.
 Let $V\subset H^0(Y,\Omega_Y)$ be the set of all spectral forms.
% We have $s_{\pi^*\btau}:Y\to S_{\pi^*\btau}$ so that $\dim Y=\dim S_{\pi^*\btau}$.
% Let $\bV\subset H^0(S_{\pi^*\btau},\Omega_{S_{\pi^*\btau}})$ be the set such that $s_{\pi^*\btau}^*\bV=V$.
 We have $\{T_{\btau}\}^{\dim S_{\btau}}\cdot S_{\btau}>0$ 
%iff (as in \cref{rem:positive})
% $${\rm Im}[\Lambda^{\dim S_{\pi^*\btau}} \bV\to H^0(S_{\pi^*\btau}, \Omega_{ S_{\pi^*\btau}}^{\dim S_{\pi^*\btau}})]\neq 0.$$ 
% This is true 
if and only if 
 $${\rm Im}[\Lambda^{\dim Y} V\to H^0(Y, \Omega_{Y}^{\dim Y})]\neq 0.$$ 
 This is true if and only if for generic smooth points $y\in Y$, the induced map $V\to (T_yY)^*$ is surjective, which is equivalent to $W_{\btau}\cap T_xX=\{0\}$ for generic points $x\in X$.
 \end{proof}

 We say $\btau$ is \emph{full with respect to $\kC$} if the following holds:
 \begin{itemize}
 \item $s_{\btau}:X\to S_{\btau}$ coincides with $s_{\kC}:X\to S_{\kC}$,
 \item $W_{\btau}$ is \emph{minimal}, i.e., for every reductive representation $\sigma$ such that $[\sigma]\in \kC$, we have $W_{\btau}\subset W_{\sigma}$.
 \end{itemize}
 
 \begin{lem}\label{lem:20240409}
 Assume that $s_{\btau}:X\to S_{\btau}$ coincides with $s_{\kC}:X\to S_{\kC}$ and that $\{T_{\btau}\}$ is a K\"ahler class. 
 Then for every $\btau'$ which is full with respect to $\kC$, $\{T_{\btau'}\}$ is a K\"ahler class. 
 \end{lem}

\begin{proof}
Let $\Sigma\subset S_{\kC}$ be a closed subvariety.
We shall show $\{T_{\btau'}\}^{\dim \Sigma}\cdot \Sigma>0$.
Let $\Sigma'\subset X$ be a closed subvariety such that $\dim \Sigma'=\dim \Sigma$ and $s_{\btau}(\Sigma')=\Sigma$.
Let $Z\to \Sigma'$ be a desingularization and let $f:Z\to X$ be the induced map.
We consider $s_{f^*\btau}:Z\to S_{f^*\btau}$.
Since $\{T_{\btau}\}$ is a K\"ahler class, we have $\{T_{f^*\btau}\}^{\dim S_{f^*\btau}}\cdot S_{f^*\btau}>0$ (cf. \cref{lem:functorial}).
By \cref{lem:20240412}, we have $W_{f^*\btau}\cap T_xZ=\{0\}$ for generic points $x\in Z$.

Since $\btau'$ is full, we have $W_{\btau'}\subset W_{\btau}$.
Hence by \cref{cor:20240412}, we have $W_{f^*\btau'}\subset W_{f^*\btau}$.
Hence we have $W_{f^*\btau'}\cap T_xZ=\{0\}$ for generic points $x\in Z$.
Hence by \cref{lem:20240412}, we have $\{T_{f^*\btau'}\}^{\dim S_{f^*\btau'}}\cdot S_{f^*\btau'}>0$.
Thus by \cref{lem:functorial}, we have $\{T_{\btau'}\}^{\dim \Sigma}\cdot \Sigma>0$.
\end{proof}

 \begin{lem}\label{lem:20240410}
% Let $\kC$ be absolutely constructible subset of $M_{\rm B}(X, N)(\bC)$. 
 Let $\btau=\{\tau_1,\ldots,\tau_k\}$ be full with respect to $\kC$.
 Suppose $\{T_{\btau}\}^{\dim \Sigma}\cdot \Sigma>0$ for every proper closed subvariety $\Sigma\subsetneqq S_{\btau}$.
 Then $\{T_{\btau}\}$ is a K\"ahler class. 
 \end{lem}
 
 \begin{proof}
We will apply \cref{thm:dichotomy}.
We have two possibilities.
Assume that the first case occurs.
In this case, we have $\{T_{\btau}\}^{\dim S_{\btau}}\cdot S_{\btau}>0$ as desired and by \cref{thm:DHP}, $\{T_\btau\}$ is K\"ahler. 
Therefore, in the following, we assume that the second case occurs.
In this case, we apply \cref{thm:positivity}.

Let us denote by $H_i$  the Zariski closure of $\tau_i(\pi_1(X))$ for each $i=1,\ldots,k$, which is a reductive algebraic group over $L_i$.  
By \cite{MO}, there is some number field $k_i$ and some non-archimedean place $v_i$ of $k_i$ such that $L_i=(k_i)_{v_i}$ and $H_i$ is defined over $k_i$.   
Denote $T_i:=H_i/\cD H_i$, which is a tori. Replacing $k_i$   by a finite extension, we may assume that $T_i$ is  defined over  $k_i$. 
Then by \cref{thm:dichotomy}, $s_{\tau_i}:X\to S_{\tau_i}$ coincides with $s_{\sigma_i}:X\to S_{\sigma_i}$, where $\sigma_i:\pi_1(X)\to T_i(L_i)$ is the composition of $\tau_i:\pi_1(X)\to H_i(L_i)$ with the group  homomorphism $H_i\to T_i$.    
Consider the morphisms of affine  $k_i$-schemes of finite type 
\begin{equation}\label{diagram00}
	 \begin{tikzcd}
	 	 & M_{\rm B}(X, N)\\ 
	M_{\rm B}(X, T_i)    &	M_{\rm B}(X, H_i)\arrow[u]\arrow[l] 
	 \end{tikzcd}
\end{equation}
Then by \cref{thm:S2}, $\kC\subset M_{\rm B}(X, N)(\bC)$ is transferred via the diagram  \eqref{diagram00}  to some  absolutely constructible subset $\kC_i$ of $M_{\rm B}(X, T_i) $.  
Consider the reduction map $s_{\kC_i}:X\to S_{\kC_i}$ defined in \cref{def:reduction ac}.  
Let $f:X\to S$ be the Stein factorisation of $s_{\kC_1}\times\cdots\times s_{\kC_k}: X\to S_{\kC_1}\times \cdots\times S_{\kC_k}$.   
\begin{claim}\label{claim:coincide}
	The reduction map $s_{\btau}:X\to S_{\btau}$ coincides with $f:X\to S$.
\end{claim}
\begin{proof}
Note that the reduction map $s_{\sigma_i}: X\to S_{\sigma_i}$ coincides with  $s_{\tau_i}:X\to S_{\tau_i}$. 
By \eqref{diagram00} and the definition of $\kC_i$,   $[\sigma_i]\in \kC_i(L_i)$. 
Therefore, $s_{\sigma_i}$ factors through $s_{\kC_i}$.  
Hence $s_{\tau_i}:X\to S_{\tau_i}$ factors through  $X\to S$.
Thus $s_{\btau}$ factors through   $X\stackrel{f}{\to}S\stackrel{q}{\to}S_{\btau}$.
By the construction of $\kC_i$, the map $s_{\kC_i}:X\to S_{\kC_i}$ factors through $s_{\kC_i}:X\to S_{\kC}\to S_{\kC_i}$.
By $S_{\kC}=S_{\btau}$, we have $X\to S_{\btau}\to S$.
\end{proof}
Since $T_i$ are all algebraic tori defined over number fields $k_i$, we apply \cref{thm:positivity} to conclude that there exists   a family of  reductive representations  $\brho_i:=\{\varrho_{ij}:\pi_1(X)\to T_i(K_{ij})\}_{j=1,\ldots,n_i}$ with $K_{ij}$ non-archimedean local field  such that
\begin{enumerate}[label=(\arabic*)]
	\item For each $i=1,\ldots,k; j=1,\ldots,n_i$, $[\varrho_{ij}]\in \kC_i(K_{ij})$; 
	\item \label{item:dominant} The reduction map $s_{\brho_i}:X\to S_{\brho_i}$ of $\brho_i$ coincides with $s_{\kC_i}:X\to S_{\kC_i}$;
	\item \label{item:kahler}for the canonical current $T_{\brho_i}$   over $S_{\kC_i}$ associated with $\brho_i$, $\{T_{\brho_i}\}$ is a K\"ahler class. 
\end{enumerate} 
By the definition of $\kC_i$, there exist a finite extension  $F_{ij}$ of $K_{ij}$ and reductive representations $\{\delta_{ij}:\pi_1(X)\to {\rm GL}_N(F_{ij})\}_{j=1,\ldots,n_i}$ such that
\begin{enumerate}[label=(\alph*)]
\item For each $i=1,\ldots,k; j=1,\ldots,n_i$, $[\delta_{ij}]\in \kC(F_{ij})$;
\item the Zariski closure of $\delta_{ij}:\pi_1(X)\to \GL_N(F_{ij})$ is contained in $H_i$; 
\item \label{item:same} $[\eta_{ij}]=[\varrho_{ij}]\in M_{\rm B}(X,T_i)(F_{ij})$, where $\eta_{ij}:\pi_1(X)\to T_i(F_{ij})$ is the composition of $\delta_{ij}:\pi_1(X)\to H_i(F_{ij})$ with the group  homomorphism $H_i\to T_i$.  
\end{enumerate} 
Therefore, $\eta_{ij}$ is conjugate to $\varrho_{ij}$ and thus their reduction map coincides.
We have the factorization $X\to S_{\delta_{ij}}\overset{q_{ij}}{\to} S_{\eta_{ij}}$.
By \cref{def:canonical2}, one can see that \begin{align}\label{eq:small}
	 q_{ij}^*T_{\varrho_{ij}}=q_{ij}^*T_{\eta_{ij}}\leq T_{\delta_{ij}}.
\end{align}  
 Consider the family of   representations  $\bm{\delta}:=\{\delta_{ij}:\pi_1(X)\to {\rm GL}_N(F_{ij})\}_{i=1,\ldots,k;j=1,\ldots,n_i}$. 
 Let  $e_i:S\to S_{\kC_i}=S_{\brho_i}$ be the natural map. Note that $e_1\times\cdots\times e_k: S\to S_{\kC_1}\times\cdots\times S_{\kC_k}$ is finite.   
By \Cref{item:dominant,item:kahler},  $\{\sum_{i=1}^{k}e_{i}^*T_{\brho_i}\}$ is K\"ahler on $S=S_\kC$.  
By \eqref{eq:small}, we conclude that  $\{T_{\bm{\delta}}\}$ is K\"ahler on $S_\kC$. %Note that $Z\to S_{\psi^*\btau}$ is birational.  
 According to \Cref{lem:20240409}, it implies that $\{T_{\btau}\}$ is K\"ahler on $S_\kC$.
 \end{proof}

 \begin{lem}\label{lem:202404091}
 Let $f:Z\to X$ be a morphism of smooth projective varieties.
 Let $j:M_{\rm B}(X,N)\to M_{\rm B}(Z,N)$ be the induced map.
 If $\btau$ is full with respect to $\kC$, then $f^*\btau$ is full with respect to $j(\kC)$.
 \end{lem}
 
 \begin{proof}
 Note that $s_{j(\kC)}:Z\to S_{j(\kC)}$ is the Stein factorization of the composite of $Z\to X$ and $X\to S_{\kC}$.
 Similarly, $s_{f^*\btau}:Z\to S_{f^*\btau}$ is the Stein factorization of the composite of $Z\to X$ and $X\to S_{\btau}$.
 Hence $s_{f^*\btau}:Z\to S_{f^*\btau}$ coincides with $s_{j(\kC)}:Z\to S_{j(\kC)}$.
 
 Let $\rho:\pi_1(X)\to \mathrm{GL}_N(\bC)$ be a reductive representation such that $[\rho]\in \kC$.
Then $W_{\btau}\subset W_{\rho}$.
 We have $(f_*)^{-1}(W_{\rho})=W_{f^*\rho}$ (cf. \cref{lem:20240411}) and $(f_*)^{-1}(W_{\btau})=W_{f^*\btau}$ (cf. \cref{cor:20240412}) under the map $f_*:TZ\to TX$.
 This shows that $W_{f^*\btau}$ is minimal.
 \end{proof}

\begin{lem}\label{lem:20240419}
Suppose $\btau$ is full with respect to $\kC$.
 Then $\{T_{\btau}\}$ is a K\"ahler class. 
\end{lem}
 
 \begin{proof}
Let $\Sigma\subset S_{\btau}$ be a closed subvariety.
It suffices to show $\{T_{\btau}\}^{\dim \Sigma}\cdot \Sigma>0$ by \cref{thm:DHP}.
We prove this by Noetherian induction.
Hence we assume that $\{T_{\btau}\}^{\dim \Sigma'}\cdot \Sigma'>0$ for every proper closed subvarieties $\Sigma'\subsetneqq \Sigma$.

We take a subvariety $Z\subset X$ such that $\dim Z=\dim \Sigma$ and $s_{\btau}(Z)=\Sigma$.
Let $Z'\to Z$ be a desingularization.
Let $f:Z'\to X$ be the induced map.
Then by \cref{lem:202404091}, $f^*\btau$ is full with respect to $j(\kC)$, where $j:M_{\rm B}(X,N)\to M_{\rm B}(Z,N)$ is the induced map.
Note that $j(\kC)$ is an absolutely constructible subset of $M_{\rm B}(Z', N)(\bC)$.
By our hypothesis of Noetherian induction and \cref{lem:functorial}, $f^*\btau$ satisfies the assumption of \cref{thm:dichotomy}; i.e., for every proper closed subvariety $\Sigma'\subsetneqq S_{f^*\btau}$, we have $\{T_{f^*\brho}\}^{\dim \Sigma'}\cdot \Sigma'>0$. 
Thus by \cref{lem:20240410}, $\{T_{f^*\btau}\}$ is a K\"ahler class. 
Hence $\{T_{\btau}\}^{\dim \Sigma}\cdot \Sigma>0$ as desired (cf. \cref{lem:functorial}).
This completes the induction step.
 \end{proof}

 \begin{proof}[Proof of \cref{thm:crucial}]
 By \cref{def:reduction ac,lem:simultaneous},  there are non-archimedean local fields $L_1,\ldots,L_\ell$ of characteristic zero and reductive representations $\tau_i:\pi_1(X)\to {\rm GL}_N(L_i)$ such that $[\tau_i]\in \kC(L_i)$ and   $s_\kC:X\to S_\kC$ is the Stein factorization of $(s_{\tau_1},\ldots,s_{\tau_\ell}):X\to S_{\tau_1}\times\cdots\times S_{\tau_\ell}$.  
 Write $\btau:=\{\tau_i\}_{i=1,\ldots,\ell}$.
 By the Noetherian property, we may assume that $W_{\btau}$ is minimal.
 Hence $\btau$ is full.
 By \cref{lem:20240419}, $\{T_{\btau}\}$ is a K\"ahler class. 
  \end{proof}

\subsection{Holomorphic convexity of Galois coverings}
In this subsection we will prove \cref{main2}.  We shall use the notations and results proven in \cref{sec:shmor,thm:Shafarevich1} without recalling the details.
\begin{thm} \label{thm:HC}
	Let $X$ be a smooth projective variety.  Let $\kC$ be an absolutely constructible subset of $M_{\rm B}(X,N)(\bC)$ defined in \cref{def:ac}.  Assume that $\kC$  is   defined on $\bQ$.  	Let $\pi:\widetilde{X}_\kC\to X$ be the covering corresponding to the group $ \cap_{\varrho}\ker \varrho\subset \pi_1(X)$ where $\varrho:\pi_1(X)\to {\rm GL}_N(\bC)$ ranges over all reductive representations  such that  $[\varrho]\in \kC(\bC)$.  Then $\widetilde{X}_\kC$ is holomorphically convex.  In particular, if $\pi_1(X)$ is a subgroup of ${\rm GL}_N(\bC)$ whose Zariski closure is reductive, then $\widetilde{X}_\kC$ is holomorphically convex. 
\end{thm}
\begin{proof}
Let $H:=\cap_{\varrho}\ker \varrho\cap \sigma$ , where $\sigma$ is the $\bC$-VHS defined in \cref{prop:nonrigid} and $\varrho:\pi_1(X)\to {\rm GL}_N(\bC)$ ranges over all reductive representation   such that $[\varrho]\in \kC $.   Denote by  $\widetilde{X}_H:=\widetilde{X}/H$.  Let $\sD$ be the period domain associated to the $\bC$-VHS $\sigma$ defined in \cref{prop:nonrigid} and let $p: \widetilde{X}_H\to \sD$ be the period mapping.  By \eqref{eq:inclusion ker},   $H=\cap_{\varrho}\ker \varrho $, where $\varrho:\pi_1(X)\to {\rm GL}_N(\bC)$ ranges over all reductive representation   such that $[\varrho]\in \kC $.   Therefore, $\widetilde{X}_\kC=\widetilde{X}_H$. 

	Consider the product 
	$$
	\Psi=s_\kC\circ\pi_H\times p: \widetilde{X}_H\to S_\kC\times \sD
	$$
	where $p:\widetilde{X}_H\to \sD$ is the period mapping of $\sigma$.  Recall that $\Psi$ factors through a proper surjective fibration ${\rm sh}_H:\widetilde{X}_H\to \widetilde{S}_H$. Moreover, there is a properly discontinuous action of $\pi_1(X)/H$ on $\widetilde{S}_H$ such that ${\rm sh}_H$ is   equivariant with respect to this action.   
 Write $g:\widetilde{S}_H\to S_\kC\times \sD$ to be the induced holomorphic map.  Denote by $\phi: \widetilde{S}_H\to \sD$ the composition of $g$ and the projection map $S_\kC\times \sD\to \sD$.      Since the period mapping $p$ is horizontal, and ${\rm sh}_H$ is surjective, it follows that $\phi$ is also  horizontal.    
 
Recall that in \cref{claim:free} we prove that 
		there is a finite index normal subgroup $N$ of  $\pi_1(X)/H$ and a homomorphism $\nu:N\to {\rm Aut}(\widetilde{S}_H)$ such that   ${\rm sh}_H:\widetilde{X}_H\to \widetilde{S}_H$ is $\nu$-equivariant and $\nu(N)$ acts on $\widetilde{S}_H$ properly discontinuous and without   fixed point.  
 Let $Y:=\widetilde{X}_H/N$. Moreover, $c: Y\to X$ is a finite Galois \'etale cover and $N$ gives rise to a proper surjective fibration
 $
Y\to \widetilde{S}_H/\nu(N) 
$  between compact normal complex spaces. Write $W:= \widetilde{S}_H/\nu(N)$. Then $\widetilde{S}_H\to W$ is a topological Galois unramified covering.  
%\begin{claim}
% Let $Y\to S_{\kC}$ be  the composition of $Y\to X$  and  $s_\kC:X\to S_{\kC}$. It then factors through  a morphism $f: W\to S_{\kC}$ induced by the reduction map $s_{\kC}:X\to S_{\kC}$.
%\end{claim} 
%\begin{proof}
%Since each fiber   of $Y\to W$ is identified with some fiber $F$ of $\widetilde{X}_H\to \widetilde{S}_H$, which is connected.  Note that  $F$ is contracted by $s_{\kC}$. Therefore,  for the composition of $Y\to X$  and  $s_\kC:X\to S_{\kC}$, it factors through  a morphism $W\to S_{\kC}$. 
%\end{proof}
Recall that the  {canonical bundle $K_{\sD}$ on the period domain $\sD$} can be endowed with a $G_0$-invariant smooth metric $h_{\sD}$ whose curvature is strictly positive-definite in the horizontal direction.  As $\phi:\widetilde{S}_H\to \sD$ is $\nu(N)$-equivariant, it follows that $\phi^*K_{\sD}$ descends to a line bundle on the quotient $W:=\widetilde{S}_H/\nu(N)$, denoted by $L_G$.      The smooth metric $h_{\sD}$ induces a smooth metric on $L_G$  whose curvature form is denoted by $T$. Let $x\in W$ be a smooth point of $W$ and let $v\in T_{\widetilde{S}_H,x}$. Then $|v|_\omega^2>0$   if $d\phi(v)\neq 0$.    

We fix a reference point $x_0$ on $\widetilde{S}_H$.  Let $\sS$ be the symmetric space associated with $\sD$ endowed with the natural metric $d_{\sS}$.  Then there exists natural quotient map $\sD\to \sS$.  It induces a  $\nu(N)$-equivariant pluri-harmonic mapping $u:\widetilde{S}_H\to \sS$.  Define $\phi_0:=2d^2_{\sS}(u(x),u(x_0))$.  By \cite[Proposition  3.3.2]{Eys04}, we have
\begin{align} \label{eq:below}
	 \hess \phi_0\geq \omega=q^*T,
\end{align} 
where   $q:\widetilde{S}_H\to \widetilde{S}_H/\nu(N)$  denotes the quotient map.

%Since $\phi:\widetilde{S}_H\to \sD$ is horizontal, we can pull back the canonical metric on $\sD$ to obtain a  semi-positive $(1,1)$-form $\omega$ on the     complex normal space $\widetilde{S}_H$. Note that $\omega$ is $\pi_1(X)/H$-invariant since $\phi$ is  $\pi_1(X)/H$-invariant.   Since $\nu(N)$ acts on  $\widetilde{S}_H$ properly continuous and freely,   $\omega$ induces a   semi-positive $(1,1)$-form  $T$   over $W:= \widetilde{S}_H/\nu(N) $ such that $q^*T=\omega$ where $q:\widetilde{S}_H\to \widetilde{S}_H/\nu(N)$  is the quotient map.    Let $F$  be the normalization of any irreducible component of $f^{-1}(t)$.  Recall that   the period mapping induces a map $p_F:F\to \sD/\Gamma$ which is finite by  \cref{claim:discrete}. Therefore, $\{T|_{F}\}$  is K\"ahler. 

We now apply \cref{thm:crucial} to find  a family of   representations  $\btau:=\{\tau_i:\pi_1(X)\to {\rm GL}_N(K_i)\}_{i=1,\ldots,m}$ where $K_i$ are non-archimedean local fields   such that
\begin{itemize}
	\item For each $i=1,\ldots,m$, $[\tau_i]\in \kC(K_i)$; 
	\item  The reduction map $s_{\btau}:X\to S_{\btau}$ of $\btau$ coincides with $s_{\kC}$.
	\item For the canonical current $T_{\btau}$ defined over $S_{\kC}$, $\{T_{\btau}\}$ is a K\"ahler class. 
\end{itemize}  
Consider 
\begin{equation*}
	\begin{tikzcd}
\widetilde{X}_H\arrow[r]\arrow[d, "{\rm sh}_H"] &		Y\arrow[r, "c"]\arrow[d] & X \arrow[d, "s_{\kC}"]\arrow[dr, "s_\btau"]\arrow[drr,bend left=20,"s_{\tau_i}"]& &\\ 
\widetilde{S}_H	\arrow[rrrr, "r_i"', bend right=30]\arrow[rr, bend right=20, "r"']\arrow[r, "q"]	& W \arrow[r, "f"]& S_\kC \arrow[r, equal] & S_\btau \arrow[r,"e_i"]& S_{\tau_i}
	\end{tikzcd}
\end{equation*}
Note that $p$ is a finite surjective morphism. %Therefore, $\{(e_i\circ p)^*T_{\btau}\}$ is also  K\"ahler by \cref{thm:DHP}. 
 
We fix a reference point $x_0$ on $\widetilde{S}_H$. For each $i=1,\ldots,m$, let $u_i:\widetilde{X}_H\to \Delta(\GL_{N})_{K_i}$ be the $\tau_i$-equivariant harmonic mapping from $\widetilde{X}_H$ to the Bruhat-Tits building of $\GL_{N}(K_i)$ whose existence was ensured by a  theorem of Gromov-Schoen \cite{GS92}. Then the function $\tilde{\phi}_i(x):=2d_i^2(u_i(x),u_i(x_0))$ defined over $\widetilde{X}_H$ is locally Lipschitz,   where $d_i:\Delta(\GL_{N})_{K_i}\times \Delta(\GL_{N})_{K_i}\to \bR_{\geq 0}$ is the distance function on the Bruhat-Tits building. By \cref{prop:Eys}, it induces a  continuous psh functions $\{\phi_i:\widetilde{S}_H\to \bR_{\geq 0}\}_{i=1,\ldots,m}$ such that  $\hess \phi_i\geq r_i^*T_{\tau_i}$ 
 for each $i$. By the definition of $T_{\btau}$, we have
\begin{align}\label{eq:below2}
	 \hess \sum_{i=1}^{m}\phi_i\geq r^*T_{\btau}.
\end{align} 
 Therefore,  putting \eqref{eq:below} and  \eqref{eq:below2} together we obtain
\begin{align}\label{eq:below3}
\hess \sum_{i=0}^{m}\phi_i\geq q^*(f^*T_\btau+T).
\end{align}
As $f$ is a finite surjective morphism, $\{f^*T_\btau\}$ is also K\"ahler by \cref{thm:DHP}. 

By \cref{claim:discrete}, we know that $g:\widetilde{S}_H\to S_\kC\times \sD$ has discrete fibers. Since $T$ is induced by the curvature form of $(K_\sD, h_{\sD})$, and $\phi:\widetilde{S}_H\to \sD$ is horizontal, we can prove that    for every  irreducible  positive dimensional closed subvariety $Z$ of $W$, $f^*T_\btau+T$ is strictly positive at general smooth points of $Z$. Therefore,   $$\{f^*T_\btau+T\}^{\dim Z}\cdot Z= \int_{Z}(f^*T_\btau+T)^{\dim Z}>0.$$ Recall that $W$ is projective by the proof of \cref{claim:projective}.     We utilize \cref{thm:DHP} to conclude that $\{f^*T_\btau+T\}$ is K\"ahler.

Given that $\widetilde{S}_H\to W$ represents a topological Galois unramified cover, we can apply \cref{prop:stein} in conjunction with \eqref{eq:below3} to deduce that $\widetilde{S}_H$ is a Stein manifold.  Furthermore, since $\widetilde{X}_H\to \widetilde{S}_H$ is a proper surjective holomorphic fibration,  the holomorphic convexity of $\widetilde{X}_H$ follows from the Cartan-Remmert theorem. Ultimately, the theorem is established by noting that $\widetilde{X}_H=\widetilde{X}_\kC$.
 \end{proof} 

\subsection{Steiness of infinite Galois coverings}
  We shall use the notations in the proof of \cref{thm:HC} without recalling their definitions.  %We first prove a lemma which enables us to deform the large representation which is  reductive and still large, and moreover defined over a number field. 
\begin{thm}\label{thm:Stein}
	Let $X$ be a smooth projective variety. Consider an absolutely constructible subset $\kC$ of $M_{\rm B}(X,{\rm GL}_N(\bC))$ as defined in \cref{def:ac}. We further assume that $\kC$ is  defined over $\bQ$. If $\kC$ is considered to be \emph{large}, meaning that for any closed subvariety $Z$ of $X$, there exists a reductive representation $\varrho:\pi_1(X)\to {\rm GL}_N(\bC)$ such that $[\varrho]\in \kC$ and $\varrho({\rm Im}[\pi_1(Z^{\rm norm})\to \pi_1(X)])$ is infinite, then all intermediate coverings between $\widetilde{X}$ and $\widetilde{X}_\kC$ of $X$ are Stein manifolds. 
\end{thm} 
\begin{proof} 
Note that  ${\rm sh}_H:\widetilde{X}_H\to \widetilde{S}_H$ is a proper holomorphic surjective fibration. 
\begin{claim}\label{claim:biho}
	${\rm sh}_H$ is biholomorphic.
\end{claim}
\begin{proof}
	Assume that there exists a positive-dimensional compact subvariety $Z$ of $\widetilde{X}_H$ which is contained in some fiber of ${\rm sh}_H$. Consider $W:=\pi_H(Z)$ which is a compact positive-dimensional irreducible subvariety of $X$.  Therefore, ${\rm Im}[\pi_1(Z^{\rm norm})\to \pi_1(W^{\rm norm})]$ is a finite index subgroup of $\pi_1(W^{\rm norm})$. By the definition of $\widetilde{X}_H$, for any reductive $\varrho:\pi_1(X)\to\GL_{N}(\bC)$ with $[\varrho]\in \kC(\bC)$, we have $\varrho({\rm Im}[\pi_1(Z^{\rm norm})\to \pi_1(X)])=\{1\}$. Therefore, $\varrho({\rm Im}[\pi_1(W^{\rm norm})\to \pi_1(X)])=\{1\}$ is finite. This contradicts with out assumption that $\kC$ is large. Hence,  ${\rm sh}_H$ is  a one-to-one proper holomorphic map of complex normal   spaces.  Consequently, it is biholomorphic. 
\end{proof}
By the proof of \cref{thm:HC}, there exist
\begin{itemize}
	\item  a topological Galois unramified covering $q: \widetilde{X}_H=\widetilde{S}_H\to W$, where $W$ is a projective normal variety;
	\item a positive $(1,1)$-current with continuous potential $f^*T_\btau+T$ over $W$ such that $\{f^*T_\btau+T\}$ is K\"ahler;
	\item    a continuous semi-positive plurisubharmonic function $\sum_{i=0}^{m}\phi_i$ on $\widetilde{S}_H$ such that we have 
	\begin{align}\label{eq:below4}
		\hess \sum_{i=0}^{m}\phi_i\geq q^*(f^*T_\btau+T). 
	\end{align}  
 \end{itemize} 
Let $p:\widetilde{X}'\to \widetilde{X}_H$ be the intermediate Galois covering of $X$ between $\widetilde{X}\to \widetilde{X}_H$. By \eqref{eq:below} we have
 \begin{align}\label{eq:below5}
	\hess \sum_{i=0}^{m}p^*\phi_i\geq (q\circ p)^*(f^*T_\btau+T). 
\end{align}   
We apply   \cref{prop:stein}  to conclude that $\widetilde{X}'$ is Stein. 
\end{proof}

 \appendix 
 \section{Shafarevich conjecture for projective normal varieties}
 \begin{center}
 	  {Ya Deng, Ludmil Katzarkov\footnote{E-mail:  \textbf{l.katzarkov@miami.edu},  University of Miami, Coral Gables, FL, USA; Institute of Mathematics
 	  		and Informatics, Bulgarian Academy of Sciences, Sofia, Bulgaria; NRU HSE, Moscow, Russia} \& Katsutoshi Yamanoi}
 \end{center}

\medspace

In this appendix, we aim to extend   \cref{thm:HC,thm:Stein} to include singular normal varieties, and thus completing the proofs of \cref{main2,main}.
\subsection{Absolutely constructible subset (II)}
Let $X$ be a projective normal variety.   
Following the recent work of Lerer \cite{Ler22}, we can also define absolutely constructible subsets in the character variety $M_{\rm B}(X, N):=M_{\rm B}(\pi_1(X), \GL_{N})$. 
\begin{dfn}\label{def:ac2}
	Let $X$ be  a normal projective variety, $\mu:Y\to X$ be a resolution of singularities, and $\iota:M_{\rm B}(X,N)\hookrightarrow M_{\rm B}(Y,N)$ be the embedding.  A subset $\kC\subset M_{\rm B}(X,N)(\bC)$ is called \emph{absolutely constructible} if $\iota(\kC)$ is an absolutely constructible subset of $M_{\rm B}(Y,N)$ in the sense of \cref{def:ac}.
\end{dfn}
Note that the above definition does not depend on the choice of the resolution of singularities (cf. \cite[Lemma 2.7]{Ler22}). Moreover, we have the following result.
\begin{proposition}[\protecting{\cite[Proposition 2.8]{Ler22}}]\label{prop:Lerer}
	Let $X$ be  a normal projective variety. Then $ M_{\rm B}(X,N)$   is absolutely constructible in the sense of \cref{def:ac2}. 
\end{proposition} 
This result holds significant importance, as it provides a fundamental example of absolutely constructible subsets for projective normal varieties.  It is worth noting that in \cite[Proposition 2.8]{Ler22}, it is explicitly stated that $\iota(M_{\rm B}(X,N))$ is $U(1)$-invariant, with $\iota$ defined in \cref{def:ac2}. However, it should be emphasized that the proof can be easily adapted to show $\bC^*$-invariance, similar to the approach used in the proof of \cref{lem:C*}.
\subsection{Reductive Shafarevich conjecture for   projective normal varieties}
\begin{thm} \label{main4}
	Let $Y$ be a projective normal variety.  Let $\kC$ be an absolutely constructible subset of $M_{\rm B}(Y,N)(\bC)$,   defined on $\bQ$ (e.g. $\kC=M_{\rm B}(Y,N)$).  	Consider the covering $\pi:\widetilde{Y}_\kC\to Y$ corresponding to the subgroup  $\cap_{\varrho}\ker \varrho$ of $\pi_1(Y)$,  where $\varrho:\pi_1(Y)\to \GL_N(\bC)$ ranges over all reductive representations such that  $[\varrho]\in \kC$.  Then the complex space $\widetilde{Y}_\kC$ is holomorphically convex. In particular, 
	\begin{itemize}
		\item The covering corresponding to the intersection of the kernels of all reductive representations of 
		$\pi_1(Y)$ in $\GL_N(\bC)$ is holomorphically convex; 
		\item  if $\pi_1(Y)$ is a subgroup of ${\rm GL}_N(\bC)$ whose Zariski closure is reductive, then the universal covering of $Y$ is holomorphically convex.
	\end{itemize}
\end{thm}
\begin{proof}
	Let $\mu:X\to Y$ be any  desingularization.   Let $j:M_{\rm B}(Y,N)\hookrightarrow  M_{\rm B}(X,N)$ the closed immersion induced by $\mu$, which is a morphism of affine $\bQ$-schemes of finite type . Then by \cref{def:ac2}, $j(\kC)$ is an absolutely constructible in the sense of \cref{def:ac}.   Since $\kC$ is  defined on $\bQ$, so   is $j(\kC)$. We shall use the notations in \cref{thm:Shafarevich1}. Let $\widetilde{X}_H$ be the covering associated with the subgroup  $H:=\cap_{\varrho}\ker \varrho$ of $\pi_1(X)$  where $\varrho:\pi_1(X)\to \GL_N(\bC)$ ranges over all reductive representations such that  $[\varrho]\in j(\kC)(\bC)$.   In other words,  $H:=\cap_{\tau}\ker \mu^*\tau$ where $\tau:\pi_1(Y)\to \GL_N(\bC)$ ranges over all reductive representations  such that  $[\tau]\in \kC(\bC)$.   Denote by $H_0:=\cap_{\tau}\ker  \tau$ where $\tau:\pi_1(Y)\to \GL_N(\bC)$ ranges over all reductive representations  such that  $[\varrho]\in \kC(\bC)$.  Therefore, $H=(\mu_*)^{-1}(H_0)$, where $\mu_*:\pi_1(X)\to \pi_1(Y)$ is a surjective homeomorphism as $Y$ is normal.  Therefore, the natural homeomorphism $\pi_1(X)/H\to \pi_1(Y)/H_0$ is an isomorphism.  Then $\widetilde{X}_\kC=\widetilde{X}/H$ and $\widetilde{Y}_H:=\widetilde{Y}/H_0$ where $\widetilde{X}$ (resp. $\widetilde{Y}$) is the universal covering of $X$ (resp. $X$). It induces a lift $p:\widetilde{X}_H\to \widetilde{Y}_\kC$ such that  
	\begin{equation*}
		\begin{tikzcd}
			\widetilde{X}_H\arrow[r,"\pi_H"] \arrow[d,"p"]& X\arrow[d, "\mu"]\\
			\widetilde{Y}_\kC \arrow[r, "\pi"] & Y
		\end{tikzcd}
	\end{equation*}
	\begin{claim}\label{claim:base change}
		$p:\widetilde{X}_H\to \widetilde{Y}_\kC$ is a proper  surjective holomorphic fibration with connected fibers. 
	\end{claim}
	\begin{proof}
		Note that ${\rm Aut}(\widetilde{X}_H/X)=\pi_1(X)/H\simeq \pi_1(Y)/H_0={\rm Aut}(\widetilde{Y}_\kC/Y)$.  
		Therefore, $\widetilde{X}_H$ is the base change $\widetilde{Y}_\kC\times_Y X$.   Note that each fiber of $\mu$ is connected as $Y$ is normal. It follows that each fiber of $p$ is connected. 	The claim is proved. 
	\end{proof}
	By    \cref{thm:Shafarevich1}, we know that there exist a proper surjective holomorphic fibration ${\rm sh}_H:\widetilde{X}_H\to \widetilde{S}_H$ such that  
	$\widetilde{S}_H$ is a Stein space. Therefore, for each connected compact subvariety $Z\subset  \widetilde{X}_H$, ${\rm sh}_H(Z)$ is a point.  By \cref{claim:base change}, it follows that each fiber of $p$ is compact and connected, and thus is contracted by ${\rm sh}_H$. Therefore,  ${\rm sh}_H$ factors through a  proper surjective fibration  $f: \widetilde{X}_\kC\to \widetilde{S}_H$:
	\begin{equation*}
\begin{tikzcd}
	\widetilde{X}_H\arrow[dr,"{\rm sh}_H"] \arrow[d,"p"]&  \\
	\widetilde{Y}_\kC \arrow[r, "f"] & \widetilde{S}_H
\end{tikzcd}
	\end{equation*} Therefore, $f$ is  a proper surjective holomorphic fibration  over a Stein space. By the Cartan-Remmert theorem, $\widetilde{Y}_\kC$ is holomorphically convex. 

 If we define $\kC$ as $M_{\rm B}(Y,N)$, then according to  \Cref{prop:Lerer}, $\kC$ is also absolutely constructible. As a result, the last two claims can be deduced. Thus, the theorem is proven.
\end{proof} 
\begin{thm} \label{thm:Stein singular}
	Let $Y$ be a projective normal variety.  Let $\kC$ be an absolutely constructible subset  of $M_{\rm B}(Y,N)(\bC)$,  defined on $\bQ$  (e.g. $\kC=M_{\rm B}(Y,N)$).  Let $\kC(\bC)$ be \emph{large} in the sense that for any closed positive dimensional subvariety $Z$ of $Y$, there exists a reductive representation $\varrho:\pi_1(Y)\to {\rm GL}_N(\bC)$	 such that $[\varrho]\in \kC(\bC)$ and $\varrho({\rm Im}[\pi_1(Z^{\rm norm})\to \pi_1(Y)])$ is infinite.  Then    all intermediate Galois coverings of $Y$ between $\widetilde{Y}$ and $\widetilde{Y}_\kC$  are Stein spaces. Here $\widetilde{Y}$ denotes the universal covering of $Y$.
\end{thm}
\begin{proof}
	%Let $\mu:X\to Y$ be any  desingularization.   Let $j:M_{\rm B}(Y,N)\hookrightarrow  M_{\rm B}(X,N)$ the closed immersion induced by $\mu$, which is a morphism of affine $\bQ$-schemes of finite type . Then by \cref{def:ac2}, $j(\kC)$ is an absolutely constructible in the sense of \cref{def:ac}.   Since $\kC$ is  defined on $\bQ$, so   is $j(\kC)$.  	
Let $\mu:X\to Y$ be any  desingularization. In the following, we will use the same notations as in the proof of \cref{main4} without explicitly recalling their definitions.    Recall that we have constructed  three proper surjective holomorphic fibrations $p$, $f$, and ${\rm sh}_H$ satisfying the following commutative diagram:
	\begin{equation*}
	\begin{tikzcd}
		&	\widetilde{X}_H\arrow[r,"\pi_H"] \arrow[d,"p"]\arrow[ld,"{\rm sh}_H"']& X\arrow[d, "\mu"]  \\
		\widetilde{S}_H&	\widetilde{Y}_\kC \arrow[l, "f"] \arrow[r,"\pi"]&  Y
	\end{tikzcd}
\end{equation*} 
\begin{claim}\label{claim:biho2}
	$f:\widetilde{Y}_\kC\to \widetilde{S}_H$ is a biholomorphism.
\end{claim}
The proof follows a similar argument to that of \cref{claim:biho}. For the sake of completeness, we will provide it here. 
\begin{proof}[Proof of \cref{claim:biho2}]
As each fibers of $f$ is compact and connected, it suffices to prove that there are no compact positive dimensional subvarieties $Z$ of $\widetilde{Y}_\kC$ such that $f(Z)$ is a point. Let us assume, for the sake of contradiction, that such a $Z$ exists. Consider $W:=\pi(Z)$ which is a compact positive-dimensional irreducible subvariety of $Y$.  Therefore, ${\rm Im}[\pi_1(Z^{\rm norm})\to \pi_1(W^{\rm norm})]$ is a finite index subgroup of $\pi_1(W^{\rm norm})$. By the definition of $\widetilde{Y}_\kC$, for any reductive $\varrho:\pi_1(Y)\to\GL_{N}(\bC)$ with $[\varrho]\in \kC(\bC)$, we have $\varrho({\rm Im}[\pi_1(Z^{\rm norm})\to \pi_1(X)])=\{1\}$. Therefore, $\varrho({\rm Im}[\pi_1(W^{\rm norm})\to \pi_1(Y)])$ is finite. This contradicts with out assumption that $\kC$ is large. Hence, $f$ is a  one-to-one proper holomorphic map of complex normal spaces. Consequently, it is biholomorphic. 
\end{proof}
The rest of the  proof is same as in \cref{thm:Stein}. 
By the proof of \cref{thm:HC}, there exist
\begin{itemize}
	\item  a topological Galois unramified covering $q:\widetilde{S}_H\to W$, where $W$ is a projective normal variety;
	\item a positive closed $(1,1)$-current with continuous potential $T_0$ over $W$ such that $\{T_0\}$ is K\"ahler;
	\item    a continuous semi-positive plurisubharmonic function $\phi$ on $\widetilde{S}_H$ such that we have 
	\begin{align} \label{eq:below6}
		\hess \phi\geq q^*T_0. 
	\end{align}  
\end{itemize} 
By \cref{claim:biho2}, $\widetilde{Y}_\kC$ can be identified with  $\widetilde{S}_H$. 
Let $p:\widetilde{Y}'\to \widetilde{Y}_\kC$ be the intermediate Galois covering of $Y$ between $\widetilde{Y}\to \widetilde{Y}_\kC$. By \eqref{eq:below6} we have
\begin{align} 
	\hess  p^*\phi\geq (q\circ p)^*T_0.
\end{align}   
We apply   \cref{prop:stein}  to conclude that $\widetilde{Y}'$ is Stein.  %
%If we define $\kC$ to be $M_{\rm B}(Y,N)$, then according to \cref{prop:Lerer}, $\kC$ is also absolutely constructible. Consequently, the last two claims hold. We accomplish the proof of the theorem.
\end{proof} 
   
%   \bibliography{biblio}
%  \bibliographystyle{ssmfalpha}

 \end{document}